\documentclass[12pt,twoside,a4paper]{amsart}
\usepackage{amssymb,amscd,amsxtra,calc,mathrsfs}
\usepackage{amsmath}
\usepackage[all]{xy}


\allowdisplaybreaks[1]

\title[Coupled stability thresholds]{On the coupled stability thresholds of 
graded linear series}
\author{Kento Fujita} 
\date{\today}
\address{Department of Mathematics, Graduate School of Science, Osaka University, 
Toyonaka, Osaka 560-0043, Japan}
\email{fujita@math.sci.osaka-u.ac.jp}

\newcommand{\pr}{\mathbb{P}}
\newcommand{\N}{\mathbb{N}}
\newcommand{\Z}{\mathbb{Z}}
\newcommand{\Q}{\mathbb{Q}}
\newcommand{\R}{\mathbb{R}}
\newcommand{\C}{\mathbb{C}}

\newcommand{\D}{\mathbb{D}}

\newcommand{\B}{{\bf B}}

\newcommand{\BIG}{\operatorname{Big}}

\newcommand{\Supp}{\operatorname{Supp}}

\newcommand{\Spec}{\operatorname{Spec}}

\newcommand{\CaCl}{\operatorname{CaCl}}
\newcommand{\DIV}{\operatorname{div}}
\newcommand{\Hom}{\operatorname{Hom}}

\newcommand{\id}{\operatorname{id}}

\newcommand{\lct}{\operatorname{lct}}

\newcommand{\ord}{\operatorname{ord}}

\newcommand{\vol}{\operatorname{vol}}
\newcommand{\Image}{\operatorname{Image}}

\newcommand{\mult}{\operatorname{mult}}

\newcommand{\interior}{\operatorname{int}}
\newcommand{\Img}{\operatorname{Image}}
\newcommand{\Ann}{\operatorname{Ann}}

\newcommand{\Biggu}{\operatorname{Big}}
\newcommand{\Cone}{\operatorname{Cone}}
\newcommand{\Conv}{\operatorname{Conv}}
\newcommand{\Gr}{\operatorname{Gr}}

\newcommand{\sC}{\mathcal{C}}

\newcommand{\sO}{\mathcal{O}}

\newcommand{\sF}{\mathcal{F}}
\newcommand{\sG}{\mathcal{G}}

\newcommand{\sS}{\mathcal{S}}

\newcommand{\dm}{\mathfrak{m}}

\newtheorem{thm}{Theorem}[section]
\newtheorem{lemma}[thm]{Lemma}
\newtheorem{proposition}[thm]{Proposition}
\newtheorem{corollary}[thm]{Corollary}
\newtheorem{claim}[thm]{Claim}

\theoremstyle{definition}
\newtheorem{definition}[thm]{Definition}
\newtheorem{remark}[thm]{Remark}

\newtheorem{example}[thm]{Example}
\newtheorem*{ack}{Acknowledgments}

\begin{document}

\maketitle 

\begin{abstract}
In this paper, we see several basic properties of graded linear series. 
We firstly see that, if a graded linear series contains an ample series, 
then so are the pullbacks of the system under birational morphisms. 
Using this proposition, we define the refinements of graded linear series 
with respects to primitive flags. Moreover, we give several formulas to compute 
the $S$-invariant of those refinements. 
Secondly, we introduce the notion of coupled stability thresholds for 
graded linear series, which is a generalization of the notion introduced by 
Rubinstein--Tian--Zhang. We see that, 
over the interior of the support for finite numbers of graded linear series containing 
an ample series, the coupled stability threshold function can be uniquely extended 
continuously, which generalizes the work by Kewei Zhang. 
Thirdly, we get a product-type formula for coupled stability thresholds, which 
generalizes the work of Zhuang. 
Fourthly, we see Abban--Zhuang's type formulas for estimating local 
coupled stability thresholds. 
\end{abstract}

\setcounter{tocdepth}{1}
\tableofcontents

\section{Introduction}\label{intro_section}

For a Fano manifold $X$ over the complex number field $\C$, it has been known 
that the existence of K\"ahler--Einstein metrics on $X$ is equivalent to the 
K-polystability of $X$. We can check K-polystability of $X$ by 
estimating its \emph{stability threshold} $\delta(X):=\delta(X; -K_X)$ 
(see \cite{FO, BJ}). 

Recently, based on the earlier work by Hultgren--Witt Nystr\"om 
\cite{HWN}, 
Rubinstein--Tian--Zhang \cite{RTZ} and Kewei Zhang 
\cite{kewei2} established its coupled version: Let $X$ be a Fano manifold over $\C$, 
let $L_1,\dots,L_k$ be ample $\Q$-divisors on $X$ satisfying 
$-K_X=\sum_{i=1}^k L_i$. In \cite[\S A]{RTZ}, the authors introduced 
the \emph{coupled stability threshold} $\delta\left(X; \left\{L_i\right\}_{i=1}^k\right)$ 
(see \S \ref{delta_section}). 
By \cite[Remark 5.3]{kewei2} (see also \cite[\S A.3]{hashimoto}), 
the author showed the existence of coupled K\"ahler--Einstein metrics 
provided that $\delta\left(X; \left\{L_i\right\}_{i=1}^k\right)>1$. 
Moreover, by \cite[Corollary A.15]{kewei2}, if $X$ is toric, then 
the existence of coupled K\"ahler--Einstein metrics is equivalent to the condition 
$\delta\left(X; \left\{L_i\right\}_{i=1}^k\right)=1$. 
The coupled stability threshold $\delta\left(X;\{L_i\}_{i=1}^k\right)$ is a 
natural generalization of the stability threshold $\delta(X; L)$ 
(for big $\Q$-divisors $L$) in \cite{FO, BJ}. 
However, systematic studies for coupled stability thresholds are not established 
so much yet. 

On the other hand, as in \cite{AZ}, it is natural and powerful for the computation 
that generalizing the notion of stability thresholds not only for big $\Q$-divisors 
but also \emph{graded linear series} $V_{\vec{\bullet}}$ which has bounded support 
and contains an ample series. In fact, in \cite{FANO, r3d28}, the authors 
got explicit formulas in order to estimate the values $\delta(Y; -K_Y)$ for 
smooth Fano threefolds $Y$, by focusing on the stability thresholds 
$\delta\left(X; V_{\vec{\bullet}}\right)$ with $X$ subvarieties of $Y$ and 
$V_{\vec{\bullet}}$ certain graded linear series on $X$. 

In this paper, we introduce the notion of the \emph{coupled stability threshold 
$\delta\left(X, B; \left\{V^i_{\vec{\bullet}}\right\}_{i=1}^k\right)$
for a series of} (the Veronese equivalence class of) \emph{graded linear series} 
$\left\{V^i_{\vec{\bullet}}\right\}_{i=1}^k$ (which have bounded supports and 
contain ample series) 
\emph{over a projective klt pair} $(X, B)$. The notion is very natural, since 
this notion is a common generalizations of the above notions
$\delta\left(X; \left\{L_i\right\}_{i=1}^k\right)$ and 
$\delta\left(X; V_{\vec{\bullet}}\right)$. Moreover, we see various basic properties 
related with the stability thresholds. 
For example, one of the purpose of the paper is to give several formulas to 
estimate or to compute the \emph{$S$-invariant} of specific graded linear series, 
which is crucial to estimate the stability thresholds. 
The concept of the Veronese equivalence classes for graded linear series was 
systematically treated in \cite[\S 3.1]{r3d28}. The concept is very natural 
in order to consider important invariants including the $S$-invariant. 

We quickly state important results of the paper. Firstly, we showed that the basic 
properties of graded linear series are stable under birational base change:

\begin{proposition}[{see 
Proposition \ref{pullback_proposition}}]\label{intro-pullback_proposition}
Let us consider a birational morphism $\sigma\colon X'\to X$ 
between (possibly non-normal) projective 
varieties, and let $V_{\vec{\bullet}}$ be the Veronese equivalence class of graded 
linear series on $X$ (see Definition \ref{gr_definition}). 
Then $V_{\vec{\bullet}}$ contains an ample series (resp., 
has bounded support) if and only if $\sigma^*V_{\vec{\bullet}}$ is so. 
\end{proposition}

Although the above proposition is technical, we can introduce the notion of 
refinement $V_{\vec{\bullet}}^{\left(
Y_1\triangleright\cdots\triangleright Y_j\right)}$ of graded 
linear series $V_{\vec{\bullet}}$ for \emph{primitive flags} 
$Y_1\triangleright\cdots\triangleright Y_j$ over $X$ in a good way 
(see Definition \ref{refinement-primitive_definition}). 
From this viewpoint, the value 
\[
S\left(V_{\vec{\bullet}};Y_1\triangleright\cdots\triangleright Y_j\right)
:=S\left(V_{\vec{\bullet}}^{\left(Y_1\triangleright\cdots\triangleright Y_{j-1}\right)};
Y_j\right)
\]
naturally appeared many times in \cite{FANO, r3d28} etc.\ in order to apply 
Abban--Zhuang's method \cite{AZ}. Thus, we are interested in computing the value 
in various situations, especially when $V_{\vec{\bullet}}$ is the complete linear 
system $H^0(\bullet L)$ of a big $\Q$-Cartier $\Q$-divisor $L$ on $X$. In this case, 
the value $S\left(V_{\vec{\bullet}};Y_1\triangleright\cdots\triangleright Y_j\right)$ 
is denoted by $S\left(L;Y_1\triangleright\cdots\triangleright Y_j\right)$.

\begin{thm}[{see Theorems \ref{toric-S_thm} and \ref{adequate_thm} 
in detail}]\label{intro}
Assume either 
\begin{itemize}
\item
$X$ is a projective $\Q$-factorial toric variety and a primitive flag is torus invariant, 
or
\item
the primitive flag admits an adequate dominant with respects to $L$ (see 
Definition \ref{adequate_definition} for the definition).
\end{itemize}
Then there is an explicit formula to compute the value 
$S\left(L; Y_1\triangleright\cdots\triangleright Y_j\right)$. 
\end{thm}

We also define the coupled global log canonical thresholds 
$\delta\left(X,B; \left\{c_i\cdot V_{\vec{\bullet}}^i\right\}_{i=1}^k\right)$ 
and the coupled stability thresholds $\delta\left(X,B; \left\{
c_i\cdot V_{\vec{\bullet}}^i\right\}_{i=1}^k\right)$
of graded linear series on projective klt pairs $(X,B)$ with $c_1,\dots,c_k\in\R_{>0}$. 
We show that both the coupled global log canonical thresholds and the coupled 
stability thresholds behaves well under changing slopes, which are generalizations 
of the result of Dervan \cite{dervan} and Kewei Zhang \cite{kewei}.

\begin{thm}[{=Corollary \ref{cont_corollary}. See Theorem \ref{cont_thm} for 
more general settings}]\label{intro_cont_thm}
For a projective klt pair $(X, B)$, the functions 
\begin{eqnarray*}
\alpha\colon\Biggu(X)_\Q^k&\to&\R_{>0}\\
(L_1,\dots,L_k)&\mapsto&\alpha\left(X,B;\left\{L_i\right\}_{i=1}^k\right), \\
\delta\colon\Biggu(X)_\Q^k&\to&\R_{>0}\\
(L_1,\dots,L_k)&\mapsto&\delta\left(X,B;\left\{L_i\right\}_{i=1}^k\right), 
\end{eqnarray*}
uniquely extend to continuous functions
\[
\alpha\colon\Biggu(X)^k\to\R_{>0}, \quad
\delta\colon\Biggu(X)^k\to\R_{>0}. 
\] 
\end{thm}

We can also show the Zhuang's product formula \cite{zhuang} for coupled settings. 

\begin{thm}[{=Theorem \ref{thm_zhuang}}]\label{intro_thm_zhuang}
Let $\left(X_1, B_1\right)$ and $\left(X_2, B_2\right)$ be projective klt. 
For any $1\leq i\leq k$, let 
$U_{\vec{\bullet}}^i$ $($resp., $V_{\vec{\bullet}}^i$$)$ be the Veronese equivalence 
class of a graded linear series on $X_1$ $($ resp., on $X_2$$)$ associated to 
$L_1^i,\dots,L_{r_i}^i\in\CaCl(X_1)\otimes_\Z\Q$ 
$($resp., $M_1^i,\dots,M_{s_i}^i\in\CaCl(X_2)\otimes_\Z\Q$$)$ which has bounded 
support and contains an ample series. Set 
$\left(X, B\right):=\left(X_1\times X_2, B_1\boxtimes B_2\right)$ and 
$W_{\vec{\bullet}}^i:=U_{\vec{\bullet}}^i\otimes V_{\vec{\bullet}}^i$ 
$($see Definition \ref{tensor_definition}$)$. 
Moreover, take any $c_1,\dots,c_k\in\R_{>0}$. 
Then we have 
\[
\delta\left(X, B;\left\{c_i W_{\vec{\bullet}}^i\right\}_{i=1}^k\right)
=\min\left\{
\delta\left(X_1, B_1;\left\{c_i U_{\vec{\bullet}}^i\right\}_{i=1}^k\right), \quad
\delta\left(X_2, B_2;\left\{c_i V_{\vec{\bullet}}^i\right\}_{i=1}^k\right)
\right\}.
\]
\end{thm}

We also show the coupled version of Abban--Zhuang's method \cite{AZ} in 
Theorem \ref{AZ_thm} and see several examples of coupled stability thresholds.

Throughout the paper, we work over an algebraically closed filed $\Bbbk$. 
From \S \ref{div_section}, we assume that the characteristic of $\Bbbk$ is 
equal to zero. For the minimal model program, we refer the readers to 
\cite{KoMo, Xu}.

\begin{ack}
The author thanks Yoshinori Hashimoto, who introduced him the notion of 
coupled stability thresholds and providing him many suggestions 
during the 28th symposium on complex geometry 
in Kanazawa; Ivan Cheltsov, who asked him about the formula in Corollary 
\ref{adequate-divide_corollary}; and the referee for suggesting many important 
improvements of the paper. 
This work was supported by JSPS KAKENHI Grant Number 22K03269, 
Royal Society International Collaboration Award 
ICA\textbackslash 1\textbackslash 23109 and Asian Young Scientist Fellowship. 
\end{ack}

\section{Graded linear series}\label{graded_section}

Let us recall basic definitions of graded linear series. 
See also \cite{LM, boucksom, AZ, FANO, r3d28}. 
In \S \ref{graded_section}, we always assume that $X$ is an 
$n$-dimensional projective variety. 
Moreover, for any $\vec{x}=(x_1,\dots,x_r)\in\R^r$ and for $L_1,\dots,L_r$ 
$\R$-Cartier $\R$-divisors on $X$, let $\vec{x}\cdot \vec{L}$ 
be the $\R$-Cartier $\R$-divisor on $X$ defined by 
$\vec{x}\cdot \vec{L}:=\sum_{i=1}^r x_i L_i$. 

\begin{definition}[{see \cite[\S 3.1]{r3d28}}]\label{gr_definition}
Let us consider $L_1,\dots,L_r\in\CaCl(X)\otimes_\Z\Q$ and let us set $m\in\Z_{>0}$ 
such that each $m L_i\in\CaCl(X)\otimes_\Z\Q$ lifts to an element 
$m L_i\in\CaCl(X)$. We fix such lifts. 
\begin{enumerate}
\renewcommand{\theenumi}{\arabic{enumi}}
\renewcommand{\labelenumi}{(\theenumi)}
\item\label{gr_definition1}
We say that $V_{m\vec{\bullet}}$ is \emph{an $\left(m\Z_{\geq 0}\right)^r$-graded 
linear series on $X$ associated to $L_1,\dots,L_r$} if $V_{m\vec{\bullet}}$ is a collection 
$\left\{V_{m\vec{a}}\right\}_{\vec{a}\in\Z_{\geq 0}^r}$ of vector subspaces 
\[
V_{m\vec{a}}\subset H^0\left(X, \vec{a}\cdot m\vec{L}\right)
\]
such that, $V_{m\vec{0}}=\Bbbk$ and 
$
V_{m\vec{a}}\cdot V_{m\vec{b}}\subset V_{m\left(\vec{a}+\vec{b}\right)}
$
holds for every $\vec{a}$, $\vec{b}\in\Z_{\geq 0}^r$. 
We note that the definition of $\left(m\Z_{\geq 0}\right)^r$-graded linear series
depends on the choices of lifts $m L_i\in\CaCl(X)$. 
\item\label{gr_definition2}
Let $V_{m\vec{\bullet}}$ be as in \eqref{gr_definition1} and take any $\in\Z_{>0}$. 
We can naturally define the \emph{Veronese subseries} $V_{k m\vec{\bullet}}$ of 
$V_{m\vec{\bullet}}$ by 
\[
V_{k m\vec{a}}:=V_{m \left(k\vec{a}\right)}\quad\quad
\left(\vec{a}\in\Z_{\geq 0}^r\right).
\]
Clearly, the series is a $\left(k m\Z_{\geq 0}\right)^r$-graded linear series on $X$ 
associated to $L_1,\dots,L_r$. 
\item\label{gr_definition3}
Let $V'_{m'\vec{\bullet}}$ be another $\left(m'\Z_{\geq 0}\right)^r$-graded linear series 
on $X$ associated to $L_1,\dots,L_r$ defined by lifts $m' L_i\in\CaCl(X)$. 
The series $V_{m\vec{\bullet}}$ and $V'_{m'\vec{\bullet}}$ are defined to be 
\emph{Veronese equivalent} if there is $d\in m m'\Z_{>0}$ such that 
$(d/m)\cdot m L_i\sim (d/m')m' L_i$ for all $1\leq i\leq r$, and 
\[
V_{\left(\frac{d}{m}\right)m\vec{\bullet}}=V'_{\left(\frac{d}{m'}\right)m'\vec{\bullet}}
\]
holds as $\left(d\Z_{\geq 0}\right)^r$-graded linear series under the above linear 
equivalences. The Veronese equivalence class of $V_{m\vec{\bullet}}$ is 
denoted by $V_{\vec{\bullet}}$. We note that the definition of $V_{\vec{\bullet}}$ 
does not depend on the choices of lifts $m L_i\in\CaCl(X)$. 
\item\label{gr_definition4}
We define \emph{the Veronese equivalence class of the complete linear series} 
$H^0\left(\vec{\bullet}\cdot\vec{L}\right)$ \emph{on $X$ associated to} 
$L_1,\dots,L_r$.
More precisely, for a sufficiently divisible $m\in\Z_{>0}$, let us consider the 
$\left(m\Z_{\geq 0}\right)^r$-graded linear series 
$H^0\left(m\vec{\bullet}\cdot\vec{L}\right)$ on $X$ defined by 
$H^0\left(m\vec{a}\cdot\vec{L}\right):=H^0\left(X, \vec{a}\cdot m\vec{L}\right)$,
and let $H^0\left(\vec{\bullet}\cdot\vec{L}\right)$ be the Veronese equivalence 
class of $H^0\left(m\vec{\bullet}\cdot\vec{L}\right)$. 
\end{enumerate}
\end{definition}

We also recall basic properties of graded linear series in \cite{LM, AZ}. 

\begin{definition}[{\cite[\S 4.3]{LM}, \cite[\S 2]{AZ}, 
\cite[Definition 3.2]{r3d28}}]\label{prop_definition}
Let $V_{\vec{\bullet}}$ be the Veronese equivalence class of an 
$\left(m\Z_{\geq 0}\right)^r$-graded linear series 
$V_{m\vec{\bullet}}$ on $X$ associated to $L_1,\dots,L_r\in\CaCl(X)\otimes_\Z\Q$. 
\begin{enumerate}
\renewcommand{\theenumi}{\arabic{enumi}}
\renewcommand{\labelenumi}{(\theenumi)}
\item\label{prop_definition1}
We set 
\begin{eqnarray*}
\sS\left(V_{m\vec{\bullet}}\right)&:=&
\left\{m\vec{a}\in\Z_{\geq 0}^r\,\,|\,\,V_{m\vec{a}}\neq 0\right\}\subset
\Z_{\geq 0}^r,\\
\Supp\left(V_{m\vec{\bullet}}\right)&:=&
\overline{\Cone\left(\sS\left(V_{m\vec{\bullet}}\right)\right)}\subset\R_{\geq 0}^r.
\end{eqnarray*}
Recall that, for any nonempty subset $\sS\subset\R^r$, the \emph{cone} 
$\Cone(\sS)\subset\R^r$ \emph{generated by $\sS$} is defined to be the set 
\[
\left\{x_1\vec{s}_1+\cdots x_m\vec{s}_m\in\R^r
\mid m\in\Z_{>0}, \, x_1,\dots,x_m\in\R_{>0},\,\vec{s}_1,\dots,\vec{s}_m\in\sS\right\}. 
\]
Thus, the subset $\Supp\left(V_{m\vec{\bullet}}\right)\subset\R^r_{\geq 0}$ is the 
closure of the cone generated by $\sS\left(V_{m\vec{\bullet}}\right)\subset
\R^r$. 
We set $\Supp\left(V_{\vec{\bullet}}\right):=\Supp\left(V_{m\vec{\bullet}}\right)$
and is well-defined by \cite[Lemma 3.4]{r3d28}. 
Moreover, let 
$\Delta_{\Supp}:=\Delta_{\Supp\left(V_{\vec{\bullet}}\right)}\subset\R_{\geq 0}^{r-1}$ 
be the closed convex set defined by the following: 
\[
\Supp\left(V_{\vec{\bullet}}\right)\cap\left(\{1\}\times\R_{\geq 0}^{r-1}\right)
=\{1\}\times\Delta_{\Supp}. 
\]
The series $V_{m\vec{\bullet}}$ (or its class $V_{\vec{\bullet}}$) 
\emph{has bounded support} if 
$\Delta_{\Supp}\subset\R_{\geq 0}^{r-1}$ is a compact set. 
For example, if $r=1$, then any series has bounded support. 
\item\label{prop_definition2}
The series $V_{m\vec{\bullet}}$ \emph{contains an ample series} if: 
\begin{enumerate}
\renewcommand{\theenumii}{\roman{enumii}}
\renewcommand{\labelenumii}{(\theenumii)}
\item\label{prop_definition21}
the sub-semigroup $\sS\left(V_{m\vec{\bullet}}\right)
\subset\left(m\Z_{\geq 0}\right)^r$ generates $\left(m\Z\right)^r$ as 
an abelian group, and 
\item\label{prop_definition22}
there exists $m\vec{a}\in\interior\left(\Supp\left(V_{m\vec{\bullet}}\right)\right)
\cap\left(m\Z_{\geq 0}\right)^r$ and a decomposition 
$m\vec{a}\cdot \vec{L}=A+E$ with $A$ ample 
Cartier divisor and $E$ effective Cartier divisor such that 
\[
k E+ H^0\left(X,k A\right)\subset V_{k m\vec{a}}
\]
holds for every $k\in\Z_{>0}$.
\end{enumerate}
We note that the above definition is equivalent to \cite[Definition 3.2 (2)]{r3d28}
by \cite[Lemme 1.13]{boucksom}. Moreover, by \cite[Lemma 3.4]{r3d28}, if 
$V_{m\vec{\bullet}}$ contains an ample series, then $V_{k m\vec{\bullet}}$ contains 
an ample series for every $k\in\Z_{>0}$. The class $V_{\vec{\bullet}}$ 
\emph{contains an ample series} if some representative $V_{m\vec{\bullet}}$ of 
$V_{\vec{\bullet}}$ contains an ample series. 
It is trivial that, if there is $\vec{x}\in\R_{\geq 0}^r$ with $\vec{x}\cdot\vec{L}$ big, 
then the complete linear series $H^0\left(\vec{\bullet}\cdot\vec{L}\right)$ 
contains an ample series. 
\end{enumerate}
\end{definition}

\begin{definition}\label{pullback_definition}
Let $V_{\vec{\bullet}}$ be the Veronese equivalence class of 
an $\left(m\Z_{\geq 0}\right)^r$-graded linear series $V_{m\vec{\bullet}}$ 
on $X$ associated to 
$L_1,\dots,L_r\in\CaCl(X)\otimes_\Z\Q$, 
let $X'$ be a projective variety together with a morphism $\sigma\colon X' \to X$. 
The \emph{pullback} $\sigma^*V_{m\vec{\bullet}}$ of $V_{m\vec{\bullet}}$ is 
an $\left(m\Z_{\geq 0}\right)^r$-graded linear series on $X'$ associated to 
$\sigma^*L_1,\dots,\sigma^*L_r$ defined by 
\[
\sigma^*V_{m\vec{a}}:=\Img\left(V_{m\vec{a}}\xrightarrow{\sigma^*}
H^0\left(X',m\vec{a}\cdot\sigma^*\vec{L}\right)\right).
\]
Let $\sigma^*V_{\vec{\bullet}}$ be the Veronese equivalence class of 
$\sigma^*V_{m\vec{\bullet}}$ and is well-defined. 
\end{definition}

We see that several basic properties on graded linear series 
are stable under birational pullbacks. When $X$ is normal, the following proposition was 
already known in \cite[Example 3.5]{r3d28}. 

\begin{proposition}\label{pullback_proposition}
Let $V_{\vec{\bullet}}$ be a $\Z_{\geq 0}^r$-graded linear series on $X$ associated to 
Cartier divisors $L_1,\dots,L_r$, let $X'$ be a projective variety together with 
a birational morphism $\sigma\colon X' \to X$.
Then we have the following: 
\begin{enumerate}
\renewcommand{\theenumi}{\arabic{enumi}}
\renewcommand{\labelenumi}{(\theenumi)}
\item\label{pullback_proposition1}
$V_{\vec{\bullet}}$ has bounded support if and only if 
$\sigma^*V_{\vec{\bullet}}$ has bounded support. 
\item\label{pullback_proposition2}
$V_{\vec{\bullet}}$ contains 
an ample series if and only if $\sigma^*V_{\vec{\bullet}}$ 
contains an ample series. 
\end{enumerate}
\end{proposition}

\begin{proof}
\eqref{pullback_proposition1} Trivial since 
$\sS\left(V_{\vec{\bullet}}\right)=\sS\left(\sigma^*V_{\vec{\bullet}}\right)$. 

\eqref{pullback_proposition2}
We may assume that $r=1$. Set $L:=L_1$. 

\noindent\underline{\textbf{Step 1}}\\
Let us assume that $\sigma_*\sO_{X'}=\sO_X$. 
Consider the case $V_\bullet$ contains an ample series. There exists $m\in\Z_{>0}$ 
and a decomposition $m L=A+E$ with $A$ ample Cartier and $E$ effective Cartier 
such that 
\[
k E+ H^0(X, k A)\subset V_{k m}
\]
holds for any $k\in\Z_{>0}$. 
Since $\sigma^*A$ is big, by replacing $m$ if necessary, we may assume that there is 
a decomposition $\sigma^*A= A'+E'$ with 
$A'$ ample Cartier and $E'$ effective Cartier on $X'$. Then we get 
\[
\sigma^*V_{k m}\supset k\sigma^*E+H^0\left(X', k A\right) 
\supset k\left(\sigma^*E+E'\right)+H^0\left(X', k A'\right)
\]
for any $k\in\Z_{>0}$, since $\sigma_*\sO_{X'}=\sO_X$. 

Consider the case $\sigma^*V_\bullet$ contains an ample series. 
There exists $m\in\Z_{>0}$ and a decomposition $\sigma^*(m L)=A'+E'$ with 
$A'$ ample Cartier and $E'$ effective Cartier such that 
\[
k E'+H^0\left(X' k A'\right)\subset \sigma^*V_{k m}
\]
holds for any $k\in \Z_{>0}$. 
Take an ample Cartier divisor $A$ on $X$. 
By replacing $m$ if necessary, we may assume that 
$|A'-\sigma^*A|\neq\emptyset$. Thus, there exists an effective Cartier divisor 
$F'$ on $X'$ and $s\in\Bbbk(X')^\times=\Bbbk(X)^\times$ such that 
$A'-\sigma^*A-F'=\DIV_{X'}(s)=\sigma^*\DIV_X(s)$, where 
$\DIV_X(s)$ is the principal Cartier divisor on $X$ defined by $s$. 
By replacing $A$ by $A+\DIV_X(s)$, we may assume that $A'=\sigma^*A+F'$. 
Since 
\[
E'+F'\in H^0\left(X', \sigma^*(m L-A)\right)=\sigma^*H^0\left(X, m L-A\right),
\]
there exists an effective Cartier divisor $E$ on $X$ 
such that $\sigma^*E=E'+F'$ holds. Thus, for any $k\in\Z_{>0}$, we have 
\[
\sigma^*V_{k m} \supset k E'+ k F' + H^0\left(X', \sigma^*(k A)\right)
=\sigma^*\left(k E+H^0\left(X, k A\right)\right).
\]
This implies that 
\[
V_{k m}\supset k E+H^0\left(X, k A\right)
\]
and thus $V_\bullet$ contains an ample series. 

\noindent\underline{\textbf{Step 2}}\\
By taking the Stein factorization, we may assume that $\sigma$ is finite and birational. 
Let $I\subset\sO_X$ be the conductor ideal of $\sigma$, i.e., 
\[I:=\Ann_{\sO_X}\left(\sigma_*\sO_{X'}/\sO_X\right).\] 
Since $\sigma$ is birational, we have $\dim\left(\sO_X/I\right)<n$. 

Consider the case $V_\bullet$ contains an ample series. Then there exists 
$m\in\Z_{>0}$ and a decomposition $m L=A+E$ with $A$ ample Cartier and $E$ 
effective Cartier such that 
\[
V_{k m}\supset k E+H^0(X, k A)
\]
for any $k\in\Z_{>0}$. By replacing $m$ if necessary, we can take $B\in|A|$ such that 
$\sO_X(-B)\subset I$. 
Write $B=A-\DIV_X(s)$ ($s\in \Bbbk(X)^\times$) and set $A_0:=A+\DIV_X(s)$. 
Obviously, we have $A_0,B\sim A$ and $A_0+B=2 A$. 
From the definition of the conductor ideal, we have 
\[
\sigma_*\sO_{X'}\otimes\sO_X(-B)\subset\sO_X
\]
as subsheaves of $\sigma_*\sO_{X'}$. Hence we get 
\[
\xymatrix{
0 \ar[r] & \sigma_*\sO_{X'}\otimes\sO_X(A_0) \ar[r]^-{\cdot B} \ar@{^{(}_->}[d]
&\sigma_*\sO_{X'}\otimes\sO_X\left(2A\right) \ar@{=}[d]\\
0 \ar[r] & \sO_X\left(2A\right) \ar[r] & 
\sigma_*\sO_{X'}\otimes\sO_X\left(2A\right).
}\]
Thus we get the inclusion 
\[
H^0\left(X', \sigma^*A_0\right)+\sigma^*B\subset 
\sigma^*H^0\left(X, 2A\right).
\]
The decomposition 
\[
\sigma^*\left(2m L\right)=\sigma^*A_0+\sigma^*\left(B+2E\right)
\]
satisfies that $\sigma^*A_0$ is ample, $\sigma^*\left(B+2E\right)$ is effective, and 
\begin{eqnarray*}
\sigma^*V_{2k m}\supset \sigma^*(2k E)+\sigma^*H^0\left(X, 2k A\right)
\supset k\sigma^*\left(B+2E\right)+H^0\left(X', k\sigma^*A_0\right)
\end{eqnarray*}
holds for any $k\in\Z_{>0}$. Thus 
the series $\sigma^*V_\bullet$ contains an ample series. 

Consider the case $\sigma^*V_\bullet$ contains an ample series. 
There exists $m\in\Z_{>0}$ and a decomposition $\sigma^*(m L)=A'+E'$ with 
$A'$ ample Cartier and $E'$ effective Cartier such that 
\[
\sigma^*V_{k m}\supset k E'+H^0\left(X', k A'\right)
\]
holds for any $k\in\Z_{>0}$. By replacing $m$ if necessary, we may assume that 
there exists an ample Cartier divisor $A$ on $X$ such that $F':=A'-2\sigma^*A$ is 
effective and there exists $B\in|A|$ such that $\sO_X(-B)\subset I$ holds. 
Let us set $E:=m L-2A$. Then 
$\sigma^*E=E'+F'$ is effective on $X'$. By the definition of the conductor ideal, 
the Cartier divisor $B+E$ is effective on $X$. The decomposition 
\[m L=(2A-B)+(B+E)\] satisfies that $2A-B$ is ample, $B+E$ is effective, and 
\[
\sigma^*V_{k m}\supset
k (E'+F') + H^0\left(X', \sigma^*(2k A)\right)\supset
k\sigma^*(B+E)+\sigma^*H^0\left(X, k(2A-B)\right), 
\]
which implies that 
\[
V_{k m}\supset k(B+E)+H^0\left(X, k(2A-B)\right)
\]
holds for any $k\in\Z_{>0}$. Thus $V_{\bullet}$ contains an ample series.
\end{proof}

\begin{remark}\label{pullback_remark}
For a finite and birational morphism $\sigma\colon X'\to X$ between varieties and 
a Cartier divisor $E$ on $X$ with $\sigma^*E$ effective on $X'$, we cannot say that 
the $E$ is effective. For example, let us consider 
$X':=\Spec \Bbbk[t]\xrightarrow{\sigma}X:=\Spec\Bbbk[t^2,t^3]$ and let 
$E$ be the Cartier divisor on $X$ defined by $E:=(t^3/t^2=0)$. Then $E$ is not 
effective but $\sigma^*E=(t=0)$ is effective. 
\end{remark}

We define several graded linear series. 

\begin{definition}\label{interior_definition}
Let $V_{\vec{\bullet}}$ be the Veronese equivalence class of 
an $\left(m\Z_{\geq 0}\right)^r$-graded linear series $V_{m\vec{\bullet}}$ 
on $X$ associated to $L_1,\dots,L_r\in\CaCl(X)\otimes_\Z\Q$ which contains 
an ample series. 
\begin{enumerate}
\renewcommand{\theenumi}{\arabic{enumi}}
\renewcommand{\labelenumi}{(\theenumi)}
\item\label{interior_definition1}
(\cite[Lemma 3.4]{r3d28}) For any $\vec{k}=(k_1,\dots,k_r)\in\Z_{>0}^r$, 
let $V^{(\vec{k})}_{m\vec{\bullet}}$ be the $(m\Z_{\geq 0})^r$-graded linear series 
on $X$ associated to $k_1 L_1,\dots,k_r L_r\in\CaCl(X)\otimes_\Z\Q$ defined by 
\[
V^{(\vec{k})}_{m(a_1,\dots,a_r)}:=V_{m(k_1a_1,\dots,k_r a_r)}. 
\]
By \cite[Lemma 3.4]{r3d28}, the series also contains an ample series. 
The Veronese equivalence class $V_{\vec{\bullet}}^{(\vec{k})}$ of 
$V_{m\vec{\bullet}}^{(\vec{k})}$ does not depend on the choice of representatives 
$V_{m\vec{\bullet}}$ of $V_{\vec{\bullet}}$. The series $V_{\vec{\bullet}}$ has 
bounded support if and only if the series $V_{\vec{\bullet}}^{(\vec{k})}$ 
has bounded support.

Similarly, for any $c\in\Q_{>0}$, let $c V_{\vec{\bullet}}$ be the Veronese equivalent 
class of an $(m'\Z_{\geq 0})^r$-graded linear series (for a sufficiently divisible $m'$) 
associated to $c L_1,\dots, c L_r$ defined by 
$c V_{m'\vec{a}}:=V_{c m'\vec{a}}$ for $\vec{a}\in\Z_{\geq 0}^r$. 
\item\label{interior_definition2}
Let us consider the sub-linear series $V_{m\vec{\bullet}}^\circ$ of 
$V_{m\vec{\bullet}}$ defined by 
\[
V_{m\vec{a}}^\circ:=\begin{cases}
V_{m\vec{a}} & \text{if $\vec{a}=\vec{0}$ or }m\vec{a}\in
\interior\left(\Supp\left(V_{\vec{\bullet}}\right)\right), \\
0 & \text{otherwise}.
\end{cases}
\]
We call the series $V_{m\vec{\bullet}}^{\circ}$ the \emph{interior series of} 
$V_{m\vec{\bullet}}$. Obviously, the series $V_{m\vec{\bullet}}^{\circ}$ satisfies that 
$\Supp\left(V_{m\vec{\bullet}}^{\circ}\right)=
\Supp\left(V_{\vec{\bullet}}\right)$ and contains an ample series. 
The Veronese equivalence class $V_{\vec{\bullet}}^{\circ}$ of 
$V_{m\vec{\bullet}}^{\circ}$ does not depend on the choice of representatives 
$V_{m\vec{\bullet}}$ of $V_{\vec{\bullet}}$. 
\item\label{interior_definition3}
More generally, for any convex subset 
$C\subset\Delta_{\Supp\left(V_{\vec{\bullet}}\right)}$ with 
$\interior(C)\neq\emptyset$, 
let us consider the sub-linear series $V^{\langle C\rangle}_{m\vec{\bullet}}$
of 
$V_{m\vec{\bullet}}$ defined by 
\[
V_{m\vec{a}}^{\langle C\rangle}:=\begin{cases}
V_{m\vec{a}} & \text{if $\vec{a}=\vec{0}$ or }m\vec{a}\in
\Cone\left(\{1\}\times C\right), \\
0 & \text{otherwise}.
\end{cases}
\]
We call it the \emph{restriction of $V_{m\vec{\bullet}}$ with respects to 
$C\subset\Delta_{\Supp\left(V_{m\vec{\bullet}}\right)}$}. 
Obviously, $V^{\langle C\rangle}_{m\vec{\bullet}}$ contains an ample series and 
$\Supp\left(V^{\langle C\rangle}_{m\vec{\bullet}}\right)=
\overline{\Cone\left(\{1\}\times C\right)}$. 
The Veronese equivalence class $V_{\vec{\bullet}}^{\langle C\rangle}$ of 
$V_{m\vec{\bullet}}^{\langle C\rangle}$ 
does not depend on the choice of representatives 
$V_{m\vec{\bullet}}$ of $V_{\vec{\bullet}}$. 
\item\label{interior_definition4}
Let us take at most countably infinite set $\Lambda$ and a decomposition 
\[
\Delta_{\Supp}=\overline{\bigcup_{\lambda\in\Lambda}
\Delta_{\Supp}^{\langle\lambda\rangle}}
\]
with 
\begin{itemize}
\item
the set $\Delta_{\Supp}^{\langle\lambda\rangle}$ is a compact convex set 
with nonempty interior for any $\lambda\in\Lambda$, 
\item
$\interior\left(\Delta_{\Supp}\right)\subset\bigcup_{\lambda\in\Lambda}
\Delta_{\Supp}^{\langle\lambda\rangle}$, 
and 
\item
$\interior\left(\Delta_{\Supp}^{\langle\lambda\rangle}\right)
\cap\interior\left(\Delta_{\Supp}^{\langle\lambda'\rangle}\right)=\emptyset$ 
for any $\lambda$, $\lambda'\in\Lambda$ with $\lambda\neq\lambda'$. 
\end{itemize}
For every $\lambda\in\Lambda$, we set 
$V^{\langle\lambda\rangle}_{\vec{\bullet}}:=
V^{\langle \Delta_{\Supp}^{\langle\lambda\rangle}\rangle}_{\vec{\bullet}}$
As in \eqref{interior_definition3}, the series $V_{\vec{\bullet}}^{\langle\lambda\rangle}$ 
has bounded support with 
$\Supp\left(V_{\vec{\bullet}}^{\langle\lambda\rangle}\right)
=\R_{\geq 0}\Delta_{\Supp}^{\langle\lambda\rangle}$ and 
contains an ample series. 
We call the procedure the \emph{decomposition of $V_{\vec{\bullet}}$ with respects 
to the decomposition $\Delta_{\Supp}=\overline{\bigcup_{\lambda\in\Lambda}
\Delta_{\Supp}^{\langle\lambda\rangle}}$.}
\item\label{interior_definition5}
Take any $\vec{a}\in\Q^r_{>0}\cap
\interior\left(\Supp\left(V_{\vec{\bullet}}\right)\right)$. We define the Veronese 
equivalence class $V_{\bullet\vec{a}}$ of the graded linear series on $X$ associated 
to $\vec{a}\cdot\vec{L}$ as follows. Fix a sufficiently divisible $m'\in m\Z_{>0}$ 
and let $V_{m'\bullet\vec{a}}$ be the $\left(m'\Z_{\geq 0}\right)$-graded linear series 
on $X$ associated to $\vec{a}\cdot\vec{L}$ whose $l$-th part is defined to be 
$V_{l\vec{a}}$ for any $l\in m'\Z_{\geq 0}$. Then $V_{\bullet\vec{a}}$ is defined to 
be the class of $V_{m'\bullet\vec{a}}$ and is well-defined. Moreover, 
by \cite[Lemma 4.18]{LM}, the series $V_{\bullet\vec{a}}$ contains an ample series.
\end{enumerate}
\end{definition}

\begin{definition}[{Refinements, \cite[Example 2.6]{AZ}, 
\cite[Definition 3.15]{r3d28}}]\label{refinement_definition}
Let us assume that $X$ is normal, let $Y$ be a prime $\Q$-Cartier divisor on 
$X$, and let 
$m$, $e\in\Z_{>0}$ such that $m e Y$ is a Cartier divisor. 
Let $V_{\vec{\bullet}}$ be the Veronese equivalence class of an 
$(m\Z_{\geq 0})^r$-graded linear series $V_{m\vec{\bullet}}$ on $X$ associated to 
$L_1,\dots,L_r\in\CaCl(X)\otimes_\Z\Q$. 
The \emph{refinement $V_{m\vec{\bullet}}^{(Y,e)}$ of $V_{m\vec{\bullet}}$ by $Y$ 
with exponent $e$}
is the $(m\Z_{\geq 0})^r$-graded linear series on $Y$ associated to 
$L_1|_Y,\dots,L_r|_Y, -e Y|_Y\in\CaCl(Y)\otimes_\Z\Q$  defined by:
\[
V_{m(\vec{a},j)}^{(Y,e)}:=\Image\left(V_{m\vec{a}}\cap
\left( j m e Y+H^0\left(X, \vec{a}\cdot m\vec{L}-j m  eY\right)\right)
\to H^0\left(Y, \vec{a}\cdot m\vec{L}|_Y-j m e Y|_Y\right)\right)
\]
for any $\vec{a}\in\Z_{\geq 0}^r$ and $j\in\Z_{\geq 0}$, where the above 
homomorphism is the natural restriction. 
By \cite[Lemma 3.16]{r3d28}, if $V_{m\vec{\bullet}}$ has bounded support 
(resp., contains an ample series), then so is $V_{m\vec{\bullet}}^{(Y, e)}$. 
Let $V_{\vec{\bullet}}^{(Y)}$ be the Veronese equivalence class of 
$V_{m\vec{\bullet}}^{(Y, 1)}$ (for a divisible $m\in\Z_{>0}$) 
and is called the \emph{refinement of $V_{\vec{\bullet}}$ 
by $Y$}, and is well-defined. Note that, if we set $\vec{e}:=(1,\dots,1,e)$, then 
$\left(V_{m\vec{\bullet}}^{(Y,1)}\right)^{(\vec{e})}=V_{m\vec{\bullet}}^{(Y,e)}$ holds. 
\end{definition}

The following lemma is trivial from the definitions. 

\begin{lemma}\label{divide_lemma}
Let us assume that $X$ is normal, let $Y$ be a prime $\Q$-Cartier divisor on $X$. 
Let $V_{\vec{\bullet}}$ be the Veronese equivalence class of a graded linear series 
on $X$ associated to $L_1,\dots,L_r\in\CaCl(X)\otimes_\Z\Q$ 
which contains an ample series. 
Let $V_{\vec{\bullet}}^{(Y)}$ be the refinement of $V_{\vec{\bullet}}$ by $Y$. 
Then the projection $\R^{r-1}\times\R\to \R^{r-1}$ gives the surjection 
\[
 q\colon\Delta_{\Supp\left(V_{\vec{\bullet}}^{(Y)}\right)}\twoheadrightarrow
 \Delta_{\Supp\left(V_{\vec{\bullet}}\right)}.
\]
Take any closed convex subset 
$C\subset\Delta_{\Supp\left(V_{\vec{\bullet}}\right)}$ 
with $\interior(C)\neq\emptyset$. 
Consider the restriction $V_{\vec{\bullet}}^{\langle C\rangle}$ of 
$V_{\vec{\bullet}}$ with respects to 
$C\subset\Delta_{\Supp\left(V_{\vec{\bullet}}\right)}$. The refinement 
$V_{\vec{\bullet}}^{\langle C\rangle,(Y)}$ of 
$V_{\vec{\bullet}}^{\langle C\rangle}$ by $Y$ is equal to 
the restriction $V_{\vec{\bullet}}^{(Y),\langle q^{-1}(C)\rangle}$ of 
$V_{\vec{\bullet}}^{(Y)}$ with respects to 
$q^{-1}(C)\subset\Delta_{\Supp\left(V_{\vec{\bullet}}^{(Y)}\right)}$ as 
Veronese equivalences of graded linear series on $Y$.
\end{lemma}

We define the notion of the tensor products for graded linear series.

\begin{definition}\label{tensor_definition}
Assume that $X$ is the product of two projective varieties $X_1$ and $X_2$. 
Let $V^i_{\vec{\bullet}}$ be the Veronese equivalence class of an 
$(m\Z_{\geq 0})^{r_i}$-graded 
linear series $V^i_{\vec{\bullet}}$ on $X_i$ associated to 
$L^i_1,\dots,L^i_{r_i}\in\CaCl(X_i)\otimes_\Z\Q$ for $i=1, 2$. 
The \emph{tensor product} $V^1_{m\vec{\bullet}}\otimes V^2_{m\vec{\bullet}}$ is 
the $(m\Z_{\geq 0}^{r_1+r_2-1})$-graded linear series $W_{m\vec{\bullet}}$ on $X$ 
associated to 
\[
L^1_1\boxtimes L^2_1, L^1_2\boxtimes\sO_{X_2},\dots,L^1_{r_1}\boxtimes\sO_{X_2}, 
\sO_{X_1}\boxtimes L^2_2,\dots,\sO_{X_1}\boxtimes L^2_{r_2}
\]
defined by 
\[
W_{m(c,\vec{a},\vec{b})}:=V^1_{m(c,\vec{a})}\otimes V^2_{m(c,\vec{b})}
\]
for any $c\in\Z_{\geq 0}$, $\vec{a}\in\Z_{\geq 0}^{r_1-1}$, 
$\vec{b}\in\Z_{\geq 0}^{r_2-1}$. 
Let $V^1_{\vec{\bullet}}\otimes V^2_{\vec{\bullet}}$ be 
the Veronese equivalence class of $V^1_{m\vec{\bullet}}\otimes V^2_{m\vec{\bullet}}$ 
and called it the \emph{tensor product} of 
$V^1_{\vec{\bullet}}$ and $V^2_{\vec{\bullet}}$, and is well-defined. 
It is obvious from the definition that, if both $V^1_{\vec{\bullet}}$ and 
$V^2_{\vec{\bullet}}$ have bounded supports (resp., contain ample series), 
then so is $V^1_{\vec{\bullet}}\otimes V^2_{\vec{\bullet}}$. 
In fact, we have 
\[
\Delta_{\Supp\left(V^1_{\vec{\bullet}}\otimes V^2_{\vec{\bullet}}\right)}
=\Delta_{\Supp\left(V^1_{\vec{\bullet}}\right)}
\times\Delta_{\Supp\left(V^2_{\vec{\bullet}}\right)}
\subset\R_{\geq 0}^{r_1+r_2-2}. 
\]
We note that, if both $V^1_{\vec{\bullet}}$ and $V^2_{\vec{\bullet}}$ are complete 
linear series, then $V^1_{\vec{\bullet}}\otimes V^2_{\vec{\bullet}}$ is also 
a complete linear series. When we furthermore assume that $r_1=r_2=1$ and 
$L^1:=L_1^1$, $L^2:=L_1^2$ (i.e., $V^1_\bullet=H^0\left(\bullet\cdot L^1\right)$ and 
$V^2_\bullet=H^0\left(\bullet\cdot L^2\right)$), then 
$V^1_\bullet\otimes V^2_\bullet
=H^0\left(\bullet\cdot \left(L^1\boxtimes L^2\right)\right)$. 
\end{definition}

We recall the notion of prime blowups \cite{Ishii} and define the notion of 
primitive flags.

\begin{definition}\label{primitive_definition}
\begin{enumerate}
\renewcommand{\theenumi}{\arabic{enumi}}
\renewcommand{\labelenumi}{(\theenumi)}
\item\label{primitive_definition1}\cite{Ishii}, \cite[Definition 1.1]{pltK}
Let $Y$ be a prime divisor over $X$. If there exists a projective birational morphism 
$\sigma\colon\tilde{X}\to X$ with $\tilde{X}$ normal such that $Y$ is a prime 
and $\Q$-Cartier divisor on $\tilde{X}$ and $-Y$ on $\tilde{X}$ is ample over $X$, 
then the $Y$ is said to be \emph{primitive} over $X$ and the morphism $\sigma$ 
is said to be \emph{the associated prime blowup}. We note that the morphism 
$\sigma$ is uniquely determined by the divisorial valuation $\ord_Y$. 
We often regard primitive prime divisors $Y$ as varieties from the embeddings 
$Y\subset\tilde{X}$. 
\item\label{primitive_definition2}
Take any $1\leq j\leq n$. 
A sequence of varieties $Y_1,\dots, Y_j$ is said to be 
a \emph{primitive flag over $X$} and is denoted by 
\[
Y_\bullet\colon X=Y_0\triangleright Y_1 \triangleright\cdots\triangleright Y_j,
\]
if $Y_k$ is a primitive prime divisor over $Y_{k-1}$ for any $1\leq k\leq j-1$, 
where we set $Y_0:=X$ and we regard $Y_k$ as a variety, as in 
\eqref{primitive_definition1}. 
If moreover $j=n$, then the primitive flag $Y_\bullet$ is said to be \emph{a complete 
primitive flag}. 
\item\label{primitive_definition3}\cite[Definition 1.1]{pltK}
Let us assume that the characteristic of $\Bbbk$ is zero. 
Fix an effective $\Q$-Weil 
divisor $B$ on $X$, i.e., $B$ is a formal $\Q$-linear sum $B=\sum_{i=1}^h b_i B_i$ 
with $b_i\geq 0$ such that each $B_i$ is an irreducible closed subvariety of 
codimension $1$ in $X$. Consider a primitive prime divisor $Y$ over $X$ 
and let $\sigma\colon\tilde{X}\to X$ be the associated primitive blowup. 
Assume that there exists a nonempty open subscheme $U\subset X$ such that 
the center of $Y$ on $X$ is contained in $U$, the pair $(U, B|_U)$ is klt, and 
the morphism $\sigma$ is a plt blowup over $(U, B|_U)$, i.e., the pair 
$\left(\tilde{X}, \tilde{B}+Y\right)$ is plt on $\sigma^{-1}(U)$, where $\tilde{B}$ is 
the effective $\Q$-Weil divisor on $\tilde{X}$ which is defined to be the closure of 
$\tilde{B}|_{\sigma^{-1}(U)}$ defined by 
\[
K_{\sigma^{-1}(U)}+\tilde{B}|_{\sigma^{-1}(U)}
+\left(1-A_{X, B}(Y)\right)Y=\sigma^*\left(K_U+B|_U\right). 
\]
We recall that the value $A_{X, B}(Y)$ is the log discrepancy of $(X, B)$ along $Y$ 
(see \cite[Definition 1.34]{Xu} for example). 
Then the $Y$ is said to be \emph{a plt-type prime divisor over $(U, B|_U)$}. 
By adjunction, if we set 
\[
K_{\sigma|_Y^{-1}(U)}+B_{\sigma|_Y^{-1}(U)}:=
\left(K_{\sigma^{-1}(U)}+\tilde{B}|_{\sigma^{-1}(U)}
+\left(1-A_{X, B}(Y)\right)Y\right)\big|_Y
\] 
and let $B_Y$ be the closure of $B_{\sigma|_Y^{-1}(U)}$ on $Y$, then the pair 
$(Y, B_Y)$ is klt over $U$ (i.e., the pair 
$\left(\sigma|_Y^{-1}(U),B_{\sigma|_Y^{-1}(U)}\right)$ is klt). 
We call the pair $(Y, B_Y)$ \emph{the associated klt pair
over $U$}. If $U=X$, then we simply say that $(Y, B_Y)$ is the \emph{associated 
klt structure}. 
\item\label{primitive_definition4}
Again, assume that the characteristic of $\Bbbk$ is zero and  $B$ is 
an effective $\Q$-Weil divisor on $X$. Consider a primitive flag 
\[
Y_\bullet\colon X=Y_0\triangleright Y_1 \triangleright\cdots\triangleright Y_j
\]
over $X$. Assume that there exists a nonempty open subscheme $U\subset X$ 
such that $Y_k$ is plt-type prime divisor over $(Y_{k-1}, B_{k-1})|_U$ for any 
$1\leq k\leq j-1$, where the pair $(Y_{k-1}, B_{k-1})$ is the associated klt pair 
over $U$. Then the primitive flag $Y_\bullet$ is said to be a \emph{plt flag 
over $(U, B|_U)$}. It is convenient to set 
\[
A_{X, B}\left(Y_1 \triangleright\cdots\triangleright Y_k\right)
:=A_{Y_{k-1}, B_{k-1}}(Y_k)
\]
for every $1\leq k\leq j$. 
Moreover, for any prime divisor $E$ over $Y_k$ with the center on $Y_k$ intersects 
with the pullback of $U$, we set 
\[
A_{X, B}\left(Y_1 \triangleright\cdots\triangleright Y_k\triangleright E\right)
:=A_{Y_k, B_k}(E).
\]
\end{enumerate}
\end{definition}

Here is a generalization of Definition \ref{refinement_definition}. 

\begin{definition}\label{refinement-primitive_definition}
Let $V_{\vec{\bullet}}$ be the Veronese equivalence class of a graded linear series 
on $X$ associated to $L_1,\dots,L_r\in\CaCl(X)\otimes_\Z\Q$. 
\begin{enumerate}
\renewcommand{\theenumi}{\arabic{enumi}}
\renewcommand{\labelenumi}{(\theenumi)}
\item\label{refinement-primitive_definition1}
Let $Y$ be a primitive prime divisor over $X$ and let $\sigma\colon\tilde{X}\to X$ 
be the associated prime blowup. The \emph{refinement $V_{\vec{\bullet}}^{(Y)}$
of $V_{\vec{\bullet}}$ by $Y$} 
is defined to be the refinement (in the sense of Definition \ref{refinement_definition}) 
of the pullback $\sigma^*V_{\vec{\bullet}}$ of $V_{\vec{\bullet}}$ by $Y$. 
Note that, by Proposition \ref{pullback_proposition} and Definition 
\ref{refinement_definition}, if $V_{\vec{\bullet}}$ has bounded support (resp., 
contains an ample series), then so is $V_{\vec{\bullet}}^{(Y)}$. 
\item\label{refinement-primitive_definition2}
Let 
\[
Y_\bullet\colon X=Y_0\triangleright Y_1 \triangleright\cdots\triangleright Y_j
\]
be a primitive flag over $X$. (We mainly consider incomplete primitive flags.)
The \emph{refinement of $V_{\vec{\bullet}}$ by $Y_\bullet$}, 
denoted by
\[
V_{\vec{\bullet}}^{\left(Y_1\triangleright\cdots\triangleright Y_j\right)}
\quad\left(\text{ or }\quad V_{\vec{\bullet}}^{\left(Y_\bullet\right)}\right),
\]
is defined to be inductively. More precisely, 
$V_{\vec{\bullet}}^{\left(Y_1\triangleright\cdots\triangleright Y_k\right)}$ 
is defined to be the refinement of 
$V_{\vec{\bullet}}^{\left(Y_1\triangleright\cdots\triangleright Y_{k-1}\right)}$ 
by $Y_k$
for any $1\leq k\leq j-1$. 
\end{enumerate}
\end{definition}

\section{Okounkov bodies}\label{okounkov_section}

In this section, we recall the notion of Okounkov bodies for graded linear series. 
See also \cite{LM, boucksom, AZ, FANO, r3d28}. In \S \ref{okounkov_section}, 
we always assume that $X$ is an $n$-dimensional projective variety 
and $Y_\bullet$ be an \emph{admissible flag} on $X$ in the sense of \cite[(1.2)]{LM}, 
i.e., 
a sequence 
\[
X=Y_0\supsetneq Y_1\supsetneq \cdots\supsetneq Y_n
\]
of irreducible subvarieties on $X$ such that each $Y_i$ is nonsingular at the point 
$Y_n$ for each $0\leq i\leq n$. 

\begin{definition}[{see \cite[\S 4.3]{LM}, \cite[Definition 2.9]{AZ}, 
\cite[Definition 3.3]{r3d28}}]\label{okounkov_definition}
Let $V_{\vec{\bullet}}$ be the Veronese equivalence class of 
an $\left(m\Z_{\geq 0}\right)^r$-graded linear series $V_{m\vec{\bullet}}$ 
on $X$ associated to $L_1,\dots,L_r\in\CaCl(X)\otimes_\Z\Q$. 
\begin{enumerate}
\renewcommand{\theenumi}{\arabic{enumi}}
\renewcommand{\labelenumi}{(\theenumi)}
\item\label{okounkov_definition1}
As in \cite[(1.2)]{LM}, we can define the valuation-like function 
\[
\nu_{Y_\bullet}\colon V_{m\vec{a}}\setminus\{0\}\to\Z_{\geq 0}^n
\]
for every $\vec{a}\in\Z_{\geq 0}^n$. 
We set 
\begin{eqnarray*}
\Gamma_{Y_\bullet}\left(V_{m\vec{\bullet}}\right)
&:=&\left\{\left(m\vec{a},\nu_{Y_\bullet}(s)\right)\,\,|\,\,
\vec{a}\in\Z_{\geq 0}^r, \,\,s\in V_{m\vec{a}}\right\}
\subset\left(m\Z_{\geq 0}\right)^r\times\Z_{\geq 0}^n, \\
\Sigma_{Y_\bullet}\left(V_{m\vec{\bullet}}\right)&:=&
\overline{\Cone\left(\Gamma_{Y_\bullet}\left(V_{m\vec{\bullet}}\right)\right)}
\subset\R_{\geq 0}^{n+r}.
\end{eqnarray*}
Moreover, let 
$\Delta_{Y_\bullet}\left(V_{m\vec{\bullet}}\right)\subset\R_{\geq 0}^{r-1+n}$ 
be the closed convex set defined by the equation
\[
\{1\}\times\Delta_{Y_\bullet}\left(V_{m\vec{\bullet}}\right)
=\Sigma_{Y_\bullet}\left(V_{m\vec{\bullet}}\right)\cap
\left(\{1\}\times\R_{\geq 0}^{r-1+n}\right), 
\]
and we say that $\Delta_{Y_\bullet}\left(V_{m\vec{\bullet}}\right)$ is the 
\emph{Okounkov body of $V_{m\vec{\bullet}}$ associated to $Y_\bullet$}. 
If $V_{m\vec{\bullet}}$ has a bounded support, then 
$\Delta_{Y_\bullet}\left(V_{m\vec{\bullet}}\right)$ is compact. 

We assume that $V_{\vec{\bullet}}$ contains an ample series. 
In this case, by \cite[Lemma 3.4]{r3d28}, the definitions 
\[
\Sigma_{Y_\bullet}\left(V_{\vec{\bullet}}\right):=
\Sigma_{Y_\bullet}\left(V_{m\vec{\bullet}}\right), \quad
\Delta_{Y_\bullet}\left(V_{\vec{\bullet}}\right):=
\Delta_{Y_\bullet}\left(V_{m\vec{\bullet}}\right)
\]
are well-defined, and we say that $\Delta_{Y_\bullet}\left(V_{\vec{\bullet}}\right)$ is 
the \emph{Okounkov body of $V_{\vec{\bullet}}$ associated to $Y_\bullet$}. 
Let $p\colon \Delta_{Y_\bullet}\left(V_{\vec{\bullet}}\right)
\twoheadrightarrow\Delta_{\Supp}\subset\R_{\geq 0}^{r-1}$ be the composition of 
\[
\Delta_{Y_\bullet}\left(V_{\vec{\bullet}}\right)\hookrightarrow
\R_{\geq 0}^{r-1}\times\R_{\geq 0}^n\xrightarrow{pr_1}\R_{\geq 0}^{r-1}, 
\]
where $pr_1$ is the first projection. The image of $p$ 
is equal to $\Delta_{\Supp}$. 
In fact, for any $\vec{a}\in\Q^{r-1}_{>0}\cap\interior\left(\Delta_{\Supp}\right)$, 
the series $V_{\bullet(1,\vec{a})}$ contains an ample series as in Definition 
\ref{interior_definition} \eqref{interior_definition5}. This implies that 
$p^{-1}(\vec{a})\neq\emptyset$. Thus we get the inclusion 
$\Delta_{\Supp}\subset p\left(\Delta_{Y_\bullet}(V_{\vec{\bullet}})\right)$. 
The reverse inclusion 
$\Delta_{\Supp}\supset p\left(\Delta_{Y_\bullet}(V_{\vec{\bullet}})\right)$ is trivial. 

If there exists $\vec{x}\in\R_{\geq 0}^r$ with $\vec{x}\cdot\vec{L}$ big, then 
we set 
\[
\Sigma_{Y_\bullet}\left(L_1,\dots,L_r\right):=
\Sigma_{Y_\bullet}\left(H^0\left(\vec{\bullet}\cdot\vec{L}\right)\right), \quad
\Delta_{Y_\bullet}\left(L_1,\dots,L_r\right):=
\Delta_{Y_\bullet}\left(H^0\left(\vec{\bullet}\cdot\vec{L}\right)\right).
\]
\item\label{okounkov_definition2}
For any $l\in m\Z_{> 0}$, we set 
\[
h^0\left(V_{l,m\vec{\bullet}}\right):=\sum_{\vec{a}\in\Z_{\geq 0}^{r-1}}
\dim V_{l,m\vec{a}}, 
\]
and 
\[
\vol\left(V_{m\vec{\bullet}}\right):=\limsup_{l\in m\Z_{>0}}
\frac{h^0\left(V_{l,m\vec{\bullet}}\right)m^{r-1}}{l^{r-1+n}/(r-1+n)!}\in[0,\infty].
\]
If $V_{m\vec{\bullet}}$ has bounded support, then the above values are finite. 
If $V_{m\vec{\bullet}}$ contains an ample series, then 
$\vol\left(V_{m\vec{\bullet}}\right)\in(0,\infty]$ and the above limsup is in fact 
the limit. Moreover, the definition 
$\vol\left(V_{\vec{\bullet}}\right):=\vol\left(V_{m\vec{\bullet}}\right)$ is well-defined, 
and 
\[
\vol\left(V_{\vec{\bullet}}\right)=
(r-1+n)!\cdot\vol\left(\Delta_{Y_\bullet}\left(V_{\vec{\bullet}}\right)\right)
\]
holds by \cite[Lemma 3.4]{r3d28}. For any big $L\in\CaCl(X)\otimes_\Z\Q$, 
we have $\vol\left(H^0\left(\bullet L\right)\right)=\vol_X(L)$, where 
$\vol_X(L)\in\R_{>0}$ is the volume of $L$ in the sense of \cite[\S 2.2]{L}. 
\end{enumerate}
\end{definition}

We also recall the notion in \cite[\S 4.5]{Xu}. 

\begin{definition}[{\cite[Definition 4.72]{Xu}}]\label{xu_definition}
Let $V_{\vec{\bullet}}$ and $W_{\vec{\bullet}}$ are the Veronese equivalence classes 
of graded linear series on $X$ associated to 
$L_1,\dots,L_r\in\CaCl(X)\otimes_\Z\Q$ such that both have bounded supports 
and contain ample series. 
If $\vol\left(W_{\vec{\bullet}}\right)=\vol\left(V_{\vec{\bullet}}\right)$
and there exist representatives $V_{m\vec{\bullet}}$ and 
$W_{m\vec{\bullet}}$ for some $m\in\Z_{>0}$ 
with $W_{m\vec{\bullet}}\subset V_{m\vec{\bullet}}$ 
(i.e., $W_{m\vec{a}}\subset V_{m\vec{a}}$ holds for any $\vec{a}\in\Z_{\geq 0}^r$), 
then we say that $W_{\vec{\bullet}}$ is \emph{asymptotically equivalent to} 
$V_{\vec{\bullet}}$. 
\end{definition}

\begin{lemma}\label{asymp-equiv_lemma}
Let us consider $\left(m\Z_{\geq 0}\right)^r$-graded linear series 
$V_{m\vec{\bullet}}$
and $W_{\vec{m\bullet}}$ on $X$ associated to $L_1,\dots,L_r\in\CaCl(X)
\otimes_\Z\Q$ which has bounded supports and contain ample series 
with $W_{m\vec{\bullet}}\subset V_{m\vec{\bullet}}$, and let 
$V_{\vec{\bullet}}$ and $W_{\vec{\bullet}}$ be their Veronese equivalence classes. 
Then the followings are equivalent: 
\begin{enumerate}
\renewcommand{\theenumi}{\arabic{enumi}}
\renewcommand{\labelenumi}{(\theenumi)}
\item\label{asymp-equiv_lemma1}
$W_{\vec{\bullet}}$ is asymptotically equivalent to $V_{\vec{\bullet}}$. 
\item\label{asymp-equiv_lemma2}
The equality $\Supp\left(V_{\vec{\bullet}}\right)=\Supp\left(W_{\vec{\bullet}}\right)$ 
and the equality 
$\vol\left(V_{\bullet\vec{a}}\right)=\vol\left(W_{\bullet\vec{a}}\right)$ holds for any 
$\vec{a}\in\Q_{>0}^r\cap\interior\left(\Supp\left(V_{\vec{\bullet}}\right)\right)$. 
\item\label{asymp-equiv_lemma3}
For any $\vec{a}\in\Q_{>0}^r\cap
\interior\left(\Supp\left(V_{\vec{\bullet}}\right)\right)$, 
the series $W_{\bullet\vec{a}}$ contains an ample series and is asymptotically 
equivalent to $V_{\bullet\vec{a}}$. 
\end{enumerate}
\end{lemma}

\begin{proof}
Let us set $\Delta^V:=\Delta_{Y_\bullet}\left(V_{\vec{\bullet}}\right)$, 
$\Delta^W:=\Delta_{Y_\bullet}\left(W_{\vec{\bullet}}\right)$, 
$\Delta_{\Supp}^V:=\Delta_{\Supp\left(V_{\vec{\bullet}}\right)}$
and $\Delta_{\Supp}^W:=\Delta_{\Supp\left(W_{\vec{\bullet}}\right)}$. 
Both $\Delta^V$ and $\Delta^W$ are compact convex sets with nonempty interiors 
with $\Delta^W\subset\Delta^V\subset\R_{\geq 0}^{r-1+n}$. 
Note that the condition \eqref{asymp-equiv_lemma1} is equivalent to 
the condition $\Delta^V=\Delta^W$. Moreover, recall that 
$p\left(\Delta^V\right)=\Delta_{\Supp}^V$ and 
$p\left(\Delta^W\right)=\Delta_{\Supp}^W$, where $p\colon\R^{r-1+n}\to\R^{r-1}$
is the projection. 
Let us set 
\begin{eqnarray*}
f^V\colon\Delta_{\Supp}^V&\to&\R_{\geq 0} \\
\vec{a}&\mapsto&\vol\left((p|_{\Delta^V})^{-1}\left(\vec{a}\right)\right),
\end{eqnarray*}
\begin{eqnarray*}
f^W\colon\Delta_{\Supp}^W&\to&\R_{\geq 0}\\
\vec{a}&\mapsto&\vol\left((p|_{\Delta^W})^{-1}\left(\vec{a}\right)\right).
\end{eqnarray*}
Both functions are continuous and $f^V|_{\Delta_{\Supp}^W}\geq f^W$ holds. 
By \cite[Theorem 4.21]{LM}, for any 
$\vec{a}\in\Q_{>0}^{r-1}\cap\interior\left(\Delta_{\Supp}^V\right)$
(resp., $\vec{a}\in\Q_{>0}^{r-1}\cap\interior\left(\Delta_{\Supp}^W\right)$), 
we have 
\[
f^V\left(\vec{a}\right)
=\frac{1}{n!}\vol\left(V_{\bullet\vec{a}}\right)\quad
\left(\text{resp., }
f^W\left(\vec{a}\right)
=\frac{1}{n!}\vol\left(W_{\bullet\vec{a}}\right)
\right).
\]
In particular, $f^V>0$ over $\interior\left(\Delta_{\Supp}^V\right)$ and 
$f^W>0$ over $\interior\left(\Delta_{\Supp}^W\right)$. 
From those observations, the condition \eqref{asymp-equiv_lemma2} (and also 
\eqref{asymp-equiv_lemma3}) is equivalent to the condition 
\begin{enumerate}\setcounter{enumi}{3}
\renewcommand{\theenumi}{\arabic{enumi}}
\renewcommand{\labelenumi}{(\theenumi)}
\item\label{asymp-equiv_lemma4}
$\Delta_{\Supp}^V=\Delta_{\Supp}^W$ and $f^V=f^W$ over 
$\interior\left(\Delta_{\Supp}^W\right)$. 
\end{enumerate}
Clearly, the condition \eqref{asymp-equiv_lemma4} is equivalent to the condition 
$\Delta^V=\Delta^W$. 
\end{proof}

\begin{example}\label{interior_example}
Let $V_{\vec{\bullet}}$ be the Veronese equivalence class of 
an $\left(m\Z_{\geq 0}\right)^r$-graded linear series $V_{m\vec{\bullet}}$ 
on $X$ associated to $L_1,\dots,L_r\in\CaCl(X)\otimes_\Z\Q$ which contains 
an ample series. 
\begin{enumerate}
\renewcommand{\theenumi}{\arabic{enumi}}
\renewcommand{\labelenumi}{(\theenumi)}
\item\label{interior_example1}
Take $\vec{k}=(k_1,\dots,k_r)\in\Z_{>0}^r$ and 
let us consider $V^{(\vec{k})}_{\vec{\bullet}}$ 
as in Definition \ref{interior_definition} \eqref{interior_definition1}. 
As in \cite[Lemma 3.4]{r3d28}, we have 
\[
f\left(\Delta_{Y_\bullet}\left(V_{\vec{\bullet}}^{(\vec{k})}\right)\right)
=\Delta_{Y_\bullet}\left(V_{\vec{\bullet}}\right)
\]
with 
\begin{eqnarray*}
f\colon \R^{r-1+n}&\to&\R^{r-1+n}\\
\left(x_1,\dots,x_{r-1+n}\right)&\mapsto&
\left((k_2/k_1)x_1,\dots,(k_r/k_1)x_{r-1},(1/k_1)x_r,\dots,(1/k_1)x_{r-1+n}\right).
\end{eqnarray*}
In particular, we have 
\[
\vol\left(V^{(\vec{k})}_{\vec{\bullet}}\right)=\frac{k_1^{r-1+n}}{k_2\cdots k_r}
\vol\left(V_{\vec{\bullet}}\right). 
\]
\item\label{interior_example2}
Let $C\subset\Delta_{\Supp\left(V_{\vec{\bullet}}\right)}$ be any closed 
convex subset 
with $\interior(C)\neq \emptyset$ as in Definition \ref{interior_definition} 
\eqref{interior_definition3}. Set $\Delta:=\Delta_{Y_\bullet}\left(V_{\vec{\bullet}}\right)\subset\R_{\geq 0}^{r-1+n}$, and 
let $p\colon\Delta\twoheadrightarrow\Delta_{\Supp\left(V_{\vec{\bullet}}\right)}
\subset\R_{\geq 0}^{r-1}$ be 
the natural projection. Then, the convex closed subset 
$p^{-1}\left(C\right)\subset\Delta$ 
is the Okounkov body 
$\Delta_{Y_\bullet}\left(V_{\vec{\bullet}}^{\langle C\rangle}\right)$
of $V_{\vec{\bullet}}^{\langle C\rangle}$, since we can check that 
$\Delta_{Y_\bullet}\left(V_{\vec{\bullet}}^{\langle C\rangle}\right)\subset
p^{-1}(C)$ and 
$\interior\left(p^{-1}(C)\right)\subset
\Delta_{Y_\bullet}\left(V_{\vec{\bullet}}^{\langle C\rangle}\right)$. 
\item\label{interior_example3}
Let us consider the decomposition of $V_{\vec{\bullet}}$ with respects 
to the decomposition $\Delta_{\Supp}=\overline{\bigcup_{\lambda\in\Lambda}
\Delta_{\Supp}^{\langle\lambda\rangle}}$ as in Definition \ref{interior_definition} 
\eqref{interior_definition4}. 
Set 
$\Delta:=\Delta_{Y_\bullet}\left(V_{\vec{\bullet}}\right)\subset\R_{\geq 0}^{r-1+n}$, and 
let $p\colon\Delta\twoheadrightarrow\Delta_{\Supp}\subset\R_{\geq 0}^{r-1}$ be 
the natural projection. Then, as in \eqref{interior_example2}, 
the compact convex subset $\Delta^{\langle\lambda\rangle}
:=p^{-1}\left(\Delta_{\Supp}^{\langle\lambda\rangle}\right)\subset\Delta$ 
is the Okounkov body 
$\Delta_{Y_\bullet}\left(V_{\vec{\bullet}}^{\langle\lambda\rangle}\right)$
of $V_{\vec{\bullet}}^{\langle\lambda\rangle}$ for any 
$\lambda\in\Lambda$. 
Obviously, we have 
\[
\Delta=\overline{\bigcup_{\lambda\in\Lambda}\Delta^{\langle\lambda\rangle}}
\]
and each $\Delta^{\langle\lambda\rangle}$ is a compact convex set with 
nonempty interior and $\interior\left(\Delta^{\langle\lambda\rangle}\right)\cap
\interior\left(\Delta^{\langle\lambda'\rangle}\right)=\emptyset$ whenever 
$\lambda\neq\lambda'$. We have 
\[
\vol\left(V_{\vec{\bullet}}^{\langle\lambda\rangle}\right)
=(r-1+n)!\cdot\vol\left(\Delta^{\langle\lambda\rangle}\right)
\]
for any $\lambda\in\Lambda$. 
Since 
\[
\vol\left(\Delta\right)=\sum_{\lambda\in\Lambda}
\vol\left(\Delta^{\langle\lambda\rangle}\right), 
\]
we get 
\[
\vol\left(V_{\vec{\bullet}}\right)=\sum_{\lambda\in\Lambda}
\vol\left(V_{\vec{\bullet}}^{\langle\lambda\rangle}\right). 
\]
\item\label{interior_example4}
Assume that $X$ is normal and $Y:=Y_1$ is a prime divisor on $X$ which is 
$\Q$-Cartier. From the flag $Y_\bullet$ on $X$, we can naturally consider the flag 
$Y'_\bullet$ on $Y$ defined by $Y'_j:=Y_{j+1}$ for any $0\leq j\leq n-1$. 
By \cite[Definition 3.15]{r3d28}, we have 
\[
\Delta_{Y_\bullet}\left(V_{\vec{\bullet}}\right)
=\Delta_{Y'_\bullet}\left(V^{(Y)}_{\vec{\bullet}}\right),
\]
where $V^{(Y)}_{\vec{\bullet}}$ is the refinement of 
$V_{\vec{\bullet}}$ by $Y$. In particular, we have 
\[\vol\left(V_{\vec{\bullet}}\right)=\vol\left(V^{(Y)}_{\vec{\bullet}}\right).\]
\item\label{interior_example5}
Let us consider the situation in Definition \ref{tensor_definition}. 
Assume moreover both $V^1_{m\vec{\bullet}}$ and $V^2_{m\vec{\bullet}}$ 
contain ample series. For any $l\in m\Z_{>0}$, we have 
\[
h^0\left(W_{l,m\vec{\bullet}}\right)
=\sum_{\vec{a}\in\Z_{\geq 0}^{r_1-1}, \vec{b}\in\Z_{\geq 0}^{r_2-1}}
\dim\left(V^1_{l,m\vec{a}}\otimes V^2_{l,m\vec{b}}\right)
=h^0\left(V^1_{l,m\vec{\bullet}}\right)\cdot h^0\left(V^2_{l,m\vec{\bullet}}\right).
\]
Thus we get 
\[
\vol\left(V^1_{\vec{\bullet}}\otimes V^2_{\vec{\bullet}}\right)
=\binom{n+r_1+r_2-2}{n_1+r_1-1}\vol\left(V^1_{\vec{\bullet}}\right)\cdot
\vol\left(V^2_{\vec{\bullet}}\right).
\]
\item\label{interior_example6}
Assume that $V_{\vec{\bullet}}$ has bounded supports. 
Take the Veronese equivalence class $W_{\vec{\bullet}}$ of a graded linear series 
on $X$ which contains an ample series such that 
$W_{\vec{\bullet}}$ is asymptotically equivalent to $V_{\vec{\bullet}}$. 
Consider any primitive prime divisor $Y$ over $X$. 
By Example \ref{interior_example} \eqref{interior_example4}, the refinement 
$W_{\vec{\bullet}}^{(Y)}$ is also asymptotically equivalent to 
$V_{\vec{\bullet}}^{(Y)}$. 
\item\label{interior_example7}
Assume that $V_{\vec{\bullet}}$ has bounded supports. 
The interior series $V_{\vec{\bullet}}^\circ$ of $V_{\vec{\bullet}}$ is trivially
asymptotically equivalent to $V_{\vec{\bullet}}$ by Lemma \ref{asymp-equiv_lemma}. 
\item\label{interior_example8}
Take any big $L\in\CaCl(X)\otimes_\Z\Q$ and any projective birational morphism 
$\sigma\colon X'\to X$ between varieties. Then 
$\sigma^*H^0\left(\bullet L\right)$ is asymptotically equivalent to 
$H^0\left(\bullet\sigma^*L\right)$ by \cite[Proposition 2.2.43]{L}. 
\end{enumerate}
\end{example}

We will use the following technical proposition. 

\begin{proposition}\label{bary_proposition}
Let us consider $n\geq 2$, let $\Delta\subset\R^n$ be a compact convex set 
with $\interior\left(\Delta\right)\neq\emptyset$, let $p_1\colon\R^n\to\R$ be the 
first projection, and let us set $\left[t_0, t_1\right]:=p_1\left(\Delta\right)\subset\R$. 
Set $V:=\vol_{\R^n}\left(\Delta\right)$ and let $\left(b_1,\dots,b_n\right)\in\Delta$ 
be the barycenter of $\Delta$. For any $x\in\left[t_0,t_1\right]$, we write 
$\Delta_x:=p_1^{-1}\left(\{x\}\right)\subset\R^{n-1}$, and set 
$g(x):=\vol_{\R^{n-1}}\left(\Delta_x\right)$.
Take any $e\in\left(t_0,t_1\right)$. 
\begin{enumerate}
\renewcommand{\theenumi}{\arabic{enumi}}
\renewcommand{\labelenumi}{(\theenumi)}
\item\label{bary_proposition1}
Assume that there exists $v\in\R$ such that either 
\[
v=\lim_{x\to e+0}\frac{g(x)-g(e)}{x-e} \quad\text{or}\quad
v=\lim_{x\to e-0}\frac{g(x)-g(e)}{x-e}. 
\]
Let $h_0\colon\left[t_0,t_1\right]\to\R_{\geq 0}$ be the function defined by 
\[
h_0(x):=\begin{cases}
g(x) & \text{if }x\in\left[t_0,e\right], \\
g(e)\cdot\left(\frac{v(x-e)}{(n-1)g(e)}+1\right)^{n-1} & \text{if }x\in\left[e,t_1\right]. 
\end{cases}\]
\begin{enumerate}
\renewcommand{\theenumii}{\roman{enumii}}
\renewcommand{\labelenumii}{(\theenumii)}
\item\label{bary_proposition11}
For any $x\in\left[t_0,t_1\right]$, we have $g(x)\leq h_0(x)$. In particular, we have 
\[
b_1\geq \frac{1}{V}\int_{t_0}^{s_0}x h_0(x)dx, 
\]
where 
\[
s_0:=\begin{cases}
e+\frac{(n-1)g(e)}{v}\left(\left(\frac{n v\left(V-\int_{t_0}^e g(x)dx\right)
+(n-1)g(e)^2}{(n-1)g(e)^2}\right)^{\frac{1}{n}}-1\right)
& \text{if }v\neq 0, \\
e+\frac{1}{g(e)}\left(V-\int_{t_0}^e g(x)dx\right) & \text{if }v=0.
\end{cases}\]
\item\label{bary_proposition12}
Assume that there exists $t\in(e,t_1]$ such that $W\geq V$ holds, where 
\[
W:=\int_{t_0}^t h_0(x) dx.
\]
In other words, 
\[
W=\begin{cases}
\int_{t_0}^e g(x)dx+\frac{(n-1)g(e)^2}{n v}
\left(\left(\frac{v(t-e)}{(n-1)g(e)}+1\right)^n-1\right) & \text{if }v\neq 0, \\
\int_{t_0}^e g(x)dx + (t-e)g(e) & \text{if }v=0.
\end{cases}\]
(For example, $t=t_1$ satisfies the above assumption.)
Set $h_1\colon [t_0,t]\to\R$ with 
\[
h_1(x):=\begin{cases}
h_0(x) & \text{if }x\in[t_0,s_1], \\
h_0(s_1)\cdot\left(\frac{t-x}{t-s_1}\right)^{n-1} & \text{if }x\in[s_1,t], 
\end{cases}\]
where $s_1\in[e,t]$ is defined by 
\[
s_1:=\begin{cases}
e+\frac{(n-1)g(e)}{v}\left(\left(\frac{n v\left(V-\int_{t_0}^e g(x)dx\right)
+(n-1)g(e)^2}{g(e)\left(v(t-e)+(n-1)g(e)\right)}\right)^{\frac{1}{n-1}}-1\right)
& \text{if }v\neq 0, \\
\frac{n\left(V-\int_{t_0}^e g(x)dx\right)-g(e)(t-n e)}{(n-1)g(e)} & 
\text{if }v=0.
\end{cases}\]
In other words, 
\[
s_1=\begin{cases}
e+\frac{(n-1)g(e)}{v}\left(\left(\frac{(n-1)g(e)^2\left(\frac{v(t-e)}{(n-1)g(e)}+1\right)^n
-n v(W-V)}{g(e)\left(v(t-e)+(n-1)g(e)\right)}\right)^{\frac{1}{n-1}}-1\right)
& \text{if }v\neq 0, \\
t-\frac{n(W-V)}{(n-1)g(e)} & \text{if }v=0.
\end{cases}\]
Then we have 
\[
b_1\geq \frac{1}{V}\int_{t_0}^t x h_1(x)dx.
\]
\end{enumerate}
\item\label{bary_proposition2}
Assume that there exists $u\in[t_1,\infty)$ such that 
\[
\int_{t_0}^e g(x)dx+\int_e^u g(e)\cdot\left(\frac{u-x}{u-e}\right)^{n-1}dx
\left(=\int_{t_0}^e g(x)dx+\frac{(u-e)g(e)}{n}\right)\leq V.
\]
$($For example, $u=t_1$ satisfies the above assumption.$)$ Fix $w\in\R_{\geq 0}$ 
satisfying the condition 
\[
(u-e)\sum_{i=0}^{n-1}g(e)^{\frac{i}{n-1}}w^{n-1-i}
-n\left(V-\int_{t_0}^e g(x)dx\right)\geq 0.
\]
Set $h_2\colon [t_0,u]\in\R_{\geq 0}$ with 
\[
h_2(x):=\begin{cases}
g(x) & \text{if }x\in[t_0,e], \\
\left(\frac{u-x}{u-e}\cdot g(e)^{\frac{1}{n-1}}+\frac{x-e}{u-e}\cdot w\right)^{n-1}
& \text{if }x\in[e,u].
\end{cases}\]
Then we have 
\[
b_1\leq\frac{1}{V}\int_{t_0}^ux h_2(x)dx.
\]
\end{enumerate}
\end{proposition}

\begin{proof}
Since $\Delta$ is a compact convex set, we have 
\begin{itemize}
\item
$g(x)\in\R_{>0}$ for any $x\in(t_0,t_1)$, 
\item
$V=\int_{t_0}^{t_1}g(x)dx$ and $b_1=\frac{1}{V}\int_{t_0}^{t_1}x g(x)dx$, and 
\item
the inequality
\[
g(x_1)^{\frac{1}{n-1}}\geq \frac{x_2-x_1}{x_2-x_0}g(x_0)^{\frac{1}{n-1}}
+\frac{x_1-x_0}{x_2-x_0}g(x_2)^{\frac{1}{n-1}}
\]
holds for any $t_0\leq x_0<x_1<x_2\leq t_1$. 
\end{itemize}

\eqref{bary_proposition1}

\noindent\underline{\textbf{Step 1}}\\
For any $e< y<x\leq t_1$ (resp., for any $t_0\leq y<e<x\leq t_1$), we have 
\[
g(x)^{\frac{1}{n-1}}
\leq\frac{x-e}{y-e}g(y)^{\frac{1}{n-1}}-\frac{x-y}{y-e}g(e)^{\frac{1}{n-1}}
=(x-y)(x-e)\cdot\frac{\frac{g(y)^{\frac{1}{n-1}}}{x-y}
-\frac{g(e)^{\frac{1}{n-1}}}{x-e}}{y-e}. 
\]
By taking $y\to e+0$ (resp., $y\to e-0$), we get
\[
g(x)^{\frac{1}{n-1}}\leq g(e)^{\frac{1}{n-1}}\left(\frac{v(x-e)}{(n-1)g(e)}+1\right)
\]
for any $x\in(e,t_1]$. Thus we have $h_0(x)\geq g(x)$ for any $x\in[t_0,t_1]$. 
Note that, for any $x\in(e,t_1]$, we have 
\[
0\leq g(x)\leq h_0(x)=g(e)\left(\frac{v(x-e)}{(n-1)g(e)}+1\right)^{n-1},
\]
and this implies that 
\[
\frac{v(x-e)}{(n-1)g(e)}+1>0
\]
for any $x\in\left(e, t_1\right)$. 

\noindent\underline{\textbf{Step 2}}\\
Since $V=\int_{t_0}^{t_1}g(x)dx$ and $0\leq g(x)\leq h_0(x)$, there is a unique value 
$\tilde{s}\in(e,t_1]$ satisfying the equality
\[
V=\int_{t_0}^{\tilde{s}}h_0(x)dx.
\]
By the definition of $s_0$, the value $\tilde{s}$ is equal to $s_0$. 
Set $\tilde{h}\colon [t_0,t_1]\to\R_{\geq 0}$ with 
\[
\tilde{h}(x):=\begin{cases}
h_0(x) & \text{if }x\in[t_0,s_0], \\
0 & \text{if } x\in(s_0, t_1].
\end{cases}\]
Then, 
\begin{eqnarray*}
\int_{t_0}^{s_0}x h_0(x)dx-s_0 V=\int_{t_0}^{t_1}(x-s_0)\tilde{h}(x)dx
\leq\int_{t_0}^{t_1}(x-s_0)g(x)dx=\int_{t_0}^{t_1} x g(x)dx-s_0 V.
\end{eqnarray*}
Thus we get the assertion \eqref{bary_proposition11}. 

\noindent\underline{\textbf{Step 3}}\\
For any $y\in[e,t]$, let us set 
\[
W(y):=\int_{t_0}^y h_0(x)dx+\int_y^t h_0(y)\cdot\left(\frac{t-x}{t-y}\right)^{n-1}dx. 
\]
Then $W(e)\leq V\leq W=W(t)$ holds. Moreover, if $y\in(e,t)$, then  
\[
\frac{d}{dy}W(y)=\frac{v(t-e)+(n-1)g(e)}{n}\cdot
\left(\frac{v(y-e)}{(n-1)g(e)}+1\right)^{n-2}\geq 0,
\]
since $h_0(y)>0$ and the end of Step 1. Therefore, there is a unique value 
$s\in[e,t]$ satisfying the condition $W(s)=V$. From the definition of $s_1$, 
we have $s=s_1$, i.e., $W(s_1)=V$ holds. 

For any $x\in[s_1,t]$, the function 
\[
g(x)^{\frac{1}{n-1}}-h_1(x)^{\frac{1}{n-1}}=g(x)^{\frac{1}{n-1}}
-h_0(s_1)^{\frac{1}{n-1}}\cdot\frac{t-x}{t-s_1}
\]
is a concave function. Note that 
\begin{eqnarray*}
g(s_1)^{\frac{1}{n-1}}-h_1(s_1)^{\frac{1}{n-1}}&\leq& 0,\\ 
g(t)^{\frac{1}{n-1}}-h_1(t)^{\frac{1}{n-1}}=g(t)^{\frac{1}{n-1}}&\geq& 0. 
\end{eqnarray*}
Let us set 
\[
\bar{s}:=\min\left\{x\in[s_1,t]\,\,\Big|\,\,g(x)^{\frac{1}{n-1}}
-h_1(x)^{\frac{1}{n-1}}\geq 0\right\}. 
\]
By concavity, we have $g(x)^{\frac{1}{n-1}}-h_1(x)^{\frac{1}{n-1}}\geq 0$ for any 
$x\in[\bar{s},t]$. Set $h_1(x):=0$ for $x\in(t, t_1]$. Then we get 
$h_1(x)\geq g(x)$ for any $x\in[t_0,\bar{s}]$ and 
$h_1(x)\leq g(x)$ for any $x\in[\bar{s},t_1]$. Hence, 
\[
\int_{t_0}^{t_1}x h_1(x)dx-\bar{s} V=\int_{t_0}^{t_1}(x-\bar{s}) h_1(x)dx
\leq\int_{t_0}^{t_1}(x-\bar{s}) g(x)dx=\int_{t_0}^{t_1}x g(x)dx-\bar{s} V.
\]
Thus we get the assertion \eqref{bary_proposition12}. 

\eqref{bary_proposition2}
We firstly note that, by the concavity of $g(x)^{\frac{1}{n-1}}$, we have 
\[
g(x)\geq g(e)\cdot\left(\frac{t_1-x}{t_1-e}\right)^{n-1}
\]
for any $x\in[e,t_1]$. Thus $u=t_1$ satisfies that assumption of 
\eqref{bary_proposition2}. 

The polynomial 
\[
F(y):=(u-e)\sum_{i=0}^{n-1}g(e)^{\frac{i}{n-1}}y^{n-1-i}
-n\left(V-\int_{t_0}^e g(x)dx\right)
\]
satisfies that, $F(0)\leq 0$, $\lim_{y\to\infty}F(y)=+\infty$ and 
$F'(y)>0$ for any $y\in\R_{\geq 0}$. 
Thus, there is a unique value $w_0\in\R_{>0}$ satisfying the condition $F(w_0)=0$. 
Note that $w\geq w_0$. We may assume that $w=w_0$ in order to prove 
\eqref{bary_proposition2}. In this case, we have 
\[
V=\int_{t_0}^u h_2(x) dx, 
\]
since we can compute that 
\[
\int_{t_0}^u h_2(x) dx=\int_{t_0}^e g(x)dx
+\frac{u-e}{n}\left(\sum_{i=0}^{n-1}g(e)^{\frac{i}{n-1}}w^{n-1-i}\right).
\]
Note that the function 
$h_2(x)^{\frac{1}{n-1}}-g(x)^{\frac{1}{n-1}}$ is convex over $x\in[e,t_1]$ 
with $h_2(e)^{\frac{1}{n-1}}-g(e)^{\frac{1}{n-1}}=0$. 

We consider the case $h_2(t_1)^{\frac{1}{n-1}}-g(t_1)^{\frac{1}{n-1}}\leq 0$. 
In this case, by the convexity, we have 
$h_2(x)^{\frac{1}{n-1}}-g(x)^{\frac{1}{n-1}}\leq 0$ for any $x\in[e,t_1]$. 
Therefore we get 
\[
\int_{t_0}^u x h_2(x)dx-t_1 V=\int_{t_0}^u(x-t_1)h_2(x)dx\geq \int_{t_0}^{t_1}(x-t_1)g(x)dx
=\int_{t_0}^{t_1}x g(x)dx-t_1 V. 
\]
Thus we get the assertion \eqref{bary_proposition2} in this case. 

We consider the remaining case $h_2(t_1)^{\frac{1}{n-1}}-g(t_1)^{\frac{1}{n-1}}> 0$. 
If $h_2(x)^{\frac{1}{n-1}}-g(x)^{\frac{1}{n-1}}\geq 0$ for any $x\in[e,t_1]$, then 
\[
V=\int_e^{t_1}g(x)dx<\int_e^{t_1}h_2(x)dx\leq V, 
\]
this leads to a contradiction. Thus, there is a unique value $s_2\in(e,t_1)$ 
satisfying the condition $h_2(s_2)^{\frac{1}{n-1}}-g(s_2)^{\frac{1}{n-1}}=0$.
Moreover, over $x\in(e,t_1)$, the condition 
$h_2(x)^{\frac{1}{n-1}}-g(x)^{\frac{1}{n-1}}>0$ (resp., $<0$) holds if and only if 
$x\in(s_2, t_1)$ (resp., $x\in(e, s_2)$). Therefore we get 
\[
\int_{t_0}^u x h_2(x)dx-s_2 V=\int_{t_0}^u(x-s_2)h_2(x)dx\geq \int_{t_0}^{t_1}(x-s_2)g(x)dx
=\int_{t_0}^{t_1}x g(x)dx-s_2 V. 
\]
Thus we get the assertion \eqref{bary_proposition2}.
\end{proof}

\section{Filtrations on graded linear series}\label{filter_section}

In this section, we recall the theory of filtrations on graded linear series. 
In \S \ref{filter_section}, we fix an $n$-dimensional projective variety $X$. 

\begin{definition}[{see \cite{BC, BJ, zhuang, AZ, r3d28}}]\label{filter_definition}
Let $V$ be a $\Bbbk$-vector space of dimension $N<\infty$. 
\begin{enumerate}
\renewcommand{\theenumi}{\arabic{enumi}}
\renewcommand{\labelenumi}{(\theenumi)}
\item\label{filter_definition1}
A \emph{filtration} $\sF$ of $V$ is given by a collection 
$\{\sF^\lambda V\}_{\lambda\in\R}$ of sub-vector spaces of $V$ satisfying the 
following conditions: 
\begin{enumerate}
\renewcommand{\theenumii}{\roman{enumii}}
\renewcommand{\labelenumii}{(\theenumii)}
\item\label{filter_definition11}
we have $\sF^{\lambda'}V\subset\sF^\lambda V$ for any $\lambda'\geq \lambda$, 
\item\label{filter_definition12}
we have $\sF^\lambda V=\bigcap_{\lambda'<\lambda}\sF^{\lambda'}V$ 
for any $\lambda\in\R$, and 
\item\label{filter_definition13}
we have $\sF^0 V=V$ and $\sF^\lambda V=0$ for any sufficiently large $\lambda$.
\end{enumerate}
For any $\lambda\in\R$, we set 
$\sF^{>\lambda}V:=\bigcup_{\lambda'>\lambda}\sF^{\lambda'}V$ and 
$\Gr^\lambda_{\sF}V:=\sF^\lambda V/\sF^{>\lambda} V$. 

A basis $\{s_1,\dots,s_N\}\subset V$ of $V$ is said to be 
\emph{compatible with $\sF$} if there is a decomposition 
\[
\{s_1,\dots,s_N\}=\bigsqcup_{\lambda\in\R}
\left\{s_1^\lambda,\dots,s^\lambda_{N_\lambda}\right\}
\]
such that $N_\lambda=\dim\Gr^\lambda_{\sF} V$, 
$\left\{s_1^\lambda,\dots,s^\lambda_{N_\lambda}\right\}\subset\sF^\lambda V$, 
and the image of 
$\left\{s_1^\lambda,\dots,s^\lambda_{N_\lambda}\right\}$ in $\Gr^\lambda_{\sF} V$ 
forms a basis of $\Gr^\lambda_{\sF} V$, for any $\lambda\in\R$. 
For a filtration $\sF$ of $V$ and $s\in V\setminus\{0\}$, we set 
\[
\ord_{\sF}(s):=\max\{\lambda\in\R_{\geq 0}\,\,|\,\,s\in\sF^\lambda V\}.
\]

\item\label{filter_definition2}
A filtration $\sF$ of $V$ is said to be \emph{an $\N$-filtration} if 
$\sF^\lambda V=\sF^{\lceil\lambda\rceil}V$ holds for any $\lambda\in\R$. 
\item\label{filter_definition3}
A filtration $\sF$ of $V$ is said to be \emph{a basis type filtration} if $\sF$ is an 
$\N$-filtration and $\dim\Gr^j_{\sF}V=1$ holds for any $j\in\{0,1,\dots,N-1\}$. 
\end{enumerate}
\end{definition}

\begin{example}\label{filter-system_example}
Let $L$ be a Cartier divisor on $X$ and let $V\subset H^0\left(X, L\right)$ be any 
sub-system with $\dim V=N$. 
\begin{enumerate}
\renewcommand{\theenumi}{\arabic{enumi}}
\renewcommand{\labelenumi}{(\theenumi)}
\item\label{filter-system_example1}
For any quasi-monomial valuation $v$ on $X$, we set 
\[
\sF^\lambda_v V:=\left\{s\in V\,\,|\,\,v(s)\geq \lambda\right\}\subset V
\]
for any $\lambda\in\R$. Then $\sF_v$ is a filtration of $V$ 
and $\ord_{\sF_v}=v$. If $v=\ord_E$ for a 
prime divisor $E$ over $X$, then we set $\sF_E:=\sF_{\ord_E}$. Note that 
$\sF_E$ is an $\N$-filtration. 
\item\label{filter-system_example2}
Assume that $X$ is normal. We recall Zhuang's construction 
\cite[Example 2.11]{zhuang} for basis type filtrations of $V$. Assume that we have 
inductively constructed $\sF^j V$ for $0\leq j\leq N-2$. Write 
$\left|\sF^j V\right|=F_j+|M_j|$, where $F_j$ is the fixed part. For a smooth point 
$x_{j+1}\in X$ with $x_{j+1}\not\in\operatorname{Bs}\left(|M_j|\right)$, 
note that the evaluation homomorphism 
\[
M_j\to M_j\otimes\Bbbk\left(x_{j+1}\right)
\]
is surjective, and the kernel $M_j\otimes\dm_{x_{j+1}}$ satisfies that 
\[
\dim M_j\otimes\dm_{x_{j+1}}=\dim\sF^j V-1. 
\] 
We set $\sF^{j+1}V\subset\sF^j V$ defined by 
\[
\left|\sF^{j+1}V\right|:= F_j+\left|M_j\otimes\dm_{x_{j+1}}\right|.
\]
We call the filtration \emph{the basis type filtration associated to $x_1,\dots,x_N$}. 
We will use following two types of basis type filtrations: 
\begin{enumerate}
\renewcommand{\theenumii}{\roman{enumii}}
\renewcommand{\labelenumii}{(\theenumii)}
\item\label{filter-system_example21}
\cite[Example 2.12]{zhuang}
The basis type filtration $\sF$ of $V$ associated to \emph{general points 
$x_1,\dots,x_N\in X$} is said to be \emph{of type (I)}. 
\item\label{filter-system_example22}
\cite[Example 2.13]{zhuang}
Let $\sigma\colon\tilde{X}\to X$ be a birational morphism such that $\tilde{X}$ is a 
normal projective variety, and let $E$ be a prime divisor on $\tilde{X}$. 
Under the identification 
\[
V\xrightarrow{\sim}\sigma^*V\subset H^0\left(\tilde{X},\sigma^*L\right), 
\]
we can choose the basis type filtration $\sF$ of $V$ associated to 
\emph{general points $x_1,\dots,x_N\in E$}. The filtration is said to be 
\emph{of type (II)}. As in \cite[Example 2.13]{zhuang}, the filtration $\sF$ refines 
$\sF_E$, i.e., for any $\lambda\in\Z_{\geq 0}$, there exists $\mu\in\Z_{\geq 0}$ 
such that $\sF_E^\lambda V=\sF^\mu V$ holds. 
\end{enumerate}
\end{enumerate}
\end{example}

\begin{definition}[{see \cite[\S 3]{AZ} and 
\cite[\S 11.2]{r3d28}}]\label{subbasis_definition}
Let $V$ be a $\Bbbk$-vector space of dimension $N<\infty$, and 
let $\sF$ and $\sG$ be filtrations of $V$. Note that $\sF$ induces the filtration 
$\bar{\sF}$ of $\Gr^\mu_{\sG}V$ with 
\[
\bar{\sF}^\lambda\left(\Gr^\mu_{\sG}V\right)
:=\left(\left(\sF^\lambda V\cap\sG^\mu V\right)+\sG^{>\mu}V\right)/\sG^{>\mu}V.
\]
Similarly, $\sG$ naturally induces the filtration $\bar{\sG}$ of $\Gr^\lambda_{\sF}V$. 
By \cite[Lemma 3.1]{AZ}, there is a canonical isomorphism 
\[
\Gr^\lambda_{\bar{\sF}}\Gr^\mu_{\sG}V\simeq
\Gr^\mu_{\bar{\sG}}\Gr^\lambda_{\sF}V
\]
for any $\lambda$, $\mu\in\R$.
\begin{enumerate}
\renewcommand{\theenumi}{\arabic{enumi}}
\renewcommand{\labelenumi}{(\theenumi)}
\item\label{subbasis_definition1}
A subset $\left\{s_1,\dots,s_N\right\}\subset V$ is said to be \emph{a basis of $V$ 
compatible with both $\sF$ and $\sG$} if there is a decomposition 
\[
\left\{s_1,\dots,s_N\right\}=\bigsqcup_{\lambda,\mu\in\R}
\left\{s_1^{\lambda,\mu},\dots,s^{\lambda,\mu}_{N_{\lambda,\mu}}\right\}
\]
such that 
$N_{\lambda,\mu}=\dim\Gr^\lambda_{\bar{\sF}}\Gr^\mu_{\sG}V$, 
$\left\{s_1^{\lambda,\mu},\dots,s^{\lambda,\mu}_{N_{\lambda,\mu}}\right\}
\subset\sF^\lambda V\cap\sG^\mu V$, and the image of 
$\left\{s_1^{\lambda,\mu},\dots,s^{\lambda,\mu}_{N_{\lambda,\mu}}\right\}$ in 
$\Gr^\lambda_{\bar{\sF}}\Gr^\mu_{\sG}V$ gives a basis of 
$\Gr^\lambda_{\bar{\sF}}\Gr^\mu_{\sG}V$ for any $\lambda$, $\mu\in\R$. 
In fact, by \cite[Lemma 3.1]{AZ}, the above subset 
$\left\{s_1,\dots,s_N\right\}\subset V$ is a basis of $V$ compatible with $\sF$
(and also with $\sG$). 
\item\label{subbasis_definition2}
Fix a subset $\Xi\subset\R_{\geq 0}$. 
\begin{enumerate}
\renewcommand{\theenumii}{\roman{enumii}}
\renewcommand{\labelenumii}{(\theenumii)}
\item\label{subbasis_definition21}
A subset $\left\{s_1,\dots,s_M\right\}\subset V$ is said to be 
\emph{a $\left(\sG,\Xi\right)$-subbasis of $V$} if there is a decomposition 
\[
\left\{s_1,\dots,s_M\right\}=\bigsqcup_{\mu\in\Xi}
\left\{s_1^\mu,\dots,s^\mu_{N_\mu}\right\}
\]
such that $N_\mu=\dim\Gr^\mu_{\sG}V$, 
$\left\{s_1^\mu,\dots,s^\mu_{N_\mu}\right\}\subset\sG^\mu V$, and the image of 
$\left\{s_1^\mu,\dots,s^\mu_{N_\mu}\right\}$ in $\Gr^\mu_{\sG}V$ gives a basis 
of $\Gr^\mu_{\sG}V$ for any $\mu\in\Xi$. 
\item\label{subbasis_definition22}
A subset $\left\{s_1,\dots,s_M\right\}\subset V$ is said to be 
\emph{a $\left(\sG,\Xi\right)$-subbasis of $V$ compatible with $\sF$} 
if there is a decomposition 
\[
\left\{s_1,\dots,s_M\right\}=\bigsqcup_{\lambda\in\R,\,\, \mu\in\Xi}
\left\{s_1^{\lambda,\mu},\dots,s^{\lambda,\mu}_{N_{\lambda,\mu}}\right\}
\]
such that $N_{\lambda,\mu}=\dim\Gr^\lambda_{\bar{\sF}}\Gr^\mu_{\sG}V$, 
$\left\{s_1^{\lambda,\mu},\dots,s^{\lambda,\mu}_{N_{\lambda,\mu}}\right\}
\subset\sF^\lambda V\cap\sG^\mu V$, and the image of 
$\left\{s_1^{\lambda,\mu},\dots,s^{\lambda,\mu}_{N_{\lambda,\mu}}\right\}$ 
in $\Gr^\lambda_{\bar{\sF}}\Gr^\mu_{\sG}V$ gives a basis 
of $\Gr^\lambda_{\bar{\sF}}\Gr^\mu_{\sG}V$ for any $\lambda\in\R$, $\mu\in\Xi$. 
As in \cite[Lemma 11.4]{r3d28}, the subset $\left\{s_1,\dots,s_M\right\}\subset V$ 
is a $\left(\sG,\Xi\right)$-subbasis of $V$. 
\end{enumerate}
\end{enumerate}
\end{definition}

\begin{definition}[{\cite[\S 1.3]{BC}, \cite[\S 2.5]{BJ}, \cite[\S 2.6]{AZ}, and
\cite[\S 3.2]{r3d28}}]\label{gr-filter_definition}
Let $V_{\vec{\bullet}}$ be the Veronese equivalence class of an 
$(m\Z_{\geq 0})^r$-graded linear series $V_{m\vec{\bullet}}$ on $X$ associated to 
$L_1,\dots,L_r\in\CaCl(X)\otimes_\Z\Q$. We assume that $V_{\vec{\bullet}}$ 
has bounded support and contains an ample series. 
A \emph{linearly bounded filtration $\sF$ of $V_{m\vec{\bullet}}$} is a filtration $\sF$ 
of $V_{m\vec{a}}$ for every $\vec{a}\in\Z_{\geq 0}^r$ such that 
\[
\sF^\lambda V_{m\vec{a}}\cdot\sF^{\lambda'}V_{m\vec{a}'}\subset
\sF^{\lambda+\lambda'}V_{m(\vec{a}+\vec{a}')}
\]
holds for every $\lambda$, $\lambda\in\R$, $\vec{a}$, $\vec{a}'\in\Z_{\geq 0}^r$, 
and there exists $C\in\R$ such that $\sF^\lambda V_{m\vec{a}}=0$ whenever 
$\lambda\geq C a_1$. 

A \emph{linearly bounded filtration $\sF$ of $V_{\vec{\bullet}}$} is a linearly bounded 
filtration $\sF$ of some representative $V_{m\vec{\bullet}}$, where we identify 
$\sF$ and its natural restriction to the Veronese subseries $V_{k m\vec{\bullet}}$ 
of $V_{m\vec{\bullet}}$. For any $t\in\R$, let 
$V^t_{\vec{\bullet}}:=V^{\sF,t}_{\vec{\bullet}}$ be the Veronese equivalence class of the 
$(m\Z_{\geq 0})^r$-graded linear series $V^t_{m\vec{\bullet}}$ on $X$ associated to 
$L_1,\dots,L_r\in\CaCl(X)\otimes_\Z\Q$ defined by 
$V^t_{m\vec{a}}:=\sF^{m a_1 t}V_{m\vec{a}}$ for any 
$\vec{a}=(a_1,\dots,a_r)\in\Z_{\geq 0}^r$. 
\end{definition}

\begin{example}\label{gr-filtration_example}
For any quasi-monomial 
valuation $v$ on $X$ (resp., for any prime divisor $E$ over $X$), 
the filtration $\sF_v$ (resp., the filtration $\sF_E$) in Example 
\ref{filter-system_example} \eqref{filter-system_example1} gives a linearly bounded 
filtration of $V_{\vec{\bullet}}$. 
\end{example}

We define the $T$-invariant and the $S$-invariant for a filtration of 
graded linear series. 

\begin{definition}[{see \cite{BJ, AZ, r3d28}}]\label{S-T_definition}
Let $V_{m\vec{\bullet}}$, $\sF$ and $V_{\vec{\bullet}}$ be as in 
Definition \ref{gr-filter_definition}. 
\begin{enumerate}
\renewcommand{\theenumi}{\arabic{enumi}}
\renewcommand{\labelenumi}{(\theenumi)}
\item\label{S-T_definition1}
For $l\in m\Z_{>0}$, we set 
\[
T_l\left(V_{m\vec{\bullet}};\sF\right):=\max\left\{\lambda\in\R_{\geq 0}\,\,
|\,\,\text{There exists $\vec{a}\in\Z_{\geq 0}^{r-1}$ with }\sF^\lambda V_{l,m\vec{a}}
\neq 0\right\}.
\]
Moreover, the definition 
\[
T\left(V_{\vec{\bullet}}; \sF\right)
:=\sup_{l\in m\Z_{>0}}\frac{T_l\left(V_{m\vec{\bullet}};\sF\right)}{l}
\]
is well-defined, since 
\[
\sup_{l\in m\Z_{>0}}\frac{T_l\left(V_{m\vec{\bullet}};\sF\right)}{l}
=\lim_{l\in m\Z_{>0}}\frac{T_l\left(V_{m\vec{\bullet}};\sF\right)}{l}
\]
(see \cite[Lemma 1.4]{BC}). As in \cite[Lemma 1.6]{BC} or \cite[Lemma 2.9]{AZ}, 
for any $t\in\left[0,T\left(V_{\vec{\bullet}};\sF\right)\right)$, the series 
$V^{\sF,t}_{\vec{\bullet}}$ has bounded support and contains an ample series. 
\item\label{S-T_definition2}
Take any $l\in m\Z_{>0}$ such that $h^0\left(V_{l,m\vec{\bullet}}\right)\neq 0$.
Let us set 
\[
S_l\left(V_{m\vec{\bullet}};\sF\right):=\frac{1}{h^0\left(V_{l,m\vec{\bullet}}\right)}
\int_0^{\frac{T_l\left(V_{m\vec{\bullet}};\sF\right)}{l}}
h^0\left(V^{\sF,t}_{l,m\vec{\bullet}}\right)dt.
\]
Moreover, the definition 
\[
S\left(V_{\vec{\bullet}};\sF\right):=\lim_{l\in m\Z_{>0}}
S_l\left(V_{m\vec{\bullet}};\sF\right)
\]
is well-defined, 
\[
S\left(V_{\vec{\bullet}};\sF\right)=\frac{1}{\vol\left(V_{\vec{\bullet}}\right)}
\int_0^{T\left(V_{\vec{\bullet}};\sF\right)}\vol\left(V^{\sF,t}_{\vec{\bullet}}\right)dt
\]
holds, and we have 
\[
\frac{T\left(V_{\vec{\bullet}};\sF\right)}{r+n}\leq
S\left(V_{\vec{\bullet}};\sF\right)\leq
T\left(V_{\vec{\bullet}};\sF\right)
\]
(see \cite[Definition 3.8]{r3d28}). 
\end{enumerate}
When $\sF=\sF_E$ for some prime divisor $E$ over $X$, then we set 
$T\left(V_{\vec{\bullet}}; E\right):=T\left(V_{\vec{\bullet}};\sF_E\right)$ and 
$S\left(V_{\vec{\bullet}}; E\right):=S\left(V_{\vec{\bullet}};\sF_E\right)$. 
When $V_{\vec{\bullet}}$ is the complete linear series $H^0\left(\bullet L\right)$ 
on $X$ associated to a big $L\in\CaCl(X)\otimes_\Z\Q$, then we set 
$T(L; \sF):=T\left(V_{\vec{\bullet}}; \sF\right)$ and 
$S(L; \sF):=S\left(V_{\vec{\bullet}}; \sF\right)$. 
More generally, when the characteristic of $\Bbbk$ is equal to zero, if $v$ is a 
valuation on $X$ with $A_{\tilde{X}}(v)<\infty$, where $\tilde{X}\to X$ is a resolution 
of singularities, then the associated filtration $\sF_v$ is 
a linearly bounded filtration by \cite[Lemma 3.1]{BJ}. Thus we can also define 
$T\left(V_{\vec{\bullet}}; v\right):=T\left(V_{\vec{\bullet}};\sF_v\right)$, 
$S\left(V_{\vec{\bullet}}; v\right):=S\left(V_{\vec{\bullet}};\sF_v\right)$, etc. 
\end{definition}

\begin{definition}\label{primitive-S_definition}
Let $V_{\vec{\bullet}}$ be as in Definition \ref{gr-filter_definition}, and 
let 
\[
Y_\bullet\colon X=Y_0\triangleright Y_1\triangleright\cdots\triangleright Y_j
\]
be a primitive flag over $X$. We set 
\[
S\left(V_{\vec{\bullet}}; Y_1\triangleright\cdots\triangleright Y_j\right):=
\begin{cases}
S\left(V_{\vec{\bullet}}; Y_1\right) & \text{if }j=1, \\
S\left(V_{\vec{\bullet}}^{\left(Y_1\triangleright\cdots\triangleright Y_{j-1}\right)}; 
Y_j\right) & \text{if }j\geq 2. 
\end{cases}\]
We also define $T\left(V_{\vec{\bullet}}; Y_1\triangleright\cdots\triangleright 
Y_j\right)$ similarly. 
Moreover, if $E$ is a prime divisor over $Y_j$, we set 
\begin{eqnarray*}
S\left(V_{\vec{\bullet}}; Y_1\triangleright\cdots\triangleright Y_j\triangleright E\right)
&:=&S\left(V_{\vec{\bullet}}^{(Y_1\triangleright\cdots\triangleright Y_j)}; E\right), \\
T\left(V_{\vec{\bullet}}; Y_1\triangleright\cdots\triangleright Y_j\triangleright E\right)
&:=&T\left(V_{\vec{\bullet}}^{(Y_1\triangleright\cdots\triangleright Y_j)}; E\right). 
\end{eqnarray*}
When $V_{\vec{\bullet}}$ is the complete linear series $H^0\left(\bullet L\right)$ 
on $X$ associated to a big $L\in\CaCl(X)\otimes_\Z\Q$, then we set 
$S\left(L; Y_1\triangleright\cdots\triangleright Y_j\right):=S\left(V_{\vec{\bullet}}; 
Y_1\triangleright\cdots\triangleright Y_j\right)$, 
$T\left(L; Y_1\triangleright\cdots\triangleright Y_j\right):=T\left(V_{\vec{\bullet}}; 
Y_1\triangleright\cdots\triangleright Y_j\right)$, and for a prime divisor $E$ over $Y_j$, 
we set
$S\left(L; Y_1\triangleright\cdots\triangleright Y_j\triangleright 
E\right):=S\left(V_{\vec{\bullet}}; Y_1\triangleright\cdots\triangleright Y_j\triangleright 
E\right)$ and 
$T\left(L; Y_1\triangleright\cdots\triangleright Y_j\triangleright 
E\right):=T\left(V_{\vec{\bullet}}; Y_1\triangleright\cdots\triangleright Y_j\triangleright 
E\right)$. 
\end{definition}

\begin{remark}[{see \cite[\S 3.2]{r3d28}}]\label{S-T_remark}
Let $V_{m\vec{\bullet}}$, $\sF$ and $V_{\vec{\bullet}}$ be as in 
Definition \ref{gr-filter_definition} and let us set 
$T:=T\left(V_{\vec{\bullet}};\sF\right)$ 
and $S:=S\left(V_{\vec{\bullet}};\sF\right)$. 
Let $Y_\bullet$ be any admissible flag on $X$. 
\begin{enumerate}
\renewcommand{\theenumi}{\arabic{enumi}}
\renewcommand{\labelenumi}{(\theenumi)}
\item\label{S-T_remark1}
Let us set $\Delta:=\Delta_{Y_\bullet}\left(V_{\vec{\bullet}}\right)
\subset\R_{\geq 0}^{r-1+n}$ and 
\[
\Delta^t:=\Delta^{\sF,t}:=\Delta_{Y_\bullet}\left(V_{\vec{\bullet}}^{\sF,t}\right)
\subset\Delta
\]
for any $t\in[0,T)$. Moreover, we define 
\begin{eqnarray*}
G:=G_{\sF}\colon\Delta&\to&[0,T]\\
\vec{x}&\mapsto&\sup\left\{t\in[0,T)\,\,|\,\,\vec{x}\in\Delta^t\right\}.
\end{eqnarray*}
Then we have 
\[
S=\frac{1}{\vol(\Delta)}\int_\Delta 
G\left(\vec{x}\right)d\vec{x}.
\]
\item\label{S-T_remark2}
For any $\vec{k}=(k_1,\dots,k_r)\in\Z_{>0}^r$, we have 
\[
T\left(V_{\vec{\bullet}}^{(\vec{k})};\sF\right)=k_1\cdot T,\quad\quad 
S\left(V_{\vec{\bullet}}^{(\vec{k})};\sF\right)=k_1\cdot S.
\] 
See \cite[Lemma 3.10]{r3d28}. 
\item\label{S-T_remark3}
Assume that $X$ is normal, $Y_1\subset X$ is a prime $\Q$-Cartier divisor on $X$ 
and $\sF=\sF_{Y_1}$. Then the above function $G_{\sF}$ is equal to the composition 
\[
\Delta\hookrightarrow\R^{r-1+n}\xrightarrow{p_r}\R,
\]
where $p_r\colon\R^{r-1+n}\to\R$ is the $r$-th projection. 
In particular, 
\begin{itemize}
\item
we can write $p_r(\Delta)=[T_0, T]$ for some $0\leq T_0<T$, and 
\item
the value $S$ is the $r$-th coordinate of the barycenter of the convex set $\Delta$. 
\end{itemize}
\end{enumerate}
\end{remark}

From Example \ref{interior_example} \eqref{interior_example4} 
and Remark \ref{S-T_remark} \eqref{S-T_remark3}, it is natural to extend 
the notion of Okounkov bodies. 

\begin{definition}\label{primitive-Okounkov_definition}
Let $V_{\vec{\bullet}}$ be as in Definition \ref{gr-filter_definition}, and let 
\[
Y_\bullet\colon X=Y_0\triangleright Y_1\triangleright\cdots\triangleright Y_n
\]
be a complete primitive flag over $X$. \emph{The Okounkov body 
$\Delta_{Y_\bullet}\left(V_{\vec{\bullet}}\right)\subset\R_{\geq 0}^{r-1+n}$ 
of $V_{\vec{\bullet}}$ associated to $Y_\bullet$} is defined to be 
\[
\Delta_{Y_\bullet}\left(V_{\vec{\bullet}}\right):=
\Delta_{\tilde{Y}_{n-1}\ni Y_n}\left(V_{\vec{\bullet}}^{\left(Y_1
\triangleright\cdots\triangleright Y_{n-1}\right)}\right), 
\]
where $\tilde{Y}_{n-1}$ is the normalization of the projective curve $Y_{n-1}$ and 
we regard $\tilde{Y}_{n-1}\ni Y_n$ as an admissible flag on $\tilde{Y}_{n-1}$. 
We note that the cone $\Sigma_{Y_\bullet}\left(V_{\vec{\bullet}}\right)$ of 
$\Delta_{Y_\bullet}\left(V_{\vec{\bullet}}\right)$ is equal to 
the closure of the cone of the support of 
$V_{\vec{\bullet}}^{\left(Y_1\triangleright\cdots\triangleright Y_n\right)}$; 
a graded linear series on the $0$-dimensional projective 
variety $Y_n$. 
If the complete primitive flag $Y_\bullet$ is an admissible flag of $X$, then 
the notion coincides with Definition \ref{okounkov_definition} by 
Example \ref{interior_example} \eqref{interior_example4} . 
Moreover, by Remark \ref{S-T_remark} \eqref{S-T_remark3}, 
the $(r+j-1)$-th coordinate 
of the barycenter of $\Delta_{Y_\bullet}\left(V_{\vec{\bullet}}\right)$ is 
equal to the value 
$S\left(V_{\vec{\bullet}}; Y_1\triangleright\cdots\triangleright Y_j\right)$
for any $1\leq j\leq n$. 
\end{definition}

\begin{example}\label{primitive-MDS_example}
Assume that the characteristic of $\Bbbk$ is zero and $n=3$. 
\begin{enumerate}
\renewcommand{\theenumi}{\arabic{enumi}}
\renewcommand{\labelenumi}{(\theenumi)}
\item\label{primitive-MDS_example1}
Assume that $X$ is a Fano manifold and $L\in\CaCl(X)$ is nef and big. 
Then, the graded linear systems $W_{\bullet,\bullet}^Y$ and 
$V_{\bullet,\bullet}^{\tilde{Y}}$ in \cite[\S 1.7]{FANO} satisfy that, the pullback of 
$W_{\bullet,\bullet}^Y$ is asymptotically equivalent to $V_{\bullet,\bullet}^{\tilde{Y}}$ 
by \cite[Theorem 1.106]{FANO}. Therefore, by \cite[Lemma 4.73]{Xu} and 
Example \ref{interior_example} \eqref{interior_example6}, the value 
$S\left(W_{\bullet,\bullet,\bullet}^{Y,Z};P\right)$ in \cite[Theorem 1.112]{FANO} 
is equal to the value $S\left(L; Y\triangleright Z\triangleright P\right)$. 
Obviously, the values $S\left(W_{\bullet,\bullet}^Y;Z\right)$ and 
$S\left(V_{\bullet,\bullet}^Y;Z\right)$ in 
\cite[Corollary 1.110 and Theorem 1.112]{FANO}
is equal to the value 
$S\left(L; Y\triangleright Z\right)$. 
Those values are the third and second coordinates of the Okounkov body of $L$ 
associated to the admissible flag $Y\supset Z\ni P$ by 
Remark \ref{S-T_remark} \eqref{S-T_remark3}. 
\item\label{primitive-MDS_example2}
Similary, assume that $X$ is a Mori dream space and $L\in\CaCl(X)\otimes_\Z\Q$ 
is big. Then the values $S\left(V_{\bullet,\bullet}^{\tilde{Y}}; C\right)$ and 
$S\left(W_{\bullet,\bullet,\bullet}^{Y', C}; p\right)$ in \cite[Corollary 4.18]{r3d28}
are nothing but the values $S\left(L; Y'\triangleright C\right)$ and 
$S\left(L; Y'\triangleright C \triangleright p\right)$, if 
$C$ is primitive over $Y$, the morpshism $\nu\colon Y'\to Y$ in 
\cite[\S 4.3]{r3d28} is the associated blowup and 
the $C$ inside $Y'$ is smooth. 
\end{enumerate}
\end{example}



\begin{proposition}\label{decomposition_proposition}
Let $V_{\vec{\bullet}}$ be the Veronese equivalence class of a graded linear series 
on $X$ associated to $L_1,\dots,L_r\in\CaCl(X)\otimes_\Z\Q$ which has bounded 
support and contains an ample series. Let $Y_\bullet$ be an admissible flag on $X$. 
Let us consider the decomposition of 
$V_{\vec{\bullet}}$ with respects 
to the decomposition $\Delta_{\Supp}=\overline{\bigcup_{\lambda\in\Lambda}
\Delta_{\Supp}^{\langle\lambda\rangle}}$ as in Definition \ref{interior_definition} 
\eqref{interior_definition4}.
Let $\sF$ be 
any linearly bounded filtration of $V_{\vec{\bullet}}$. 
\begin{enumerate}
\renewcommand{\theenumi}{\arabic{enumi}}
\renewcommand{\labelenumi}{(\theenumi)}
\item\label{decomposition_proposition1}
We have 
\[
T\left(V_{\vec{\bullet}};\sF\right)=\sup_{\lambda\in\Lambda}
T\left(V_{\vec{\bullet}}^{\langle\lambda\rangle};\sF\right). 
\]
\item\label{decomposition_proposition2}
For any $t\in\R$, let us set $\Delta^{\sF, t}:=
\Delta_{Y_\bullet}\left(V_{\vec{\bullet}}^{\sF, t}\right)\subset\Delta$ 
and $G_{\sF}\colon\Delta\to\left[0,T\left(V_{\vec{\bullet}};\sF\right)\right)$ be 
as in Remark \ref{S-T_remark} \eqref{S-T_remark1}. 
Moreover, for any $\lambda\in\Lambda$, let  
$V_{\vec{\bullet}}^{\langle\lambda\rangle, \sF, t}$ be the subsystem of 
$V_{\vec{\bullet}}^{\langle\lambda\rangle}$ obtained by $\sF$, set 
$\Delta^{\langle\lambda\rangle,\sF,t}:=
\Delta_{Y_\bullet}\left(V_{\vec{\bullet}}^{\langle\lambda\rangle, \sF, t}\right)
\subset\Delta^{\langle\lambda\rangle}$, 
and let us set 
\begin{eqnarray*}
G_{\sF}^{\langle\lambda\rangle}\colon\Delta^{\langle\lambda\rangle}&\to&
\left[0, T\left(V_{\vec{\bullet}}^{\langle\lambda\rangle};\sF\right)\right], \\
\vec{x}&\mapsto&\sup\left\{
t\in\left[0,T\left(V_{\vec{\bullet}}^{\langle\lambda\rangle};\sF\right)\right)\,\,\Big|
\,\,\vec{x}\in\Delta^{\langle\lambda\rangle,\sF,t}\right\}, 
\end{eqnarray*}
as in Remark \ref{S-T_remark} \eqref{S-T_remark1}. 
\begin{enumerate}
\renewcommand{\theenumii}{\roman{enumii}}
\renewcommand{\labelenumii}{\rm{(\theenumii)}}
\item\label{decomposition_proposition21}
If $t\in\left[0,T\left(V_{\vec{\bullet}}^{\langle\lambda\rangle};\sF\right)\right)$ satisfies 
that $\interior\left(\Delta^{\langle\lambda\rangle}\right)\cap
\interior\left(\Delta^{\sF, t}\right)\neq\emptyset$, then we have 
\[
\Delta^{\langle\lambda\rangle,\sF,t}
=\Delta^{\langle\lambda\rangle}\cap\Delta^{\sF,t}
=p^{-1}\left(\Delta_{\Supp}^{\langle\lambda\rangle}\right)\cap
\Delta^{\sF,t}.
\]
\item\label{decomposition_proposition22}
The function $G_{\sF}^{\langle\lambda\rangle}$ is equal to $G_{\sF}$ over 
$\interior\left(\Delta^{\langle\lambda\rangle}\right)$. 
\item\label{decomposition_proposition23}
We have the equality
\[
\vol\left(V_{\vec{\bullet}}\right)\cdot S\left(V_{\vec{\bullet}};\sF\right)
=\sum_{\lambda\in\Lambda}
\vol\left(V_{\vec{\bullet}}^{\langle\lambda\rangle}\right)\cdot 
S\left(V_{\vec{\bullet}}^{\langle\lambda\rangle};\sF\right).
\]
\end{enumerate}
\end{enumerate}
\end{proposition}

\begin{proof}
By \cite[Lemma 4.73]{Xu}, Remark \ref{S-T_remark} and 
\cite[Lemma 3.10]{r3d28}, we may assume that 
$V_{\vec{\bullet}}$ is $\Z_{\geq 0}^r$-graded with $L_1,\dots,L_r\in\CaCl(X)$ and 
$V_{\vec{\bullet}}=V_{\vec{\bullet}}^\circ$. 

\eqref{decomposition_proposition1}
Since $V_{\vec{\bullet}}=V_{\vec{\bullet}}^\circ$, for any 
$m\in\Z_{>0}$, we have 
\[
T_m\left(V_{\vec{\bullet}};\sF\right)=\max_{\lambda\in\Lambda}
T_m\left(V_{\vec{\bullet}}^{\langle\lambda\rangle};\sF\right). 
\]
Thus we have 
\[
T\left(V_{\vec{\bullet}};\sF\right)=\sup_m\sup_\lambda
\frac{T_m\left(V_{\vec{\bullet}}^{\langle\lambda\rangle};\sF\right)}{m}=\sup_\lambda
T\left(V_{\vec{\bullet}}^{\langle\lambda\rangle};\sF\right). 
\]

\eqref{decomposition_proposition2}
The assertion (i) is trivial from the definition of Okounkov bodies. 
Let us consider (ii). Take any 
$\vec{x}\in\interior\left(\Delta^{\langle\lambda\rangle}\right)$. For any 
$t<G_{\sF}\left(\vec{x}\right)$, we have 
$\vec{x}\in\interior\left(\Delta^{\sF,t}\right)
\cap\interior\left(\Delta^{\langle\lambda\rangle}\right)(\neq\emptyset). $
By (i), we get $\vec{x}\in\Delta^{\langle\lambda\rangle,\sF,t}$. Thus we get 
$G_{\sF}^{\langle\lambda\rangle}\left(\vec{x}\right)\geq G_{\sF}\left(\vec{x}\right)$. 
Conversely, for any $t<G_{\sF}^{\langle\lambda\rangle}\left(\vec{x}\right)$, we have 
$\vec{x}\in\interior\left(\Delta^{\langle\lambda\rangle,\sF,t}\right)\subset
\Delta^{\sF, t}$. Thus we immediately get the reverse inequality 
$G_{\sF}^{\langle\lambda\rangle}\left(\vec{x}\right)\leq G_{\sF}\left(\vec{x}\right)$ and 
we get (ii). The assertion (iii) follows from (ii), since we know that 
\[
S\left(V_{\vec{\bullet}}^{\langle\lambda\rangle};\sF\right)
=\frac{(r-1+n)!}{\vol\left(V_{\vec{\bullet}}^{\langle\lambda\rangle}\right)}\cdot
\int_{\Delta^{\langle\lambda\rangle}}G_{\sF}^{\langle\lambda\rangle}d\vec{x}
\]
for any $\lambda\in\Lambda$. 
\end{proof}

The following lemma is essentially due to Kewei Zhang. 

\begin{lemma}[{cf.\ \cite[Proposition 4.1]{kewei}}]\label{kewei_lemma}
Let $V_{\vec{\bullet}}$ be the Veronese equivalence class of a graded linear series 
on $X$ associated to $L_1,\dots,L_r\in\CaCl(X)\otimes_\Z\Q$ which 
contains an ample series. Let $\sF$ be a linearly bounded filtration of 
$V_{\vec{\bullet}}$. Set $\sC:=\interior\left(\Supp\left(V_{\vec{\bullet}}\right)\right)
\subset\R_{>0}^r$. 
Take any $\vec{a}$, $\vec{b}\in\sC\cap\Q^r$ 
\begin{enumerate}
\renewcommand{\theenumi}{\arabic{enumi}}
\renewcommand{\labelenumi}{(\theenumi)}
\item\label{kewei_lemma1}
Assume that $\vec{b}-\vec{a}\in\sC$. Then we have 
\[
\vol\left(V_{\bullet\vec{a}}^{\sF,t}\right)\leq
\vol\left(V_{\bullet\vec{b}}^{\sF,t}\right)
\]
for any $t\in\left[0,T\left(V_{\bullet\vec{a}};\sF\right)\right)$. 
In particular, we have 
\[
T\left(V_{\bullet\vec{a}};\sF\right)\leq T\left(V_{\bullet\vec{b}};\sF\right)
\]
and
\[
\vol\left(V_{\bullet\vec{a}}\right)\cdot S\left(V_{\bullet\vec{a}};\sF\right)
\leq
\vol\left(V_{\bullet\vec{b}}\right)\cdot S\left(V_{\bullet\vec{b}};\sF\right).
\]
\item\label{kewei_lemma2}
Take any $\varepsilon\in\Q$ with $0<\varepsilon<1/(2n)$. 
If $(1+\varepsilon)\vec{a}-\vec{b}\in\sC$ and 
$\vec{b}-(1-\varepsilon)\vec{a}\in\sC$, then we have 
\begin{eqnarray*}
S\left(V_{\bullet\left(\vec{a}+\varepsilon\vec{b}\right)};\sF\right)
&\geq&\left(\frac{1+\varepsilon-\varepsilon^2}{1+\varepsilon+\varepsilon^2}\right)^n
\left(1+\varepsilon-\varepsilon^2\right)\cdot S\left(V_{\bullet\vec{a}};\sF\right),\\
S\left(V_{\bullet\left(\vec{a}-\varepsilon\vec{b}\right)};\sF\right)
&\leq&\left(\frac{1-\varepsilon+\varepsilon^2}{1-\varepsilon-\varepsilon^2}\right)^n
\left(1-\varepsilon+\varepsilon^2\right)\cdot S\left(V_{\bullet\vec{a}};\sF\right).
\end{eqnarray*}
\end{enumerate}
\end{lemma}

\begin{proof}
\eqref{kewei_lemma1}
Set $\vec{c}:=\vec{b}-\vec{a}$. Take a sufficiently divisible $l\in\Z_{>0}$. 
Since $\vec{c}\in\sC$, there exists an effective $\Q$-divisor 
$C\sim_\Q\vec{c}\cdot\vec{L}$ such that $l C\in\left|V_{l\vec{c}}\right|$. Thus 
we have a natural inclusion
\[
\sF^{l t}V_{l \vec{a}}\hookrightarrow\sF^{l t} V_{l\vec{b}}
\]
by multiplying $l C$. In particular, 
\begin{eqnarray*}
\vol\left(V_{\bullet\vec{a}}\right)\cdot S\left(V_{\bullet\vec{a}};\sF\right)
&=&
\int_0^{T\left(V_{\bullet\vec{a}};\sF\right)}\vol\left(V_{\bullet\vec{a}}^{\sF,t}\right)d t\\
\leq
\int_0^{T\left(V_{\bullet\vec{b}};\sF\right)}\vol\left(V_{\bullet\vec{b}}^{\sF,t}\right)d t
&=&
\vol\left(V_{\bullet\vec{b}}\right)\cdot S\left(V_{\bullet\vec{b}};\sF\right)
\end{eqnarray*}
holds. 

\eqref{kewei_lemma2}
By \eqref{kewei_lemma1}, we have 
\begin{eqnarray*}
S\left(V_{\bullet\left(\vec{a}-\varepsilon\vec{b}\right)};\sF\right)
&\leq&
\frac{\vol\left(V_{\bullet\left((1-\varepsilon+\varepsilon^2)\vec{a}\right)}\right)}
{\vol\left(V_{\bullet\left(\vec{a}-\varepsilon\vec{b}\right)}\right)}\cdot
S\left(V_{\bullet(1-\varepsilon+\varepsilon^2)\vec{a}};\sF\right)\\
&\leq&
\frac{\vol\left(V_{\bullet\left((1-\varepsilon+\varepsilon^2)\vec{a}\right)}\right)}
{\vol\left(V_{\bullet\left((1-\varepsilon-\varepsilon^2)\vec{a}\right)}\right)}\cdot
S\left(V_{\bullet(1-\varepsilon+\varepsilon^2)\vec{a}};\sF\right)\\
&=&
\left(\frac{1-\varepsilon+\varepsilon^2}{1-\varepsilon-\varepsilon^2}\right)^n
\left(1-\varepsilon+\varepsilon^2\right)\cdot S\left(V_{\bullet\vec{a}};\sF\right).
\end{eqnarray*}
We can get the other inequality similarly. 
\end{proof}

We recall the notion of basis type $\Q$-divisors. 

\begin{definition}[{see \cite[Definition 11.8]{r3d28}}]\label{bt-div_definition}
Let $V_{\vec{\bullet}}$ be the Veronese equivalence class of an 
$(m\Z_{\geq 0})^r$-graded linear series $V_{m\vec{\bullet}}$ 
on $X$ associated to $L_1,\dots,L_r\in\CaCl(X)\otimes_\Z\Q$ which has bounded 
support and contains an ample series. Let $\sF$ be a linearly bounded filtration 
of $V_{m\vec{\bullet}}$. 
\begin{enumerate}
\renewcommand{\theenumi}{\arabic{enumi}}
\renewcommand{\labelenumi}{(\theenumi)}
\item\label{bt-div_definition1}
Consider $l\in m\Z_{>0}$ with $h^0\left(V_{l,m\vec{\bullet}}\right)\neq 0$. 
An effective $\Q$-Cartier $\Q$-divisor $D$ on $X$ is said to be 
\emph{an $l$-basis type $\Q$-divisor of $V_{m\vec{\bullet}}$} (resp., 
\emph{compatible with $\sF$}) if there is a basis 
\[
\left\{s_1^{\vec{a}},\dots,s_{N_{\vec{a}}}^{\vec{a}}\right\}\subset V_{l,m\vec{a}}
\]
of $V_{l,m\vec{a}}$ (resp., compatible with $\sF$) for any $\vec{a}\in\Z_{\geq 0}^{r-1}$ 
such that 
\[
D=\frac{1}{l\cdot h^0\left(V_{l,m\vec{\bullet}}\right)}\sum_{\vec{a}\in\Z_{\geq 0}^{r-1}}
\sum_{i=1}^{N_{\vec{a}}}\left(s_i^{\vec{a}}=0\right)
\]
holds. 
\item\label{bt-div_definition2}
Let $\sigma\colon X'\to X$ be a projective birational morphism with $X'$ normal, 
let $Y\subset X'$ be a prime $\Q$-Cartier divisor on $X'$, and let $e\in\Z_{>0}$ 
with $e Y$ Cartier. Let $V_{m\vec{\bullet}}^{(Y,e)}$ be the refinement of 
$\sigma^*V_{m\vec{\bullet}}$ by $Y$ with exponent $e$. 
Consider $l\in m\Z_{>0}$ with $h^0\left(V_{l,m\vec{\bullet}}^{(Y,e)}\right)\neq 0$. 
An effective $\Q$-Cartier $\Q$-divisor $D'$ on $X$ is said to be 
\emph{an $l$-$(Y,e)$-subbasis type $\Q$-divisor of $V_{m\vec{\bullet}}$} 
(resp., \emph{compatible with $\sF$})
if there exists an $\left(\sF_Y,e\Z_{\geq 0}\right)$-subbasis 
\[
\left\{s_1^{\vec{a}},\dots,s_{M_{\vec{a}}}^{\vec{a}}\right\}\subset V_{l,m\vec{a}}
\]
of $V_{l,m\vec{a}}$ for any $\vec{a}\in\Z_{\geq 0}^{r-1}$ 
(resp., compatible with $\sF$)
such that 
\[
D=\frac{1}{l\cdot h^0\left(V^{(Y,e)}_{l,m\vec{\bullet}}\right)}
\sum_{\vec{a}\in\Z_{\geq 0}^{r-1}}
\sum_{i=1}^{M_{\vec{a}}}\left(s_i^{\vec{a}}=0\right)
\]
holds. 
Note that $\sum_{\vec{a}\in\Z_{\geq 0}^{r-1}}M_{\vec{a}}
=h^0\left(V^{(Y,e)}_{l,m\vec{\bullet}}\right).$
\end{enumerate}
\end{definition}

\begin{remark}[{see \cite[\S 3.1]{AZ} and 
\cite[Proposition 11.9]{r3d28}}]\label{bt-div_remark}
\begin{enumerate}
\renewcommand{\theenumi}{\arabic{enumi}}
\renewcommand{\labelenumi}{(\theenumi)}
\item\label{bt-div_remark1}
We have 
\[
\ord_{\sF}D\leq S_l\left(V_{m\vec{\bullet}};\sF\right)
\]
for any $l$-basis type $\Q$-divisor $D$ of $V_{m\vec{\bullet}}$. Moreover, 
the equality attains if $D$ is compatible with $\sF$. 
\item\label{bt-div_remark2}
We have 
\[
\ord_Y D'=S_l\left(V_{m\vec{\bullet}}^{(Y,e)};\bar{\sF}_Y\right)
\]
for any $l$-$(Y,e)$-basis type $\Q$-divisor $D'$ of $V_{m\vec{\bullet}}$, 
where $\bar{\sF}_Y$ on $V_{m\vec{\bullet}}^{(Y,e)}$ is the natural filtration 
induced by $\sF_Y$ on $\sigma^*V_{m\vec{\bullet}}$ (in the sense of 
Definition \ref{subbasis_definition}). 
Moreover, if we set 
\[
D'':=\sigma^*D'-S_l\left(V_{m\vec{\bullet}}^{(Y,e)};\bar{\sF}_Y\right) Y, \quad
\quad D_Y:=D''|_Y, 
\]
then $D_Y$ is an $l$-basis type $\Q$-divisor of $V_{m\vec{\bullet}}^{(Y,e)}$.
\end{enumerate}
\end{remark}

We will use the following well-known lemma in \S \ref{delta_section}. 

\begin{lemma}[{\cite[Corollary 2.10]{BJ}, \cite[Lemma 2.9]{AZ21} 
and \cite[Lemma 11.6]{r3d28}}]\label{bt-div_lemma}
Let $V_{\vec{\bullet}}$ be the Veronese equivalence class of an 
$(m\Z_{\geq 0})^r$-graded linear series $V_{m\vec{\bullet}}$ 
on $X$ associated to $L_1,\dots,L_r\in\CaCl(X)\otimes_\Z\Q$ which has bounded 
support and contains an ample series. 
\begin{enumerate}
\renewcommand{\theenumi}{\arabic{enumi}}
\renewcommand{\labelenumi}{(\theenumi)}
\item\label{bt-div_lemma1}
For any $\varepsilon\in\Q_{>0}$, there exists $l_0\in m\Z_{>0}$ such that, 
for any linearly bounded filtration $\sF$ on $V_{m\vec{\bullet}}$ and for any 
$l\in m\Z_{>0}$ with $l\geq l_0$, we have 
\[
S_l\left(V_{m\vec{\bullet}};\sF\right)\leq (1+\varepsilon) 
S\left(V_{\vec{\bullet}};\sF\right).
\]
\item\label{bt-div_lemma2}
Let $\sigma\colon X'\to X$ be a projective birational morphism with $X'$ normal, 
let $Y\subset X'$ be a prime $\Q$-Cartier divisor on $X'$ and let $e\in\Z_{>0}$ 
with $e Y$ Cartier. For any linearly bounded filtration $\sF$ of $V_{\vec{\bullet}}$, 
we have 
\[T\left(V_{\vec{\bullet}};\sF\right)=T\left(V_{\vec{\bullet}}^{(Y)};\bar{\sF}\right), 
\quad
S\left(V_{\vec{\bullet}};\sF\right)=S\left(V_{\vec{\bullet}}^{(Y)};\bar{\sF}\right),
\]
where $V_{\vec{\bullet}}^{(Y)}$ is the refinement of $\sigma^*V_{\vec{\bullet}}$ by $Y$ 
and $\bar{\sF}$ is the filtration on $V_{\vec{\bullet}}^{(Y)}$ induced by $\sF$ 
in the sense of Definition \ref{subbasis_definition}. 
\end{enumerate}
\end{lemma}

\section{Toric plt flags}\label{toric_section}

In this section, we observe the Okounkov bodies of big divisors on $\Q$-factorial 
projective toric varieties associated to torus invariant complete primitive flags, 
which is a generalization of \cite[Proposition 6.1 (1)]{LM}. In this section, we fix 
$N^0:=\Z^n$, $M^0:=\Hom_\Z\left(N^0,\Z\right)$ and the $n$-dimensional 
$\Q$-factorial projective toric variety associated with a fan $\Sigma$ in 
$N^0_\R:=N^0\otimes_\Z\R$. In this section, we follow the notations in \cite{CLS}.

\begin{definition}\label{toric-plt-flag_definition}
Fix $1\leq j\leq n$. For any $1\leq k\leq j$, let us fix a primitive element 
$v_k\in N^{k-1}$, and set $N^k:=N^{k-1}/\Z v_k\left(\simeq\Z^{n-k}\right)$. 
Let $\pi_k\colon N^{k-1}\twoheadrightarrow N^k$ be the quotient homomorphism. 
From those $v_1,\dots,v_j$, we inductively define 
\begin{itemize}
\item
a fan $\Sigma_k$ on $N^k_\R:=N^k\otimes_\Z\R$ for any $0\leq k\leq j$, and 
\item
a fan $\tilde{\Sigma}_k$ on $N^k_\R$ for any $0\leq k\leq j-1$ 
\end{itemize}
as follows: 
\begin{itemize}
\item
We set $\Sigma_0:=\Sigma$. 
\item
$\tilde{\Sigma}_k$ is the star subdivision of $\Sigma_k$ at $v_{k+1}$ 
in the sense of \cite[\S 11.1]{CLS}. 
\item
$\Sigma_{k+1}$ is defined to be 
\[
\Sigma_{k+1}:=\left\{\left(\pi_{k+1}\right)_\R\left(\tau\right)\subset N^{k+1}_\R
\,\,|\,\,v_{k+1}\in\tau\in\tilde{\Sigma}_k\right\},
\]
as in \cite[\S 3.2]{CLS}. 
\end{itemize}
Let $Y_k$ be the toric variety associated with the fan $\Sigma_k$, and let 
$\tilde{Y}_k$ be the toric variety associated with the fan $\tilde{\Sigma}_k$. 
Then both are $\Q$-factorial projective toric varieties, $X=Y_0$, and there is 
a natural projective birational morphism $\sigma_k\colon\tilde{Y}_k\to Y_k$ such that 
the morphism $\sigma_k$ is the prime blowup of $Y_{k+1}\subset\tilde{Y}_k$ 
by \cite[Proposition 11.1.6]{CLS}. The sequence 
\[
Y_\bullet\colon X=Y_0\triangleright Y_1\triangleright\cdots\triangleright Y_j
\]
is a plt flag over $X$. Conversely, any torus invariant primitive flag over $X$ can be 
obtained in the way of above. We call the flag \emph{the torus invariant plt flag 
over $X$ associated with $v_1,\dots,v_j$}. 
\end{definition}

We are mainly interested in torus invariant \emph{complete} plt flags over $X$. 

\begin{definition}\label{toric-cone-sequence_definition}
Let
\[
Y_\bullet\colon X=Y_0\triangleright Y_1\triangleright\cdots\triangleright Y_n
\]
be the torus invariant complete plt flag over $X$ associated with $v_1,\dots,v_n$
as in Definition \ref{toric-plt-flag_definition}. 
We inductively define 
\begin{itemize}
\item
an $(n-j)$-dimensional cone $\tau_j\in\Sigma_j$ and an $(n-j)$-dimensional 
cone $\gamma_j\in\tilde{\Sigma}_j$ with $v_{j+1}\in\gamma_j\subset\tau_j$
for any $0\leq j\leq n-1$, and 
\item
a primitive element $v_{j,k}\in N^{j-1}$ 
with 
\begin{eqnarray*}
\tau_{j-1}&=&\Cone\left(v_{j,j},v_{j,j+1},\dots,v_{j,n}\right), \\
\gamma_{j-1}&=&\Cone\left(v_j,v_{j,j+1},\dots,v_{j,n}\right)
\end{eqnarray*}
for any $1\leq j\leq k\leq n$, satisfying 
$\left(\pi_j\right)_\R\left(\R_{\geq 0}v_{j,k}\right)=\R_{\geq 0}v_{j+1,k}$ if $j<k$, 
\end{itemize}
as follows: 
\begin{enumerate}
\renewcommand{\theenumi}{\arabic{enumi}}
\renewcommand{\labelenumi}{(\theenumi)}
\item\label{toric-cone-sequence_definition1}
We set $\gamma_{n-1}:=\R_{\geq 0}v_n$, $\tau_{n-1}:=\R_{\geq 0} v_n$ 
and $v_{n,n}:=v_n$. 
\item\label{toric-cone-sequence_definition2}
Assume that we have defined $\tau_j\in\Sigma_j$ and $v_{j+1,j+1},\dots,v_{j+1,n}
\in N^j$ primitive with $\tau_j=\Cone\left(v_{j+1,j+1},\dots,v_{j+1,n}\right)$. 
There is a unique $(n-j+1)$-dimensional cone 
$\gamma_{j-1}\in\tilde{\Sigma}_{j-1}$ with $v_j\in\gamma_{j-1}$ and 
$\left(\pi_j\right)_\R\left(\gamma_{j-1}\right)=\tau_j$.
We can uniquely determine primitive elements 
$v_{j,j+1},\dots,v_{j,n}\in N^{j-1}$ such that 
\begin{itemize}
\item
$\gamma_{j-1}=\Cone\left(v_{j,j+1},\dots,v_{j,n},v_j\right)$ and 
\item
$\left(\pi_j\right)_\R\left(\R_{\geq 0}v_{j,k}\right)=\R_{\geq 0}v_{j+1,k}$
for any $j+1\leq k\leq n$.
\end{itemize}
Since $\tilde{\Sigma}_{j-1}$ is the subdivision of $\Sigma_{j-1}$ at $v_j$, 
there is a unique $(n-j+1)$-dimensional cone $\tau_{j-1}\in\Sigma_{j-1}$ such that 
$\gamma_{j-1}\subset\tau_{j-1}$. Both $\gamma_{j-1}$ and $\tau_{j-1}$ admit the 
$(n-j)$-dimensional face 
$\Cone\left(v_{j,j+1},v_{j,j+2},\dots,v_{j,n}\right)$, we can uniquely take the primitive 
element $v_{j,j}\in N^{j-1}$ such that 
$\tau_{j-1}=\Cone\left(v_{j,j},v_{j,j+1},\dots,v_{j,n}\right)$. 
\end{enumerate}
For any $2\leq j\leq k\leq n$, since 
$\left(\pi_{j-1}\right)_\R\left(\R_{\geq 0}v_{j-1,k}\right)=\R_{\geq 0}v_{j,k}$, 
there uniquely exists a positive integer $m_{j,k}\in\Z_{>0}$ such that 
$\left(\pi_{j-1}\right)\left(v_{j-1,k}\right)=m_{j,k}v_{j,k}$ holds. We also set $m_{1,k}:=1$ 
for any $1\leq k\leq n$. 

For any $1\leq j\leq k\leq n$, since $v_j\in\gamma_{j-1}\subset\tau_{j-1}$, 
we can define a nonnegative rational number $c_{j,k}\in\Q_{\geq 0}$ such that 
\[
v_j=\sum_{k=j}^n c_{j,k}v_{j,k}
\]
holds. We also set 
\[
c'_{j,k}:=\frac{c_{j,k}}{\prod_{i=1}^j m_{i,k}}
\]
for any $1\leq j\leq k\leq n$. 
\end{definition}

\begin{lemma}\label{toric-cone-sequence_lemma}
\begin{enumerate}
\renewcommand{\theenumi}{\arabic{enumi}}
\renewcommand{\labelenumi}{(\theenumi)}
\item\label{toric-cone-sequence_lemma1}
We have $c_{n,n}=1$. In general, $c_{j,j}\in\Q_{>0}$ holds for any $1\leq j\leq n$. 
\item\label{toric-cone-sequence_lemma2}
For any $1\leq j\leq n$, the multiplicity $\mult\left(\tau_{j-1}\right)\in\Z_{>0}$ 
in the sense of \cite[\S 6.4]{CLS} satisfies that 
\[
\mult\left(\tau_{j-1}\right)\cdot\prod_{k=j}^n c_{k,k}=\prod_{j+1\leq i\leq k\leq n}m_{i,k}.
\]
In particular, we have 
\[
\mult\left(\tau_0\right)^{-1}=\prod_{i=1}^n c'_{i,i}. 
\]
\end{enumerate}
\end{lemma}

\begin{proof}
\eqref{toric-cone-sequence_lemma1}
Since $v_{n,n}=v_n$, we have $c_{n,n}=1$. For any $1\leq j\leq n-1$, since 
$v_j\not\in\Cone\left(v_{j,j+1},v_{j,j+2},\dots,v_{j,n}\right)$, we have $c_{j,j}>0$. 

\eqref{toric-cone-sequence_lemma2}
Since $\mult\left(\tau_{n-1}\right)=1$, we may assume that $j\leq n-1$. Note that 
\begin{eqnarray*}
\mult\left(\gamma_{j-1}\right)&=&
\left[N^{j-1}\colon\Z v_j+\Z v_{j,j+1}+\cdots+\Z v_{j,n}\right] \\
&=&\left[N^j\colon\Z m_{j+1,j+1}v_{j+1,j+1}+\cdots+\Z m_{j+1,n}v_{j+1,n}\right] \\
&=&\mult\left(\tau_j\right)\cdot\prod_{k=j+1}^n m_{j+1,k}.
\end{eqnarray*}
On the other hand, since $v_j=\sum_{k=j}^n c_{j,k}v_{j,k}$, we have 
\[
\mult\left(\gamma_{j-1}\right)=\left|\det\left(v_j,v_{j,j+1},\dots,v_{j,n}\right)\right|
=c_{j,j}\left|\det\left(v_{j,j},v_{j,j+1},\dots,v_{j,n}\right)\right|
=c_{j,j}\mult\left(\tau_{j-1}\right). 
\]
Thus we get the assertion \eqref{toric-cone-sequence_lemma2}. 
\end{proof}

We consider the log discrepancies. 

\begin{proposition}\label{toric-A_proposition}
Let $B$ be an effective torus invariant $\Q$-divisor on $X$ such that any coefficient 
of $B$ is less than $1$. It is well-known that the pair $(X,B)$ is a klt pair. 
Let 
\[
Y_\bullet\colon X=Y_0\triangleright Y_1\triangleright\cdots\triangleright Y_n
\]
be the torus invariant complete plt flag over $X$ associated with $v_1,\dots,v_n$
as in Definition \ref{toric-plt-flag_definition}, and let $v_{j,k}$ and $c'_{j,k}$ be 
as in Definition \ref{toric-cone-sequence_definition}. 
The complete flag $Y_\bullet$ is a plt flag over $(X, B)$. For any $1\leq j\leq n$, 
let $(Y_{j-1}, B_{j-1})$ be 
the associated klt structure in the sense of 
Definition \ref{primitive_definition} \eqref{primitive_definition4}. 
Then we have the equality 
\[
A_{Y_{j-1},B_{j-1}}\left(Y_j\right)=\sum_{k=j}^n c'_{j,k}A_{X,B}\left(V\left(
v_{1,k}\right)\right),
\]
where $V\left(v_{1,k}\right)\subset X$ is the torus invariant prime divisor 
associates to the $1$-dimensional cone $\R_{\geq 0}v_{1,k}\in\Sigma$ 
(see \cite[\S 3.2]{CLS}). 
\end{proposition}

\begin{proof}
Clearly, each $(Y_{j-1},B_{j-1})$ is a toric klt pair. Moreover, since 
$v_j=\sum_{k=j}^n c_{j,k}v_{j,k}$, we have 
\[
A_{Y_{j-1},B_{j-1}}\left(Y_j\right)=\sum_{k=j}^n c_{j,k}A_{Y_{j-1},B_{j-1}}\left(V\left(
v_{j,k}\right)\right).
\]
Therefore, it is enough to show the equality 
\[
A_{Y_{j-1},B_{j-1}}\left(V\left(v_{j,k}\right)\right)
=\frac{A_{X,B}\left(V\left(
v_{1,k}\right)\right)}{\prod_{i=1}^j m_{i,k}}
\]
for any $1\leq j\leq k\leq n$. However, it is well-known that 
\[
A_{Y_1,B_1}\left(V\left(v_{2,k}\right)\right)
=\frac{A_{X,B}\left(V\left(v_{1,k}\right)\right)}{m_{2,k}}. 
\]
By doing the procedure $(j-1)$ times, we get the desired equality. 
\end{proof}

Recall that, for any torus invariant $\Q$-divisor $D$ on $X$, we have the 
associated rational polytope $P_D\subset M_\R:=M\otimes_\Z\R$ 
such that, for any sufficiently 
divisible $m\in\Z_{>0}$, the set $m P_D\cap M$ is equal to 
\[
\left\{u\in M\,\,|\,\, \DIV\left(\chi^u\right)+ m D\geq 0\right\},
\]
a basis of isotypical sections of $H^0\left(X, m D\right)$. 
Here is a generalization of \cite[Proposition 6.1]{LM}. 

\begin{thm}\label{toric-okounkov_thm}
Let 
\[
Y_\bullet\colon X=Y_0\triangleright Y_1\triangleright\cdots\triangleright Y_n
\]
be the torus invariant complete plt flag over $X$ associated with $v_1,\dots,v_n$
as in Definition \ref{toric-plt-flag_definition}, and let $v_{j,k}$ and $c'_{j,k}$ be 
as in Definition \ref{toric-cone-sequence_definition}. Take any big 
$L\in\CaCl(X)\otimes_\Z\Q$. Then there exists a unique torus invariant 
$\Q$-divisor $D$ on $X$ with $L\sim_\Q D$ and $D|_{U_{\tau_0}}=0$. Moreover, 
we have 
\[
\Delta_{Y_\bullet}\left(L\right)=\psi\circ\phi\left(P_D\right), 
\]
where 
\begin{eqnarray*}
\phi\colon M_\R &\xrightarrow{\sim}&\R^n\\
u&\mapsto&\left(\langle u,v_{1,k}\rangle\right)_{1\leq k\leq n}
\end{eqnarray*}
and 
\begin{eqnarray*}
\psi\colon\R^n&\xrightarrow{\sim}&\R^n\\
\begin{pmatrix}
x_1 \\ \vdots \\ x_n
\end{pmatrix}&\mapsto&
\begin{pmatrix}
c'_{1,1} & \cdots & c'_{1,n} \\ 
 & \ddots & \vdots \\
O & & c'_{n,n}
\end{pmatrix}
\begin{pmatrix}
x_1 \\ \vdots  \\ x_n
\end{pmatrix}.
\end{eqnarray*}
\end{thm}

\begin{proof}
The proof is similar to the proof of \cite[Proposition 6.1 (1)]{LM}. 
We follow the notations in Definition \ref{toric-cone-sequence_definition}. 
The existence and the uniqueness of $D$ is essentially same as the argument in 
\cite[\S 6.1]{LM}; we have a natural exact sequence
\[
0\to M_\Q\xrightarrow{\DIV\left(\chi^\bullet\right)}\Q^{\Sigma(1)}\to
\CaCl(X)\otimes_\Z\Q\to 0,
\]
where $\Sigma(1)\subset\Sigma$ is the set of $1$-dimensional cones of $\Sigma$ 
(see \cite[Definition 3.1.2]{CLS}). 
Let us consider the Okounkov body $\Delta_{Y_\bullet}\left(L\right)$. 
Fix a sufficiently divisible $m\in\Z_{>0}$ and set 
$V_{m\bullet}:=H^0\left(X,\bullet m L\right)$ and 
\[
\Gamma_{Y_\bullet}\left(L\right):=\sS\left(V_{m\vec{\bullet}}^{\left(
Y_1\triangleright\cdots\triangleright Y_n\right)}\right)
\subset\left(m\Z_{\geq 0}\right)^{n+1}
\]
(see Definition \ref{prop_definition}). 
Since $Y_n$ is $0$-dimensional, for any $\left(a_0,\dots,a_n\right)
\in\left(m\Z_{\geq 0}\right)^{n+1}$, the space 
$V_{a_0,\dots,a_n}^{\left(
Y_1\triangleright\cdots\triangleright Y_n\right)}$ is either zero or $1$-dimensional. 
As in Definition \ref{primitive-Okounkov_definition}, we have
\[
\{1\}\times\Delta_{Y_\bullet}\left(L\right)=\Supp\left(V_{m\vec{\bullet}}^{\left(
Y_1\triangleright\cdots\triangleright Y_n\right)}\right)\cap
\left(\{1\}\times\R_{\geq 0}^n\right).
\]
Take $m'$, $a_0\in m\Z_{>0}$ sufficiently divisible. Let us take any isotypical 
section 
\[
u^0\in a_0 P_D\cap m' M\subset H^0\left(X, a_0 D\right)=V_{a_0}.
\]
Since $m$ and $m'$ are sufficiently divisible, for each $1\leq j\leq n$, we can 
inductively take $a_j\in m\Z_{\geq 0}$ and a nonzero isotypical section 
$u^j\in V_{a_0,\dots,a_j}^{\left(Y_1\triangleright\cdots\triangleright Y_j\right)}$ such 
that the section $u^{j-1}$ maps to $u^j$. 
By the construction of $D_0:=D$, if we set 
$D_j:=\left(\sigma_{j-1}^*D_{j-1}\right)|_{Y_j}$, then we can inductively show that 
$D_j|_{U_{\tau_j}}=0$, where $U_{\tau_j}\subset Y_j$ is the affine toric open subset 
defined by the cone $\tau_j\in\Sigma_j$ (see \cite[\S 1.2]{CLS}). 
. Set $\alpha_k:=\langle u^0,v_{1,k}\rangle$ for any 
$1\leq k\leq n$. 

\begin{claim}\label{toric-okounkov_claim}
For any $1\leq j\leq k\leq n$, we have 
\[
\langle u^{j-1}, v_{j,k}\rangle=\frac{\alpha_k}{\prod_{i=1}^k m_{i,k}}. 
\]
In particular, we have 
\[
a_j=\sum_{k=j}^n c'_{j,k}\alpha_k
\]
for any $1\leq j\leq n$. 
\end{claim}

\begin{proof}[Proof of Claim \ref{toric-okounkov_claim}]
Since $\langle u^{j-1}, m_{j,k}v_{j,k}\rangle=\langle u^{j-2}, v_{j-1,k}\rangle$, we get 
$\alpha_k=\langle u^{j-1}, v_{j,k}\rangle\prod_{i=1}^k m_{i,k}$. 
In particular, since $a_j=\langle u^{j-1}, v_j\rangle$, we complete the proof 
of Claim \ref{toric-okounkov_claim}. 
\end{proof}

Claim \ref{toric-okounkov_claim} implies that 
\[
\left(1\times\left(\psi\circ\phi\right)\right)\left(\Cone\left(\{1\}\times P_D\right)\right)\subset\Gamma_{Y_\bullet}\left(
L\right). 
\]
In particular, we have 
\[
\psi\circ\phi\left(P_D\right)\subset\Delta_{Y_\bullet}\left(L\right). 
\]
Both sets are compact convex sets in $\R_{\geq 0}^n$. 
On the other hand, by Lemma \ref{toric-cone-sequence_lemma}, we have 
\[
\vol\left(\psi\circ\phi\left(P_D\right)\right)
=\vol\left(P_D\right)=\frac{1}{n!}\vol_X(L)
=\vol\left(\Delta_{Y_\bullet}\left(L\right)\right). 
\]
Thus we must have $\psi\circ\phi\left(P_D\right)=\Delta_{Y_\bullet}\left(L\right)$.
\end{proof}

We consider a special case of toric complete plt flags. 

\begin{definition}\label{toric-admissible_definition}
We follow the notations in Definitions \ref{toric-plt-flag_definition} and 
\ref{toric-cone-sequence_definition}. Assume moreover that, the morphism 
$\sigma_{j-1}\colon\tilde{Y}_{j-1}\to Y_{j-1}$ is an isomorphism, i.e., $Y_j$ is a 
prime divisor \emph{on} $Y_{j-1}$, for any $1\leq j\leq n$. In this case, 
for any $1\leq j\leq n$, we have $\R_{\geq 0}v_j\in\Sigma_{j-1}$ and 
$\tilde{\Sigma}_{j-1}=\Sigma_{j-1}$. Therefore, we have 
\begin{itemize}
\item
$v_j=v_{j,j}$, 
\item
$\tau_{j-1}=\gamma_{j-1}=\Cone\left(v_{j,j},\dots,v_{j,n}\right)$, 
\item
$c_{j,k}=\delta_{j,k}$ for any $j\leq k\leq n$. 
\end{itemize}
In this case, the sequence $v_1,\dots,v_n$ is uniquely determined by the sequence 
$v_{1,1},\dots,v_{1,n}\in N^0$. We call 
\[
Y_\bullet\colon X=Y_0\supsetneq Y_1\supsetneq\cdots\supsetneq Y_n
\]
the \emph{complete toric plt flag on $X$ associated to $v_{1,1},\dots,v_{1,n}\in N^0$}. 
For any $0\leq j\leq n$, let us define $l_j:=l_j\left(Y_\bullet\right)\in\Z_{>0}$ 
as follows: 
\[
l_0:=1, \quad l_j:=\mult\left(\R_{\geq 0}v_{1,1}+\cdots+\R_{\geq 0}v_{1,j}\right). 
\]
\end{definition}

\begin{lemma}\label{toric-admissible_lemma}
Under the notations in Definitions \ref{toric-cone-sequence_definition} 
and \ref{toric-admissible_definition}, we have 
\[
\prod_{i=1}^j m_{i,j}=\frac{l_j}{l_{j-1}}. 
\]
for any $1\leq j\leq n$. 
\end{lemma}

\begin{proof}
We may assume that $j\geq 2$. Observe that 
\begin{eqnarray*}
l_j&=&\left[N^0\colon\Z v_{1,1}+\cdots+\Z v_{1,j}\right] \\
&=&\left(\prod_{k=2}^j m_{2,k}\right) 
\left[N^1\colon\Z v_{2,2}+\cdots+\Z v_{2,j}\right]\\
&=&\cdots\\
&=&\prod_{i=2}^j\prod_{k=i}^j m_{i,k}.
\end{eqnarray*}
Thus we get the assertion. 
\end{proof}

The following two corollaries are direct consequences of 
Proposition \ref{toric-A_proposition}, Theorem \ref{toric-okounkov_thm}, 
and 
Lemmas \ref{toric-cone-sequence_lemma} and \ref{toric-admissible_lemma}. 

\begin{corollary}\label{toric-admissible-A_corollary}
Under the notation in Definition \ref{toric-admissible_definition}, we have 
\[
A_{Y_{j-1},B_{j-1}}\left(Y_j\right)=\frac{l_{j-1}}{l_j}A_{X,B}\left(V\left(v_{1,j}\right)\right)
\]
for any $1\leq j\leq n$. 
\end{corollary}

\begin{corollary}\label{toric-okounkov-admissible_cor}
We follow the notation in Definition \ref{toric-admissible_definition}. 
Take any big $L\in\CaCl(X)\otimes_\Z\Q$, and let us take the torus invariant 
$\Q$-divisor $D$ on $X$ with $L\sim_\Q D$ and $D|_{U_{\tau_0}}=0$, where 
$\tau_0=\Cone\left(v_{1,1},\dots,v_{1,n}\right)$. Then we have the equality 
\[
\Delta_{Y_\bullet}\left(L\right)=\phi'\left(P_D\right), 
\]
where 
\begin{eqnarray*}
\phi'\colon M_\R&\xrightarrow{\sim}&\R^n \\
u&\mapsto&\left(\frac{l_{j-1}}{l_j}\langle u, v_{1,j}\rangle\right)_{1\leq j\leq n}.
\end{eqnarray*}
\end{corollary}

Let us compute the value $S\left(L;Y_1\triangleright\cdots\triangleright Y_j\right)$. 

\begin{proposition}\label{toric-S-admissible_proposition}
We follow the notation in Definition \ref{toric-admissible_definition}. 
Let $B$ be an effective torus invariant $\Q$-divisor on $X$ such that $(X, B)$ 
is a klt pair. For any big $L\in\CaCl(X)\otimes_\Z\Q$ and for any $1\leq j\leq n$, 
we have 
\[
S\left(L;Y_1\triangleright\cdots\triangleright Y_j\right)
=\frac{l_{j-1}}{l_j}\cdot S\left(L; V\left(v_{1,j}\right)\right).
\]
\end{proposition}

\begin{proof}
We may assume that $j\geq 2$. Let 
\[
Y'_\bullet\colon X=Y'_0\supsetneq Y'_1\supsetneq\cdots\supsetneq Y'_n
\]
be the complete toric plt flag on $X$ associated to 
\[
v'_{1,1}:=v_{1,j},\,\,v'_{1,2}:=v_{1,2},\dots,v'_{1,j-1}:=v_{1,j-1},\,\,
v'_{1,j}:=v_{1,1},\,\,
v'_{1,j+1}:=v_{1,j+1},\dots,v'_{1,n}:=v_{1,n} \in N^0.
\]
Then, by Corollary \ref{toric-okounkov-admissible_cor}, we have 
$\Delta_{Y_\bullet}\left(L\right)=f\left(\Delta_{Y'_\bullet}\left(L\right)\right)$, where 
the linear transform $f\colon \R^n\to\R^n$ corresponds to the matrix 
\[
\operatorname{diag}\left(\frac{l_0}{l_1},\dots,\frac{l_{n-1}}{l_n}\right)
(1,j)
\operatorname{diag}\left(\frac{l'_1}{l'_0},\dots,\frac{l'_n}{l'_{n-1}}\right),
\]
where $(1,j)$ is the square matrix corresponds to the transposition between 
$1$-st and $j$-th columns, and 
\[
l'_i:=\mult\left(\R_{\geq 0}v'_{1,1}+\cdots+\R_{\geq 0}v'_{1,i}\right)
\]
for any $1\leq i\leq n$. 
Recall that the $1$-st coordinate of the barycenter of 
$\Delta_{Y'_\bullet}\left(L\right)$ is nothing but the value 
$S\left(L; V\left(v_{1,j}\right)\right)$. Moreover, the $j$-th coordinate of the 
barycenter of $\Delta_{Y_\bullet}\left(L\right)$ is equal to the value 
$S\left(L;Y_1\triangleright\cdots\triangleright Y_j\right)$. 
Since $(l_{j-1}/l_j)(l'_1/l'_0)=l_{j-1}/l_j$, we get the assertion. 
\end{proof}

\begin{thm}\label{toric-S_thm}
Let 
\[
Y_\bullet\colon X=Y_0\triangleright Y_1\triangleright\cdots\triangleright Y_n
\]
be the torus invariant complete plt flag over $X$ associated with $v_1,\dots,v_n$
as in Definition \ref{toric-plt-flag_definition}, and let $v_{j,k}$ and $c'_{j,k}$ be 
as in Definition \ref{toric-cone-sequence_definition}. Take any big 
$L\in\CaCl(X)\otimes_\Z\Q$. Then we have 
\[
S\left(L;Y_1\triangleright\cdots\triangleright Y_j\right)
=\sum_{k=j}^n c'_{j,k}\cdot S\left(L; V\left(v_{1,k}\right)\right)
\]
for any $1\leq j \leq n$. 
\end{thm}

\begin{proof}
Let 
\[
Y'_\bullet\colon X=Y'_0\supsetneq Y'_1\supsetneq\cdots\supsetneq Y'_n
\]
be the complete toric plt flag on $X$ associated to 
$v_{1,1},\dots,v_{1,n}\in N^0$. By Theorem \ref{toric-okounkov_thm} and
Corollary \ref{toric-okounkov-admissible_cor}, we have 
$\Delta_{Y_\bullet}\left(L\right)=f\left(\Delta_{Y'_\bullet}\left(L\right)\right)$, where 
the linear transform $f\colon \R^n\to\R^n$ corresponds to the matrix
\[
\begin{pmatrix}
c'_{1,1} & \cdots & c'_{1,n} \\ 
 & \ddots & \vdots \\
O & & c'_{n,n}
\end{pmatrix}
\begin{pmatrix}
\frac{l_1}{l_0} &  &  \\ 
 & \ddots &  \\
 & & \frac{l_n}{l_{n-1}}
\end{pmatrix}.
\]
Therefore we get 
\begin{eqnarray*}
S\left(L;Y_1\triangleright\cdots\triangleright Y_j\right)
&=&\sum_{k=j}^n c'_{j,k}\cdot \frac{l_k}{l_{k-1}}\cdot 
S\left(L;Y'_1\triangleright\cdots\triangleright Y'_k\right)\\
&=&\sum_{k=j}^n c'_{j,k}\cdot S\left(L; V\left(v_{1,k}\right)\right), 
\end{eqnarray*}
where the last equality follows from Proposition \ref{toric-S-admissible_proposition}. 
\end{proof}

\begin{example}[{see \cite[\S 3.2]{CFKP}}]\label{CFKP_example}
Set $X:= \pr^1\times\pr_{\pr^1}\left(\sO\oplus\sO(1)\right)$. 
Then $X$ corresponds to the fan $\Sigma_0$ in 
$N^0_\R=\Z^3\otimes_\Z\R$ such that the set of 
primitive generators of the rays in $\Sigma_0$ is 
\begin{eqnarray*}
\{v_{1,1}=(0,1,0),\quad v_{1,2}=(1,0,0), \quad v_{1,3}=(0,0,-1), \\
v_{1,4}=(0,0,1), \quad v_{1,5}=(0,-1,1), \quad v_{1,6}=(-1,0,0)\}, 
\end{eqnarray*}
and the set of $3$-dimensional cones in $\Sigma_0$ is 
\begin{eqnarray*}
\{ [1,2,3],\quad [1,2,4], \quad [2,3,5], \quad [2,4,5], \\
\, [1,3,6], \quad [1,4,6], \quad [3,5,6], \quad [4,5,6]\}
\end{eqnarray*}
where $[i,j,k]:=\Cone\left(v_{1,i}, v_{1,j}, v_{1,k}\right)$. 

Set $v_1:=(1,3,-1)=3v_{1,1}+1v_{1,2}+1v_{1,3}\in N^0$, let $\tilde{\Sigma}_0$ be 
the subdivision of $\Sigma_0$ at $v_1$, and let $\Sigma_1$ be the fan in 
$N^1:=N^0/\Z v_1$ as in Definition \ref{toric-plt-flag_definition}. 
The set of primitive generators of the rays of $\Sigma_1$ is 
\[
v_{2,1}:=\pi_1\left(v_{1,1}\right), \quad 
v_{2,2}:=\pi_1\left(v_{1,2}\right), \quad
v_{2,3}:=\pi_1\left(v_{1,3}\right).
\]
The lattice $N^1$ is generated by $v_{2,1}$ and $v_{2,2}$. Moreover, we have the 
equality $v_{2,3}=-3 v_{2,1}-v_{2,2}$. Thus the variety $Y_1$ in $\tilde{Y}_0$ 
in the sense of Definition \ref{toric-plt-flag_definition} is isomorphic to 
$\pr(1,1,3)$. 

Set $v_2:=v_{2,2}$. Then the subdivision $\tilde{\Sigma}_1$ of $\Sigma_1$ at $v_2$ 
is equal to $\Sigma_1$. Moreover, the fan $\Sigma_2$ 
in $N^2:=N^1/\Z v_2$ as in Definition  \ref{toric-plt-flag_definition} satisfies that, 
the set of primitive generators of the rays in $\Sigma_2$ is 
\[
v_{3,1}:=\pi_2\left(v_{2,1}\right), \quad v_{3,3}:=\frac{1}{3}\pi_2\left(v_{2,3}\right).
\]
Of course, we have $v_{3,1}=-v_{3,3}$. 

We set $v_3:=v_{3,3}$. Then the sequence $v_1$, $v_2$, $v_3$ determines a 
torus invariant complete plt flag 
$X=Y_0\triangleright Y_1\triangleright Y_2\triangleright Y_3$ over $X$. Moreover, 
we have 
\begin{eqnarray*}
m_{2,2}=m_{2,3}=1, \quad&& m_{3,3}=3, \\
\begin{pmatrix}
c_{1,1} & c_{1,2} & c_{1,3} \\ 
 & c_{2,2} & c_{2,3} \\
 & & c_{3,3}
\end{pmatrix}
=\begin{pmatrix}
3 & 1 & 1 \\ 
 & 1 & 0 \\
 & & 1
\end{pmatrix}, 
\quad&&
\begin{pmatrix}
c'_{1,1} & c'_{1,2} & c'_{1,3} \\ 
 & c'_{2,2} & c'_{2,3} \\
 & & c'_{3,3}
\end{pmatrix}
=\begin{pmatrix}
3 & 1 & 1 \\ 
 & 1 & 0 \\
 & & \frac{1}{3}
\end{pmatrix}.
\end{eqnarray*}
Let us consider the ample divisor 
\[
L\sim D=V\left(v_{1,4}\right)+2V\left(v_{1,5}\right)+V\left(v_{1,6}\right).
\]
Then we have 
\[
\phi\left(P_D\right)=\left\{(x_1,x_2,x_3)\in\R^3\,\,\middle|\,\,\begin{split}
x_1&\geq 0,\quad x_2\geq 0, \quad x_3\geq 0, \\
-x_3&\geq -1, \quad -x_1-x_3\geq -2,\quad -x_2\geq -1
\end{split}\right\}, 
\]
where $\phi$ is as in Theorem \ref{toric-okounkov_thm}. 
Thus, by Theorem \ref{toric-okounkov_thm}, we have 
\[
\Delta_{Y_\bullet}\left(L\right)=\Conv\left\{
\begin{split}
\left(0,0,0\right),\quad \left(6,0,0\right), \quad \left(4,0,\frac{1}{3}\right), 
\quad\left(1,0,\frac{1}{3}\right), \\
\left(1,1,0\right),\quad \left(7,1,0\right), \quad \left(5,1,\frac{1}{3}\right), \quad
\left(2,1,\frac{1}{3}\right)
\end{split}
\right\}.
\]
We can check that 
\[
S\left(L; V\left(v_{1,1}\right)\right)=\frac{7}{9}, \quad
S\left(L; V\left(v_{1,2}\right)\right)=\frac{1}{2}, \quad
S\left(L; V\left(v_{1,3}\right)\right)=\frac{4}{9}.
\]
Therefore, we have 
\begin{eqnarray*}
S\left(L; Y_1\right)&=&3\cdot\frac{7}{9}+1\cdot\frac{1}{2}+1\cdot
\frac{4}{9}=\frac{59}{18}, \\
S\left(L; Y_1\triangleright Y_2\right)&=&1\cdot\frac{1}{2}=\frac{1}{2}, \\
S\left(L; Y_1\triangleright Y_2\triangleright Y_3\right)&=&\frac{1}{3}\cdot\frac{4}{9}
=\frac{4}{27}.
\end{eqnarray*}
The values coincide with the values $S_L(G)$, $S\left(W_{\bullet,\bullet}^G; C\right)$ 
with $C=\bar{\alpha}_1$, and $S\left(W_{\bullet,\bullet,\bullet}^{G,C};Q\right)$ with 
$(C,Q)=(\bar{\alpha}_1,Q_{16})$ in \cite[\S 3.2]{CFKP}, respectively. 
\end{example}

\section{Locally divisorial series}\label{div_section}

From this section until the end of the article, 
we assume that the characteristic of $\Bbbk$ is zero. 
In this section, as a warm-up of \S \ref{dominant_section} and \S 
\ref{adequate_section}, we consider locally divisorial series. 
In this section, we assume that $X$ is an $n$-dimensional projective variety and 
$L_1,\dots,L_r\in\CaCl(X)\otimes_\Z\Q$. 

\begin{definition}\label{divisorial-series_definition}
Let $V_{\vec{\bullet}}$ be the Veronese equivalence class of a graded linear series 
on $X$ associated to $L_1,\dots,L_r$ which contains an ample series. 
\begin{enumerate}
\renewcommand{\theenumi}{\arabic{enumi}}
\renewcommand{\labelenumi}{(\theenumi)}
\item\label{divisorial-series_definition1}
We say that $V_{\vec{\bullet}}$ is \emph{a divisorial series} if 
there exist a representative $V_{m\vec{\bullet}}$ of $V_{\vec{\bullet}}$, 
an effective Cartier divisor $N$ on $X$, and a rational linear function 
$f\colon\R^r\to\R$ which satisfies 
$f\left(\Supp\left(V_{\vec{\bullet}}\right)\right)\subset\R_{\geq 0}$, such that 
\[
V_{m\vec{a}}=f\left(m\vec{a}\right)N+H^0\left(X, m\vec{a}\cdot\vec{L}
-f\left(m\vec{a}\right)N\right)
\]
holds for any $\vec{a}\in\Z_{\geq 0}^r\cap
\interior\left(\Supp\left(V_{\vec{\bullet}}\right)\right)$. 
\item\label{divisorial-series_definition2}
Assume that $V_{\vec{\bullet}}$ is a divisorial series as in 
\eqref{divisorial-series_definition1}. Take any birational morphism 
$\sigma\colon X'\to X$ between projective varieties. Let 
$\sigma_{\DIV}^*V_{\vec{\bullet}}^\circ$ be the Veronese equivalence class of 
the $\left(m\Z_{\geq 0}\right)^r$-graded linear series 
$\sigma_{\DIV}^*V_{m\vec{\bullet}}^\circ$ on $X'$ associated to 
$\sigma^*L_1,\dots,\sigma^*L_r$ defined by 
\[
\sigma_{\DIV}^*V_{m\vec{a}}^\circ:=\begin{cases}
f\left(m\vec{a}\right)\sigma^*N+H^0\left(X', m\vec{a}\cdot\sigma^*\vec{L}
-f\left(m\vec{a}\right)\sigma^*N\right) & \text{if }
\vec{a}\in\Z_{\geq 0}^r\cap\interior\left(\Supp\left(V_{\vec{\bullet}}\right)\right), \\
\Bbbk & \text{if }\vec{a}=\vec{0}, \\
0 & \text{otherwise}.
\end{cases}\]
Obviously, the series $\sigma_{\DIV}^*V_{m\vec{\bullet}}^\circ$ is an interior series 
with $\Supp\left(\sigma_{\DIV}^*V_{m\vec{\bullet}}^\circ\right)
=\Supp\left(V_{\vec{\bullet}}\right)$ which contains an ample series. We call it 
\emph{the interior divisorial pullback of $V_{\vec{\bullet}}$}. Note that, if $X$ is 
normal, then $\sigma_{\DIV}^*V_{m\vec{\bullet}}^\circ$ is nothing but 
the interior series of $\sigma^*V_{\vec{\bullet}}$. Moreover, 
by Lemma \ref{asymp-equiv_lemma} and \cite[Proposition 2.2.43]{L} 
(see also Example \ref{interior_example} \eqref{interior_example8}), 
the interior series of $\sigma^*V_{\vec{\bullet}}$ is asymptotically equivalent to 
$\sigma_{\DIV}^*V_{m\vec{\bullet}}^\circ$. 
\end{enumerate}
\end{definition}

For divisorial series $V_{\vec{\bullet}}$, it is easy to compute 
$S\left(V_{\vec{\bullet}}; E\right)$ for any prime divisor $E$ over $X$. 

\begin{proposition}[{see also \cite[\S 1.7]{FANO}, 
\cite[Theorem 4.8]{r3d28}}]\label{divisorial-series_proposition}
Let $V_{\vec{\bullet}}$ be a divisorial series as in 
Definition \ref{divisorial-series_definition} \eqref{divisorial-series_definition1}, 
and assume moreover that $V_{\vec{\bullet}}$ has bounded support. Let us set 
$\Delta_{\Supp}:=\Delta_{\Supp\left(V_{\vec{\bullet}}\right)}\subset\R_{\geq 0}^{r-1}$. 
Take any prime divisor $E$ over $X$. 
\begin{enumerate}
\renewcommand{\theenumi}{\arabic{enumi}}
\renewcommand{\labelenumi}{(\theenumi)}
\item\label{divisorial-series_proposition1}
We have 
\begin{eqnarray*}
\vol\left(V_{\vec{\bullet}}\right)
&=&\frac{(r-1+n)!}{n!}\int_{\Delta_{\Supp}}
\vol_X\left((1,\vec{x})\cdot\vec{L}-f(1,\vec{x})N\right)d\vec{x}, \\
S\left(V_{\vec{\bullet}};E\right)&=&
\frac{(r-1+n)!}{n!}\cdot\frac{1}{\vol\left(V_{\vec{\bullet}}\right)}
\int_{\Delta_{\Supp}}\Biggl(f(1,\vec{x})\ord_E N\cdot
\vol_X\left((1,\vec{x})\cdot\vec{L}-f(1,\vec{x})N\right)\\
&&+\int_0^\infty\vol_X\left((1,\vec{x})\cdot\vec{L}-f(1,\vec{x})N-t E\right)d t
\Biggr)d\vec{x}.
\end{eqnarray*}
\item\label{divisorial-series_proposition2}
We have 
\[
\vol\left(\sigma^*_{\DIV}V_{\vec{\bullet}}^\circ\right)
=\vol\left(V_{\vec{\bullet}}\right), \quad
S\left(\sigma^*_{\DIV}V_{\vec{\bullet}}^\circ;E\right)
=S\left(V_{\vec{\bullet}};E\right)
\]
for any birational morphism $\sigma\colon X'\to X$ between projective varieties. 
\end{enumerate}
\end{proposition}

\begin{proof}
The proof is essentially same as the proofs of \cite[Theorem 1.7.19]{FANO} and 
\cite[Theorem 4.8]{r3d28}. By \cite[Lemma 4.73]{Xu}, we may assume that 
$V_{m\vec{\bullet}}=V_{m\vec{\bullet}}^\circ$. Then we have 
\begin{eqnarray*}
&&\vol\left(V_{\vec{\bullet}}\right)\\
&=&\lim_{l\in m\Z_{>0}}\frac{h^0\left(V_{l,m\vec{\bullet}}\right)m^{r-1}}
{l^{r-1+n}/(r-1+n)!}\\
&=&\lim_{l\in m\Z_{>0}}\sum_{\vec{a}\in\Z_{\geq 0}^{r-1}\cap
\interior\left(\frac{l}{m}\Delta_{\Supp}\right)}\left(\frac{m}{l}\right)^{r-1}
\frac{(r-1+n)!}{n!}
\frac{h^0\left(X,l\left(L_1+\sum_{j=2}^r\frac{m}{l}a_j L_j-
f\left(1,\frac{m}{l}\vec{a}\right)N\right)\right)}{l^n/n!}\\
&=&\frac{(r-1+n)!}{n!}\int_{\Delta_{\Supp}}
\vol_X\left((1,\vec{x})\cdot\vec{L}-f(1,\vec{x})N\right)d\vec{x},
\end{eqnarray*}
and 
\begin{eqnarray*}
&&S\left(V_{\vec{\bullet}};E\right)\\
&=&\lim_{l\in m\Z_{>0}}S_l\left(V_{m\vec{\bullet}};E\right)\\
&=&\lim_{l\in m\Z_{>0}}\frac{l^{r-1+n}/(r-1+n)!}
{h^0\left(V_{l,m\vec{\bullet}}\right)m^{r-1}}\int_0^{T\left(V_{\vec{\bullet}};E\right)}
\frac{h^0\left(V_{l,m\vec{\bullet}}^{E,t}\right)m^{r-1}}{l^{r-1+n}/(r-1+n)!}dt\\
&=&\frac{1}{\vol\left(V_{\vec{\bullet}}\right)}\lim_{l\in m\Z_{>0}}
\sum_{\vec{a}\in\Z_{\geq 0}^{r-1}\cap\interior\left(\frac{l}{m}\Delta_{\Supp}\right)}
\left(\frac{m}{l}\right)^{r-1}\frac{(r-1+n)!}{n!}
\int_0^{T\left(V_{\vec{\bullet}};E\right)}\frac{\dim\sF_E^{l t} V_{l,m\vec{a}}}{l^n/n!}dt\\
&=&\frac{(r-1+n)!}{n!}\cdot\frac{1}{\vol\left(V_{\vec{\bullet}}\right)}
\int_{\Delta_{\Supp}}\Biggl(f(1,\vec{x})\ord_E N\cdot
\vol_X\left((1,\vec{x})\cdot\vec{L}-f(1,\vec{x})N\right)\\
&&+\int_{f\left(1,\vec{x}\right)\ord_E N}^{T\left(V_{\vec{\bullet}};E\right)}
\vol_X\left((1,\vec{x})\cdot\vec{L}-f(1,\vec{x})N-\left(t-f\left(1,\vec{x}\right)
\ord_E N\right)E\right)d t
\Biggr)d\vec{x}.
\end{eqnarray*}
\eqref{divisorial-series_proposition2} is trivial from 
\eqref{divisorial-series_proposition1} and \cite[Lemma 4.73]{Xu}. 
\end{proof}

We will rephrase Proposition \ref{divisorial-series_proposition}. To begin with, 
we prepare the following elementary lemma: 

\begin{lemma}\label{NB_lemma}
Let $X$ be a normal projective variety, let $E\subset X$ be a prime $\Q$-Cartier 
divisor on $X$ and let $D$ be a big $\R$-Cartier $\R$-divisor on $X$. 
Let $\sigma_E(D)$, $\tau_E(D)\in\R_{\geq 0}$ be the values in the sense of 
\cite[III, Definitions 1.1, 1.2 and 1.6]{N}, i.e., 
\begin{eqnarray*}
\sigma_E(D)&=&\inf\left\{\ord_E D'\,\,| D'\equiv D\text{ effective}\right\}, \\
\tau_E(D)&=&\max\left\{t\in\R_{\geq 0}\,\,|\,\,D-t E\text{ is pseudo-effective}\right\}. 
\end{eqnarray*} 
$($Note that $\tau_E(D)=T\left(D; E\right)$. 
Moreover, note that $\sigma_E(D)<\tau_E(D)$ holds. See \cite[III, Lemma 1.4 (4)]{N}.$)$
\begin{enumerate}
\renewcommand{\theenumi}{\arabic{enumi}}
\renewcommand{\labelenumi}{(\theenumi)}
\item\label{NB_lemma1}
If $E\not\subset\B_+(D)$, then we have $E\not\subset\B_+(D-t E)$ 
for any $t\in\left[0,\tau_E(D)\right)$, where $\B_+$ is the augmented base locus 
$($see \cite{ELMNP-b} and \cite{Birkar}$)$. 
\item\label{NB_lemma2}
If $\sigma_E(D)=0$, then we have $E\not\subset\B_+(D-t E)$ 
for any $t\in\left(0,\tau_E(D)\right)$.
\end{enumerate}
\end{lemma}

\begin{proof}
\eqref{NB_lemma1}
Since $E\not\subset\B_+(D)$, there exists an effective $\R$-divisor $D'$ 
with $D-D'$ ample such that $E\not\subset\Supp D'$ holds. 
On the other hand, for any $\tau\in\left(t,\tau_E(D)\right)$, there exists 
effective $D''\equiv D$ such that $\ord_E D''=\tau$. The effective divisor 
\[
\tilde{D}:=\frac{\tau-t}{\tau}D'+\frac{t}{\tau}D''
\]
satisfies that $D-\tilde{D}$ is ample and $\ord_E\tilde{D}=t$. Thus we have 
$E\not\subset\B_+(D-t E)$. 

\eqref{NB_lemma2}
Since $D$ is big, we may assume that $D=A+B$ with $A$ ample and effective, 
$B$ effective and
$\ord_E A=0$. Set $m_0:=\ord_E D=\ord_E B$. By \eqref{NB_lemma1}, 
it is enough to show that there exists a sequence $\{n_i\}_{i\in\Z_{>0}}$ of nonnegative 
real numbers with $\lim_{i\to\infty}n_i=0$ such that $E\not\subset\B_+(D-n_i E)$ 
holds for any $i\in\Z_{>0}$. 

Since $\sigma_E(D)=0$, for any $i\in\Z_{>0}$, there exists an effective $\R$-divisor 
$D_i\equiv D$ with $m_i:=\ord_E D_i\leq 1/i$. Set 
\[
n_i:=\frac{m_0+i m_i}{i+1}\leq\frac{m_0+1}{i+1}.
\]
Then, since 
\[
D-n_i E-\frac{1}{i+1}A\equiv\frac{1}{i+1}(B-m_0E)+\frac{i}{i+1}(D_i-m_i E), 
\]
we have $E\not\subset\B_+(D-n_i E)$ for any $i\in\Z_{>0}$. 
\end{proof}

\begin{proposition}\label{NB_proposition}
Under the notation in Proposition \ref{divisorial-series_proposition}
\eqref{divisorial-series_proposition2}, assume moreover $X'$ is normal and 
$E\subset X'$ is a prime $\Q$-Cartier divisor on $X'$. 
For any $\vec{x}\in\interior\left(\Delta_{\Supp}\right)$, let us consider the 
big $\R$-Cartier $\R$-divisor 
\[
M_{\vec{x}}:=\sigma^*\left(\left(1,\vec{x}\right)\cdot\vec{L}
-f\left(1,\vec{x}\right)N\right)
\]
on $X'$. 
\begin{enumerate}
\renewcommand{\theenumi}{\arabic{enumi}}
\renewcommand{\labelenumi}{(\theenumi)}
\item\label{NB_proposition1}
Set 
$t_0\left(\vec{x}\right):=\sigma_E\left(M_{\vec{x}}\right)$ and 
$t_1\left(\vec{x}\right):=\tau_E\left(M_{\vec{x}}\right)$ in the sense of 
\cite[III, Definitions 1.1, 1.2 and 1.6]{N}. Then we have 
$t_0\left(\vec{x}\right)<t_1\left(\vec{x}\right)$. Moreover, for any 
$u\in\left(t_0\left(\vec{x}\right), t_1\left(\vec{x}\right)\right)$, the big $\R$-divisor 
$M_{\vec{x}}-u E$ satisfies that $E\not\subset\B_+\left(M_{\vec{x}}-u E\right)$. 
Thus we can set the restricted volume 
\[
\vol_{X'|E}\left(M_{\vec{x}}-u E\right)\in\R_{>0}
\]
as in \cite[Corollary 4.27 (iii)]{LM}, \cite[Theorems A and B]{BFJ}. 
In particular, for any admissible flag $Y_\bullet$ on $X$ 
with $Y_1=E$, we have 
\[
\vol_{X|E}\left(M_{\vec{x}}-u E\right)=(n-1)!\cdot\vol_{\R^{n-1}}
\left(\Delta_{Y_\bullet}\left(M_{\vec{x}}\right)|_{\nu_1=u}\right), 
\]
where $\Delta_{Y_\bullet}\left(M_{\vec{x}}\right)|_{\nu_1=u}\subset
\Delta_{Y_\bullet}\left(M_{\vec{x}}\right)$ is the subset whose $1$-st coordinate 
is equal to $u$. 
\item\label{NB_proposition2}
We have 
\[
\vol_{X'}\left(M_{\vec{x}}\right)=n\cdot
\int_{t_0\left(\vec{x}\right)}^{t_1\left(\vec{x}\right)}
\vol_{X'|E}\left(M_{\vec{x}}-u E\right)d u
\]
and \[
\vol_{X'}\left(M_{\vec{x}}-t E\right)=\begin{cases}
\vol_{X'}\left(M_{\vec{x}}-t_0\left(\vec{x}\right) E\right) & 
\text{if }t\in\left[0,t_0\left(\vec{x}\right)\right], \\
n\cdot\int_t^{t_1\left(\vec{x}\right)}\vol_{X'|E}\left(M_{\vec{x}}-u E\right)d u &
\text{if }t\in\left[t_0\left(\vec{x}\right),t_1\left(\vec{x}\right)\right], \\
0 & \text{if }t\in\left[t_1\left(\vec{x}\right),\infty\right).
\end{cases}
\]
In particular, we have 
\begin{eqnarray*}
\vol\left(V_{\vec{\bullet}}\right)&=&\frac{(r-1+n)!}{(n-1)!}\int_{\Delta_{\Supp}}
\int_{t_0\left(\vec{x}\right)}^{t_1\left(\vec{x}\right)}
\vol_{X'|E}\left(M_{\vec{x}}-u E\right)d u d\vec{x}, \\
S\left(V_{\vec{\bullet}}; E\right)&=&\frac{(r-1+n)!}{(n-1)!\vol\left(V_{\vec{\bullet}}\right)}
\int_{\Delta_{\Supp}}
\int_{t_0\left(\vec{x}\right)}^{t_1\left(\vec{x}\right)}
\left(u+f\left(1,\vec{x}\right)\ord_E N\right)
\vol_{X'|E}\left(M_{\vec{x}}-u E\right)d u d\vec{x}.
\end{eqnarray*}
\end{enumerate}
\end{proposition}

\begin{proof}
\eqref{NB_proposition1} is an immediate consequence of Lemma \ref{NB_lemma} 
and \cite[Corollary 4.27]{LM}. 
\eqref{NB_proposition2} follows from \cite[Corollary 4.27]{LM} and Fubini's theorem. 
Note that the continuity of the function $\vol_{X'|E}\left(M_{\vec{x}}-u E\right)$ 
follows from \cite[Theorem A]{BFJ}. 
\end{proof}

\begin{corollary}\label{kojutsu_corollary}
Let $X'$ be an $n$-dimensional projective variety, let $L'$ be a big 
$\Q$-Cartier $\Q$-divisor on $X'$, let $\phi\colon X\to X'$ be 
a birational morphism with 
$X$ normal, and let $F\subset X$ be a prime $\Q$-Cartier divisor. 
Then, for any $x\in\left(\sigma_F\left(\phi^*L'\right), \tau_F\left(\phi^*L'\right)
\right)\cap\Q$, we have 
\[
\limsup_{m\to\infty}\frac{\dim\Image\left(\substack{
\phi^*H^0\left(X',m L'\right)\cap\left(m x F+H^0\left(X, \phi^*m L'-m x F\right)
\right)\\
\to H^0\left(F, \phi^*m L'|_F-m x F|_F\right)}\right)}{m^{n-1}/(n-1)!}
=\vol_{X|F}\left(\phi^*L'-x F\right).
\]
\end{corollary}

\begin{proof}
Set $V'_{\vec{\bullet}}:=H^0\left(\bullet L'\right)$. From the definition of the 
refinement $\left(\phi^*V'_{\vec{\bullet}}\right)^{(F)}$, 
the left hand side is equal to the volume of 
$\left(\phi^*V'_{\bullet(1,x)}\right)^{(F)}$, where 
$(1,x)\in\interior\left(\Supp\left(\phi^*V'_{\vec{\bullet}}\right)^{(F)}\right)$. 
We know that 
$\phi^*V'_{\vec{\bullet}}$ is asymptotically equivalent to 
$V_{\vec{\bullet}}:=H^0\left(\bullet\phi^*L'\right)$ by Example \ref{interior_example}
\eqref{interior_example8}. Thus, by 
Example \ref{interior_example} \eqref{interior_example6} and 
Proposition \ref{NB_proposition}, the left hand side is equal to 
$\vol\left(V_{\bullet(1,x)}^{(F)}\right)$. Take any admissible flag $Y_\bullet$ on $X$ 
with $Y_1=F$, and consider the admissible flag $Y'_\bullet$ on $F$ with 
$Y'_i:=Y_{i+1}$. By \cite[Theorem 4.21]{LM}, we have 
\[
\Delta_{Y'_\bullet}\left(V_{\bullet(1,x)}^{(F)}\right)
=\Delta_{Y'_\bullet}\left(V_{\vec{\bullet}}^{(F)}\right)\big|_{\nu_1=x}. 
\]
Therefore we get 
\[
\vol\left(V_{\bullet(1,x)}^{(F)}\right)=(n-1)!\cdot\vol_{\R^{n-1}}
\left(\Delta_{Y'_\bullet}\left(V_{\vec{\bullet}}^{(F)}\right)\big|_{\nu_1=x}\right)
=\vol_{X|F}\left(\phi^*L'-x F\right), 
\]
where the last equality follows from Example \ref{interior_example} 
\eqref{interior_example4} and Proposition \ref{NB_proposition} 
\eqref{NB_proposition1}. 
\end{proof}

The most typical examples of divisorial series are the complete linear series 
$H^0\left(\bullet L\right)$ with big $L\in\CaCl(X)\otimes_\Z\Q$. 
In this case, Proposition \ref{NB_proposition} can be rephrased as follows: 

\begin{corollary}[{cf.\ \cite[Proposition 3.12]{r3d28}}]\label{NB_corollary}
Assume that $L\in\CaCl(X)\otimes_\Z\R$ is big. Take any 
birational morphism $\sigma\colon X'\to X$ with $X'$ normal projective, and 
let $E$ be a prime $\Q$-Cartier divisor on $X'$. Set 
$t_0:=\sigma_E\left(\sigma^*L\right)$ and $t_1:=\tau_E\left(\sigma^*L\right)$ 
in the sense of \cite[III, Definitions 1.1, 1.2 and 1.6]{N}. 
\begin{enumerate}
\renewcommand{\theenumi}{\arabic{enumi}}
\renewcommand{\labelenumi}{(\theenumi)}
\item\label{NB_corollary1}
Take any $u\in(t_0, t_1)$. 

\begin{enumerate}
\renewcommand{\theenumii}{\roman{enumii}}
\renewcommand{\labelenumii}{\rm{(\theenumii)}}
\item\label{NB_corollary11}
We can define the restricted volume 
\[
\vol_{X'|E}\left(\sigma^*L-u E\right)\in\R_{>0}, 
\]
and we have 
\begin{eqnarray*}
\vol_X(L)=n\int_{t_0}^{t_1}\vol_{X'|E}\left(\sigma^*L-u E\right)d u. 
\end{eqnarray*}
Moreover, we have 
\begin{eqnarray*}
\frac{1}{\vol_X(L)}\int_0^\infty\vol_{X'}\left(\sigma^*L-t E\right)d t
=\frac{n}{\vol_X(L)}\int_{t_0}^{t_1}u \cdot\vol_{X'|E}\left(\sigma^*L-u E\right)d u. 
\end{eqnarray*}
We note that, if $L\in\CaCl(X)\otimes_\Z\Q$, then the above value is nothing but 
the value $S\left(L; E\right)$. 
\item\label{NB_corollary12}
Assume that $X$ is $\Q$-factorial. 
Let $\sigma^*L-u E=N_u+P_u$ be the Nakayama--Zariski decomposition in the sense 
of \cite[III, Definition 1.12]{N}.
Then, the restriction $P_u|_E\in\CaCl(E)\otimes_\Z\R$ is big. 
\end{enumerate}
\item\label{NB_corollary2}
Let $Y_\bullet$ be an admissible flag on $X'$ with $Y_1=E$, and let us set 
$\Delta:=\Delta_{Y_\bullet}\left(\sigma^*L\right)\subset\R_{\geq 0}^n$. 
Then we have $p_1(\Delta)=[t_0, t_1]\subset\R$, where $p_1\colon\R^n\to\R$ 
be the first projection. Moreover, for any $u\in(t_0, t_1)$, we have 
\[
\vol_{X'|E}\left(\sigma^*L-u E\right)=(n-1)!\vol_{\R^{n-1}}\left(\Delta_u\right), 
\]
where $\Delta_u:=p_1^{-1}\left(\{u\}\right)\subset\R^{n-1}$. 
In particular, if $L\in\CaCl(X)\otimes_\Z\Q$, 
then the value $S(L; E)$ is the first coordinate of the barycenter of $\Delta$. 
\end{enumerate}
\end{corollary}

\begin{proof}
\eqref{NB_corollary11} and \eqref{NB_corollary2} are immediate consequences of 
Proposition \ref{NB_proposition}. We consider \eqref{NB_corollary12}. 
By \cite[III, Lemma 1.4 (4) and Corollary 1.9]{N}, $E\not\subset\Supp N_u$ holds 
for any $u\in(t_0, t_1)$. 

Let us fix $u'\in(t_0,u)\cap\Q$ and $u''\in(u,t_1)\cap\Q$. We know that 
the $\R$-divisor 
\[
\frac{u''-u}{u''-u'}N_{u'}+\frac{u-u'}{u''-u'}N_{u''}-N_u
\]
is effective and the support does not conitain $E$. Since
\[
P_u|_E=\frac{u''-u}{u''-u'}\left(P_{u'}|_E\right)+\frac{u-u'}{u''-u'}\left(P_{u''}|_E\right)
+\left(\frac{u''-u}{u''-u'}N_{u'}+\frac{u-u'}{u''-u'}N_{u''}-N_u\right)\Big|_E, 
\]
we may assume that $u\in\Q$. 

Recall that both $\sigma_E$ and $\tau_E$ are continuous over the big cone 
$\BIG(X)$ (\cite[III, Lemma 1.7 (1)]{N}). Thus, we can take big 
$L_1,\dots,L_p\in\CaCl(X)\otimes_\Z\Q$ and $c_1,\dots,c_p\in\R_{>0}$ 
with $\sum_{i=1}^p c_i=1$ such that $L=\sum_{i=1}^p c_i L_i$ and 
$u\in\left(\sigma_E(\sigma^*L_i),\tau_E(\sigma^*L_i)\right)$ holds 
for any $1\leq i\leq p$. 
By the same argument as above, we get the bigness of $P_u|_E$ provided that 
the bigness of $P_\sigma(\sigma^*L_i-u E)|_E$ for all $1\leq i\leq p$. 
Thus we may further assume that $L\in\CaCl(X)\otimes_\Z\Q$. 

Let us fix $r_0\in\Z_{>0}$ such that $r_0(\phi^*L-u E)$ is Cartier. 
For any $m\in r_0\Z_{>0}$, set 
\[
W_m:=\Image\left(H^0\left(X, m(\phi^*L-u E)\right)\to 
H^0\left(E, m(\phi^*L-u E)|_E\right)\right).
\]
By Lemma \ref{NB_lemma}, we have $E\not\subset\B_+\left(\phi^*L-u E\right)$. 
Thus, by \cite{ELMNP} (see also \cite{BCL}, \cite{lopez}), 
for any $r\in r_0\Z_{>0}$, 
\[
a:=\limsup_{m\in r\Z_{>0}}\frac{\dim W_m}{m^{n-1}/(n-1)!}
\left(=\vol_{X|E}\left(\phi^*L-u E\right)\right)
\]
satisfies that $a\in\R_{>0}$ and is independent of $r\in r_0\Z_{>0}$. 

For any $i\in\Z_{>0}$, take an effective $\Q$-divisor $N^i$ on $X$ 
with $N^i\leq N_u$ such that 
\[
\ord_F\left(N_u-N^i\right)\leq \frac{1}{i}
\] 
holds for any prime divisor $F$ on $X$. 
Fix $r_i\in r_0\Z_{>0}$ such that $r_i N^i$ is Cartier. Then, since $N^i\leq N_u$, 
we have 
\[
m N^i+H^0\left(X, m\left(\phi^*L-u E- N^i\right)\right)
\xrightarrow{\sim}H^0\left(X, m\left(\phi^*L-u E\right)\right)
\]
for any $m\in r_i\Z_{>0}$. Thus we get 
\[
\dim W_m\leq h^0\left(E, m\left(\phi^*L-u E- N^i\right)|_E\right)
\]
for any $m\in r_i\Z_{>0}$. Therefore, 
\[
a\leq\limsup_{m\in r_i\Z_{>0}}
\frac{h^0\left(E, m\left(\phi^*L-u E- N^i\right)|_E\right)}{m^{n-1}/(n-1)!}
=\vol\left(\left(\phi^*L-u E- N^i\right)|_E\right)
\]
for any $i\in\Z_{>0}$. Since 
\[
a\leq\lim_{i\to\infty}\vol\left(\left(\phi^*L-u E- N^i\right)|_E\right)
=\vol\left(P_u|_E\right), 
\]
we get the assertion. 
\end{proof}

By applying Proposition \ref{bary_proposition}, we can estimate $S(L; E)$ in 
various situations. Here we give one specific example. 

\begin{example}\label{line-S_example}
Let us assume that $X$ is smooth with $n\geq 2$, 
and let $L$ be a very ample Cartier divisor on $X$ 
and let $Z\subset X$ be a line with respects to $L$, i.e., $(L\cdot Z)=1$. 
Consider the blowup $\sigma\colon \tilde{X}\to X$ along $Z$ and let 
$E\subset\tilde{X}$ be the exceptional divisor. 
Set $\tau:=\tau_E\left(\sigma^*L\right)$, $d:=\left(-K_X\cdot Z\right)$
 and $V_0:=(L^{\cdot n})$. We assume that $d\leq n$. (In fact, when 
 $X\not\simeq\pr^n$ and the characteristic of $\Bbbk$ is zero, then 
$d\leq n$ holds by \cite{CMSB, kebekus}.) Let 
$\Delta:=\Delta_{Y_\bullet}\left(\sigma^*L\right)\subset\R_{\geq 0}^n$ be the 
Okounkov body such that $Y_1=E$. Then the values $t_0$, $t_1$ in 
Proposition \ref{bary_proposition} is equal to $0$, $\tau$, respectively. 
Moreover, the function $g\colon[0,\tau]\to\R_{\geq 0}$ in Proposition 
\ref{bary_proposition} is equal to 
\[
\frac{1}{(n-1)!}\vol_{\tilde{X}|E}\left(\sigma^*L-x E\right)
\]
and 
\[
\int_0^\tau g(x)dx=\frac{V_0}{n!}=:V=\vol_{\R^n}\left(\Delta\right)
\]
holds by Corollary \ref{NB_corollary}. 
Note that $\sigma^*L-x E$ is nef for $x\in [0,1]$ since $L$ is very ample and 
$Z$ is a line. Thus we have 
\[
g(x)=\frac{1}{(n-1)!}\left((2-d)x^{n-1}+(n-1)x^{n-2}\right)
\]
for any $x\in [0,1]$. Thus we always have $V_0\geq n+2-d$. 
We note that the $1$st coordinate $b_1$ of the barycenter of $\Delta$ 
is equal to $S\left(L; E\right)$ by Corollary \ref{NB_corollary}. 
Let us apply Proposition \ref{bary_proposition} \eqref{bary_proposition1} 
for $e=1$ and 
\[
v=\lim_{x\to1-0}\frac{g(x)-g(1)}{x-1}=\frac{n-d}{(n-2)!}. 
\]
Note that $v=0$ if and only if $d=n$. The function $h_0$ in Proposition 
\ref{bary_proposition} satisfies that 
\[
h_0(x)=\frac{n+1-d}{(n-1)!}\left(\frac{(n-d)x+1}{n+1-d}\right)^{n-1}
\]
for $x\in[1,\tau]$. Let us fix $t\in (1,\tau]$ satisfying the condition $W\geq V$ in 
Proposition \ref{bary_proposition} \eqref{bary_proposition12}. The condition is 
equivalent to the condition 
\[\begin{cases}
2+n(t-1)\geq V_0 & \text{if }d=n, \\
1+\left(V_0-(n+2-d)\right)\frac{n-d}{(n+1-d)^2}\leq
\left(\frac{t(n-d)+1}{n+1-d}\right)^n & \text{if }d\neq n. 
\end{cases}\]
Then the value $s_1$ in Proposition \ref{bary_proposition} \eqref{bary_proposition12} 
is equal to 
\[\begin{cases}
\frac{V_0-2+n-t}{n-1} & \text{if }d=n, \\
1+\frac{n+1-d}{n-d}\left(\beta^{\frac{1}{n-1}}-1\right)
& \text{if }d\neq n,
\end{cases}\]
where 
\[
\beta:=\frac{(n-d)\left(V_0-(n+2-d)\right)+(n+1-d)^2}{(n+1-d)(t(n-d)+1)}.
\]
This implies that 
\[
h_0(s_1)=\frac{(n-d)\left(V_0-(n+2-d)\right)+(n+1-d)^2}{(n-1)!\cdot (t(n-d)+1)}.
\]
Therefore, by Proposition \ref{bary_proposition} \eqref{bary_proposition12}, 
the value $S(L; E)$ is bigger than or equal to 
\[\begin{cases}
\frac{1}{2(n+1)V_0}\left(\frac{n}{n-1}(V_0-2+n-t)^2+2t(V_0-2+n-t)
-(n-1)(n-2)+2t^2\right) & \text{if }d=n, \\
\frac{1}{V_0}\Biggl(\beta^{\frac{n+1}{n-1}}\frac{n(n+1-d)^2}{(n+1)(n-d)}
-\beta^{\frac{n}{n-1}}\frac{2(n+1-d)^2}{(n+1)(n-d)^2} & \text{if }d\neq n.\\
\quad\quad\quad
+\frac{(d-1)(2d-1)+n(2-3d-d^2+n+2d n-n^2)+t(n-d)+(n-d)((n-d)t-n)V_0}{
(n+1-d)(n-d)^2(n+1)}\Biggr) & 
\end{cases}\]
\end{example}

Now, we define the notion of locally divisorial series. 

\begin{definition}\label{locally-div_definition}
Let $V_{\vec{\bullet}}$ be the Veronese equivalence class of an 
$(m\Z_{\geq 0})^r$-graded linear series $V_{m\vec{\bullet}}$ on $X$ associated to 
$L_1,\dots,L_r$. The series $V_{\vec{\bullet}}$ is said to be 
a \emph{locally divisorial series} if there is 
a decomposition $\Delta_{\Supp}=\overline{\bigcup_{\lambda\in\Lambda}
\Delta_{\Supp}^{\langle\lambda\rangle}}$ as in Definition \ref{interior_definition} 
\eqref{interior_definition4} such that the restriction 
$V_{\vec{\bullet}}^{\langle\lambda\rangle}$ (in the sense of Definition 
\ref{interior_definition} \eqref{interior_definition4})
is a divisorial series for any $\lambda\in\Lambda$. 
\end{definition}

By Propositions \ref{divisorial-series_proposition} and 
\ref{decomposition_proposition} \eqref{decomposition_proposition2}, 
for locally divisorial series $V_{\vec{\bullet}}$ and a prime divisor $E$ over $X$, 
we can compute $S\left(V_{\vec{\bullet}};E\right)$.

Finally, we prepare the notion of the Zariski decomposition in a strong sense. 

\begin{definition}\label{ZDS_definition}
Assume that $X$ is normal and take a big $L\in\CaCl(X)\otimes_\Z\Q$. 
We say that $L$ admits \emph{the Zariski decomposition $L=N+P$ in a strong sense} 
if $N$ is an effective $\Q$-Cartier $\Q$-divisor on $X$, $P$ is a \emph{nef} 
and big $\Q$-divisor on $X$ such that 
\[
H^0\left(X, m L\right)=m N+ H^0\left(X, m P\right)
\]
holds for any sufficiently divisible $m\in\Z_{>0}$. 
(We only allow that $N$ is a $\Q$-divisor.) 
The decomposition must be the Nakayama--Zariski decomposition of $L$, and hence 
the decomposition is unique if exists. 
\end{definition}

\begin{example}\label{ZDS_example}
\begin{enumerate}
\renewcommand{\theenumi}{\arabic{enumi}}
\renewcommand{\labelenumi}{(\theenumi)}
\item\label{ZDS_example1}
Assume that $X$ is $\Q$-factorial. If $n\leq 2$ or if $X$ is a Mori dream space 
\cite{HK}, then any big $\Q$-divisor on $X$ admits the Zariski decomposition 
in a strong sense. See, for example, \cite[Example 2.19]{ELMNP}, 
\cite[\S 2.3]{okawa}. 
\item\label{ZDS_example2}
Assume that a big $L\in\CaCl(X)\otimes_\Z\Q$ admits the Zariski decomposition 
$L=N+P$ in a strong sense. Take any projective birational morphism 
$\sigma\colon\tilde{X}\to X$ with $\tilde{X}$ normal. Then the decomposition 
$\sigma^*L=\sigma^*N+\sigma^*P$ is the Zariski decomposition 
of $\sigma^*L$ in a strong sense. 
Moreover, if $E$ is an effective and $\sigma$-exceptional $\Q$-Cartier 
$\Q$-divisor on $\tilde{X}$, then 
$\sigma^*L+E=\left(\sigma^*N+E\right)+\sigma^*P$ is the Zariski decomposition 
of $\sigma^*L+E$ in a strong sense. 
\end{enumerate}
\end{example}

\section{Dominants of primitive flags}\label{dominant_section}

In this section, we assume that the characteristic of $\Bbbk$ is zero. 
In this section, we also assume that $X$ is an $n$-dimensional projective variety, 
and let
\[
Y_\bullet\colon X=Y_0\triangleright Y_1\triangleright\cdots\triangleright Y_j
\]
be a primitive flag over $X$ with the associated prime blowups 
$\sigma_k\colon\tilde{Y}_k\to Y_k$ for any $0\leq k\leq j-1$, and  
let $V_{\vec{\bullet}}$ be the Veronese equivalence class of a graded linear series 
on $X$ associated to $L_1,\dots,L_r\in\CaCl(X)\otimes_\Z\Q$ which has 
bounded support and contains an ample series.

\begin{definition}\label{dominant_definition}
\begin{enumerate}
\renewcommand{\theenumi}{\arabic{enumi}}
\renewcommand{\labelenumi}{(\theenumi)}
\item\label{dominant_definition1}
A \emph{dominant} of $Y_\bullet$ is a collection of projective birational morphisms 
$\left\{\gamma_k\colon\bar{Y}_k\to\tilde{Y}_k\right\}_{0\leq k\leq j-1}$
satisfying: 
\begin{enumerate}
\renewcommand{\theenumii}{\roman{enumii}}
\renewcommand{\labelenumii}{\rm{(\theenumii)}}
\item\label{dominant_definition11}
for any $0\leq k\leq j-1$, the variety $\bar{Y}_k$ is a normal projective variety 
and $\hat{Y}_{k+1}:=\left(\gamma_k\right)^{-1}_*Y_{k+1}$ is a $\Q$-Cartier prime 
divisor in $\bar{Y}_k$, and 
\item\label{dominant_definition12}
for any $1\leq k\leq j-1$, there exists a morphism 
$\phi_k\colon\bar{Y}_k\to\hat{Y}_k$ such that the following diagram 
\[
\xymatrix{
& \bar{Y}_k \ar[ld]_{\gamma_k} \ar[rd]^{\phi_k} & \\
\tilde{Y}_k \ar[rd]_{\sigma_k} & & \hat{Y}_k \ar[ld]^{\gamma_{k-1}|_{\hat{Y}_k}} \\
& Y_k &
}\]
makes commutative. 
\end{enumerate}
Obviously, the morphism $\phi_k$ is unique. 
We say that a dominant $\left\{\gamma_k\right\}_{0\leq k\leq j-1}$ is a \emph{smooth}
(resp., a \emph{$\Q$-factorial}) dominant of $Y_\bullet$ if $\bar{Y}_k$
is smooth (resp., $\Q$-factorial) for any $0\leq k\leq j-1$. 
\item\label{dominant_definition2}
Assume that both 
$\left\{\gamma_k\colon\bar{Y}_k\to\tilde{Y}_k\right\}_{0\leq k\leq j-1}$
and 
$\left\{\gamma'_k\colon\bar{Y}'_k\to\tilde{Y}_k\right\}_{0\leq k\leq j-1}$
are dominants of $Y_\bullet$. 
When a collection $\left\{\psi_k\colon\bar{Y}'_k\to\bar{Y}_k\right\}_{0\leq k\leq j-1}$ 
satisfies $\gamma'_k=\gamma_k\circ\psi_k$ for any $0\leq k\leq j-1$, then 
$\left\{\psi_k\right\}_{0\leq k\leq j-1}$ is said to be 
\emph{a morphism between dominants}
$\left\{\gamma'_k\right\}_{0\leq k\leq j-1}$ and 
$\left\{\gamma_k\right\}_{0\leq k\leq j-1}$ of $Y_\bullet$. 
\end{enumerate}
\end{definition}

When $\left\{\gamma_k\colon\bar{Y}_k\to\tilde{Y}_k\right\}_{0\leq k\leq j-1}$ is a 
dominant of $Y_\bullet$, 
from the definition of dominants, we have the following commutative diagram: 
\[
\xymatrix{
& {} \ar@{.}[dr] &&&&&&&\\
\ar@{.}[dr] & & \hat{Y}_3 \ar[dl]_{\gamma_2|_{\hat{Y}_3}} \ar@{}[r]|{\subset} & \bar{Y}_2 
\ar[dl]_{\gamma_2} \ar[dr]^{\phi_2}&&&&&\\
& Y_3 \ar@{}[r]|{\subset} & \tilde{Y}_2 \ar[dr]_{\sigma_2}& &\hat{Y}_2
\ar[dl]_{\gamma_1|_{\hat{Y}_2}} \ar@{}[r]|{\subset} & \bar{Y}_1 \ar[dl]_{\gamma_1}
\ar[dr]^{\phi_1}&&&\\
&&& Y_2 \ar@{}[r]|{\subset} & \tilde{Y}_1\ar[dr]_{\sigma_1}&&
\hat{Y}_1 \ar[dl]_{\gamma_0|_{\hat{Y}_1}} \ar@{}[r]|{\subset} & \bar{Y}_0
\ar[dl]_{\gamma_0}&\\
&&&&&Y_1\ar@{}[r]|{\subset} &\tilde{Y}_0\ar[dr]_{\sigma_0}&&\\
&&&&&&&Y_0\ar@{}[r]|{=} & X.
}\]

\begin{definition}\label{dominant-compare_definition}
Let $\left\{\gamma_k\colon\bar{Y}_k\to\tilde{Y}_k\right\}_{0\leq k\leq j-1}$ 
be a dominant 
of $Y_\bullet$. For any $1\leq l\leq k\leq j$, 
we define $d_{l,k}\in\Q_{\geq 0}$ and an effective $\Q$-Cartier $\Q$-divisor 
$\Sigma_{l,k}$ on $\bar{Y}_{k-1}$ with $\hat{Y}_k\not\subset\Supp\Sigma_{l,k}$ 
inductively as follows: 
\begin{itemize}
\item
$d_{l,l}:=0$, $\Sigma_{l,l}:=\gamma_{l-1}^*Y_l-\hat{Y}_l$, and 
\item
if $l<k$, then we set $d_{l,k}:=\ord_{\hat{Y}_k}\phi_{k-1}^*
\left(\Sigma_{l,k-1}|_{\hat{Y}_{k-1}}\right)$ and 
$\Sigma_{l,k}:=\phi_{k-1}^*
\left(\Sigma_{l,k-1}|_{\hat{Y}_{k-1}}\right)-d_{l,k}\hat{Y}_k$.
\end{itemize}
We also set 
\[
\begin{pmatrix}
1 & & & & \\
g_{1,2} & 1 & & O & \\
g_{1,3} & g_{2,3} & 1 & & \\
\vdots & & \ddots & \ddots & \\
g_{1,j} & g_{2,j} & \cdots & g_{j-1,j} & 1
\end{pmatrix}:=
\begin{pmatrix}
1 & & & & \\
d_{1,2} & 1 & & O & \\
d_{1,3} & d_{2,3} & 1 & & \\
\vdots & & \ddots & \ddots & \\
d_{1,j} & d_{2,j} & \cdots & d_{j-1,j} & 1
\end{pmatrix}^{-1}.
\]
We sometimes denote $d_{l,k}$, $\Sigma_{l,k}$ and $g_{l,k}$ by 
\[
d_{l,k}\left(\left\{\gamma_k\right\}_{0\leq k\leq j-1}\right),\quad 
\Sigma_{l,k}\left(\left\{\gamma_k\right\}_{0\leq k\leq j-1}\right)\quad \text{and}
\quad
g_{l,k}\left(\left\{\gamma_k\right\}_{0\leq k\leq j-1}\right).
\] 
\end{definition}

The following lemma is trivial. We omit the proof. 

\begin{lemma}\label{dominants-compare_lemma}
Let $\left\{\gamma_k\colon\bar{Y}_k\to\tilde{Y}_k\right\}_{0\leq k\leq j-1}$
and 
$\left\{\gamma'_k\colon\bar{Y}'_k\to\tilde{Y}_k\right\}_{0\leq k\leq j-1}$
be dominants of $Y_\bullet$, and let 
$\left\{\psi_k\colon\bar{Y}'_k\to\bar{Y}_k\right\}_{0\leq k\leq j-1}$ 
be a morphism between dominants
$\left\{\gamma'_k\right\}_{0\leq k\leq j-1}$ and 
$\left\{\gamma_k\right\}_{0\leq k\leq j-1}$. 
We set 
\begin{eqnarray*}
d_{l,k}:=d_{l,k}\left(\left\{\gamma_k\right\}_{0\leq k\leq j-1}\right),\quad 
d'_{l,k}:=d_{l,k}\left(\left\{\gamma'_k\right\}_{0\leq k\leq j-1}\right),\\ 
\Sigma_{l,k}:=\Sigma_{l,k}\left(\left\{\gamma_k\right\}_{0\leq k\leq j-1}\right),\quad
\Sigma'_{l,k}:=\Sigma_{l,k}\left(\left\{\gamma'_k\right\}_{0\leq k\leq j-1}\right).
\end{eqnarray*}
Moreover, for any $1\leq l\leq k\leq j$, let us inductively define $e_{l,k}\in\Q_{\geq 0}$
and an effective $\Q$-Cartier $\Q$-divisor $\Theta_{l,k}$ on $\bar{Y}'_{k-1}$ with 
$\hat{Y}'_k\not\subset\Supp\Theta_{l,k}$ such that 
\begin{itemize}
\item
$e_{l,l}:=0$, $\Theta_{l,l}:=\psi_{l-1}^*\hat{Y}_l-\hat{Y}'_l$, and 
\item
for any $1\leq l<k\leq j$, we set 
$e_{l,k}:=\ord_{\hat{Y}'_k}
\left({\phi'_k}^*\left(\Theta_{l,k-1}|_{\hat{Y}'_{k-1}}\right)\right)$ and 
$\Theta_{l,k}:={\phi'_k}^*\left(\Theta_{l,k-1}|_{\hat{Y}'_{k-1}}\right)-e_{l,k}\hat{Y}'_k$.
\end{itemize}
Then, for any $1\leq l\leq k\leq j$, we have 
\[
\Sigma'_{l,k}=\psi_{k-1}^*\Sigma_{l,k}+\Theta_{l,k}+\sum_{i=l+1}^k d_{l,i}\Theta_{i,k}
\]
and 
\[
d'_{l,k}=d_{l,k}+e_{l,k}+\sum_{i=l+1}^{k-1} d_{l,i}e_{i,k}. 
\]
In other words, we have 
\[
\begin{pmatrix}
1 & & & & \\
d'_{1,2} & 1 & & O & \\
d'_{1,3} & d'_{2,3} & 1 & & \\
\vdots & & \ddots & \ddots & \\
d'_{1,j} & d'_{2,j} & \cdots & d'_{j-1,j} & 1
\end{pmatrix}=
\begin{pmatrix}
1 & & & & \\
e_{1,2} & 1 & & O & \\
e_{1,3} & e_{2,3} & 1 & & \\
\vdots & & \ddots & \ddots & \\
e_{1,j} & e_{2,j} & \cdots & e_{j-1,j} & 1
\end{pmatrix}\begin{pmatrix}
1 & & & & \\
d_{1,2} & 1 & & O & \\
d_{1,3} & d_{2,3} & 1 & & \\
\vdots & & \ddots & \ddots & \\
d_{1,j} & d_{2,j} & \cdots & d_{j-1,j} & 1
\end{pmatrix}.
\]
\end{lemma}

\begin{proposition}\label{refinement-dominant_proposition}
Let $\left\{\gamma_k\colon\bar{Y}_k\to\tilde{Y}_k\right\}_{0\leq k\leq j-1}$ 
be a dominant of $Y_\bullet$. For any $1\leq k\leq j$, let us define the Veronese 
equivalence class $V_{\vec{\bullet}}^{\left(\hat{Y}_1>\dots>\hat{Y}_k\right)}$ on 
$\hat{Y}_k$ as follows: 
\begin{itemize}
\item
we set $V_{\vec{\bullet}}^{\left(\hat{Y}_1\right)}:=
\left(\gamma_0^*\sigma_0^*V_{\vec{\bullet}}\right)^{\left(\hat{Y}_1\right)}$, and 
\item
if $k\geq 2$, we set 
$V_{\vec{\bullet}}^{\left(\hat{Y}_1>\dots>\hat{Y}_k\right)}:=
\left(\phi_{k-1}^*V_{\vec{\bullet}}^{\left(\hat{Y}_1>\dots>\hat{Y}_{k-1}\right)}
\right)^{\left(\hat{Y}_k\right)}$.
\end{itemize}
Then, for any sufficiently divisible $m\in\Z_{>0}$ and for any 
$\left(\vec{a},b_1,\dots,b_k\right)\in\left(m\Z_{\geq 0}\right)^{r+k}$, 
the space $V_{\vec{a},b_1,\dots,b_k}^{\left(\hat{Y}_1>\dots>\hat{Y}_k\right)}$ 
is equal to 
\[\begin{cases}
0 & \text{if there exists }2\leq l\leq k\text{ such that } b'_l<0,\\
\left(\gamma_{k-1}|_{\hat{Y}_k}\right)^*
V_{\vec{a},b'_1,\dots,b'_k}^{\left(Y_1\triangleright\dots\triangleright Y_k\right)}
+\sum_{l=1}^k b'_l\left(\Sigma_{l,k}|_{\hat{Y}_k}\right) & \text{otherwise},
\end{cases}\]
where we set $b'_1,\dots,b'_k\in\Z$ as follows: 
\begin{eqnarray*}
b'_1&:=&b_1, \\
b'_l&:=&b_l-\sum_{i=1}^{l-1}d_{i,l}b'_i\quad(2\leq l\leq k).
\end{eqnarray*}
In other words, 
\[
\begin{pmatrix}b_1\\  \\ \vdots\\  \\  b_k\end{pmatrix}
=\begin{pmatrix}
1 & & & & \\
d_{1,2} & 1 & & O & \\
d_{1,3} & d_{2,3} & 1 & & \\
\vdots & & \ddots & \ddots & \\
d_{1,k} & d_{2,k} & \cdots & d_{k-1,k} & 1
\end{pmatrix}\begin{pmatrix}b'_1\\ \\ \vdots \\  \\  b'_k\end{pmatrix}.
\]
We note that $V_{\vec{\bullet}}^{\left(Y_1\triangleright\dots\triangleright Y_k\right)}$
is defined in Definition \ref{refinement-primitive_definition}. 
\end{proposition}

\begin{proof}
The proof is just applying \cite[Remark 3.17]{r3d28} inductively. 
By \cite[Remark 3.17]{r3d28}, we may assume that $k \geq 2$. We may also 
assume that $b'_1,\dots,b'_{k-1}\geq 0$. Since 
\[
\ord_{\hat{Y}_k}\left(\sum_{l=1}^{k-1}b'_l\phi_{k-1}^*\left(\Sigma_{l,k-1}|_{\hat{Y}_{k-1}}
\right)\right)=\sum_{l=1}^{k-1}b'_l d_{l,k}=b_k-b'_k
\]
and 
\[
\sum_{l=1}^{k-1}b'_l\phi_{k-1}^*\left(\Sigma_{l,k-1}|_{\hat{Y}_{k-1}}\right)
-\left(\sum_{l=1}^{k-1}b'_l d_{l,k}\right)\hat{Y}_k=\sum_{l=1}^{k-1}b'_l\Sigma_{l,k}, 
\]
we get the assertion by applying \cite[Remark 3.17]{r3d28}. 
\end{proof}

\begin{corollary}\label{refinement-dominant_corollary}
Under the assumption of Proposition \ref{refinement-dominant_proposition}, 
let us consider any general admissible flag 
\[
Z_\bullet\colon\hat{Y}_j=Z_0\supsetneq Z_1\supsetneq\cdots\supsetneq Z_{n-j}
\]
of $\hat{Y}_j$, where ``general'' means, the support of any $\Sigma_{l,j}|_{\hat{Y}_j}$
does not contain the point $Z_{n-j}$. 
Let \[
\Delta\subset\R_{\geq 0}^{(r-1+j)+(n-j)}\quad (\text{resp., }
\hat{\Delta}\subset\R_{\geq 0}^{(r-1+j)+(n-j)})
\]
be the Okounkov body of 
\[
\left(\gamma_{j-1}|_{\hat{Y}_j}\right)^*V_{\vec{\bullet}}^{\left(Y_1\triangleright
\cdots\triangleright Y_j\right)}\quad(\text{resp., } 
V_{\vec{\bullet}}^{\left(\hat{Y}_1>\dots>\hat{Y}_j\right)})
\]
associated to $Z_\bullet$. Let us set the linear transform 
\begin{eqnarray*}
f\colon\R^{r-1+j+n-j}&\to&\R^{r-1+j+n-j}\\
\begin{pmatrix}\vec{x}\\ \vec{y}'\\ \vec{z}\end{pmatrix}
&\mapsto&
\begin{pmatrix}\vec{x}\\ \vec{y}\\ \vec{z}\end{pmatrix}
\end{eqnarray*}
with $\vec{x}\in\R^{r-1}$, $\vec{y}$, $\vec{y}'\in\R^j$, $\vec{z}\in\R^{n-j}$
defined by 
\[
\begin{pmatrix}y_1\\ \\ \vdots \\  \\  y_j\end{pmatrix}
=\begin{pmatrix}
1 & & & & \\
d_{1,2} & 1 & & O & \\
d_{1,3} & d_{2,3} & 1 & & \\
\vdots & & \ddots & \ddots & \\
d_{1,j} & d_{2,j} & \cdots & d_{j-1,j} & 1
\end{pmatrix}\begin{pmatrix}y'_1\\ \\ \vdots \\  \\  y'_j\end{pmatrix}.
\]
Then we have the equality $\hat{\Delta}=f\left(\Delta\right)$. 
In particular, if $\left(\hat{b}_1,\dots,\hat{b}_{r-1+n}\right)\in\hat{\Delta}$ be the 
barycenter of $\hat{\Delta}$, then we have 
\[
\hat{b}_{r-1+k}=S\left(V_{\vec{\bullet}}; Y_1\triangleright\cdots\triangleright Y_k\right)
+\sum_{l=1}^{k-1} d_{l,k}
S\left(V_{\vec{\bullet}}; Y_1\triangleright\cdots\triangleright Y_l\right)
\]
for any $1\leq k\leq j$. 
\end{corollary}

\begin{proof}
The assertion $\hat{\Delta}=f\left(\Delta\right)$ is a direct consequence 
of Proposition \ref{refinement-dominant_proposition}. We already know in 
Remark \ref{S-T_remark} \eqref{S-T_remark3} that, the value 
$S\left(V_{\vec{\bullet}}; Y_1\triangleright\cdots\triangleright Y_k\right)$ is 
the $(r-1+k)$-th coordinate of the barycenter of $\Delta$. Since 
\[
\hat{b}_{r-1+k}=\frac{1}{\vol\left(\hat{\Delta}\right)}
\int_{\left(\vec{x},\vec{y},\vec{z}\right)\in\hat{\Delta}}y_kd\vec{x}d\vec{y}d\vec{z}
=\frac{1}{\vol\left(\Delta\right)}
\int_{\left(\vec{x},\vec{y}',\vec{z}\right)\in\hat{\Delta}}
\left(y'_k+\sum_{l=1}^{k-1}d_{l,k}y'_l\right)d\vec{x}d\vec{y}'d\vec{z},
\]
we get the assertion. 
\end{proof}

\section{Adequate dominants}\label{adequate_section}

In this section, we assume that the characteristic of $\Bbbk$ is zero. 
As in \S \ref{dominant_section}, in this section, 
we assume that $X$ is an $n$-dimensional projective variety, 
\[
Y_\bullet\colon X=Y_0\triangleright Y_1\triangleright\cdots\triangleright Y_j
\]
is a primitive flag over $X$ with the associated prime blowups 
$\sigma_k\colon\tilde{Y}_k\to Y_k$ for any $0\leq k\leq j-1$. 
We also fix a \emph{$\Q$-factorial} dominant 
$\left\{\gamma_k\colon\bar{Y}_k\to\tilde{Y}_k\right\}_{0\leq k\leq j-1}$ of $Y_\bullet$, 
and we follow the notation in Definition \ref{dominant_definition}. 
Let us fix a big $L\in\CaCl(X)\otimes_\Z\Q$ and set 
$V_{\vec{\bullet}}:=H^0(\bullet L)$. 

\begin{definition}\label{P-N_definition}
For any $1\leq k\leq j$, let us define 
\begin{itemize}
\item
a subset $\D_k\subset\R_{>0}^k$, 
\item
a big $\R$-divisor $P_{k-1}\left(x_1,\dots,x_k\right)$ on $\bar{Y}_{k-1}$ 
such that the restriction $P_{k-1}\left(x_1,\dots,x_k\right)|_{\hat{Y}_k}$ is big and 
$\hat{Y}_k\not\subset\B_+\left(P_{k-1}\left(x_1,\dots,x_k\right)\right)$ holds 
for any $\left(x_1,\dots,x_k\right)\in\D_k$, 
\item
an effective $\R$-divisor $N_{l-1,k-1}\left(x_1,\dots,x_l\right)$ on $\bar{Y}_{k-1}$
with $\hat{Y}_k\not\subset\Supp N_{l-1,k-1}\left(x_1,\dots,x_l\right)$
for any $1\leq l\leq k$ and for any $\left(x_1,\dots,x_l\right)\in\D_l$, 
\item
real numbers $u_k\left(x_1,\dots,x_{k-1}\right)$, 
$t_k\left(x_1,\dots,x_{k-1}\right)\in\R_{\geq 0}$ for any 
$\left(x_1,\dots,x_{k-1}\right)\in\D_{k-1}$ with 
$u_k\left(x_1,\dots,x_{k-1}\right)<t_k\left(x_1,\dots,x_{k-1}\right)$, 
\item
a real number $u_{l,k}\left(x_1,\dots,x_l\right)\in\R_{\geq 0}$ for any $1\leq l<k$
and for any $\left(x_1,\dots,x_l\right)\in\D_l$, and 
\item
a real number $v_k\left(x_1,\dots,x_{k-1}\right)\in\R_{\geq 0}$
for any $\left(x_1,\dots,x_{k-1}\right)\in\D_{k-1}$
\end{itemize}
inductively as follows: 
\begin{enumerate}
\renewcommand{\theenumi}{\arabic{enumi}}
\renewcommand{\labelenumi}{(\theenumi)}
\item\label{P-N_definition1}
We set $v_1:=0$, $u_1:=\sigma_{\hat{Y}_1}\left(\gamma_0^*\sigma_0^*L\right)$, 
$t_1:=\tau_{\hat{Y}_1}\left(\gamma_0^*\sigma_0^*L\right)$ and 
$\D_1:=(u_1, t_1)$. By Lemma \ref{NB_lemma}, we have $u_1<t_1$ and 
$\hat{Y}_1\not\subset\B_+\left(\gamma_0^*\sigma_0^*L-x_1\hat{Y}_1\right)$ 
for any $x_1\in(u_1,t_1)$. Let 
\[
\gamma_0^*\sigma_0^*L-x_1\hat{Y}_1=:N_{0,0}(x_1)+P_0(x_1)
\]
on $\bar{Y}_0$ be the Nakayama--Zariski decomposition 
in the sense of \cite[III, Definition 1.12]{N}. 
More precisely, we set 
\[
N_{0,0}(x_1):=N_\sigma\left(\gamma_0^*\sigma_0^*L-x_1\hat{Y}_1\right), \quad
P_0(x_1):=P_\sigma\left(\gamma_0^*\sigma_0^*L-x_1\hat{Y}_1\right).
\]
Then we know that $P_0(x_1)|_{\hat{Y}_1}$ is big, 
$\hat{Y}_1\not\subset\B_+\left(P_0(x_1)\right)$ and 
$\hat{Y}_1\not\subset\Supp\left(N_{0,0}(x_1)\right)$
for any $x_1\in\D_1$ by Lemma \ref{NB_lemma} and Corollary \ref{NB_corollary}. 

\item\label{P-N_definition2}
Assume that $k\geq 2$. Take any $\left(x_1,\dots,x_{k-1}\right)\in\D_{k-1}$. 
By an inductive assumption, the 
$\R$-divisor 
$\phi_{k-1}^*\left(P_{k-2}\left(x_1,\dots,x_{k-1}\right)|_{\hat{Y}_{k-1}}\right)$ on 
$\bar{Y}_{k-1}$ is a big $\R$-divisor. Let us set 
\begin{eqnarray*}
u_k\left(x_1,\dots,x_{k-1}\right)&:=&\sigma_{\hat{Y}_k}
\left(\phi_{k-1}^*\left(P_{k-2}\left(x_1,\dots,x_{k-1}\right)|_{\hat{Y}_{k-1}}\right)\right), 
\\
t_k\left(x_1,\dots,x_{k-1}\right)&:=&\tau_{\hat{Y}_k}
\left(\phi_{k-1}^*\left(P_{k-2}\left(x_1,\dots,x_{k-1}\right)|_{\hat{Y}_{k-1}}\right)\right).
\end{eqnarray*}
By Lemma \ref{NB_lemma}, we have 
$u_k\left(x_1,\dots,x_{k-1}\right)<t_k\left(x_1,\dots,x_{k-1}\right)$. 
We set 
\begin{eqnarray*}
\D_k&:=&\left\{\left(x_1,\dots,x_k\right)\in\R_{>0}^k\,\,\bigg|\,\,
\begin{split}\left(x_1,\dots,x_{k-1}\right)&\in\D_{k-1}\text{ and }\\
x_k&\in\left(u_k\left(x_1,\dots,x_{k-1}\right),
t_k\left(x_1,\dots,x_{k-1}\right)\right)\end{split}\right\}\\
&=&\left\{\left(x_1,\dots,x_k\right)\in\R_{>0}^k\,\,\Big|\,\,
x_l\in\left(u_l\left(x_1,\dots,x_{l-1}\right),
t_l\left(x_1,\dots,x_{l-1}\right)\right)\text{ for all }1\leq l\leq k\right\}.
\end{eqnarray*}
Moreover, for any $x_k\in\left(u_k\left(x_1,\dots,x_{k-1}\right), 
t_k\left(x_1,\dots,x_{k-1}\right)\right)$, by Lemma \ref{NB_lemma} and Corollary 
\ref{NB_corollary}, the Nakayama--Zariski decomposition 
\[
\phi_{k-1}^*\left(P_{k-2}\left(x_1,\dots,x_{k-1}\right)|_{\hat{Y}_{k-1}}\right)
-x_k\hat{Y}_k=:N_{k-1,k-1}\left(x_1,\dots,x_k\right)
+P_{k-1}\left(x_1,\dots,x_k\right)
\]
on $\bar{Y}_{k-1}$ satisfies 
$P_{k-1}\left(x_1,\dots,x_k\right)|_{\hat{Y}_k}$ is big, 
$\hat{Y}_k\not\subset\B_+\left(P_{k-1}\left(x_1,\dots,x_k\right)\right)$ and
$\hat{Y}_k\not\subset\Supp\left(N_{k-1,k-1}\left(x_1,\dots,x_k\right)\right)$. 

For any $1\leq l<k$ and for any $\left(x_1,\dots,x_l\right)\in\D_l$, we have 
already defined the effective $\R$-divisor 
$N_{l-1,k-2}\left(x_1,\dots,x_l\right)$ on $\bar{Y}_{k-2}$ with $\hat{Y}_{k-1}
\not\subset\Supp\left(N_{l-1,k-2}\left(x_1,\dots,x_l\right)\right)$
by inductive assumption. We set 
\begin{eqnarray*}
u_{l,k}\left(x_1,\dots,x_l\right)&:=&\ord_{\hat{Y}_k}\left(\phi_{k-1}^*
\left(N_{l-1,k-2}\left(x_1,\dots,x_l\right)|_{\hat{Y}_{k-1}}\right)\right), \\
N_{l-1,k-1}\left(x_1,\dots,x_l\right)&:=&
\phi_{k-1}^*\left(N_{l-1,k-2}\left(x_1,\dots,x_l\right)|_{\hat{Y}_{k-1}}\right)
-u_{l,k}\left(x_1,\dots,x_l\right)\hat{Y}_k.
\end{eqnarray*}
Finally, we set 
\[
v_k\left(x_1,\dots,x_{k-1}\right):=\sum_{l=1}^{k-1}u_{l,k}\left(x_1,\dots,x_l\right)
\]
for any $2\leq k$ and for any $\left(x_1,\dots,x_{k-1}\right)\in\D_{k-1}$. 
\end{enumerate}
\end{definition}

From now on, instead 
$\phi_{k-1}^*\left(P_{k-2}\left(x_1,\dots,x_{k-1}\right)|_{\hat{Y}_{k-1}}\right)$ and 
$\phi_{k-1}^*\left(N_{l-1,k-2}\left(x_1,\dots,x_l\right)|_{\hat{Y}_{k-1}}\right)$, 
we simply write $P_{k-2}\left(x_1,\dots,x_{k-1}\right)|_{\bar{Y}_{k-1}}$ and 
$N_{l-1,k-2}\left(x_1,\dots,x_l\right)|_{\bar{Y}_{k-1}}$, etc. 

\begin{definition}\label{tilde_definition}
Let us define 
\begin{itemize}
\item
a subset $\tilde{\D}_k\subset\R_{>0}^k$ for any $1\leq k\leq j$, 
\item
a real number 
$\tilde{u}_{l,k}\left(y_1,\dots,y_l\right)$, 
for any $1\leq l<k\leq j$ and for any $\left(y_1,\dots,y_l\right)\in\tilde{\D}_l$, and 
\item
real numbers $\tilde{v}_k\left(y_1,\dots,y_{k-1}\right)$, 
$\tilde{u}_k\left(y_1,\dots,y_{k-1}\right)$, 
$\tilde{t}_k\left(y_1,\dots,y_{k-1}\right)\in\R_{\geq 0}$ for any 
$2\leq k\leq j$ and for any 
$\left(y_1,\dots,y_{k-1}\right)\in\tilde{\D}_{k-1}$
\end{itemize}
inductively as follows: 
\begin{enumerate}
\renewcommand{\theenumi}{\arabic{enumi}}
\renewcommand{\labelenumi}{(\theenumi)}
\item\label{tilde_definition1}
We set $\tilde{\D}_1:=\D_1=(u_1,t_1)$ and $\tilde{v}_2:=v_2$. 
\item\label{tilde_definition2}
For $1\leq l<k\leq j$ and for any $\left(y_1,\dots,y_l\right)\in\tilde{\D}_l$, we set 
\[
\tilde{u}_{l,k}\left(y_1,\dots,y_l\right):=u_{l,k}\left(y_1,y_2-\tilde{v}_2(y_1),
y_3-\tilde{v}_3(y_1,y_2),\dots,y_l-\tilde{v}_l\left(y_1,\dots,y_{l-1}\right)\right). 
\]
For any $2\leq k\leq j$ and for any $\left(y_1,\dots,y_{k-1}\right)\in\tilde{\D}_{k-1}$, 
we set 
\begin{eqnarray*}
\tilde{v}_k\left(y_1,\dots,y_{k-1}\right)&:=&\sum_{l=1}^{k-1}\tilde{u}_{l,k}
\left(y_1,\dots,y_l\right), \\
\tilde{u}_k\left(y_1,\dots,y_{k-1}\right)&:=&u_k\left(y_1,y_2-\tilde{v}_2(y_1),
y_3-\tilde{v}_3(y_1,y_2),\dots,
y_{k-1}-\tilde{v}_{k-1}\left(y_1,\dots,y_{k-2}\right)\right),\\
\tilde{t}_k\left(y_1,\dots,y_{k-1}\right)&:=&t_k\left(y_1,y_2-\tilde{v}_2(y_1),
y_3-\tilde{v}_3(y_1,y_2),\dots,y_{k-1}-\tilde{v}_{k-1}\left(y_1,\dots,y_{k-2}\right)\right).
\end{eqnarray*}
We define 
\begin{eqnarray*}
\tilde{\D}_k&:=&\left\{\left(y_1,\dots,y_k\right)\in\R_{>0}^k\,\,\bigg|\,\,
\begin{split}
\left(y_1,\dots,y_{k-1}\right)&\in\tilde{\D}_{k-1} \text{ and}\\
y_k-\tilde{v}_k\left(y_1,\dots,y_{k-1}\right)
&\in\left(\tilde{u}_k\left(y_1,\dots,y_{k-1}\right),\tilde{t}_k\left(y_1,\dots,y_{k-1}\right)
\right)
\end{split}\right\}\\
&=&\left\{\left(y_1,\dots,y_k\right)\in\R_{>0}^k\,\,\bigg|\,\,
\begin{split}y_1&\in(u_1,t_1) \text{ and, for any }2\leq l\leq k, \\
y_l-\tilde{v}_l\left(y_1,\dots,y_{l-1}\right)
&\in\left(\tilde{u}_l\left(y_1,\dots,y_{l-1}\right),\tilde{t}_l\left(y_1,\dots,y_{l-1}\right)
\right)\end{split}\right\}.
\end{eqnarray*}
\end{enumerate}
Moreover, we set 
\[
\tilde{P}_{k-1}\left(y_1,\dots,y_k\right)
:=P_{k-1}\left(y_1,y_2-\tilde{v}_2(y_1),\dots,y_k-\tilde{v}_k\left(y_1,\dots,y_{k-1}
\right)\right)
\]
for any $1\leq k\leq j$ and 
for any $\left(y_1,\dots,y_k\right)\in\tilde{\D}_k$, and 
\[
\tilde{N}_{l-1,k-1}\left(y_1,\dots,y_l\right)
:=N_{l-1,k-1}\left(y_1,y_2-\tilde{v}_2(y_1),\dots,y_l-\tilde{v}_l\left(y_1,\dots,y_{l-1}
\right)\right)
\]
for any $1\leq l\leq k\leq j$ and for any $\left(y_1,\dots,y_l\right)\in\tilde{\D}_l$. 
\end{definition}

The following lemma is trivial from the definition. 

\begin{lemma}\label{tilde_lemma}
\begin{enumerate}
\renewcommand{\theenumi}{\arabic{enumi}}
\renewcommand{\labelenumi}{(\theenumi)}
\item\label{tilde_lemma1}
For any $2\leq k\leq j$ and for any $\left(x_1,\dots,x_{k-1}\right)\in\D_{k-1}$, 
we have 
\begin{eqnarray*}
u_k\left(x_1,\dots,x_{k-1}\right)&=&\tilde{u}_k
\left(x_1,x_2+v_2\left(x_1\right), x_3+v_3\left(x_1,x_2\right),\dots,x_{k-1}
+v_{k-1}\left(x_1,\dots,x_{k-2}\right)\right), \\
t_k\left(x_1,\dots,x_{k-1}\right)&=&\tilde{t}_k
\left(x_1,x_2+v_2\left(x_1\right), x_3+v_3\left(x_1,x_2\right),\dots,x_{k-1}
+v_{k-1}\left(x_1,\dots,x_{k-2}\right)\right), \\
v_k\left(x_1,\dots,x_{k-1}\right)&=&\tilde{v}_k
\left(x_1,x_2+v_2\left(x_1\right), x_3+v_3\left(x_1,x_2\right),\dots,x_{k-1}
+v_{k-1}\left(x_1,\dots,x_{k-2}\right)\right).
\end{eqnarray*}
\item\label{tilde_lemma2}
For any $1\leq k\leq j$, the map 
\begin{eqnarray*}
\tilde{\D}_k &\to&\D_k \\
\left(y_1,\dots,y_k\right)&\mapsto&
\left(y_1,y_2-\tilde{v}_2\left(y_1\right),y_3-\tilde{v}_3
\left(y_1,y_2\right),\dots,y_k-\tilde{v}_k\left(y_1,\dots,y_{k-1}
\right)\right)
\end{eqnarray*}
is a bijection, and the inverse is given by 
\[
\left(x_1,\dots,x_k\right)\mapsto\left(x_1,x_2+v_2\left(x_1\right), 
x_3+v_3\left(x_1,x_2\right),\dots,x_k+v_k\left(x_1,\dots,x_{k-1}\right)\right). 
\]
$($We will see later that the map is a homeomorphism.$)$
\item\label{tilde_lemma3}
For any $1\leq k\leq j$ and for any $\left(y_1,\dots,y_k\right)\in\tilde{\D}_k$, 
we have 
\begin{eqnarray*}
&&L|_{\bar{Y}_{k-1}}-y_1\hat{Y}_1|_{\bar{Y}_{k-1}}-\cdots-
y_{k-1}\hat{Y}_{k-1}|_{\bar{Y}_{k-1}}-y_k\hat{Y}_k\\
&\sim_\R&\tilde{P}_{k-1}\left(y_1,\dots,y_k\right)
+\sum_{l=1}^k\tilde{N}_{l-1,k-1}\left(y_1,\dots,y_l\right)
\end{eqnarray*}
on $\bar{Y}_{k-1}$. 
\end{enumerate}
\end{lemma}

\begin{proof}
We only prove the assertion \eqref{tilde_lemma3} by induction on $k$, 
since the other assertions are trivial from the definition. 
For any $y_1\in\left(u_1,t_1\right)$, since
\[
L|_{\bar{Y}_0}-y_1\hat{Y}_1\sim_\R P_0\left(y_1\right)+N_{0,0}\left(y_1\right)
=\tilde{P}_0\left(y_1\right)+\tilde{N}_{0,0}\left(y_1\right),
\]
the assertion is true for $k=1$. Assume that the assertion is true in $k$ with $k<j$. 
For any $\left(y_1,\dots,y_{k+1}\right)\in\tilde{\D}_{k+1}$, since 
$y_{k+1}-\tilde{v}_{k+1}\left(y_1,\dots,y_k\right)\in\left(
\tilde{u}_{k+1}\left(y_1,\dots,y_k\right), \tilde{t}_{k+1}\left(y_1,\dots,y_k\right)\right)$, 
we have 
\[
\tilde{P}_{k-1}\left(y_1,\dots,y_k\right)|_{\bar{Y}_k}-\left(y_{k+1}-\tilde{v}_{k+1}
\left(y_1,\dots,y_k\right)\right)\hat{Y}_{k+1}\sim_\R
\tilde{P}_k\left(y_1,\dots,y_{k+1}\right)+\tilde{N}_{k,k}\left(y_1,\dots,y_{k+1}\right). 
\]
On the other hand, for any $1\leq l\leq k$, we have 
\[
\tilde{N}_{l-1,k-1}\left(y_1,\dots,y_l\right)|_{\bar{Y}_k}
=\tilde{N}_{l-1,k}\left(y_1,\dots,y_l\right)+\tilde{u}_{l,k+1}\left(y_1,\dots,y_l\right)
\hat{Y}_{k+1}. 
\]
Therefore, 
\begin{eqnarray*}
&&L|_{\bar{Y}_k}-y_1\hat{Y}_1|_{\bar{Y}_k}-\cdots-y_{k+1}\hat{Y}_{k+1}\\
&\sim_\R&
\tilde{P}_{k-1}\left(y_1,\dots,y_k\right)|_{\bar{Y}_k}-\left(y_{k+1}-\tilde{v}_{k+1}
\left(y_1,\dots,y_k\right)\right)\hat{Y}_{k+1}\\
&+&\sum_{l=1}^k\tilde{N}_{l-1,k-1}\left(y_1,\dots,y_l\right)|_{\bar{Y}_k}-
\tilde{v}_{k+1}\left(y_1,\dots,y_k\right)\hat{Y}_{k+1}\\
&=&\tilde{P}_k\left(y_1,\dots,y_{k+1}\right)+\sum_{l=1}^{k+1}\tilde{N}_{l-1,k}
\left(y_1,\dots,y_l\right).
\end{eqnarray*}
Thus the assertion is also true in $k+1$. 
\end{proof}

The following proposition is technically important in this section. 

\begin{proposition}\label{tilde_proposition}
Take any $1\leq k\leq j$. 
\begin{enumerate}
\renewcommand{\theenumi}{\arabic{enumi}}
\renewcommand{\labelenumi}{(\theenumi)}
\item\label{tilde_proposition1}
The subset $\tilde{\D}_k\subset\R_{>0}^k$ is an open convex set. 
\item\label{tilde_proposition2}
If $k\geq 2$, then all of the functions $\tilde{v}_k$, $\tilde{u}_k+\tilde{v}_k$ 
and $-\tilde{t}_k$ 
from $\tilde{\D}_{k-1}$ to $\R$ are convex functions. In particular, they are 
continuous functions. 
\item\label{tilde_proposition3}
For any $1\leq l\leq k$, the divisors $\tilde{N}_{l-1,k-1}$ behave convex in 
$\tilde{\D}_l$. More precisely, for any $\left(y_1,\dots,y_l\right)$, 
$\left(y'_1,\dots,y'_l\right)\in\tilde{\D}_l$ and for any $t\in(0,1)$, if we set 
\[
\left(y''_1,\dots,y''_l\right):=t\cdot\left(y_1,\dots,y_l\right)+(1-t)\cdot
\left(y'_1,\dots,y'_l\right),
\]
then we have 
\[
t\tilde{N}_{l-1,k-1}\left(y_1,\dots,y_l\right)+(1-t)\tilde{N}_{l-1,k-1}
\left(y'_1,\dots,y'_l\right)\geq\tilde{N}_{l-1,k-1}\left(y''_1,\dots,y''_l\right). 
\]
\item\label{tilde_proposition4}
For any $1\leq l\leq k$ and for any $\vec{y}\in\tilde{\D}_l$, there exists an open 
neighborhood $U\subset\tilde{\D}_l$ of $\vec{y}$ such that the possibility of 
irreducible components of the support of $\tilde{N}_{l-1,k-1}\left(\vec{y}'\right)$
for $\vec{y}'\in U$ is at most finite. In particular, together with 
\eqref{tilde_proposition3} and Lemma \ref{tilde_lemma} \eqref{tilde_lemma3}, 
the $\R$-divisor $\tilde{P}_{k-1}\left(y_1,\dots,y_k\right)$ moves continuously 
in the space 
$N^1\left(\bar{Y}_{k-1}\right)$ over $\left(y_1,\dots,y_k\right)\in\tilde{\D}_k$. 
\end{enumerate}
\end{proposition}

\begin{proof}
We prove by induction on $k$. If $k=1$, then the assertions are trivial. 
Assume that $k\geq 2$. We firstly show that $\tilde{\D}_k$ is a convex set. 
Take any $\vec{y}=\left(y_1,\dots,y_k\right)$, $\vec{y}'=\left(y'_1,\dots,y'_k\right)
\in\tilde{\D}_k$ and any $t\in(0,1)$. Set 
\[
\vec{y}''=\left(y''_1,\dots,y''_k\right):=t\cdot\vec{y}+(1-t)\cdot
\vec{y}'
\]
as in \eqref{tilde_proposition3}. 
Let us set 
\[
L\left(y_1,\dots,y_{k-1}\right):=L|_{\bar{Y}_{k-2}}-y_1\hat{Y}_1|_{\bar{Y}_{k-2}}-\cdots-
y_{k-1}\hat{Y}_{k-1}
\]
for simplicity. By Lemma \ref{tilde_lemma} \eqref{tilde_lemma3}, we have 
\begin{equation}\tag{*}
\begin{split}
L\left(y''_1,\dots,y''_{k-1}\right) & \sim_\R
\tilde{P}_{k-2}\left(y''_1,\dots,y''_{k-1}\right)+\sum_{l=1}^{k-1}\tilde{N}_{l-1,k-2}
\left(y''_1,\dots,y''_l\right) \\
&\sim_\R t\left(\tilde{P}_{k-2}\left(y_1,\dots,y_{k-1}\right)
+\sum_{l=1}^{k-1}\tilde{N}_{l-1,k-2}
\left(y_1,\dots,y_l\right)\right)\\
&+(1-t)\left(\tilde{P}_{k-2}\left(y'_1,\dots,y'_{k-1}\right)
+\sum_{l=1}^{k-1}\tilde{N}_{l-1,k-2}
\left(y'_1,\dots,y'_l\right)\right).
\end{split}
\end{equation}
Moreover, by induction, we may assume that 
\begin{equation}\tag{**}
\begin{split}
t\tilde{N}_{l-1,k-2}\left(y_1,\dots,y_l\right)
+(1-t)\tilde{N}_{l-1,k-2}\left(y'_1,\dots,y'_l\right)
-\tilde{N}_{l-1,k-2}\left(y''_1,\dots,y''_l\right)\geq 0.
\end{split}
\end{equation}
Therefore, we have 
\begin{eqnarray*}
&&\tilde{u}_k\left(y''_1,\dots,y''_{k-1}\right)+\tilde{v}_k\left(y''_1,\dots,y''_{k-1}\right)\\
&=&\sigma_{\hat{Y}_k}\left(\tilde{P}_{k-2}\left(y''_1,\dots,y''_{k-1}\right)|_{\bar{Y}_{k-1}}
\right)+\ord_{\hat{Y}_k}\left(\sum_{l=1}^{k-1}\tilde{N}_{l-1,k-2}\left(y''_1,\dots,y''_l
\right)|_{\bar{Y}_{k-1}}\right)\\
&\leq&\sigma_{\hat{Y}_k}\left(
t\tilde{P}_{k-2}\left(y_1,\dots,y_{k-1}\right)|_{\bar{Y}_{k-1}}
+(1-t)\tilde{P}_{k-2}\left(y'_1,\dots,y'_{k-1}\right)|_{\bar{Y}_{k-1}}
\right)\\
&&+\ord_{\hat{Y}_k}\left(t\sum_{l=1}^{k-1}\tilde{N}_{l-1,k-2}\left(y_1,\dots,y_l
\right)|_{\bar{Y}_{k-1}}
+(1-t)\sum_{l=1}^{k-1}\tilde{N}_{l-1,k-2}\left(y'_1,\dots,y'_l
\right)|_{\bar{Y}_{k-1}}\right)\\
&\leq&
t\sigma_{\hat{Y}_k}\left(\tilde{P}_{k-2}\left(y_1,\dots,y_{k-1}\right)|_{\bar{Y}_{k-1}}\right)
+t\ord_{\hat{Y}_k}\left(\sum_{l=1}^{k-1}\tilde{N}_{l-1,k-2}\left(y_1,\dots,y_l
\right)|_{\bar{Y}_{k-1}}\right)
\\
&&+(1-t)\sigma_{\hat{Y}_k}\left(
\tilde{P}_{k-2}\left(y'_1,\dots,y'_{k-1}\right)|_{\bar{Y}_{k-1}}\right)
+(1-t)\ord_{\hat{Y}_k}\left(\sum_{l=1}^{k-1}\tilde{N}_{l-1,k-2}\left(y'_1,\dots,y'_l
\right)|_{\bar{Y}_{k-1}}\right)\\
&=&t\left(\tilde{u}_k\left(y_1,\dots,y_{k-1}\right)
+\tilde{v}_k\left(y_1,\dots,y_{k-1}\right)\right)
+(1-t)\left(\tilde{u}_k\left(y'_1,\dots,y'_{k-1}\right)
+\tilde{v}_k\left(y'_1,\dots,y'_{k-1}\right)\right)\\
&<&t y_k+(1-t)y'_k=y''_k,
\end{eqnarray*}
where the inequality in the third line follows from $(*)$ and $(**)$. 
Indeed, for a big $\R$-divisor $\mathbf{P}$ and an effective $\R$-divisor 
$\mathbf{N}$ on $\bar{Y}_{k-1}$, we have 
$\sigma_{\hat{Y}_k}(\mathbf{P})+\ord_{\hat{Y}_k}(\mathbf{N})
\geq\sigma_{\hat{Y}_k}(\mathbf{P}+\mathbf{N})$. The inequality in the third line 
can be obtained if we set $\mathbf{P}:=t\tilde{P}_{k-2}\left(y_1,\dots,y_{k-1}\right)
|_{\bar{Y}_{k-1}}+(1-t)\tilde{P}_{k-2}\left(y'_1,\dots,y'_{k-1}\right)|_{\bar{Y}_{k-1}}$
and $\mathbf{N}$ to be the sum of the restrictions of the left hand side of $(**)$ 
for $l=1,\dots,k-1$. 
Similarly, we get 
\begin{eqnarray*}
&&\tilde{t}_k\left(y''_1,\dots,y''_{k-1}\right)+\tilde{v}_k\left(y''_1,\dots,y''_{k-1}\right)\\
&=&\tau_{\hat{Y}_k}\left(\tilde{P}_{k-2}\left(y''_1,\dots,y''_{k-1}\right)|_{\bar{Y}_{k-1}}
\right)+\ord_{\hat{Y}_k}\left(\sum_{l=1}^{k-1}\tilde{N}_{l-1,k-2}\left(y''_1,\dots,y''_l
\right)|_{\bar{Y}_{k-1}}\right)\\
&\geq&\tau_{\hat{Y}_k}\left(
t\tilde{P}_{k-2}\left(y_1,\dots,y_{k-1}\right)|_{\bar{Y}_{k-1}}
+(1-t)\tilde{P}_{k-2}\left(y'_1,\dots,y'_{k-1}\right)|_{\bar{Y}_{k-1}}
\right)\\
&&+\ord_{\hat{Y}_k}\left(t\sum_{l=1}^{k-1}\tilde{N}_{l-1,k-2}\left(y_1,\dots,y_l
\right)|_{\bar{Y}_{k-1}}
+(1-t)\sum_{l=1}^{k-1}\tilde{N}_{l-1,k-2}\left(y'_1,\dots,y'_l
\right)|_{\bar{Y}_{k-1}}\right)\\
&\geq&
t\cdot\tau_{\hat{Y}_k}\left(\tilde{P}_{k-2}\left(y_1,\dots,y_{k-1}\right)|_{\bar{Y}_{k-1}}\right)
+t\ord_{\hat{Y}_k}\left(\sum_{l=1}^{k-1}\tilde{N}_{l-1,k-2}\left(y_1,\dots,y_l
\right)|_{\bar{Y}_{k-1}}\right)
\\
&&+(1-t)\tau_{\hat{Y}_k}\left(
\tilde{P}_{k-2}\left(y'_1,\dots,y'_{k-1}\right)|_{\bar{Y}_{k-1}}\right)
+(1-t)\ord_{\hat{Y}_k}\left(\sum_{l=1}^{k-1}\tilde{N}_{l-1,k-2}\left(y'_1,\dots,y'_l
\right)|_{\bar{Y}_{k-1}}\right)\\
&=&t\left(\tilde{t}_k\left(y_1,\dots,y_{k-1}\right)
+\tilde{v}_k\left(y_1,\dots,y_{k-1}\right)\right)
+(1-t)\left(\tilde{t}_k\left(y'_1,\dots,y'_{k-1}\right)
+\tilde{v}_k\left(y'_1,\dots,y'_{k-1}\right)\right)\\
&>&t y_k+(1-t)y'_k=y''_k.
\end{eqnarray*}
Hence we get 
\[
\tilde{u}_k\left(y''_1,\dots,y''_{k-1}\right)+\tilde{v}_k\left(y''_1,\dots,y''_{k-1}\right)
<y''_k<
\tilde{t}_k\left(y''_1,\dots,y''_{k-1}\right)+\tilde{v}_k\left(y''_1,\dots,y''_{k-1}\right), 
\]
which implies that the set $\tilde{\D}_k$ is a convex set. 

We check the assertion \eqref{tilde_proposition3}. By induction, we may assume that 
$l=k$. Note that 
\begin{eqnarray*}
L\left(y''_1,\dots,y''_k\right) & \sim_\R &
t\tilde{P}_{k-1}\left(y_1,\dots,y_k\right)+(1-t)\tilde{P}_{k-1}\left(y'_1,\dots,y'_k\right)\\
&+& t\sum_{l=1}^k\tilde{N}_{l-1,k-1}\left(y_1,\dots,y_l\right)
+(1-t)\sum_{l=1}^k\tilde{N}_{l-1,k-1}\left(y'_1,\dots,y'_l\right), 
\end{eqnarray*}
the $\R$-divisor 
\[
t\tilde{P}_{k-1}\left(y_1,\dots,y_k\right)+(1-t)\tilde{P}_{k-1}\left(y'_1,\dots,y'_k\right)
\]
is movable and big, and 
\begin{eqnarray*}
&&t\sum_{l=1}^k\tilde{N}_{l-1,k-1}\left(y_1,\dots,y_l\right)
+(1-t)\sum_{l=1}^k\tilde{N}_{l-1,k-1}\left(y'_1,\dots,y'_l\right)\\
&\geq&
\sum_{l=1}^{k-1}\tilde{N}_{l-1,k-1}\left(y''_1,\dots,y''_l\right)
+t\tilde{N}_{k-1,k-1}\left(y_1,\dots,y_k\right)
+(1-t)\tilde{N}_{k-1,k-1}\left(y'_1,\dots,y'_k\right)
\end{eqnarray*}
by induction. Since the decomposition 
\[
L\left(y''_1,\dots,y''_k\right)-\sum_{l=1}^{k-1}\tilde{N}_{l-1,k-1}\left(y''_1,
\dots,y''_l\right)  \sim_\R 
\tilde{P}_{k-1}\left(y''_1,\dots,y''_k\right)+\tilde{N}_{k-1,k-1}\left(y''_1,\dots,y''_k\right)
\]
is the Nakayama--Zariski decomposition, we get the inequality 
\[
t\tilde{N}_{k-1,k-1}\left(y_1,\dots,y_k\right)+(1-t)\tilde{N}_{k-1,k-1}
\left(y'_1,\dots,y'_k\right)\geq\tilde{N}_{k-1,k-1}\left(y''_1,\dots,y''_k\right)
\]
by the definition of the Nakayama--Zariski decomposition. 
Thus we get the assertion \eqref{tilde_proposition3}. 

We check \eqref{tilde_proposition2}. 
As in the proof for the convexity of $\tilde{\D}_k$, we know that 
the convexities of $\tilde{u}_k+\tilde{v}_k$ and $-\left(\tilde{t}_k+\tilde{v}_k\right)$.
Thus it is enough to check the convexity for $\tilde{v}_k$. By 
\eqref{tilde_proposition3}, we know the convexity of $\tilde{v}_k$. Thus the 
assertion \eqref{tilde_proposition2} follows. 

We see the openness of $\tilde{\D}_k$. Take any $\left(y_1,\dots,y_k\right)
\in\tilde{\D}_k$. By induction, there exists an open neighborhood 
$U\subset\tilde{\D}_{k-1}$ of $\left(y_1,\dots,y_{k-1}\right)$. 
The functions $\tilde{u}_k$, $\tilde{t}_k$, $\tilde{v}_k$ are continuous over $U$ by 
\eqref{tilde_proposition2}. Thus $\tilde{\D}_k$ is also open, and we get 
the assertion \eqref{tilde_proposition1}. 

Finally, let us show the assertion \eqref{tilde_proposition4}. 
Let us take any 
\[
\vec{y}^{(1)},\dots,\vec{y}^{(l+1)}\in\tilde{\D}_l
\]
with 
\[
\vec{y}\in\interior\left(\Conv\left(\vec{y}^{(1)},\dots,\vec{y}^{(l+1)}\right)\right). 
\]
For any $\vec{y}'\in\Conv\left(\vec{y}^{(1)},\dots,\vec{y}^{(l+1)}\right)$, there 
exists $t_1,\dots,t_{l+1}\in\R_{> 0}$ with $\sum_{i=1}^{l+1}t_i=1$ such that 
$\vec{y}'=\sum_{i=1}^{l+1}t_i\vec{y}^{(i)}$. As in \eqref{tilde_proposition3}, 
we have 
\[
\tilde{N}_{l-1,k-1}\left(\vec{y}'\right)\leq
\sum_{i=1}^{l+1}t_i\tilde{N}_{l-1,k-1}\left(\vec{y}^{(i)}\right). 
\]
This implies that 
\[
\Supp\tilde{N}_{l-1,k-1}\left(\vec{y}'\right)\subset
\bigcup_{i=1}^{l+1}\Supp\tilde{N}_{l-1,k-1}\left(\vec{y}^{(i)}\right), 
\]
thus we get the assertion \eqref{tilde_proposition4}. 
\end{proof}

We are ready to define the notion of adequate dominants. 

\begin{definition}\label{adequate_definition}
A $\Q$-factorial dominant $\left\{\gamma_k\right\}_{0\leq k\leq j-1}$ 
of $Y_\bullet$ is said to be \emph{an adequate dominant of $Y_\bullet$} 
with respects to $L$ if: 
\begin{enumerate}
\renewcommand{\theenumi}{\arabic{enumi}}
\renewcommand{\labelenumi}{(\theenumi)}
\item\label{adequate_definition1}
for any $x_1\in\left(u_1,t_1\right)\cap\Q$, the Nakayama--Zariski decomposition 
\[
\gamma_0^*\sigma_0^*L-x_1\hat{Y}_1=N_{0,0}(x_1)+P_0(x_1)
\]
on $\bar{Y}_0$ is the Zariski decomposition in a strong sense, and 
\item\label{adequate_definition2}
for any $2\leq k\leq j$ and for any $\left(x_1,\dots,x_k\right)\in\D_k\cap\Q^k$, 
the Nakayama--Zariski decomposition 
\[
P_{k-2}\left(x_1,\dots,x_{k-1}\right)|_{\bar{Y}_{k-1}}
-x_k\hat{Y}_k=N_{k-1,k-1}\left(x_1,\dots,x_k\right)
+P_{k-1}\left(x_1,\dots,x_k\right)
\]
on $\bar{Y}_{k-1}$ is the Zariski decomposition in a strong sense. 
\end{enumerate}
\end{definition}

\begin{remark}\label{adequate_remark}
Assume that $\left\{\gamma_k\right\}_{0\leq k\leq j-1}$ is an adequate dominant 
of $Y_\bullet$ with respects to $L$. 
\begin{enumerate}
\renewcommand{\theenumi}{\arabic{enumi}}
\renewcommand{\labelenumi}{(\theenumi)}
\item\label{adequate_remark1}
For any $1\leq k\leq j$ and for any $\left(x_1,\dots,x_k\right)\in\D_k\cap\Q^k$, 
the divisor $P_{k-1}\left(x_1,\dots,x_k\right)$ is a nef and big $\Q$-divisor 
on $\bar{Y}_{k-1}$. Thus, for any $2\leq k\leq j$ and for any 
$\left(x_1,\dots,x_{k-1}\right)\in\D_{k-1}$, we have the equality 
$u_k\left(x_1,\dots,x_{k-1}\right)=0$. 
\item\label{adequate_remark2}
By Proposition \ref{tilde_proposition} \eqref{tilde_proposition4}, for any $1\leq k\leq j$
and for any $\left(x_1,\dots,x_k\right)\in\D_k$, the divisor $P_{k-1}
\left(x_1,\dots,x_k\right)$ is a \emph{nef} and big $\R$-divisor on $\bar{Y}_{k-1}$ 
with $\hat{Y}_k\not\subset\B_+\left(P_{k-1}\left(x_1,\dots,x_k\right)\right)$. 
In particular, we have 
\[
\vol\left(P_{k-1}\left(x_1,\dots,x_k\right)|_{\hat{Y}_k}\right)
=\left(P_{k-1}\left(x_1,\dots,x_k\right)^{\cdot n-k}\cdot\hat{Y}_k\right)
=\vol_{\bar{Y}_{k-1}|\hat{Y}_k}\left(P_{k-1}\left(x_1,\dots,x_k\right)\right)
\]
(see Proposition \ref{NB_proposition}). 
\end{enumerate}
\end{remark}

\begin{lemma}\label{adequate-compare_lemma}
Assume that $\left\{\gamma_k\colon\bar{Y}_k\to\tilde{Y}_k\right\}_{0\leq k\leq j-1}$
is an adequate dominant of $Y_\bullet$ with respects to $L$. 
Let $\left\{\gamma'_k\colon\bar{Y}'_k\to\tilde{Y}_k\right\}_{0\leq k\leq j-1}$
be another $\Q$-factorial dominant of $Y_\bullet$, and let 
$\left\{\psi_k\colon\bar{Y}'_k\to\bar{Y}_k\right\}_{0\leq k\leq j-1}$ 
be a morphism between dominants 
$\left\{\gamma'_k\right\}_{0\leq k\leq j-1}$ and 
$\left\{\gamma_k\right\}_{0\leq k\leq j-1}$, as in Lemma 
\ref{dominants-compare_lemma}. 
\begin{enumerate}
\renewcommand{\theenumi}{\arabic{enumi}}
\renewcommand{\labelenumi}{(\theenumi)}
\item\label{adequate-compare_lemma1}
The dominant  
$\left\{\gamma'_k\right\}_{0\leq k\leq j-1}$ is also adequate with respects to $L$. 
\item\label{adequate-compare_lemma2}
Let 
\[
\D'_k,\quad t'_k, \quad v'_k, \quad u'_{l,k},\quad 
P'_{k-1}\left(x'_1,\dots,x'_k\right), \quad
N'_{l-1,k-1}\left(x'_1,\dots,x'_l\right)
\]
be the notions for $\left\{\gamma'_k\right\}_{0\leq k\leq j-1}$ and $L$ 
in Definition \ref{P-N_definition}. Moreover, for any $1\leq l\leq k\leq j$
let $e_{l,k}$ and $\Theta_{l,k}$ be as in Lemma 
\ref{dominants-compare_lemma}. 
Then, for any $1\leq k\leq j$, we have 
\begin{enumerate}
\renewcommand{\theenumii}{\roman{enumii}}
\renewcommand{\labelenumii}{(\theenumii)}
\item\label{adequate-compare_lemma21}
$\D'_k=\D_k$, 
\item\label{adequate-compare_lemma22}
$t'_k=t_k$ over $\D'_k=\D_k$, 
\item\label{adequate-compare_lemma23}
$P'_{k-1}\left(x_1,\dots,x_k\right)=\psi_{k-1}^*P_{k-1}\left(x_1,\dots,x_k\right)$
for any $\left(x_1,\dots,x_k\right)\in\D'_k=\D_k$, 
\item\label{adequate-compare_lemma24}
\begin{eqnarray*}
N'_{l-1,k-1}\left(x_1,\dots,x_l\right)&=&\psi^*_{k-1}N_{l-1,k-1}\left(x_1,\dots,x_l\right)\\
&+&x_l\Theta_{l,k}+\sum_{i=l+1}^k u_{l,i}\left(x_1,\dots,x_l\right)\Theta_{i,k}
\end{eqnarray*}
for any $1\leq l\leq k$ and for any $\left(x_1,\dots,x_l\right)\in\D'_l=\D_l$, 
\item\label{adequate-compare_lemma25}
\begin{eqnarray*}
u'_{l,k}\left(x_1,\dots,x_l\right)
=u_{l,k}\left(x_1,\dots,x_l\right)+x_l e_{l,k}
+\sum_{i=l+1}^{k-1}u_{l,i}\left(x_1,\dots,x_l\right)e_{i,k}
\end{eqnarray*}
for any $1\leq l<k$ and for any $\left(x_1,\dots,x_l\right)\in\D'_l=\D_l$, and 
\item\label{adequate-compare_lemma26}
if $k\geq 2$, then 
\[
v'_k\left(x_1,\dots,x_{k-1}\right)=v_k\left(x_1,\dots,x_{k-1}\right)
+\sum_{l=1}^{k-1}\left(x_l+v_l\left(x_1,\dots,x_{l-1}\right)\right)e_{l,k}
\]
for any $\left(x_1,\dots,x_{k-1}\right)\in\D'_{k-1}=\D_{k-1}$. 
\end{enumerate}
\end{enumerate}
\end{lemma}

\begin{proof}
We give a proof by induction on $k$. If $k=1$, since $\Theta_{1,1}$ is a 
$\psi_0$-exceptional effective $\Q$-divisor on $\bar{Y}_0$ and 
\[
\gamma_0^*\sigma_0^*L-x_1\hat{Y}_1=N_{0,0}\left(x_1\right)+P_0\left(x_1\right)
\]
for any $x_1\in\left(u_1,t_1\right)\cap\Q$ is the Zariski decomposition in a strong 
sense, the decomposition 
\[
\left(\gamma'_0\right)^*\sigma_0^*L-x_1\hat{Y}'_1
=\left(\psi^*_0 N_{0,0}\left(x_1\right)+x_1\Theta_{1,1}\right)
+\psi_0^*P_0\left(x_1\right)
\]
is the Zariski decomposition in a strong sense. Thus the assertions are trivial when 
$k=1$. 

Assume that $k\geq 2$ the assertions are true up to $k-1$. For any 
$\left(x_1,\dots,x_{k-1}\right)\in\D_{k-1}\cap\Q^{k-1}=\D'_{k-1}\cap\Q^{k-1}$, 
since 
\[
P'_{k-2}\left(x_1,\dots,x_{k-1}\right)|_{\bar{Y}'_{k-1}}
=\psi_{k-1}^*\left(P_{k-2}\left(x_1,\dots,x_{k-1}\right)|_{\bar{Y}_{k-1}}\right)
\]
is nef and big, we have $t_k=t'_k$ and $u'_k\equiv 0$ over $\D'_{k-1}=\D_{k-1}$, 
where $u'_k$ is the ``$u_k$ function'' for $\left\{\gamma'_k\right\}_{0\leq k\leq j-1}$ 
and $L$ in Definition \ref{P-N_definition}. 
(We remark that both are continuous functions.)
Moreover, since $\Theta_{k,k}$ is an effective and $\psi_{k-1}$-exceptional 
$\Q$-divisor on $\bar{Y}'_{k-1}$, the decomposition 
\[
P'_{k-2}\left(x_1,\dots,x_{k-1}\right)|_{\bar{Y}'_{k-1}}-x_k\hat{Y}_k
=\left(\psi_{k-1}^*N_{k-1,k-1}\left(x_1,\dots,x_k\right)+x_k\Theta_{k,k}\right)
+\psi_{k-1}^*P_{k-1}\left(x_1,\dots,x_k\right)
\]
is the Zariski decomposition in a strong sense 
for any $x_k\in\left(0,t_k\left(x_1,\dots,x_{k-1}\right)\right)\cap\Q$. 

Let us consider the assertion \eqref{adequate-compare_lemma24}. 
We may assume that $l<k$ since we already know the case $l=k$. We see by 
induction on $k-l$. We may assume that, for any $\left(x_1,\dots,x_l\right)\in
\D'_l=\D_l$, the equality  
\[
N'_{l-1,k-2}\left(x_1,\dots,x_l\right)=\psi_{k-2}^*N_{l-1,k-2}\left(x_1,\dots,x_l\right)
+x_l\Theta_{l,k-1}+\sum_{i=l+1}^{k-1}u_{l,i}\left(x_1,\dots,x_l\right)\Theta_{i,k-1}
\]
holds on $\bar{Y}'_{k-2}$. Note that
\begin{eqnarray*}
u'_{l,k}\left(x_1,\dots,x_l\right)
&=&\ord_{\hat{Y}'_k}\Biggl(\psi_{k-1}^*\left(N_{l-1,k-2}\left(
x_1,\dots,x_l\right)|_{\bar{Y}_{k-1}}\right)+x_l\left(\Theta_{l,k-1}|_{\bar{Y}'_{k-1}}
\right)\\
&+&\sum_{i=l+1}^{k-1}u_{l,i}\left(x_1,\dots,x_l\right)\left(\Theta_{i,k-1}
|_{\bar{Y}'_{k-1}}\right)\Biggr)\\
&=&u_{l,k}\left(x_1,\dots,x_l\right)+x_l e_{l,k}
+\sum_{i=l+1}^{k-1}u_{l,i}\left(x_1,\dots,x_l\right)e_{i,k}. 
\end{eqnarray*}
Thus we get 
\begin{eqnarray*}
N'_{l-1,k-1}\left(x_1,\dots,x_l\right)&=&
\psi_{k-1}^*\left(N_{l-1,k-2}\left(x_1,\dots,x_l\right)|_{\bar{Y}_{k-1}}\right)
+x_l\left(\Theta_{l,k-1}|_{\bar{Y}'_{k-1}}\right)\\
&+&\sum_{i=l+1}^{k-1}u_{l,i}\left(x_1,\dots,x_l\right)\left(\Theta_{i,k-1}
|_{\bar{Y}'_{k-1}}\right)\\
&-&\left(u_{l,k}\left(x_1,\dots,x_l\right)+x_l e_{l,k}
+\sum_{i=l+1}^{k-1}u_{l,i}\left(x_1,\dots,x_l\right)e_{i,k}\right)\hat{Y}'_k\\
&=&\psi_{k-1}^*\left(N_{l-1,k-2}\left(x_1,\dots,x_l\right)|_{\bar{Y}_{k-1}}
-u_{l,k}\left(x_1,\dots,x_l\right)\hat{Y}_k\right)\\
&+&u_{l,k}\left(x_1,\dots,x_l\right)\Theta_{k,k}
+x_l\left(\Theta_{l,k-1}|_{\bar{Y}'_{k-1}}-e_{l,k}\hat{Y}'_k\right)\\
&+&\sum_{i=l+1}^{k-1}u_{l,i}\left(x_1,\dots,x_l\right)
\left(\Theta_{i,k-1}|_{\bar{Y}'_{k-1}}-e_{i,k}\hat{Y}'_k\right)\\
&=&\psi_{k-1}^*N_{l-1,k-1}\left(x_1,\dots,x_l\right)+x_l\Theta_{l,k}
+\sum_{i=l+1}^k u_{l,i}\left(x_1,\dots,x_l\right)\Theta_{i,k}. 
\end{eqnarray*}
Thus we get the assertion \eqref{adequate-compare_lemma24}, and also 
the assertion \eqref{adequate-compare_lemma25}. 

Since 
\begin{eqnarray*}
v'_k\left(x_1,\dots,x_{k-1}\right)&=&\sum_{l=1}^{k-1}
\left(u_{l,k}\left(x_1,\dots,x_l\right)+x_l e_{l,k}+\sum_{i=l+1}^{k-1}u_{l,i}
\left(x_1,\dots,x_l\right)e_{i,k}\right)\\
&=&v_k\left(x_1,\dots,x_{k-1}\right)+\sum_{l=1}^{k-1}x_l e_{l,k}
+\sum_{i=2}^{k-1}\sum_{l=1}^{i-1}e_{i,k} u_{l,i}\left(x_1,\dots,x_l\right)\\
&=&v_k\left(x_1,\dots,x_{k-1}\right)+\sum_{l=1}^{k-1}\left(
x_l+v_l\left(x_1,\dots,x_{l-1}\right)\right)e_{l,k}, 
\end{eqnarray*}
we get the assertion \eqref{adequate-compare_lemma26}. 
\end{proof}

We state the main theorem in this section. 

\begin{thm}\label{adequate_thm}
Assume that $\left\{\gamma_k\colon\bar{Y}_k\to\tilde{Y}_k\right\}_{0\leq k\leq j-1}$
is an adequate dominant of $Y_\bullet$ with respects to $L$. 
Then, for any $1\leq k\leq j$, we have 
\begin{eqnarray*}
S\left(L; Y_1\triangleright\cdots\triangleright Y_k\right)
&=&\frac{1}{\vol_X\left(L\right)}\cdot\frac{n!}{(n-j)!}
\int_{\left(x_1,\dots,x_j\right)\in\D_j}\Biggl(x_k+v_k\left(x_1,\dots,x_{k-1}\right)\\
&&+\sum_{l=1}^{k-1}g_{l,k}\left(x_l+v_l\left(x_1,\dots,x_{l-1}\right)\right)\Biggr)
\cdot\left(P_{j-1}\left(x_1,\dots,x_j\right)^{\cdot n-j}\cdot\hat{Y}_j\right)d\vec{x},
\end{eqnarray*}
where $g_{l,k}:=g_{l,k}\left(\left\{\gamma_k\right\}_{1\leq k\leq j-1}\right)$ is 
as in Definition \ref{dominant-compare_definition}. 
\end{thm}

\begin{remark}\label{adequate-formula_remark}
\begin{enumerate}
\renewcommand{\theenumi}{\arabic{enumi}}
\renewcommand{\labelenumi}{(\theenumi)}
\item\label{adequate-formula_remark1}
If $Y_\bullet$ is a complete primitive flag over $X$, i.e., if $j=n$, then 
\[
\left(P_{j-1}\left(x_1,\dots,x_j\right)^{\cdot n-j}\cdot\hat{Y}_j\right)
\]
in Theorem \ref{adequate_thm} is identically equal to $1$ by the definition 
of intersection numbers. 
\item\label{adequate-formula_remark2}
In the proof of Theorem \ref{adequate_thm}, we can also show that 
\[
\vol_X\left(L\right)=\frac{n!}{(n-j)!}\int_{\vec{x}\in\D_j}
\left(P_{j-1}\left(\vec{x}\right)^{\cdot n-j}\cdot\hat{Y}_j\right)
d\vec{x}.
\]
\end{enumerate}
\end{remark}

\begin{proof}[Proof of Theorem \ref{adequate_thm}]
The proof is divided into 7 numbers of steps. 

\noindent\underline{\textbf{Step 1}}\\
Let $\left\{\gamma'_k\colon\bar{Y}'_k\to\tilde{Y}_k\right\}_{0\leq k\leq j-1}$
be any $\Q$-factorial dominant of $Y_\bullet$, let 
$\left\{\psi_k\colon\bar{Y}'_k\to\bar{Y}_k\right\}_{0\leq k\leq j-1}$ 
be any morphism between dominants 
$\left\{\gamma'_k\right\}_{0\leq k\leq j-1}$ and 
$\left\{\gamma_k\right\}_{0\leq k\leq j-1}$, as in Lemma 
\ref{adequate-compare_lemma}. 
We see that the right hand side of the equation in Theorem \ref{adequate_thm} 
takes the same value after replacing $\left\{\gamma_k\right\}_{0\leq k\leq j-1}$
with $\left\{\gamma'_k\right\}_{0\leq k\leq j-1}$. 
Set $g_{l,k}:=g_{l,k}\left(\left\{\gamma_k\right\}_{1\leq k\leq j-1}\right)$ 
and $g'_{l,k}:=g_{l,k}\left(\left\{\gamma'_k\right\}_{1\leq k\leq j-1}\right)$. 
We also use the terminologies in Lemma \ref{adequate-compare_lemma}. 
Note that 
\[
g_{l,k}=g'_{l,k}+e_{l,k}+\sum_{i=l+1}^{k-1}e_{l,i}g'_{i,k}
\]
holds for any $1\leq l< k$, where $e_{l,k}$ is as in 
Lemma \ref{dominants-compare_lemma}. 
For any $\left(x_1,\dots,x_j\right)\in\D_j$, we have 
\begin{eqnarray*}
&&x_k+v'_k\left(x_1,\dots,x_{k-1}\right)+\sum_{l=1}^{k-1}g'_{l,k}
\left(x_l+v'_l\left(x_1,\dots,x_{l-1}\right)\right)\\
&-&\left(x_k+v_k\left(x_1,\dots,x_{k-1}\right)+\sum_{l=1}^{k-1}g_{l,k}
\left(x_l+v_l\left(x_1,\dots,x_{l-1}\right)\right)\right)\\
&=&\sum_{l=1}^{k-1}\left(x_l+v_l\left(x_1,\dots,x_{l-1}\right)\right)e_{l,k}
+\sum_{l=1}^{k-1}g'_{l,k}\left(x_l+v_l\left(x_1,\dots,x_{l-1}\right)\right)\\
&+&\sum_{l=1}^{k-1}\sum_{i=1}^{l-1}
g'_{l,k}\left(x_i+v_i\left(x_1,\dots,x_{i-1}\right)\right)e_{i,l}\\
&-&\sum_{l=1}^{k-1}\left(g'_{l,k}+e_{l,k}\right)\left(x_l+v_l\left(x_1,\dots,x_{l-1}\right)
\right)
-\sum_{i=2}^{k-1}\sum_{l=1}^{i-1}\left(x_l+v_l\left(x_1,\dots,x_{l-1}\right)\right)
e_{l,i}g'_{i,k}\\
&=&\sum_{l=2}^{k-1}\sum_{i=1}^{l-1}g'_{l,k}\left(x_i+v_i\left(x_1,\dots,x_{i-1}\right)\right)
e_{i,l}-\sum_{i=2}^{k-1}\sum_{l=1}^{i-1}\left(x_l+v_l\left(x_1,\dots,x_{l-1}\right)\right)
e_{l,i}g'_{i,k}\\
&=&0.
\end{eqnarray*}
Thus, as in Lemma \ref{adequate-compare_lemma} \eqref{adequate-compare_lemma1}, 
since the characteristic of $\Bbbk$ is equal to zero, we may assume that 
$\left\{\gamma_k\right\}_{0\leq k\leq j-1}$ is a \emph{smooth} adequate 
dominant of $Y_\bullet$ with respects to $L$. 

\noindent\underline{\textbf{Step 2}}\\
We see that the right hand side of the equation in Theorem \ref{adequate_thm} 
is equal to the value 
\begin{eqnarray*}
\frac{1}{\vol_X\left(L\right)}\frac{n!}{(n-j)!}
\int_{\left(y_1,\dots,y_j\right)\in\tilde{\D}_j}
\left(y_k+\sum_{l=1}^{k-1}g_{l,k}y_l\right)\left(\tilde{P}_{j-1}
\left(y_1,\dots,y_j\right)^{\cdot n-j}\cdot\hat{Y}_j\right)d\vec{y}.
\end{eqnarray*}
This is trivial from Fubini's theorem by changing the coordinates 
\[
x_1=y_1,\quad x_2=y_2-\tilde{v}_2\left(y_1\right),\dots\dots,
x_j=y_j-\tilde{v}_j\left(y_1,\dots,y_{j-1}\right)
\]
step-by-step. Indeed, we have 
\begin{eqnarray*}
\tilde{P}_{j-1}\left(y_1,\dots,y_j\right)&=&P_{j-1}\left(x_1,\dots,x_j\right), \\
y_k+\sum_{l=1}^{k-1}g_{l,k}y_l&=&x_k+v_k\left(x_1,\dots,x_{k-1}\right)
+\sum_{l=1}^{k-1}g_{l,k}\left(x_l+v_l\left(x_1,\dots,x_{l-1}\right)\right).
\end{eqnarray*}

\noindent\underline{\textbf{Step 3}}\\
For $V_{\vec{\bullet}}=H^0\left(\bullet L\right)$, let us consider the series 
$V_{\vec{\bullet}}^{\left(\hat{Y}_1>\cdots>\hat{Y}_j\right)}$
as in Proposition \ref{refinement-dominant_proposition}. Moreover, 
let us fix a general admissible flag 
\[
Z_\bullet\colon\hat{Y}_j=Z_0\supsetneq Z_1\supsetneq\cdots\supsetneq Z_{n-j}
\]
of $\hat{Y}_j$
in the sense of Corollary \ref{refinement-dominant_corollary}. 
Set 
\[
\hat{\Delta}:=\Delta_{Z_\bullet}
\left(V_{\vec{\bullet}}^{\left(\hat{Y}_1>\cdots>\hat{Y}_j\right)}\right)
\subset\R_{\geq 0}^n,
\]
and let $\left(\hat{b}_1,\dots,\hat{b}_n\right)\in\hat{\Delta}$ be the barycenter 
of $\hat{\Delta}$. By Step 2 and Corollary \ref{refinement-dominant_corollary}, 
it is enough to show the equality 
\[
\hat{b}_k=\frac{1}{\vol_X\left(L\right)}\frac{n!}{(n-j)!}
\int_{\vec{y}\in\tilde{\D}_j}
y_k\left(\tilde{P}_{j-1}
\left(\vec{y}\right)^{\cdot n-j}\cdot\hat{Y}_j\right)d\vec{y}
\]
for any $1\leq k\leq j$ in order to prove Theorem \ref{adequate_thm}. 

\noindent\underline{\textbf{Step 4}}\\
For any $1\leq k\leq j$, 
the series $V_{\vec{\bullet}}^{\left(\hat{Y}_1>\cdots>\hat{Y}_k\right)}$ 
on $\hat{Y}_k$ is associated to $L|_{\hat{Y}_k}, -\hat{Y}_1|_{\hat{Y}_k},\dots,
-\hat{Y}_k|_{\hat{Y}_k}$. Let us construct a similar series 
$V_{\vec{\bullet}}^{\left(\DIV,\hat{Y}_1>\cdots>\hat{Y}_k\right)}$ on $\hat{Y}_k$ 
associated to $L|_{\hat{Y}_k}, -\hat{Y}_1|_{\hat{Y}_k},\dots,
-\hat{Y}_k|_{\hat{Y}_k}$. (Recall that, by Step 1, we assume that 
$\left\{\gamma_k\right\}_{0\leq k\leq j-1}$ is a smooth and adequate with respects 
to $L$.) For any sufficiently divisible $m\in\Z_{>0}$ and for any 
$\left(a,b_1,\dots,b_k\right)\in\left(m\Z_{\geq 0}\right)^{k+1}$, 
let us define the subspace
\[
V_{a,b_1,\dots,b_k}^{\left(\DIV,\hat{Y}_1>\cdots>\hat{Y}_k\right)}
\subset
H^0\left(\hat{Y}_k,a L|_{\hat{Y}_k}-b_1\hat{Y}_1|_{\hat{Y}_k}-\cdots-b_k
\hat{Y}_k|_{\hat{Y}_k}\right)
\]
as follows: 
\[\begin{cases}
\left\lceil\sum_{l=1}^k a\tilde{N}_{l-1,k-1}\left(\frac{b_1}{a},\dots,\frac{b_l}{a}\right)
\right\rceil
\Big|_{\hat{Y}_k}+H^0\left(\hat{Y}_k,\left\lfloor a\tilde{P}_{k-1}
\left(\frac{b_1}{a},\dots,\frac{b_k}{a}\right)\right\rfloor\Big|_{\hat{Y}_k}\right) 
& \text{if }\left(\frac{b_1}{a},\dots,\frac{b_k}{a}\right)\in\tilde{\D}_k,\\
0 & \text{otherwise}.
\end{cases}\]
This definition gives the Veronese equivalence class 
$V_{\vec{\bullet}}^{\left(\DIV,\hat{Y}_1>\cdots>\hat{Y}_k\right)}$ of graded linear 
series by Lemma \ref{tilde_lemma} \eqref{tilde_lemma3} and Proposition 
\ref{tilde_proposition} \eqref{tilde_proposition1}, \eqref{tilde_proposition3}. 
From the construction, the series 
$V_{\vec{\bullet}}^{\left(\DIV,\hat{Y}_1>\cdots>\hat{Y}_k\right)}$
contains an ample series and has bounded support with 
\[
\Delta_{\Supp\left(V_{\vec{\bullet}}^{\left(\DIV,\hat{Y}_1>\cdots>\hat{Y}_k\right)}
\right)}=\overline{\left(\tilde{\D}_k\right)}.
\]
Moreover, for any $\vec{y}\in\tilde{\D}_k\cap\Q^k$, 
we have 
\begin{eqnarray*}
\vol\left(V_{\bullet\left(1,\vec{y}\right)}^{\left(
\DIV,\hat{Y}_1>\cdots>\hat{Y}_k\right)}\right)
=\limsup_{p\to\infty}\frac{h^0\left(\hat{Y}_k,\left\lfloor p\tilde{P}_{k-1}
\left(\vec{y}\right)\right\rfloor\Big|_{\hat{Y}_k}\right)}{p^{n-k}/(n-k)!}
=\left(\tilde{P}_{k-1}\left(\vec{y}\right)^{\cdot n-k}\cdot\hat{Y}_k\right).
\end{eqnarray*}

\noindent\underline{\textbf{Step 5}}\\
We show the following claim. 

\begin{claim}\label{adequate-compare_claim}
Take any $1\leq k\leq j$. Let 
$V_{\vec{\bullet}}^{\left(\DIV,\hat{Y}_1>\cdots>\hat{Y}_{k-1}\right)
\left(\hat{Y}_k\right)}$ be the refinement of 
\[\begin{cases}
\phi_{k-1}^*V_{\vec{\bullet}}^{\left(\DIV,\hat{Y}_1>\cdots>\hat{Y}_{k-1}\right)}
& \text{if }k\geq 2, \\
\gamma_0^*\sigma_0^*V_{\vec{\bullet}} & \text{if }k=1, 
\end{cases}\]
by $\hat{Y}_k\subset\bar{Y}_{k-1}$. 
\begin{enumerate}
\renewcommand{\theenumi}{\arabic{enumi}}
\renewcommand{\labelenumi}{(\theenumi)}
\item\label{adequate-compare_claim1}
We have 
\[
\Delta_{\Supp\left(V_{\vec{\bullet}}^{\left(\DIV,\hat{Y}_1>\cdots>\hat{Y}_{k-1}\right)
\left(\hat{Y}_k\right)}\right)}
=\Delta_{\Supp\left(V_{\vec{\bullet}}^{\left(\DIV,\hat{Y}_1>\cdots>\hat{Y}_k\right)}
\right)}=\overline{\left(\tilde{\D}_k\right)}.
\]
\item\label{adequate-compare_claim2}
There exists the Veronese equivalence class $W_{\vec{\bullet}}^k$ of graded 
linear series on $\hat{Y}_k$ associated to $L|_{\hat{Y}_k}, -\hat{Y}_1|_{\hat{Y}_k},\dots,
-\hat{Y}_k|_{\hat{Y}_k}$ such that 
\begin{itemize}
\item
the series $V_{\vec{\bullet}}^{\left(\DIV,\hat{Y}_1>\cdots>\hat{Y}_{k-1}\right)
\left(\hat{Y}_k\right)}$ is asymptotically equivalent to $W_{\vec{\bullet}}^k$, and 
\item
the series $V_{\vec{\bullet}}^{\left(\DIV,\hat{Y}_1>\cdots>\hat{Y}_k\right)}$ is 
asymptotically equivalent to $W_{\vec{\bullet}}^k$. 
\end{itemize}
\item\label{adequate-compare_claim3}
For any $\vec{y}\in\tilde{\D}_k\cap\Q^k$, 
we have 
\begin{eqnarray*}
\vol\left(V_{\bullet\left(1,\vec{y}\right)}^{\left(
\DIV,\hat{Y}_1>\cdots>\hat{Y}_{k-1}\right)\left(\hat{Y}_k\right)}\right)
=\vol\left(V_{\bullet\left(1,\vec{y}\right)}^{\left(
\DIV,\hat{Y}_1>\cdots>\hat{Y}_k\right)}\right)
=\left(\tilde{P}_{k-1}\left(\vec{y}\right)^{\cdot n-k}\cdot\hat{Y}_k\right).
\end{eqnarray*}
\end{enumerate}
\end{claim}

\begin{proof}[Proof of Claim \ref{adequate-compare_claim}]
Take any $\vec{y}=\left(y_1,\dots,y_k\right)\in\Q_{>0}^k$ and take any sufficiently 
divisible $a\in\Z_{>0}$. If $V_{a,a\vec{y}}^{\left(
\DIV,\hat{Y}_1>\cdots>\hat{Y}_{k-1}\right)\left(\hat{Y}_k\right)}\neq 0$, then 
we must have $\left(y_1,\dots,y_{k-1}\right)\in\tilde{\D}_{k-1}$ since the space 
$V_{a,a\left(y_1,\dots,y_{k-1}\right)}^{\left(
\DIV,\hat{Y}_1>\cdots>\hat{Y}_{k-1}\right)}$ must be nonzero. 
Recall that, the space $V_{a,a\vec{y}}^{\left(
\DIV,\hat{Y}_1>\cdots>\hat{Y}_{k-1}\right)\left(\hat{Y}_k\right)}$ is defined by 
the image of the homomorphism 
\begin{eqnarray*}
&&\left(a y_k\hat{Y}_k+ H^0\left(\bar{Y}_{k-1},a L|_{\bar{Y}_{k-1}}
-a y_1\hat{Y}_1|_{\bar{Y}_{k-1}}-\cdots-a y_k\hat{Y}_k\right)\right)\\
&&\cap\left(\left\lceil\sum_{l=1}^{k-1}a\tilde{N}_{l-1,k-2}
\left(y_1,\dots,y_l\right)\right\rceil\Big|_{\bar{Y}_{k-1}}+\phi_{k-1}^*
H^0\left(\hat{Y}_{k-1}, \left\lfloor a\tilde{P}_{k-2}\left(y_1,\dots,y_{k-1}
\right)\right\rfloor\Big|_{\hat{Y}_{k-1}}\right)\right)\\
&\xrightarrow{\bullet|_{\hat{Y}_k}}&
H^0\left(\hat{Y}_k, a L|_{\hat{Y}_k}-a y_1\hat{Y}_1|_{\hat{Y}_k}-\cdots-a y_k
\hat{Y}_k|_{\hat{Y}_k}\right). 
\end{eqnarray*}
Assume that the homomorphism is not the zero map. Then we have 
\begin{itemize}
\item
for any sufficiently divisible $a\in\Z_{>0}$, we have 
\[
a y_k\geq\ord_{\hat{Y}_k}\left(\left\lceil\sum_{l=1}^{k-1}a\tilde{N}_{l-1,k-2}
\left(y_1,\dots,y_l\right)\right\rceil\Big|_{\bar{Y}_{k-1}}\right),
\]
and
\item
for any sufficiently divisible $a\in\Z_{>0}$, we have 
\[
a y_k-\ord_{\hat{Y}_k}\left(\left\lceil\sum_{l=1}^{k-1}a\tilde{N}_{l-1,k-2}
\left(y_1,\dots,y_l\right)\right\rceil\Big|_{\bar{Y}_{k-1}}\right)
\leq\tau_{\hat{Y}_k}\left(\left\lfloor a\tilde{P}_{k-2}\left(y_1,\dots,y_{k-1}
\right)\right\rfloor\Big|_{\hat{Y}_{k-1}}\right).
\]
\end{itemize}
Thus we have 
\[
0\leq y_k-\tilde{v}_k\left(y_1,\dots,y_{k-1}\right)\leq
\tilde{t}_k\left(y_1,\dots,y_{k-1}\right).
\]
This implies that 
\[
\Delta_{\Supp\left(V_{\vec{\bullet}}^{\left(\DIV,\hat{Y}_1>\cdots>\hat{Y}_{k-1}\right)
\left(\hat{Y}_k\right)}\right)}
\subset\overline{\left(\tilde{\D}_k\right)}.
\]

Conversely, assume that $\vec{y}\in\tilde{\D}_k\cap\Q^k$. Then for any sufficiently 
divisible $a\in\Z_{>0}$, let $M_a$ be the image of the homomorphism 
\begin{eqnarray*}
&&\phi_{k-1}^*H^0\left(\hat{Y}_{k-1}, a \tilde{P}_{k-2}\left(y_1,\dots,y_{k-1}\right)
|_{\hat{Y}_{k-1}}\right)\\
&\cap& \biggl(a\left(y_k-\tilde{v}_k\left(y_1,\dots,y_{k-1}\right)\right)\hat{Y}_k\\
&&+H^0\left(\bar{Y}_{k-1}, a\left(\tilde{P}_{k-2}\left(y_1,\dots,y_{k-1}\right)
|_{\bar{Y}_{k-1}}-\left(y_k-\tilde{v}_k\left(y_1,\dots,y_{k-1}\right)\right)
\hat{Y}_k\right)\right)\biggr)\\
&=&\phi_{k-1}^*H^0\left(\hat{Y}_{k-1}, a \tilde{P}_{k-2}\left(y_1,\dots,y_{k-1}\right)
|_{\hat{Y}_{k-1}}\right)\\
&\cap&\left(a\left(y_k-\tilde{v}_k\left(y_1,\dots,y_{k-1}\right)\right)\hat{Y}_k
+a\tilde{N}_{k-1,k-1}\left(y_1,\dots,y_k\right)
+H^0\left(\bar{Y}_{k-1}, a\tilde{P}_{k-1}\left(y_1,\dots,y_k\right)\right)\right)\\
&\xrightarrow{\bullet|_{\hat{Y}_k}}&
a\tilde{N}_{k-1,k-1}\left(y_1,\dots,y_k\right)|_{\hat{Y}_k}
+H^0\left(\hat{Y}_k, a\tilde{P}_{k-1}\left(y_1,\dots,y_k\right)|_{\hat{Y}_k}\right)
\end{eqnarray*}
just for simplicity. As we have seen above, $M_a$ is canonically isomorphic to the 
space $V_{a,a\vec{y}}^{\left(
\DIV,\hat{Y}_1>\cdots>\hat{Y}_{k-1}\right)\left(\hat{Y}_k\right)}$. 
By Corollary \ref{kojutsu_corollary}, we have 
\begin{eqnarray*}
\limsup_{a\to\infty}\frac{\dim M_a}{a^{n-k}/(n-k)!}&=&\vol_{\bar{Y}_{k-1}|\hat{Y}_k}
\left(\tilde{P}_{k-2}\left(y_1,\dots,y_{k-1}\right)
|_{\bar{Y}_{k-1}}-\left(y_k-\tilde{v}_k\left(y_1,\dots,y_{k-1}\right)\right)
\hat{Y}_k\right)\\
&=&\vol_{\bar{Y}_{k-1}|\hat{Y}_k}\left(\tilde{P}_{k-1}\left(y_1,\dots,y_k\right)\right)
=\left(\tilde{P}_{k-1}\left(y_1,\dots,y_k\right)^{\cdot n-k}\cdot\hat{Y}_k\right). 
\end{eqnarray*}
Thus we get the assertions \eqref{adequate-compare_claim1} and 
\eqref{adequate-compare_claim3} in Claim \ref{adequate-compare_claim}. 

Let us consider the assertion \eqref{adequate-compare_claim2}. 
For any sufficiently divisible $m\in\Z_{>0}$ and for any $\left(a,b_1,\dots,b_k\right)
\in\left(m\Z_{\geq 0}\right)^{k+1}$, 
let 
\[
W_{a,b_1,\dots,b_k}^k\subset H^0\left(\hat{Y}_k,a L|_{\hat{Y}_k}-b_1\hat{Y}_1|_{\hat{Y}_k}
-\cdots-b_k\hat{Y}_k|_{\hat{Y}_k}\right)
\]
be the subspace defined by the sum 
\[
W_{a,b_1,\dots,b_k}^k:=V_{a,b_1,\dots,b_k}^{\left(\DIV,\hat{Y}_1>\cdots>\hat{Y}_{k-1}
\right)\left(\hat{Y}_k\right)}
+V_{a,b_1,\dots,b_k}^{\left(\DIV,\hat{Y}_1>\cdots>\hat{Y}_k\right)}
\]
of the subspaces. Obviously, $W_{\vec{\bullet}}^k$ is the Veronese equivalence class 
of a graded linear series which contains an ample series and has bounded support 
with 
\[
\Delta_{\Supp\left(W_{\vec{\bullet}}^k\right)}=\overline{\left(\tilde{\D}_k\right)}.
\]
Moreover, for any $\vec{y}\in\tilde{\D}_k\cap\Q^k$ and for any sufficiently 
divisible $a\in\Z_{>0}$, we have 
\[
W_{a,a\vec{y}}^k=V_{a,a\vec{y}}^{\left(\DIV, \hat{Y}_1>\cdots>\hat{Y}_k\right)}
\]
by construction. This implies that 
\begin{eqnarray*}
\vol\left(W_{\bullet\left(1,\vec{y}\right)}^k\right)
=\vol\left(V_{\bullet\left(1,\vec{y}\right)}^{\left(\DIV, 
\hat{Y}_1>\cdots\hat{Y}_k\right)}\right)
=\left(\tilde{P}_{k-1}\left(\vec{y}\right)^{\cdot n-k}\cdot\hat{Y}_k\right)
=\vol\left(V_{\bullet\left(1,\vec{y}\right)}^{\left(\DIV, 
\hat{Y}_1>\cdots\hat{Y}_{k-1}\right)\left(\hat{Y}_k\right)}\right). 
\end{eqnarray*}
Thus the assertion \eqref{adequate-compare_claim2} follows by Lemma 
\ref{asymp-equiv_lemma} and we complete the proof of 
Claim \ref{adequate-compare_claim}.
\end{proof}

\noindent\underline{\textbf{Step 6}}\\
Recall that, in Step 3, we fix a general admissible flag $Z_\bullet$ of $\hat{Y}_j$. 

\begin{claim}\label{adequate-body_claim}
We have $\Delta_{Z_\bullet}\left(V_{\vec{\bullet}}^{\left(\DIV, 
\hat{Y}_1>\cdots>\hat{Y}_j\right)}\right)=\hat{\Delta}$. 
\end{claim}

\begin{proof}[Proof of Claim \ref{adequate-body_claim}]
For every $1\leq k<l\leq j$, let $V_{\vec{\bullet}}^{\left(\DIV, \hat{Y}_1
>\cdots>\hat{Y}_k\right)\left(\hat{Y}_{k+1}>\cdots>\hat{Y}_l\right)}$
be the refinement of $\phi_{l-1}^*V_{\vec{\bullet}}^{\left(\DIV, \hat{Y}_1
>\cdots>\hat{Y}_k\right)\left(\hat{Y}_{k+1}>\cdots>\hat{Y}_{l-1}\right)}$ 
by $\hat{Y}_l\subset\bar{Y}_{l-1}$. For any $1\leq k\leq j$, by Claim 
\ref{adequate-compare_claim} and Example \ref{interior_example} 
\eqref{interior_example6}, both 
\[
V_{\vec{\bullet}}^{\left(\DIV, \hat{Y}_1
>\cdots>\hat{Y}_{k-1}\right)\left(\hat{Y}_k>\cdots>\hat{Y}_j\right)}
\quad\text{and}\quad
V_{\vec{\bullet}}^{\left(\DIV, \hat{Y}_1
>\cdots>\hat{Y}_{k}\right)\left(\hat{Y}_{k+1}>\cdots>\hat{Y}_j\right)}
\]
are asymptotically equivalent to $W_{\vec{\bullet}}^{k, \left(\hat{Y}_{k+1}
>\cdots>\hat{Y}_j\right)}$. By \cite[Lemma 4.73]{Xu}, we have 
\[
\Delta_{Z_\bullet}\left(V_{\vec{\bullet}}^{\left(\DIV, \hat{Y}_1
>\cdots>\hat{Y}_{k-1}\right)\left(\hat{Y}_k>\cdots>\hat{Y}_j\right)}\right)
=\Delta_{Z_\bullet}\left(V_{\vec{\bullet}}^{\left(\DIV, \hat{Y}_1
>\cdots>\hat{Y}_{k}\right)\left(\hat{Y}_{k+1}>\cdots>\hat{Y}_j\right)}\right)
\]
for any $1\leq k\leq j$. Thus we complete the proof of 
Claim \ref{adequate-body_claim}.
\end{proof}

\noindent\underline{\textbf{Step 7}}\\
Let $p\colon\hat{\Delta}\twoheadrightarrow\overline{\left(\tilde{\D}_j\right)}
\subset\R_{\geq 0}^j$ be the composition of the natural maps 
\[
\hat{\Delta}\hookrightarrow\R_{\geq 0}^n=\R_{\geq 0}^{j}\times\R_{\geq 0}^{n-j}
\to\R_{\geq 0}^j. 
\]
By \cite[Theorem 4.21]{LM}, Claims \ref{adequate-compare_claim} 
and \ref{adequate-body_claim}, for any $\vec{y}\in\tilde{\D}_j\cap\Q^j$, we have 
\[
\vol_{\R^{n-j}}\left(p^{-1}\left(\vec{y}\right)\right)=\frac{1}{(n-j)!}
\vol\left(V_{\bullet\left(1,\vec{y}\right)}^{\left(\DIV, 
\hat{Y}_1>\cdots>\hat{Y}_j\right)}\right)=\frac{1}{(n-j)!}
\left(\tilde{P}_{j-1}\left(\vec{y}\right)^{\cdot n-j}\cdot\hat{Y}_j\right).
\]
By Proposition \ref{tilde_proposition} \eqref{tilde_proposition4}, we can also get 
\[
\vol_{\R^{n-j}}\left(p^{-1}\left(\vec{y}\right)\right)=\frac{1}{(n-j)!}
\left(\tilde{P}_{j-1}\left(\vec{y}\right)^{\cdot n-j}\cdot\hat{Y}_j\right)
\]
for any $\vec{y}\in\tilde{\D}_j$. This implies that 
\[
\vol_X\left(L\right)=n!\vol\left(\hat{\Delta}\right)
=\frac{n!}{(n-j)!}\int_{\vec{y}\in\tilde{\D}_j}
\left(\tilde{P}_{j-1}\left(\vec{y}\right)^{\cdot n-j}\cdot\hat{Y}_j\right)d\vec{y}. 
\]
Moreover, for any $1\leq k\leq j$, we have 
\begin{eqnarray*}
\hat{b}_k&=&\frac{1}{\vol\left(\hat{\Delta}\right)}\frac{1}{(n-j)!}
\int_{\vec{y}\in\tilde{\D}_j}
y_k\left(\tilde{P}_{j-1}\left(\vec{y}\right)^{\cdot n-j}\cdot\hat{Y}_j\right)d\vec{y}\\
&=&\frac{n!}{\vol_X\left(L\right)}\frac{1}{(n-j)!}
\int_{\vec{y}\in\tilde{\D}_j}
y_k\left(\tilde{P}_{j-1}\left(\vec{y}\right)^{\cdot n-j}\cdot\hat{Y}_j\right)d\vec{y}. 
\end{eqnarray*}
As a consequence, we complete the proof of Theorem \ref{adequate_thm}. 
\end{proof}

\section{Special cases of Theorem \ref{adequate_thm}}

We assume that the characteristic of $\Bbbk$ is equal to zero. 
Let us consider special cases of Theorem \ref{adequate_thm} for convenience, 
since the formula in Theorem \ref{adequate_thm} is a bit complicated. 

When $X$ is a surface, the following formula 
is probably well-known for specialists. 
See \cite[Lemma 4.8]{AZ}, \cite[Theorem 1.106]{FANO} 
and \cite[Theorem 4.8]{r3d28}. 

\begin{corollary}\label{ad-surface_corollary}
Let $X$ be a normal $\Q$-factorial projective surface, let $L$ be a big $\Q$-divisor 
on $X$, and let $Y_\bullet$ be 
a complete primitive flag over $X$. Let $\sigma_k\colon\tilde{Y}_k\to Y_k$ be 
the associated prime blowups for $k=0$, $1$. Then we have 
\begin{eqnarray*}
S\left(L;Y_1\right)&=&
\frac{2}{\vol_X(L)}\int_{u_1}^{t_1}x_1\left(P_0(x_1)\cdot Y_1\right)dx_1, \\
S\left(L;Y_1\triangleright Y_2\right)&=&
\frac{2}{\vol_X(L)}\int_{u_1}^{t_1}
\left(\left(P_0(x_1)\cdot Y_1\right)\left(\frac{1}{2}\left(P_0(x_1)\cdot Y_1\right)
+\ord_{Y_2}\left(\sigma_1^*N_{0,0}(x_1)|_{Y_1}\right)\right)\right)dx_1,
\end{eqnarray*}
where $u_1=\sigma_{Y_1}\left(\sigma_0^*L\right)$, 
$t_1=\tau_{Y_1}\left(\sigma_0^*L\right)$ and 
\[
\sigma_0^*L-x_1Y_1=N_{0,0}(x_1)+P_0(x_1)
\]
is the Zariski decomposition. 
\end{corollary}

\begin{proof}
The trivial dominant 
$\left\{\id_{\tilde{Y}_k}\colon\tilde{Y}_k\to\tilde{Y}_k\right\}_{k=0,1}$ is an 
adequate dominant of $Y_\bullet$ with respects to $L$ by Example 
\ref{ZDS_example}. 
Since $g_{1,2}=0$ and $t_2(x_1)=\left(P_0(x_1)\cdot Y_1\right)$, 
the assertion is trivial from Theorem \ref{adequate_thm}. 
\end{proof}

We consider the case $X$ is of dimension three. In this case, we get 
a slight generalization of \cite[Theorem 1.112]{FANO}, \cite[Theorem 4.17]{r3d28}, 
since the papers assumed that $\tilde{Y}_0$ is a Mori dream space. 

\begin{corollary}\label{ad-threefold_corollary}
Under the assumptions in Definition \ref{P-N_definition} and 
Theorem \ref{adequate_thm}, 
assume moreover that $n=j=3$. Then we have 
\begin{eqnarray*}
S\left(L; Y_1\triangleright Y_2\right)
&=&\frac{6}{\vol_X(L)}\int_{u_1}^{t_1}\int_0^{t_2(x_1)}\left(P_1(x_1,x_2)\cdot 
\hat{Y}_2\right)\left(x_2+u_{1,2}(x_1)-x_1d_{1,2}\right)dx_2dx_1 \\
&=&\frac{3}{\vol_X(L)}\int_{u_1}^{t_1}
\bigg(\left(u_{1,2}(x_1)-x_1d_{1,2}\right)\left(P_0(x_1)^{\cdot 2}\cdot \hat{Y}_1\right) \\
&&+\int_0^\infty\vol_{\hat{Y}_1}\left(P_0(x_1)|_{\hat{Y}_1}-x_2\hat{Y}_2\right)
dx_2\bigg)dx_1, \\
S\left(L; Y_1\triangleright Y_2\triangleright Y_3\right)
&=&\frac{6}{\vol_X(L)}\int_{u_1}^{t_1}\int_0^{t_2(x_1)}
\bigg(\left(P_1(x_1,x_2)\cdot\hat{Y}_2\right)\bigg(
\frac{1}{2}\left(P_1(x_1,x_2)\cdot\hat{Y}_2\right)+u_{1,3}(x_1)\\
&&+u_{2,3}(x_1,x_2)
-(d_{1,3}-d_{1,2}d_{2,3})x_1-d_{2,3}(x_2+u_{1,2}(x_1))\bigg)\bigg)dx_2dx_1.
\end{eqnarray*}
\end{corollary}

\begin{proof}
We know that $g_{1,2}=-d_{1,2}$, $g_{1,3}=-d_{1,3}+d_{1,2}d_{2,3}$, $g_{2,3}=-d_{2,3}$ 
and $t_3(x_1,x_2)=\left(P_1(x_1,x_2)\cdot \hat{Y}_2\right)$. Thus the assertion follows 
from Theorem \ref{adequate_thm} and Corollary \ref{NB_corollary}. 
\end{proof}

\begin{remark}\label{ad-threefold_remark}
Let us compare Corollary \ref{ad-threefold_corollary} and \cite[Theorem 4.17]{r3d28}. 
The $\R$-divisors $N(x_1)$, $P(x_1)$ and $N'(x_1)$ in \cite{r3d28} are equal to 
$N_{0,0}(x_1)-x_1\Sigma_{1,1}$, $P_0(x_1)$ and $N_{0,1}(x_1)-x_1\Sigma_{1,2}$ 
in our sense, respectively. 
Moreover, the value $d(x_1)$ in \cite{r3d28} is equal to $u_{1,2}(x_1)-x_1d_{1,2}$ 
in our sense. 
Thus the above formula is same as the formula in \cite[Theorem 4.17]{r3d28}. 
\end{remark}

Here is an answer of the question by Cheltsov: 

\begin{corollary}\label{adequate-divide_corollary}
Under the assumptions in Definitions \ref{P-N_definition}, \ref{tilde_definition} and 
Theorem \ref{adequate_thm}, 
take any $1\leq l\leq k\leq j$. 
Let $C\subset\Delta_{\Supp\left(V_{\vec{\bullet}}^{\left(Y_1\triangleright\cdots
\triangleright Y_{l-1}\right)}\right)}$ be a closed convex set with 
$\interior\left(C\right)\neq\emptyset$ and 
let us consider the natural projection 
\[
q_k\colon\Delta_{\Supp\left(V_{\vec{\bullet}}^{\left(Y_1\triangleright\cdots
\triangleright Y_k\right)}\right)}\twoheadrightarrow
\Delta_{\Supp\left(V_{\vec{\bullet}}^{\left(Y_1\triangleright\cdots
\triangleright Y_{l-1}\right)}\right)}
\]
and its inverse image $q_k^{-1}(C)\subset\Delta_{\Supp\left(
V_{\vec{\bullet}}^{\left(Y_1
\triangleright\cdots\triangleright Y_k\right)}\right)}$. 
Set $W_{\vec{\bullet}}:=V_{\vec{\bullet}}^{\left(Y_1\triangleright\cdots
\triangleright Y_{l-1}\right), \langle C\rangle}$. 
Let us take the linear transform 
\begin{eqnarray*}
f_k\colon \R^k&\to&\R^k \\
\begin{pmatrix}
y'_1 \\ \vdots \\ \vdots \\ y'_k
\end{pmatrix}
&\mapsto&
\begin{pmatrix}
1 & & & \\
d_{1,2} & 1 && \\
\vdots &\ddots& \ddots& \\
d_{1,k} &\cdots & d_{k-1,k}&1
\end{pmatrix}
\begin{pmatrix}
y'_1 \\ \vdots \\ \vdots \\ y'_k
\end{pmatrix}.
\end{eqnarray*}
Then we have 
\begin{eqnarray*}
&&\vol\left(W_{\vec{\bullet}}\right)=\vol\left(
W_{\vec{\bullet}}^{\left(Y_l\triangleright\cdots\triangleright Y_k\right)}\right)\\
&=&\frac{n!}{(n-k)!}\int_{\vec{y}\in f_k\left(q_k^{-1}(C)\right)}
\left(\tilde{P}_{k-1}\left(\vec{y}\right)^{\cdot n-k}\cdot\hat{Y}_k\right)d\vec{y}, \\
&&S\left(W_{\vec{\bullet}};Y_l\triangleright\cdots\triangleright Y_k\right)\\
&=&\frac{1}{\vol\left(W_{\vec{\bullet}}\right)}\frac{n!}{(n-j)!}
\int_{\vec{y}\in f_j\left(q_j^{-1}(C)\right)}
\left(y_k+\sum_{i=1}^{k-1}g_{i,k} y_i\right)
\left(\tilde{P}_{j-1}\left(\vec{y}\right)^{\cdot n-j}\cdot\hat{Y}_j\right)d\vec{y}.
\end{eqnarray*}
\end{corollary}

\begin{proof}
For the equalities on $\vol\left(W_{\vec{\bullet}}\right)$, we may assume that 
$k=j$. 
By Lemma \ref{divide_lemma}, we know that 
\[
W_{\vec{\bullet}}^{\left(Y_l\triangleright\cdots\triangleright Y_j\right)}
=V_{\vec{\bullet}}^{\left(Y_1\triangleright\cdots\triangleright Y_j\right),\langle
q_j^{-1}(C)\rangle}.
\]
Take any general admissible flag $Z_\bullet$ of $\hat{Y}_j$ in the sense of 
Corollary \ref{refinement-dominant_corollary}, and let $\Delta$ (resp., 
$\hat{\Delta}$) be the Okounkov body of 
$V_{\vec{\bullet}}^{\left(Y_1\triangleright\cdots
\triangleright Y_j\right)}$ (resp., $V_{\vec{\bullet}}^{\left(\hat{Y}_1>\dots>\hat{Y}_j\right)}$)
associated to $Z_\bullet$. By Corollary \ref{refinement-dominant_corollary}, 
we have $\hat{\Delta}=f\left(\Delta\right)$, where $f:=f_j\oplus\id_{\R^{n-j}}$. 
Note that the value $S\left(W_{\vec{\bullet}};Y_l\triangleright\cdots\triangleright 
Y_k\right)$ is equal to the $k$-th coordinate of the barycenter of 
$\Delta^{\langle C\rangle}$, where $\Delta^{\langle C\rangle}\subset\Delta$ 
is defined to be $p^{-1}\left(q_j^{-1}(C)\right)$ with 
\[
p\colon \Delta\twoheadrightarrow\Delta_{\Supp\left(
V_{\vec{\bullet}}^{\left(Y_1\triangleright\cdots\triangleright Y_j\right)}\right)}.
\]
Obviously, under the natural projection 
\[
p\colon \hat{\Delta}\twoheadrightarrow\Delta_{\Supp\left(
V_{\vec{\bullet}}^{\left(\hat{Y}_1>\cdots> \hat{Y}_j\right)}\right)}, 
\]
if we set $\hat{\Delta}^{\langle C\rangle}:=p^{-1}\left(f_j\left(q_j^{-1}(C)\right)\right)$, 
then $\hat{\Delta}^{\langle C\rangle}=f\left(\Delta^{\langle C\rangle}\right)$ holds. 
Thus the assertions follow from the proof (more precisely, Step 2) of Theorem 
\ref{adequate_thm}. 
\end{proof}

In Corollary \ref{adequate-divide_corollary}, if $l=2$, then $C$ is a segment. 
We state the case $l=2$, $n=j=3$. 

\begin{corollary}\label{divide-three_corollary}
Under the assumption in Corollary \ref{adequate-divide_corollary}, assume that 
$n=j=3$, $l=2$ and $C=\left[u_1^C, t_1^C \right]$ with $u_1\leq u_1^C<t_1^C\leq t_1$. 
For $W_{\vec{\bullet}}:=V_{\vec{\bullet}}^{(Y_1),\langle C\rangle}$, 
we have 
\begin{eqnarray*}
\vol\left(W_{\vec{\bullet}}\right)&=&
6\int_{u_1^C}^{t_1^C}\int_0^{t_2(x_1)}
\left(P_1(x_1,x_2)\cdot\hat{Y}_2\right)dx_2dx_1, \\
S\left(W_{\vec{\bullet}}; Y_2\right)
&=&\frac{6}{\vol\left(W_{\vec{\bullet}}\right)}\int_{u_1^C}^{t_1^C}
\int_0^{t_2(x_1)}\left(P_1(x_1,x_2)\cdot 
\hat{Y}_2\right)\left(x_2+u_2(x_1)-x_1d_{1,2}\right)dx_2dx_1, \\
S\left(W_{\vec{\bullet}}; Y_2\triangleright Y_3\right)
&=&\frac{6}{\vol\left(W_{\vec{\bullet}}\right)}\int_{u_1^C}^{t_1^C}\int_0^{t_2(x_1)}
\bigg(\left(P_1(x_1,x_2)\cdot\hat{Y}_2\right)\bigg(
\frac{1}{2}\left(P_1(x_1,x_2)\cdot\hat{Y}_2\right)+u_{1,3}(x_1)\\
&&+u_{2,3}(x_1,x_2)
-(d_{1,3}-d_{1,2}d_{2,3})x_1-d_{2,3}(x_2+u_{1,2}(x_1))\bigg)\bigg)dx_2dx_1.
\end{eqnarray*}
\end{corollary}

\begin{proof}
We just apply Corollary \ref{adequate-divide_corollary}. 
We note that $\D_1=\tilde{\D}_1$.
\end{proof}

\section{Stability thresholds}\label{delta_section}

In this section, we assume that the characteristic of $\Bbbk$ is zero. 
Let $X$ be an $n$-dimensional projective variety and 
let $B$ be an effective $\Q$-Weil 
divisor on $X$. For any $1\leq i\leq k$, let $V_{\vec{\bullet}}^i$ be 
the Veronese equivalence class of an $(m\Z_{\geq 0})^{r_i}$-graded linear series 
$V_{m\vec{\bullet}}^i$ on $X$ associated to 
$L_1^i,\dots,L_{r_i}^i\in\CaCl(X)\otimes_\Z\Q$ which has bounded support and 
contains an ample series. Take any $c_1,\dots,c_k\in\R_{>0}$.

\begin{definition}[{cf.\ \cite[\S 4]{BJ}, 
\cite[\S 11.2]{r3d28}}]\label{alpha-delta_definition}
\begin{enumerate}
\renewcommand{\theenumi}{\arabic{enumi}}
\renewcommand{\labelenumi}{(\theenumi)}
\item\label{alpha-delta_definition1}
Assume that $(X, B)$ is klt. 
\begin{enumerate}
\renewcommand{\theenumii}{\roman{enumii}}
\renewcommand{\labelenumii}{(\theenumii)}
\item\label{alpha-delta_definition11}
For any $l\in m\Z_{>0}$ with $\prod_{i=1}^k h^0\left(V_{l,m\vec{\bullet}}^i\right)\neq 0$, 
we set 
\begin{eqnarray*}
\alpha_l\left(X, B; \left\{c_i\cdot V_{m\vec{\bullet}}^i\right\}_{i=1}^k\right)
&:=&
\inf_{\substack{D^i\in\left|V_{l,m\vec{a}^i}^i\right| \\ \text{for all }1\leq i\leq k\\
\text{and for some }\vec{a}^i\in\Z_{\geq 0}^{r_i-1}}}
\lct\left(X,B; \frac{1}{l}\sum_{i=1}^k c_i D^i\right)\\
&=&\inf_{D^i\in\left|V_{l,m\vec{a}^i}^i\right|}\inf_{\substack{E \text{ prime divisor}\\ 
\text{over }X}}\frac{A_{X, B}(E)}{\frac{1}{l}\sum_{i=1}^k c_i\ord_E D^i}\\
&=&\inf_{\substack{E \text{ prime divisor}\\ 
\text{over }X}}\frac{A_{X, B}(E)}{\sum_{i=1}^k c_i\frac{1}{l}
T_l\left(V_{m\vec{\bullet}}^i;E\right)},
\end{eqnarray*}
where $\lct$ is the log canonical threshold. 
Similarly, we set 
\begin{eqnarray*}
\delta_l\left(X, B; \left\{c_i\cdot V_{m\vec{\bullet}}^i\right\}_{i=1}^k\right)
&:=&
\inf_{\substack{D^i \text{ $l$-basis type} \\ 
\text{$\Q$-divisor of }V_{m\vec{\bullet}}^i\\
\text{for all }1\leq i\leq k}}
\lct\left(X,B; \sum_{i=1}^k c_i D^i\right)\\
&=&\inf_{\substack{D^i \text{ $l$-basis type} \\ 
\text{$\Q$-divisor of }V_{m\vec{\bullet}}^i}}
\inf_{\substack{E \text{ prime divisor}\\ 
\text{over }X}}\frac{A_{X, B}(E)}{\sum_{i=1}^k c_i\ord_E D^i}\\
&=&\inf_{\substack{E \text{ prime divisor}\\ 
\text{over }X}}\frac{A_{X, B}(E)}{\sum_{i=1}^k
c_i S_l\left(V_{m\vec{\bullet}}^i;E\right)}.
\end{eqnarray*}
\item\label{alpha-delta_definition12}
We set 
\begin{eqnarray*}
\alpha\left(X, B; \left\{c_i\cdot V_{\vec{\bullet}}^i\right\}_{i=1}^k\right)
&:=&\lim_{l\in m\Z_{>0}}
\alpha_l\left(X, B; \left\{c_i\cdot V_{m\vec{\bullet}}^i\right\}_{i=1}^k\right)\\
&=&\inf_{\substack{E \text{ prime divisor}\\ 
\text{over }X}}\frac{A_{X, B}(E)}{\sum_{i=1}^k
c_i T\left(V_{\vec{\bullet}}^i;E\right)},
\end{eqnarray*}
and 
\begin{eqnarray*}
\delta\left(X, B; \left\{c_i\cdot V_{\vec{\bullet}}^i\right\}_{i=1}^k\right)
&:=&\lim_{l\in m\Z_{>0}}
\delta_l\left(X, B; \left\{c_i\cdot V_{m\vec{\bullet}}^i\right\}_{i=1}^k\right)\\
&=&\inf_{\substack{E \text{ prime divisor}\\ 
\text{over }X}}\frac{A_{X, B}(E)}{\sum_{i=1}^k
c_i S\left(V_{\vec{\bullet}}^i;E\right)}.
\end{eqnarray*}
By the next proposition, the above definitions are well-defined. 
We call the value $\alpha\left(X, B; \left\{c_i\cdot V_{\vec{\bullet}}^i
\right\}_{i=1}^k\right)$ 
\emph{the coupled global log canonical threshold of $(X, B)$ with respects to 
$\{c_i\cdot V_{\vec{\bullet}}^i\}_{i=1}^k$}, and we call the value
$\delta\left(X, B; \left\{c_i\cdot V_{\vec{\bullet}}^i\right\}_{i=1}^k\right)$ 
\emph{the coupled stability threshold of $(X, B)$ with respects to 
$\{c_i\cdot V_{\vec{\bullet}}^i\}_{i=1}^k$}. 
\end{enumerate}
\item\label{alpha-delta_definition2}
Assume that $\eta\in X$ is a scheme-theoretic point such that $(X, B)$ is 
klt at $\eta$. 
\begin{enumerate}
\renewcommand{\theenumii}{\roman{enumii}}
\renewcommand{\labelenumii}{(\theenumii)}
\item\label{alpha-delta_definition21}
For any $l\in m\Z_{>0}$ with $\prod_{i=1}^k h^0\left(V_{l,m\vec{\bullet}}^i\right)\neq 0$, 
we set 
\begin{eqnarray*}
\alpha_{\eta,l}\left(X, B; \left\{c_i\cdot V_{m\vec{\bullet}}^i\right\}_{i=1}^k\right)
&:=&
\inf_{\substack{D^i\in\left|V_{l,m\vec{a}^i}^i\right| \\ \text{for all }1\leq i\leq k\\
\text{and for some }\vec{a}^i\in\Z_{\geq 0}^{r_i-1}}}
\lct_\eta\left(X,B; \frac{1}{l}\sum_{i=1}^k c_i D^i\right)\\
&=&\inf_{D^i\in\left|V_{l,m\vec{a}^i}^i\right|}\inf_{\substack{E \text{ prime divisor}\\ 
\text{over }X \\ \text{with }\eta\in C_X(E)}}
\frac{A_{X, B}(E)}{\frac{1}{l}\sum_{i=1}^k c_i\ord_E D^i}\\
&=&\inf_{\substack{E \text{ prime divisor}\\ 
\text{over }X \\ \text{with }\eta\in C_X(E)}}\frac{A_{X, B}(E)}{\sum_{i=1}^k c_i\frac{1}{l}
T_l\left(V_{m\vec{\bullet}}^i;E\right)},
\end{eqnarray*}
where $\lct_\eta$ is the log canonical threshold at $\eta$ and 
$C_X(E)\subset X$ is the center of $E$ on $X$.   
Similarly, we set 
\begin{eqnarray*}
\delta_{\eta,l}\left(X, B; \left\{c_i\cdot V_{m\vec{\bullet}}^i\right\}_{i=1}^k\right)
&:=&
\inf_{\substack{D^i \text{ $l$-basis type} \\ 
\text{$\Q$-divisor of }V_{m\vec{\bullet}}^i\\
\text{for all }1\leq i\leq k}}
\lct_\eta\left(X,B; \sum_{i=1}^k c_i D^i\right)\\
&=&\inf_{\substack{D^i \text{ $l$-basis type} \\ 
\text{$\Q$-divisor of }V_{m\vec{\bullet}}^i}}
\inf_{\substack{E \text{ prime divisor}\\ 
\text{over }X \\ \text{with }\eta\in C_X(E)}}\frac{A_{X, B}(E)}{\sum_{i=1}^k
c_i\ord_E D^i}\\
&=&\inf_{\substack{E \text{ prime divisor}\\ 
\text{over }X \\ \text{with }\eta\in C_X(E)}}\frac{A_{X, B}(E)}{\sum_{i=1}^k
c_i S_l\left(V_{m\vec{\bullet}}^i;E\right)}.
\end{eqnarray*}
\item\label{alpha-delta_definition22}
We set 
\begin{eqnarray*}
\alpha_\eta\left(X, B; \left\{c_i\cdot V_{\vec{\bullet}}^i\right\}_{i=1}^k\right)
&:=&\lim_{l\in m\Z_{>0}}
\alpha_{\eta,l}\left(X, B; \left\{c_i\cdot V_{m\vec{\bullet}}^i\right\}_{i=1}^k\right)\\
&=&\inf_{\substack{E \text{ prime divisor}\\ 
\text{over }X\\ \text{with }\eta\in C_X(E)}}\frac{A_{X, B}(E)}{\sum_{i=1}^k
c_i T\left(V_{\vec{\bullet}}^i;E\right)},
\end{eqnarray*}
and 
\begin{eqnarray*}
\delta_\eta\left(X, B; \left\{c_i\cdot V_{\vec{\bullet}}^i\right\}_{i=1}^k\right)
&:=&\lim_{l\in m\Z_{>0}}
\delta_{\eta,l}\left(X, B; \left\{c_i\cdot V_{m\vec{\bullet}}^i\right\}_{i=1}^k\right)\\
&=&\inf_{\substack{E \text{ prime divisor}\\ 
\text{over }X \\ \text{with }\eta\in C_X(E)}}\frac{A_{X, B}(E)}{\sum_{i=1}^k
c_i S\left(V_{\vec{\bullet}}^i;E\right)}.
\end{eqnarray*}
By the next proposition, the above definitions are well-defined. 
We call the value $\alpha_\eta\left(X, B; \left\{c_i\cdot 
V_{\vec{\bullet}}^i\right\}_{i=1}^k\right)$ 
\emph{the local coupled global log canonical threshold of $\eta\in (X, B)$ 
with respects to 
$\{c_i\cdot V_{\vec{\bullet}}^i\}_{i=1}^k$}, and we call the value 
$\delta_\eta\left(X, B; \left\{c_i\cdot V_{\vec{\bullet}}^i\right\}_{i=1}^k\right)$ 
\emph{the local coupled stability threshold of $\eta\in (X, B)$ with respects to 
$\{c_i\cdot V_{\vec{\bullet}}^i\}_{i=1}^k$}. 
\end{enumerate}
\item\label{alpha-delta_definition3}
Assume that $L_1,\dots,L_k$ are big $\Q$-Cartier $\Q$-divisors on $X$, 
$r_i=1$ and 
$V_{\vec{\bullet}}^i=H^0\left(\bullet L_i\right)$ for every $1\leq i\leq k$. 
Then we set 
\begin{eqnarray*}
\alpha\left(X, B; \left\{c_i\cdot L_i\right\}_{i=1}^k\right)&:=&
\alpha\left(X, B; \left\{c_i\cdot V_{\vec{\bullet}}^i\right\}_{i=1}^k\right), \\
\delta\left(X, B; \left\{c_i\cdot L_i\right\}_{i=1}^k\right)&:=&
\delta\left(X, B; \left\{c_i\cdot V_{\vec{\bullet}}^i\right\}_{i=1}^k\right), 
\end{eqnarray*}
and so on. 
\item\label{alpha-delta_definition4}
If $c_1=\cdots=c_k=1$, then we write 
$\alpha\left(X, B; \left\{V_{\vec{\bullet}}^i\right\}_{i=1}^k\right)$, 
$\delta\left(X, B; \left\{V_{\vec{\bullet}}^i\right\}_{i=1}^k\right)$, etc.; 
if $k=1$, then we write 
$\alpha\left(X, B; c_1\cdot V_{\vec{\bullet}}^1\right)$, 
$\delta\left(X, B; c_1\cdot V_{\vec{\bullet}}^1\right)$, etc.; if $B=0$, then we write 
$\alpha\left(X; \left\{c_i\cdot L_i\right\}_{i=1}^k\right)$, 
$\delta\left(X; \left\{c_i\cdot L_i\right\}_{i=1}^k\right)$, etc., just for simplicity. 
\end{enumerate}
\end{definition}

The above definitions are well-defined thanks to the following well-known proposition. 
See \cite[Theorem 4.4]{BJ}, \cite[Lemma 2.21]{AZ}, \cite[Proposition 11.13]{r3d28}
and \cite[\S A]{hashimoto}. 

\begin{proposition}[{cf.\ \cite[Theorem 4.4]{BJ}, 
\cite[Proposition 11.13]{r3d28}}]\label{alpha-delta_proposition}
\begin{enumerate}
\renewcommand{\theenumi}{\arabic{enumi}}
\renewcommand{\labelenumi}{(\theenumi)}
\item\label{alpha-delta_proposition1}
Under the notion in Definition \ref{alpha-delta_definition} 
\eqref{alpha-delta_definition1}, we have 
\[
\lim_{l\in m\Z_{>0}}
\alpha_l\left(X, B; \left\{c_i\cdot V_{m\vec{\bullet}}^i\right\}_{i=1}^k\right)
=\inf_{\substack{E \text{ prime divisor} \\
\text{over }X}}\frac{A_{X, B}(E)}{\sum_{i=1}^k
c_i T\left(V_{\vec{\bullet}}^i;E\right)},
\]
and 
\[
\lim_{l\in m\Z_{>0}}
\delta_l\left(X, B; \left\{c_i\cdot V_{m\vec{\bullet}}^i\right\}_{i=1}^k\right)
=\inf_{\substack{E \text{ prime divisor} \\
\text{over }X}}\frac{A_{X, B}(E)}{\sum_{i=1}^k
c_i S\left(V_{\vec{\bullet}}^i;E\right)}.
\]
\item\label{alpha-delta_proposition2}
Under the notion in Definition \ref{alpha-delta_definition} 
\eqref{alpha-delta_definition2}, we have 
\[
\lim_{l\in m\Z_{>0}}
\alpha_{\eta,l}\left(X, B; \left\{c_i\cdot V_{m\vec{\bullet}}^i\right\}_{i=1}^k\right)
=\inf_{\substack{E \text{ prime divisor} \\
\text{over }X\\ \text{with }\eta\in C_X(E)}}\frac{A_{X, B}(E)}{\sum_{i=1}^k
c_i T\left(V_{\vec{\bullet}}^i;E\right)},
\]
and 
\[
\lim_{l\in m\Z_{>0}}
\delta_{\eta,l}\left(X, B; \left\{c_i\cdot V_{m\vec{\bullet}}^i\right\}_{i=1}^k\right)
=\inf_{\substack{E \text{ prime divisor} \\
\text{over }X\\ \text{with }\eta\in C_X(E)}}\frac{A_{X, B}(E)}{\sum_{i=1}^k
c_i S\left(V_{\vec{\bullet}}^i;E\right)}.
\]
\end{enumerate}
\end{proposition}

\begin{proof}
We only see \eqref{alpha-delta_proposition1}. Since 
\begin{eqnarray*}
\inf_{\substack{E \text{ prime divisor} \\
\text{over }X}}\frac{A_{X, B}(E)}{\sum_{i=1}^k
c_i T\left(V_{\vec{\bullet}}^i;E\right)}&\geq&
\limsup_{l\in m\Z_{>0}}\alpha_l\left(X, B; 
\left\{c_i\cdot V_{m\vec{\bullet}}^i\right\}_{i=1}^k\right)\\
\geq\inf_{l\in m\Z_{>0}}\alpha_l\left(X, B; 
\left\{c_i\cdot V_{m\vec{\bullet}}^i\right\}_{i=1}^k\right)
&=&\inf_{\substack{E \text{ prime divisor} \\
\text{over }X}}\frac{A_{X, B}(E)}{\sum_{i=1}^k
c_i T\left(V_{\vec{\bullet}}^i;E\right)}, 
\end{eqnarray*}
the first assertion follows. Similarly, we have 
\[
\limsup_{l\in m\Z_{>0}}\delta_l\left(X, B; 
\left\{c_i\cdot V_{m\vec{\bullet}}^i\right\}_{i=1}^k\right)\leq
\inf_{\substack{E \text{ prime divisor} \\
\text{over }X}}\frac{A_{X, B}(E)}{\sum_{i=1}^k
c_i S\left(V_{\vec{\bullet}}^i;E\right)}. 
\]
On the other hand, by Lemma \ref{bt-div_lemma} \eqref{bt-div_lemma1}, 
for any $\varepsilon\in\Q_{>0}$, we have 
\[
\liminf_{l\in m\Z_{>0}}\delta_l\left(X, B; 
\left\{c_i\cdot V_{m\vec{\bullet}}^i\right\}_{i=1}^k\right)\geq
\frac{1}{1+\varepsilon}\cdot\inf_{\substack{E \text{ prime divisor} \\
\text{over }X}}\frac{A_{X, B}(E)}{\sum_{i=1}^k
c_i S\left(V_{\vec{\bullet}}^i;E\right)}. 
\]
Thus we also get the second assertion. 
\end{proof}

\begin{remark}\label{valuation_remark}
Assume that $(X, B)$ is klt at a scheme-theoretic point $\eta$. 
As in \cite{BJ}, we have the following equalities: 
\begin{eqnarray*}
\alpha_{\eta,l}\left(X, B;\left\{c_i\cdot V^i_{m\vec{\bullet}}\right\}_{i=1}^k\right)
&=&\inf_v\frac{A_{X,B}(v)}{\sum_{i=1}^k c_i
\frac{1}{l}T_l\left(V_{m\vec{\bullet}}^i; v\right)}, \\
\delta_{\eta,l}\left(X, B;\left\{c_i\cdot V^i_{m\vec{\bullet}}\right\}_{i=1}^k\right)
&=&\inf_v\frac{A_{X,B}(v)}{\sum_{i=1}^k c_i
S_l\left(V_{m\vec{\bullet}}^i; v\right)}, \\
\alpha_{\eta}\left(X, B;\left\{c_i\cdot V^i_{\vec{\bullet}}\right\}_{i=1}^k\right)
&=&\inf_v\frac{A_{X,B}(v)}{\sum_{i=1}^k c_i
T\left(V_{\vec{\bullet}}^i; v\right)}, \\
\delta_{\eta}\left(X, B;\left\{c_i\cdot V^i_{\vec{\bullet}}\right\}_{i=1}^k\right)
&=&\inf_v\frac{A_{X,B}(v)}{\sum_{i=1}^k c_i
S\left(V_{\vec{\bullet}}^i; v\right)}, 
\end{eqnarray*}
where $v$ runs through all valuations on $X$ with $A_{X, B}(v)<\infty$ and 
$\eta\in C_X(v)$. See \cite{BJ} for detail. 
\end{remark}

\begin{definition}\label{alpha-delta-flag_definition}
\begin{enumerate}
\renewcommand{\theenumi}{\arabic{enumi}}
\renewcommand{\labelenumi}{(\theenumi)}
\item\label{alpha-delta-flag_definition1}
Let $U\subset X$ be an open subscheme and let 
\[
Y_\bullet\colon X=Y_0\triangleright Y_1\triangleright\cdots\triangleright Y_j
\]
be a plt flag over $(U, B|_U)$. For any scheme-theoretic point $\eta\in Y_j$ over $U$, 
we set 
\begin{eqnarray*}
\alpha_\eta\left(X, B\triangleright Y_1\triangleright\cdots\triangleright Y_j; 
\left\{c_i\cdot V_{\vec{\bullet}}^i\right\}_{i=1}^k\right)
&:=&\alpha_\eta\left(Y_j, B_j;\left\{c_i\cdot V_{\vec{\bullet}}^{i, 
\left(Y_1\triangleright\cdots\triangleright Y_j\right)}\right\}_{i=1}^k
\right), \\ 
\delta_\eta\left(X, B\triangleright Y_1\triangleright\cdots\triangleright Y_j; 
\left\{c_i\cdot V_{\vec{\bullet}}^i\right\}_{i=1}^k\right)
&:=&\delta_\eta\left(Y_j, B_j;\left\{c_i\cdot V_{\vec{\bullet}}^{i, 
\left(Y_1\triangleright\cdots\triangleright Y_j\right)}\right\}_{i=1}^k
\right), 
\end{eqnarray*}
where $(Y_j, B_j)$ is the associated klt pair over $U$. In other words, 
\begin{eqnarray*}
\alpha_\eta\left(X, B\triangleright Y_1\triangleright\cdots\triangleright Y_j; 
\left\{c_i\cdot V_{\vec{\bullet}}^i\right\}_{i=1}^k\right)&=&
\inf_{\substack{E \text{ prime divisor} \\
\text{over }Y_j\\ \text{with }\eta\in C_{Y_j}(E)}}\frac{A_{X, B}(Y_1\triangleright\cdots
\triangleright Y_j\triangleright E)}{\sum_{i=1}^k
c_i T\left(V_{\vec{\bullet}}^i;Y_1\triangleright\cdots\triangleright Y_j\triangleright 
E\right)}, \\
\delta_\eta\left(X, B\triangleright Y_1\triangleright\cdots\triangleright Y_j; 
\left\{c_i\cdot V_{\vec{\bullet}}^i\right\}_{i=1}^k\right)&=&
\inf_{\substack{E \text{ prime divisor} \\
\text{over }Y_j\\ \text{with }\eta\in C_{Y_j}(E)}}\frac{A_{X, B}(Y_1\triangleright\cdots
\triangleright Y_j\triangleright E)}{\sum_{i=1}^k
c_i S\left(V_{\vec{\bullet}}^i;Y_1\triangleright\cdots\triangleright Y_j\triangleright 
E\right)}.
\end{eqnarray*}
\item\label{alpha-delta-flag_definition2}
If 
\[
Y_\bullet\colon X=Y_0\triangleright Y_1\triangleright\cdots\triangleright Y_j
\]
is a plt flag over $(X, B)$, we set 
\begin{eqnarray*}
\alpha\left(X, B\triangleright Y_1\triangleright\cdots\triangleright Y_j; 
\left\{c_i\cdot V_{\vec{\bullet}}^i\right\}_{i=1}^k\right)
&:=&\alpha\left(Y_j, B_j;\left\{c_i\cdot V_{\vec{\bullet}}^{i, 
\left(Y_1\triangleright\cdots\triangleright Y_j\right)}\right\}_{i=1}^k
\right), \\ 
\delta\left(X, B\triangleright Y_1\triangleright\cdots\triangleright Y_j; 
\left\{c_i\cdot V_{\vec{\bullet}}^i\right\}_{i=1}^k\right)
&:=&\delta\left(Y_j, B_j;\left\{c_i\cdot V_{\vec{\bullet}}^{i, 
\left(Y_1\triangleright\cdots\triangleright Y_j\right)}\right\}_{i=1}^k
\right), 
\end{eqnarray*}
where $(Y_j, B_j)$ is the associated klt pair.
In other words, 
\begin{eqnarray*}
\alpha\left(X, B\triangleright Y_1\triangleright\cdots\triangleright Y_j; 
\left\{c_i\cdot V_{\vec{\bullet}}^i\right\}_{i=1}^k\right)&=&
\inf_{\substack{E \text{ prime divisor} \\
\text{over }Y_j}}\frac{A_{X, B}(Y_1\triangleright\cdots
\triangleright Y_j\triangleright E)}{\sum_{i=1}^k
c_i T\left(V_{\vec{\bullet}}^i;Y_1\triangleright\cdots\triangleright Y_j\triangleright 
E\right)}, \\
\delta\left(X, B\triangleright Y_1\triangleright\cdots\triangleright Y_j; 
\left\{c_i\cdot V_{\vec{\bullet}}^i\right\}_{i=1}^k\right)&=&
\inf_{\substack{E \text{ prime divisor} \\
\text{over }Y_j}}\frac{A_{X, B}(Y_1\triangleright\cdots
\triangleright Y_j\triangleright E)}{\sum_{i=1}^k
c_i S\left(V_{\vec{\bullet}}^i;Y_1\triangleright\cdots\triangleright Y_j\triangleright 
E\right)}.
\end{eqnarray*}
If $r_i=1$ and $L^i:=L_1^i\in\CaCl(X)\otimes_\Z\Q$ are big for all $1\leq i\leq r$, 
then we set 
\begin{eqnarray*}
\alpha\left(X, B\triangleright Y_1\triangleright\cdots\triangleright Y_j; 
\left\{c_i\cdot L^i\right\}_{i=1}^k\right)
&:=&\alpha\left(X, B\triangleright Y_1\triangleright\cdots\triangleright Y_j; 
\left\{c_i\cdot V_{\vec{\bullet}}^i\right\}_{i=1}^k\right), \\ 
\delta\left(X, B\triangleright Y_1\triangleright\cdots\triangleright Y_j; 
\left\{c_i\cdot L^i\right\}_{i=1}^k\right)
&:=&\delta\left(X, B\triangleright Y_1\triangleright\cdots\triangleright Y_j; 
\left\{c_i\cdot V_{\vec{\bullet}}^i\right\}_{i=1}^k\right), \\ 
\end{eqnarray*}
and so on. 
\end{enumerate}
\end{definition}

We see basic properties of coupled global log canonical thresholds and 
coupled stability thresholds. The following proposition is true even if we replace 
``$(X, B)$ is klt", ``$\alpha$" and ``$\delta$", 
with ``$\eta\in X$ is a scheme-theoretic point which is not the 
generic point of $X$ such that $(X, B)$ is klt at $\eta$", ``$\alpha_\eta$" and 
``$\delta_\eta$", respectively. 

\begin{proposition}\label{delta-basic_proposition}
Assume that $(X, B)$ is klt. 
\begin{enumerate}
\renewcommand{\theenumi}{\arabic{enumi}}
\renewcommand{\labelenumi}{(\theenumi)}
\item\label{delta-basic_proposition1}
We have 
\begin{eqnarray*}
\alpha\left(X,B;\left\{c_i\cdot V_{\vec{\bullet}}^i\right\}_{i=1}^k\right)\leq
\delta\left(X,B;\left\{c_i\cdot V_{\vec{\bullet}}^i\right\}_{i=1}^k\right)\leq
\left(\max_{1\leq i\leq k}\left\{r_i\right\}+n\right)\cdot
\alpha\left(X,B;\left\{c_i\cdot V_{\vec{\bullet}}^i\right\}_{i=1}^k\right).
\end{eqnarray*}
\item\label{delta-basic_proposition2}
If $c'_1,\dots,c'_k\in\R_{>0}$ satisfies that $c'_i\geq c_i$ for any $1\leq i\leq k$, 
then we have 
\begin{eqnarray*}
\alpha\left(X,B;\left\{c_i\cdot V_{\vec{\bullet}}^i\right\}_{i=1}^k\right)&\geq&
\alpha\left(X,B;\left\{c'_i\cdot V_{\vec{\bullet}}^i\right\}_{i=1}^k\right), \\
\delta\left(X,B;\left\{c_i\cdot V_{\vec{\bullet}}^i\right\}_{i=1}^k\right)&\geq&
\delta\left(X,B;\left\{c'_i\cdot V_{\vec{\bullet}}^i\right\}_{i=1}^k\right).
\end{eqnarray*}
\item\label{delta-basic_proposition3}
We have 
$\alpha\left(X,B;\left\{c_i\cdot V_{\vec{\bullet}}^i\right\}_{i=1}^k\right)\in\R_{>0}$ and 
$\delta\left(X,B;\left\{c_i\cdot V_{\vec{\bullet}}^i\right\}_{i=1}^k\right)\in\R_{>0}$.
\item\label{delta-basic_proposition4}
For any $c'_1,\dots,c'_k\in\Q_{>0}$, we have 
\begin{eqnarray*}
\alpha\left(X,B;\left\{c_i c'_i\cdot V_{\vec{\bullet}}^i\right\}_{i=1}^k\right)&=&
\alpha\left(X,B;\left\{c_i\cdot c'_i V_{\vec{\bullet}}^i\right\}_{i=1}^k\right), \\
\delta\left(X,B;\left\{c_ic'_i\cdot V_{\vec{\bullet}}^i\right\}_{i=1}^k\right)&=&
\delta\left(X,B;\left\{c_i\cdot c'_i V_{\vec{\bullet}}^i\right\}_{i=1}^k\right),
\end{eqnarray*}
where $c'_i V_{\vec{\bullet}}^i$ is as in Definition \ref{interior_definition} 
\eqref{interior_definition1}. 
\item\label{delta-basic_proposition5}
Take any $p\in\Z_{>0}$. For any $1\leq i\leq k$, take any $\vec{p}^i
=(p^i_1,\dots,p^i_{r_i})\in\Z_{>0}^{r_i}$ with $p^i_1=p$. Then we have 
\begin{eqnarray*}
\alpha\left(X,B;\left\{c_i \cdot V_{\vec{\bullet}}^{i,(\vec{p}^i)}\right\}_{i=1}^k\right)&=&
\frac{1}{p}\cdot
\alpha\left(X,B;\left\{c_i \cdot V_{\vec{\bullet}}^i\right\}_{i=1}^k\right), \\
\delta\left(X,B;\left\{c_i \cdot V_{\vec{\bullet}}^{i,(\vec{p}^i)}\right\}_{i=1}^k\right)&=&
\frac{1}{p}\cdot\delta\left(X,B;\left\{c_i \cdot V_{\vec{\bullet}}^i\right\}_{i=1}^k\right),
\end{eqnarray*}
\item\label{delta-basic_proposition6}
For any $c\in \R_{>0}$, we have 
\begin{eqnarray*}
\alpha\left(X,B;\left\{c c_i\cdot V_{\vec{\bullet}}^i\right\}_{i=1}^k\right)&=&
\frac{1}{c}\cdot\alpha\left(X,B;\left\{c_i\cdot 
V_{\vec{\bullet}}^i\right\}_{i=1}^k\right), \\
\delta\left(X,B;\left\{c c_i\cdot V_{\vec{\bullet}}^i\right\}_{i=1}^k\right)&=&
\frac{1}{c}\cdot\delta\left(X,B;\left\{c_i\cdot V_{\vec{\bullet}}^i\right\}_{i=1}^k\right). 
\end{eqnarray*}
\item\label{delta-basic_proposition7}
Assume that there exists $0\leq k'\leq k-1$, $c'_{k'+1},\dots,c'_k\in\Q_{>0}$ and 
a graded series $V_{\vec{\bullet}}$ such that $r_{k'+1}=\cdots=r_k$ and 
$V_{\vec{\bullet}}^j=c'_j V_{\vec{\bullet}}$ for any $k'+1\leq j\leq k$. Then we have 
\begin{eqnarray*}
\alpha\left(X,B;\left\{c_i\cdot V_{\vec{\bullet}}^i\right\}_{i=1}^k\right)&=&
\alpha\left(X,B;\left\{c_i\cdot V_{\vec{\bullet}}^i\right\}_{i=1}^{k'}\cup
\left\{\left(\sum_{j=k'+1}^k c_j c'_j\right)\cdot V_{\vec{\bullet}}\right\}\right), \\
\delta\left(X,B;\left\{c_i\cdot V_{\vec{\bullet}}^i\right\}_{i=1}^k\right)&=&
\delta\left(X,B;\left\{c_i\cdot V_{\vec{\bullet}}^i\right\}_{i=1}^{k'}\cup
\left\{\left(\sum_{j=k'+1}^k c_j c'_j\right)\cdot V_{\vec{\bullet}}\right\}\right). 
\end{eqnarray*}
\item\label{delta-basic_proposition8}
Let $L, L_1,\dots,L_k$ are big $\Q$-Cartier $\Q$-divisors on $X$. 
Assume that there exists $0\leq k'\leq k-1$ and $c'_{k'+1},\dots,c'_k\in\Q_{>0}$ 
such that $L_j\equiv c'_j L$ for any $k'+1\leq j\leq k$. Then we have 
\begin{eqnarray*}
\alpha\left(X,B;\left\{c_i\cdot L_i\right\}_{i=1}^k\right)&=&
\alpha\left(X,B;\left\{c_i\cdot L_i\right\}_{i=1}^{k'}\cup
\left\{\left(\sum_{j=k'+1}^k c_j c'_j\right)\cdot L\right\}\right), \\
\delta\left(X,B;\left\{c_i\cdot L_i\right\}_{i=1}^k\right)&=&
\delta\left(X,B;\left\{c_i\cdot L_i\right\}_{i=1}^{k'}\cup
\left\{\left(\sum_{j=k'+1}^k c_j c'_j\right)\cdot L\right\}\right). 
\end{eqnarray*}
In particular, when moreover $k'=0$, we have 
\begin{eqnarray*}
\alpha\left(X,B;\left\{c_i\cdot L_i\right\}_{i=1}^k\right)&=&
\frac{1}{\sum_{i=1}^k c_i c'_i}\cdot\alpha\left(X,B;L\right), \\
\delta\left(X,B;\left\{c_i\cdot L_i\right\}_{i=1}^k\right)&=&
\frac{1}{\sum_{i=1}^k c_i c'_i}\cdot\delta\left(X,B;L\right). 
\end{eqnarray*}
\item\label{delta-basic_proposition9}
Take any division 
\[
\left\{1,\dots, k\right\}=I_1\sqcup\dots\sqcup I_l
\]
with $I_j\neq \emptyset$ for any $1\leq j\leq l$. 
We have the inequalities 
\begin{eqnarray*}
\alpha\left(X, B;\left\{c_i\cdot V_{\vec{\bullet}}^i\right\}_{i=1}^k\right)^{-1}
\leq\sum_{j=1}^l
\alpha\left(X, B;\left\{c_i\cdot V_{\vec{\bullet}}^i\right\}_{i\in I_j}\right)^{-1}, \\
\delta\left(X, B;\left\{c_i\cdot V_{\vec{\bullet}}^i\right\}_{i=1}^k\right)^{-1}
\leq\sum_{j=1}^l
\delta\left(X, B;\left\{c_i\cdot V_{\vec{\bullet}}^i\right\}_{i\in I_j}\right)^{-1}.
\end{eqnarray*}
In particular, we have 
\begin{eqnarray*}
\alpha\left(X, B;\left\{c_i\cdot V_{\vec{\bullet}}^i\right\}_{i=1}^k\right)^{-1}
\leq\sum_{i=1}^k
c_i\cdot\alpha\left(X, B;V_{\vec{\bullet}}^i\right)^{-1}, \\
\delta\left(X, B;\left\{c_i\cdot V_{\vec{\bullet}}^i\right\}_{i=1}^k\right)^{-1}
\leq\sum_{i=1}^k
c_i\cdot\delta\left(X, B;V_{\vec{\bullet}}^i\right)^{-1}.
\end{eqnarray*}
\item\label{delta-basic_proposition10}
For any $1\leq i\leq k$, let $\Lambda_i$ be a finite set and let us consider 
a decomposition 
\[
\Delta_{\Supp\left(V_{\vec{\bullet}}^i\right)}=\overline{\bigcup_{\lambda\in\Lambda_i}
\Delta_{\Supp}^{i, \langle\lambda\rangle}}
\]
and consider $V_{\vec{\bullet}}^{i,\langle\lambda\rangle}$ in the sense of 
Definition \ref{interior_definition} \eqref{interior_definition4}. Then we have 
\[
\delta\left(X, B;\left\{c_i\cdot V_{\vec{\bullet}}^i\right\}_{i=1}^k\right)
=\delta\left(X, B;\left\{c_i
\frac{\vol\left(V_{\vec{\bullet}}^{i,\langle\lambda
\rangle}\right)}{\vol\left(V_{\vec{\bullet}}^i\right)}\cdot 
V_{\vec{\bullet}}^{i,\langle\lambda\rangle}\right\}_{1\leq i\leq k, \lambda\in\Lambda_i}\right).
\]
\item\label{delta-basic_proposition11}
Both the functions 
\begin{eqnarray*}
\alpha\colon\R^k_{>0}&\to&\R_{>0}\\
(t_1,\dots,t_k)&\mapsto&
\alpha\left(X, B;\left\{t_i\cdot V_{\vec{\bullet}}^i\right\}_{i=1}^k\right), \\
\delta\colon\R^k_{>0}&\to&\R_{>0}\\
(t_1,\dots,t_k)&\mapsto&
\delta\left(X, B;\left\{t_i\cdot V_{\vec{\bullet}}^i\right\}_{i=1}^k\right)
\end{eqnarray*}
are continuous. 
\end{enumerate}
\end{proposition}

\begin{proof}
The assertion \eqref{delta-basic_proposition1} follows from 
Definition \ref{S-T_definition} \eqref{S-T_definition2}. 
The assertions \eqref{delta-basic_proposition2} and \eqref{delta-basic_proposition6} 
are trivial. 
The assertion \eqref{delta-basic_proposition3} follows from 
\eqref{delta-basic_proposition1} and the argument in \cite[Proposition 11.1]{r3d28}. 
The assertion \eqref{delta-basic_proposition5} follows from \cite[Lemma 3.10]{r3d28}. 
The assertions 
\eqref{delta-basic_proposition4}, \eqref{delta-basic_proposition7}, 
\eqref{delta-basic_proposition8} follow from 
the facts $T\left(c V_{\vec{\bullet}}; E\right)=c\cdot T\left(V_{\vec{\bullet}}; E\right)$
and $S\left(c V_{\vec{\bullet}}; E\right)=c\cdot S\left(V_{\vec{\bullet}}; E\right)$ 
for $c\in\Q_{>0}$. 
The assertion \eqref{delta-basic_proposition9} follows from 
\begin{eqnarray*}
\delta\left(X, B;\left\{c_i\cdot V_{\vec{\bullet}}^i\right\}_{i=1}^k\right)^{-1}
&=&\sup_{E/X}\frac{\sum_{j=1}^l\sum_{i\in I_j}c_i\cdot S\left(V_{\vec{\bullet}}^i; E
\right)}{A_{X,B}(E)}\\
\leq\sum_{j=1}^l\sup_{E/X}\frac{\sum_{i\in I_j} c_i\cdot S\left(V_{\vec{\bullet}}^i; E
\right)}{A_{X, B}(E)}
&=&\sum_{j=1}^l
\delta\left(X, B;\left\{c_i\cdot V_{\vec{\bullet}}^i\right\}_{i\in I_j}\right)^{-1}.
\end{eqnarray*}
The assertion \eqref{delta-basic_proposition10} follows from 
Proposition \ref{decomposition_proposition}. 
Let us consider the assertion \eqref{delta-basic_proposition11}. 
Take any $\vec{t}=(t_1,\dots,t_k)\in\R_{>0}^k$ and $\varepsilon\in\R_{>0}$. 
By \eqref{delta-basic_proposition6}, we have 
$\delta\left(a\vec{t}\right)=a^{-1}\delta\left(\vec{t}\right)$ for any $a\in\R_{>0}$. 
Take any small $\varepsilon_1\in\R_{>0}$ with 
\[
\delta\left(\vec{t}\right)-\varepsilon<\frac{\delta\left(\vec{t}\right)}{1+\varepsilon_1}
\quad\text{and}\quad
\delta\left(\vec{t}\right)+\varepsilon>\frac{\delta\left(\vec{t}\right)}{1-\varepsilon_1}.
\]
Fix a norm $\|\cdot\|$ on $\R^k$. 
By Lemma \ref{cone_lemma}, there exists $\delta'\in\R_{>0}$ such that 
for any $\vec{t}'=(t'_1,\dots,t'_k)\in\R^k_{>0}$ with $\|\vec{t}'-\vec{t}\|<\delta'$, 
we have 
\[
(1+\varepsilon_1)t_i\geq t'_i\quad\text{and}\quad
t'_i\geq (1-\varepsilon_1)t_i
\]
hold for all $1\leq i\leq k$. This implies that 
\[
\frac{\delta\left(\vec{t}\right)}{1+\varepsilon_1}\leq
\delta\left(\vec{t}'\right)\leq
\frac{\delta\left(\vec{t}\right)}{1-\varepsilon_1}
\]
by \eqref{delta-basic_proposition2}. Thus we get the assertion. 
\end{proof}

\begin{lemma}\label{cone_lemma}
Fix a norm $\|\cdot\|$ on $\R^r$. Take any open cone $\sC\subset\R^r$.
For any compact subset $K\subset \R^r$ 
with $K\subset\sC$ and for any $\varepsilon\in\R_{>0}$, there exists 
$\delta\in\R_{>0}$ such that, for any $\vec{x}$, $\vec{y}\in K$ with 
$\|\vec{y}-\vec{x}\|<\delta$, we have $(1+\varepsilon)\vec{x}-\vec{y}\in\sC$ and 
$\vec{y}-(1-\varepsilon)\vec{x}\in\sC$. 
\end{lemma}

\begin{proof}
Fix $\vec{c}\in\sC$ such that $K\subset\vec{c}+\sC$ and set 
\[
U:=\left(-\varepsilon\vec{c}+\sC\right)\cap\left(\varepsilon\vec{c}-\sC\right)
\subset\R^r. 
\]
Since $U$ is open with $\vec{0}\in U$, there exists $\delta\in\R_{>0}$ such that 
\[
\left\{\vec{z}\in\R^r\,\,|\,\,\|z\|<\delta\right\}\subset U
\]
holds. For any $\vec{x}$, $\vec{y}\in K$ with 
$\|\vec{y}-\vec{x}\|<\delta$, we have $\vec{x}$, $\vec{y}\in K$ with 
$\|\vec{y}-\vec{x}\|<\delta$, we have 
\begin{eqnarray*}
\vec{y}\in\vec{x}+U&=&\left(-\varepsilon\vec{c}+\vec{x}+\sC\right)\cap
\left(\varepsilon\vec{c}+\vec{x}-\sC\right)\\
&\subset&\left((1-\varepsilon)\vec{x}+\sC\right)\cap
\left((1+\varepsilon)\vec{x}-\sC\right),
\end{eqnarray*}
since $\vec{x}-\vec{c}\in\sC$.
\end{proof}

\begin{remark}\label{delta-basic_remark}
\begin{enumerate}
\renewcommand{\theenumi}{\arabic{enumi}}
\renewcommand{\labelenumi}{(\theenumi)}
\item\label{delta-basic_remark1}
By Proposition \ref{delta-basic_proposition} \eqref{delta-basic_proposition4}, 
there is no confusion if we write 
\[\alpha\left(X,B;\left\{c_i V_{\vec{\bullet}}^i\right\}_{i=1}^k\right), \quad
\delta\left(X,B;\left\{c_i V_{\vec{\bullet}}^i\right\}_{i=1}^k\right), 
\]
etc., in place of 
\[
\alpha\left(X,B;\left\{c_i\cdot V_{\vec{\bullet}}^i\right\}_{i=1}^k\right), \quad
\delta\left(X,B;\left\{c_i\cdot V_{\vec{\bullet}}^i\right\}_{i=1}^k\right).
\]
\item\label{delta-basic_remark2}
By Proposition \ref{delta-basic_proposition} 
\eqref{delta-basic_proposition4} and \eqref{delta-basic_proposition11}, 
we are mainly interested in the case $c_1=\cdots=c_k=1$. 
\end{enumerate}
\end{remark}

From now on, 
we assume that $(X, B)$ is klt and 
the Veronese equivalence class $V_{\vec{\bullet}}^i$ 
of an $(m\Z_{\geq 0})^{r_i}$-graded linear series 
$V_{m\vec{\bullet}}^i$ on $X$ associated to 
$L_1^i,\dots,L_{r_i}^i\in\CaCl(X)\otimes_\Z\Q$ containing an ample series 
which \emph{does not need to have bounded support in general} 
for any $1\leq i\leq k$. 
We consider a generalization of Dervan and Kewei Zhang's results 
\cite[Theorem 1.4]{dervan}, \cite[Theorem 1.7]{kewei}. 
Let us set 
\begin{eqnarray*}
\sC_i&:=&\interior\left(\Supp\left(V_{\vec{\bullet}}^i\right)\right)
\subset\R_{>0}^{r_i}, \\
\sC&:=&\prod_{i=1}^k \sC_i\subset\R_{>0}^{r_1+\cdots+r_k}.
\end{eqnarray*}
For any $\vec{a}^i\in\sC_i\cap\Q^{r_i}$, we considered the series 
$V_{\bullet\vec{a}^i}^i$ in Definition \ref{interior_definition} \eqref{interior_definition5}. 
Consider the function 
\begin{eqnarray*}
\sC_i\cap\Q^{r_i}&\to&\R_{>0}\\
\vec{a}^i&\mapsto&\vol\left(V_{\bullet\vec{a}^i}^i\right)^{\frac{1}{n}}. 
\end{eqnarray*}
By \cite[Corollary 4.22]{LM}, the function 
uniquely extends to a concave (in particular, continuous) and homogeneous function 
\[
\vol_{V_{\vec{\bullet}}^i}^{\frac{1}{n}}\colon
\sC_i\to\R_{>0}.
\]
Let us consider the behaviors of the values 
\[
\alpha(\vec{a}):=\alpha\left(X, B; \left\{V_{\bullet\vec{a}^i}^i\right\}_{i=1}^k\right), \quad
\delta(\vec{a}):=\delta\left(X, B; \left\{V_{\bullet\vec{a}^i}^i\right\}_{i=1}^k\right)
\]
for every $\vec{a}=\left(\vec{a}^1,\dots,\vec{a}^k\right)\in
\sC\cap\Q^{r_1+\cdots+r_k}$. 

\begin{lemma}\label{alpha-compare_lemma}
Take $\vec{a}$, $\vec{b}\in\sC\cap\Q^{r_1+\cdots+r_k}$ 
with $\vec{b}-\vec{a}\in\sC$. Fix a sufficiently divisible $m\in\Z_{>0}$ 
such that $V_{\bullet\vec{a}^i}^i$, $V_{\bullet\vec{b}^i}^i$ are obtained by 
$V_{m\bullet\vec{a}^i}^i$, $V_{m\bullet\vec{b}^i}^i$ 
for any $1\leq i\leq k$, respectively. 
Then, for any sufficiently divisible $l\in m\Z_{>0}$, we have 
\begin{eqnarray*}
\alpha_l\left(X, B;\left\{V_{m\bullet\vec{a}^i}^i\right\}_{i=1}^k\right)
&\geq&
\alpha_l\left(X, B;\left\{V_{m\bullet\vec{b}^i}^i\right\}_{i=1}^k\right), \\
\frac{\delta_l\left(X, B;\left\{V_{m\bullet\vec{a}^i}^i\right\}_{i=1}^k\right)}
{\min_{1\leq i\leq k}\dim V_{l\vec{a}^i}^i}
&\geq&
\frac{\delta_l\left(X, B;\left\{V_{m\bullet\vec{b}^i}^i\right\}_{i=1}^k\right)}
{\max_{1\leq i\leq k}\dim V_{l\vec{b}^i}^i}.
\end{eqnarray*}
In particular, we have 
\begin{eqnarray*}
\alpha\left(X, B;\left\{V_{\bullet\vec{a}^i}^i\right\}_{i=1}^k\right)
&\geq&
\alpha\left(X, B;\left\{V_{\bullet\vec{b}^i}^i\right\}_{i=1}^k\right), \\
\frac{\delta\left(X, B;\left\{V_{\bullet\vec{a}^i}^i\right\}_{i=1}^k\right)}
{\min_{1\leq i\leq k}\vol\left(V_{\bullet\vec{a}^i}^i\right)}
&\geq&
\frac{\delta\left(X, B;\left\{V_{\bullet\vec{b}^i}^i\right\}_{i=1}^k\right)}
{\max_{1\leq i\leq k}\vol\left(V_{\bullet\vec{b}^i}^i\right)}.
\end{eqnarray*}
\end{lemma}

\begin{proof}
Set $\vec{c}:=\vec{b}-\vec{a}\in\sC$. By \cite[Lemma 4.18]{LM}, 
we may assume that there exist effective $\Q$-divisors 
$C^i\sim_\Q\vec{c}^i\cdot\vec{L}^i$ with 
$l C^i\in\left|V_{l\vec{c}^i}^i\right|$ for all $1\leq i\leq k$. 
For any $1\leq i\leq k$ and for any $D^i\in\left|V_{l\vec{a}^i}^i\right|$, since 
$D^i+l C^i\in\left|V_{l \vec{b}^i}^i\right|$, we get 
\[
\lct\left(X, B;\frac{1}{l}\sum_{i=1}^k\left(D^i+l C^i\right)\right)
\leq\lct\left(X, B;\frac{1}{l}\sum_{i=1}^k D^i\right). 
\]
This implies that 
\[
\alpha_l\left(X, B;\left\{V_{m\bullet\vec{a}^i}^i\right\}_{i=1}^k\right)
\geq
\alpha_l\left(X, B;\left\{V_{m\bullet\vec{b}^i}^i\right\}_{i=1}^k\right).
\]
Let us set 
\[
N^i:=\dim V_{l\vec{a}^i}^i, \quad M^i:=\dim V_{l\vec{b}^i}^i. 
\]
Take any basis 
\[
\left\{s_1^i,\dots,s_{N^i}^i\right\}\subset V_{l\vec{a}^i}^i
\]
and set 
\[
D_j^i:=\left(s_j^i=0\right)\in\left|V_{l\vec{a}^i}^i\right|, \quad
D^i:=\frac{1}{l N^i}\sum_{j=1}^{N^i}D_j^i. 
\]
Of course, $D^i$ is an $l$-basis type $\Q$-divisor of $V_{m\bullet\vec{a}^i}^i$. 
Let $t_j^i\in V_{l\vec{b}^i}^i$ be the image of $s_j^i$ under the natural inclusion 
\[
V_{l\vec{a}^i}^i\xrightarrow{\cdot l C^i}V_{l\vec{b}^i}^i.
\]
Take $t_{N^i+1}^i,\dots,t_{M^i}^i\in V_{l\vec{b}^i}^i$ such that 
$\left\{t_j^i\right\}_{j=1}^{M^i}$ is a basis of $V_{l\vec{b}^i}^i$, and set 
\[
E_j^i:=\left(t_j^i=0\right)\in\left|V_{l\vec{b}^i}^i\right|, \quad
E^i:=\frac{1}{l M^i}\sum_{j=1}^{M^i}E_j^i.
\]
The $\Q$-divisor $E^i$ is an $l$-basis type $\Q$-divisor of 
$V_{m\bullet\vec{b}^i}^i$. Moreover, for any $1\leq j\leq N^i$, we have 
$E_j^i=D_j^i+l C^i$. Thus we have $M^i E^i\geq N^i D^i$. 
In particular, 
\[
\max_{1\leq i\leq k}\left\{M^i\right\}\sum_{i=1}^k E^i
\geq \min_{1\leq i\leq k}\left\{N^i\right\}\sum_{i=1}^k D^i
\]
holds. This immediately implies that
\[
\lct\left(X, B; \max_{1\leq i\leq k}\left\{M^i\right\}\sum_{i=1}^k E^i\right)
\leq
\lct\left(X, B; \min_{1\leq i\leq k}\left\{N^i\right\}\sum_{i=1}^k D^i\right)
\]
and we get the assertion. 
\end{proof}

Now we state the following generalization of 
Dervan and Kewei Zhang's result 
\cite[Theorem 1.4]{dervan}, \cite[Theorem 1.7]{kewei}. 

\begin{thm}\label{cont_thm}
The functions 
\[
\alpha\colon\sC\cap\Q^{r_1+\cdots+r_k}\to\R_{>0}, \quad
\delta\colon\sC\cap\Q^{r_1+\cdots+r_k}\to\R_{>0}
\]
introduced above can extend uniquely to continuous functions 
\[
\alpha\colon\sC\to\R_{>0}, \quad
\delta\colon\sC\to\R_{>0}, 
\]
respectively. 
\end{thm}

\begin{proof}
The proof is similar to the proof of \cite[Theorem 4.2]{kewei}.
Fix a norm $\|\cdot\|$ on $\R^{r_1+\cdots+r_k}$ and 
take any compact subset $K\subset \R^{r_1+\cdots+r_k}$ 
with $K\subset\sC$ as in Lemma \ref{cone_lemma}. Let us fix 
$\vec{c}\in\sC\cap\Q^{r_1+\cdots+r_k}$ with $K\subset\vec{c}+\sC$. 
By the compactness of $K$, there exists $\delta_1\in\Q_{>0}$ such that 
\[
\left\{\vec{y}\in\R^{r_1+\cdots+r_k}\,\,|\,\,\left\|\vec{y}-\vec{x}\right\|
<\delta_1\right\}\subset\sC
\]
holds for any $\vec{x}\in K$. 
Take any sufficiently small $\varepsilon\in\Q_{>0}$ with $\varepsilon<1/(2n)$,
\[
\left(\frac{1+\varepsilon-\varepsilon^2}{1+\varepsilon+\varepsilon^2}\right)^n
(1+\varepsilon-\varepsilon^2)\geq 1, \quad
\left(\frac{1-\varepsilon+\varepsilon^2}{1-\varepsilon-\varepsilon^2}\right)^n
(1-\varepsilon+\varepsilon^2)\leq 1.
\]

\noindent\underline{\textbf{Step 1}}\\
By Lemma \ref{alpha-compare_lemma}, for any $\vec{a}\in K\cap\Q^{r_1+\cdots+r_k}$, 
we have $\alpha(\vec{a})\leq\alpha(\vec{c})$. 
Moreover, if we take $\delta_2\in\R_{>0}$ as in Lemma \ref{cone_lemma} 
from the $\varepsilon$, then we have 
\[
\frac{1}{1+\varepsilon}\alpha(\vec{a})\leq\alpha\left(\vec{b}\right)
\leq\frac{1}{1-\varepsilon}\alpha(\vec{a})
\]
for any $\vec{a}$, $\vec{b}\in K\cap\Q^{r_1+\cdots+r_k}$ with 
$\|\vec{b}-\vec{a}\|<\delta_2$. 
In particular, we have 
\[
\left|\alpha(\vec{b})-\alpha(\vec{a})\right|<2\alpha(\vec{c})\varepsilon. 
\]
Thus we can extend the function $\alpha$ continuously over $K$, hence over $\sC$. 

\noindent\underline{\textbf{Step 2}}\\
By Step 1 and Proposition \ref{delta-basic_proposition} 
\eqref{delta-basic_proposition1}, 
there exists a positive constant $M$ 
satisfying 
$\delta(\vec{a})\leq M$ for any $\vec{a}\in K\cap\Q^{r_1+\cdots+r_k}$. 
Let us fix such $M$. 
Note that, for any $\vec{a}$, $\vec{b}\in \sC\cap\Q^{r_1+\cdots+r_k}$ with 
$(1+\varepsilon)\vec{a}-\vec{b}$, $\vec{b}-(1-\varepsilon)\vec{a}\in\sC$, 
we have 
\[
\delta(\vec{a}+\varepsilon\vec{b})\leq\delta(\vec{a})\leq
\delta(\vec{a}-\varepsilon\vec{b})
\]
holds. Indeed, by Lemma \ref{kewei_lemma} \eqref{kewei_lemma2}, we have 
\begin{eqnarray*}
\delta(\vec{a}+\varepsilon\vec{b})
&=&\inf_{E/X}\frac{A_{X,B}(E)}{\sum_{i=1}^k 
S\left(V_{\bullet\left(\vec{a}^i+\varepsilon\vec{b}^i\right)}^i;E\right)}
\leq
\inf_{E/X}\frac{A_{X,B}(E)}{\sum_{i=1}^k S\left(V_{\bullet\vec{a}^i}^i;E\right)}\\
&=&\delta(\vec{a})
\leq
\inf_{E/X}\frac{A_{X,B}(E)}{\sum_{i=1}^k 
S\left(V_{\bullet\left(\vec{a}^i-\varepsilon\vec{b}^i\right)}^i;E\right)}
=\delta(\vec{a}-\varepsilon\vec{b}).
\end{eqnarray*}

\noindent\underline{\textbf{Step 3}}\\
Let us set $\delta_0:=\frac{\varepsilon^2\delta_1}{2}$. 
Take any $\vec{a}$, $\vec{b}\in K\cap\Q^{r_1+\cdots+r_k}$ such that 
$\vec{e}:=\vec{b}-\vec{a}$ satisfies that 
$\|\vec{e}\|<\delta_0$. From the definition of $\delta_1$, we have 
\[
\vec{a}+\frac{1+\varepsilon}{\varepsilon}\vec{e}\in\sC, \quad
\vec{a}-\frac{1-\varepsilon}{\varepsilon}\vec{e}\in\sC. 
\]
Note that
\begin{eqnarray*}
\left\|\frac{1+\varepsilon}{\varepsilon^2}\vec{e}\right\|
&<&\delta_1, \quad
(1+\varepsilon)\vec{b}=\vec{a}+\varepsilon
\left(\vec{a}+\frac{1+\varepsilon}{\varepsilon}\vec{e}\right), \\
\left\|\frac{1-\varepsilon}{\varepsilon^2}\vec{e}\right\|
&<&\delta_1, \quad
(1-\varepsilon)\vec{b}=\vec{a}-\varepsilon
\left(\vec{a}-\frac{1-\varepsilon}{\varepsilon}\vec{e}\right). 
\end{eqnarray*}
Then, 
\begin{eqnarray*}
(1+\varepsilon)\vec{a}-\left(\vec{a}+\frac{1+\varepsilon}{\varepsilon}\vec{e}\right)
&=&\varepsilon\left(\vec{a}-\frac{1+\varepsilon}{\varepsilon^2}\vec{e}\right), \\
\left(\vec{a}+\frac{1+\varepsilon}{\varepsilon}\vec{e}\right)-(1-\varepsilon)\vec{a}
&=&\varepsilon\left(\vec{a}+\frac{1+\varepsilon}{\varepsilon^2}\vec{e}\right), \\
(1+\varepsilon)\vec{a}-\left(\vec{a}-\frac{1-\varepsilon}{\varepsilon}\vec{e}\right)
&=&\varepsilon\left(\vec{a}+\frac{1-\varepsilon}{\varepsilon^2}\vec{e}\right), \\
\left(\vec{a}-\frac{1-\varepsilon}{\varepsilon}\vec{e}\right)-(1-\varepsilon)\vec{a}
&=&\varepsilon\left(\vec{a}-\frac{1-\varepsilon}{\varepsilon^2}\vec{e}\right)
\end{eqnarray*}
are elements in $\sC$ from the definition of $\delta_1$. By Step 2, we get 
\[
\delta\left((1+\varepsilon)\vec{b}\right)
=\delta\left(\vec{a}+\varepsilon\left(\vec{a}+\frac{1+\varepsilon}{\varepsilon}\vec{e}
\right)\right)\leq\delta\left(\vec{a}\right)\leq
\delta\left(\vec{a}-\varepsilon\left(\vec{a}-\frac{1-\varepsilon}{\varepsilon}\vec{e}
\right)\right)=\delta\left((1-\varepsilon)\vec{b}\right). 
\]
In other words, we have 
\[
(1-\varepsilon)\delta(\vec{a})
\leq\delta\left(\vec{b}\right)\leq(1+\varepsilon)\delta(\vec{a}). 
\]
Moreover, we have $\delta(\vec{a})\leq M$. 
Therefore we get the following: for any $0<\varepsilon\ll 1$, there exists $\delta_0>0$ 
such that, for any $\vec{a}$, $\vec{b}\in K\cap\Q^{r_1+\cdots+r_k}$ with 
$\|\vec{b}-\vec{a}\|<\delta_0$, we have 
\[
\left|\delta(\vec{b})-\delta(\vec{a})\right|\leq M\varepsilon. 
\]
Thus we get the assertion. 
\end{proof}

We remark that the local version of Theorem \ref{cont_thm} also holds 
by the completely same proof. We only state the result just for readers' convenience. 

\begin{thm}
Let $\eta\in X$ be a scheme-theoretic point which is not the generic point of $X$ 
and assume that $(X, B)$ is klt at $\eta$. 
Let $V_{\vec{\bullet}}^i$ be the Veronese equivalence class of 
a graded linear series on $X$ associated to 
$L_1^i,\dots,L_{r_i}^i\in\CaCl(X)\otimes_\Z\Q$ which contains an ample series 
for any $1\leq i\leq k$. 
Let us set $\sC_i:=\interior\left(\Supp\left(V_{\vec{\bullet}}^i\right)\right)$ and 
$\sC:=\prod_{i=1}^k\sC_i$. 
Then the functions $\alpha_\eta\colon\sC\cap\Q^{r_1+\cdots+r_k}\to\R_{>0}$ and 
$\delta_\eta\colon\sC\cap\Q^{r_1+\cdots+r_k}\to\R_{>0}$ with 
\[
\alpha_\eta(\vec{a}):=\alpha_\eta
\left(X, B; \left\{V_{\bullet\vec{a}^i}^i\right\}_{i=1}^k\right), \quad
\delta_\eta(\vec{a}):=\delta_\eta
\left(X, B; \left\{V_{\bullet\vec{a}^i}^i\right\}_{i=1}^k\right)
\]
uniquely extend to continuous functions 
$\alpha_\eta\colon\sC\to\R_{>0}$ and 
$\delta_\eta\colon\sC\to\R_{>0}$, 
respectively. 
\end{thm}

As an immediate consequence of Theorem \ref{cont_thm}, we have the following 
corollary. Note that the local version of Corollary \ref{cont_corollary} is also true. 
Let $\Biggu(X)\subset N^1(X)$
(resp., $\Biggu(X)_\Q\subset N^1(X)_\Q$) be the set of the numerical classes of big 
$\R$-Cartier $\R$-divisors (resp., $\Q$-Cartier $\Q$-divisors) on $X$. 

\begin{corollary}[{cf.\ \cite[Theorem 1.4]{dervan}, 
\cite[Theorem 1.7]{kewei}}]\label{cont_corollary}
Assume that $(X, B)$ is klt. The functions 
\begin{eqnarray*}
\alpha\colon\Biggu(X)_\Q^k&\to&\R_{>0}\\
(L_1,\dots,L_k)&\mapsto&\alpha\left(X,B;\left\{L_i\right\}_{i=1}^k\right), \\
\delta\colon\Biggu(X)_\Q^k&\to&\R_{>0}\\
(L_1,\dots,L_k)&\mapsto&\delta\left(X,B;\left\{L_i\right\}_{i=1}^k\right), 
\end{eqnarray*}
uniquely extend to continuous functions
\[
\alpha\colon\Biggu(X)^k\to\R_{>0}, \quad
\delta\colon\Biggu(X)^k\to\R_{>0}. 
\] 
\end{corollary}

\begin{proof}
The values 
$\alpha\left(X,B;\left\{L_i\right\}_{i=1}^k\right)$ 
and $\delta\left(X,B;\left\{L_i\right\}_{i=1}^k\right)$ depend only on 
the numerical class of $L_1,\dots,L_k$. See the proof of 
\cite[Lemma 3.7 (iii)]{BJ}. 
Then the assertion is a direct consequence of Theorem \ref{cont_thm}. 
\end{proof}

\begin{remark}\label{cont_remark}
If $L_1,\dots,L_k\in\CaCl(X)\otimes_\Z\Q$, then the values 
\[
\alpha\left(X,B;\left\{c_i\cdot L_i\right\}_{i=1}^k\right), \quad 
\delta\left(X,B;\left\{c_i\cdot L_i\right\}_{i=1}^k\right), 
\]
etc., in Definition \ref{alpha-delta_definition} \eqref{alpha-delta_definition3} 
coincide with the values in Corollary \ref{cont_corollary} 
by Proposition \ref{delta-basic_proposition} 
\eqref{delta-basic_proposition11} and Theorem \ref{cont_thm}. 
\end{remark}

\section{Zhuang's product formula}\label{ziquan_section}

In this section, we assume that the characteristic of $\Bbbk$ is zero. 
We consider the product formula \cite{zhuang} for collections of 
tensor products of graded linear series. 
The proof is almost same as the proof in \cite{zhuang}, but the argument is 
more complicated. 

\begin{thm}[{cf.\ \cite[Theorem 1.2]{zhuang}}]\label{thm_zhuang}
Let $\left(X_1, B_1\right)$ and $\left(X_2, B_2\right)$ be projective klt. 
For any $1\leq i\leq k$, let 
$U_{\vec{\bullet}}^i$ $($resp., $V_{\vec{\bullet}}^i$$)$ be the Veronese equivalence 
class of a graded linear series on $X_1$ $($resp., on $X_2$$)$ associated to 
$L_1^i,\dots,L_{r_i}^i\in\CaCl(X_1)\otimes_\Z\Q$ 
$($resp., $M_1^i,\dots,M_{s_i}^i\in\CaCl(X_2)\otimes_\Z\Q$$)$ which has bounded 
support and contains an ample series. Set 
$\left(X, B\right):=\left(X_1\times X_2, B_1\boxtimes B_2\right)$ and 
$W_{\vec{\bullet}}^i:=U_{\vec{\bullet}}^i\otimes V_{\vec{\bullet}}^i$ 
$($see Definition \ref{tensor_definition}$)$. 
Moreover, take any $c_1,\dots,c_k\in\R_{>0}$. 
Then we have 
\[
\delta\left(X, B;\left\{c_i W_{\vec{\bullet}}^i\right\}_{i=1}^k\right)
=\min\left\{
\delta\left(X_1, B_1;\left\{c_i U_{\vec{\bullet}}^i\right\}_{i=1}^k\right), \quad
\delta\left(X_2, B_2;\left\{c_i V_{\vec{\bullet}}^i\right\}_{i=1}^k\right)
\right\}.
\]
\end{thm}

As an immediate corollary of Theorem \ref{thm_zhuang} and Corollary 
\ref{cont_corollary}, we get the following: 

\begin{corollary}\label{corollary_ziquan}
Let $\left(X_1, B_1\right)$ and $\left(X_2, B_2\right)$ be projective klt. 
Take any $\theta_1,\dots,\theta_k\in\BIG(X_1)$ and $\xi_1,\dots,\xi_k\in\BIG(X_2)$. 
Then we have 
\[
\delta\left(X_1\times X_2, B_1\boxtimes B_2; 
\left\{\theta_i\boxtimes\xi_i\right\}_{i=1}^k\right)
=\min\left\{\delta\left(X_1,B_1, \left\{\theta_i\right\}_{i=1}^k\right), \quad
\delta\left(X_2,B_2, \left\{\xi_i\right\}_{i=1}^k\right)
\right\}.
\]
\end{corollary}

\begin{proof}[Proof of Theorem \ref{thm_zhuang}]
We heavily follow the argument in \cite[\S 3]{zhuang}. 
We firstly remark that 
$c \left(U_{\vec{\bullet}}\otimes V_{\vec{\bullet}}\right)=\left(c U_{\vec{\bullet}}\right)
\otimes\left(c V_{\vec{\bullet}}\right)$ holds as Veronese equivalence classes of 
graded linear series for any $c\in\Q_{>0}$. 
Thus, by Proposition \ref{delta-basic_proposition} 
\eqref{delta-basic_proposition4} and \eqref{delta-basic_proposition11}, 
we may assume that $c_1=\cdots=c_k=1$. 
By Proposition \ref{delta-basic_proposition} \eqref{delta-basic_proposition5}, 
we may assume that 
$U_{\vec{\bullet}}^i$ (resp., $V_{\vec{\bullet}}^i$) are $\Z_{\geq 0}^{r_i}$-graded 
(resp., $\Z_{\geq 0}^{s_i}$-graded) and $L_j^i$ (resp. $M_j^i$) are Cartier divisors. 
Set $\delta:=\delta\left(X, B;\left\{W_{\vec{\bullet}}^i\right\}_{i=1}^k\right)$, 
$\delta_1:=\delta\left(X_1, B_1;\left\{U_{\vec{\bullet}}^i\right\}_{i=1}^k\right)$ and 
$\delta_2:=\delta\left(X_2, B_2;\left\{V_{\vec{\bullet}}^i\right\}_{i=1}^k\right)$. 

We firstly show that $\delta\leq\min\{\delta_1,\delta_2\}$. For any 
$\varepsilon\in\Q_{>0}$, there exists a prime divisor $F_1$ over $X_1$ such that 
\[
\frac{A_{X_1,B_1}(F_1)}{\sum_{i=1}^k S\left(U_{\vec{\bullet}}^i; F_1\right)}
<\delta_1+\varepsilon
\]
holds. Take any resolution $\sigma_1\colon\tilde{X}_1\to X_1$ of 
singularities with $F_1\subset\tilde{X}_1$, and set $\tilde{X}:=\tilde{X}_1\times X_2
\xrightarrow{\sigma}X$ and $E_1:=\pi_1^*F_1\subset\tilde{X}$, where 
$\pi_1\colon\tilde{X}\to \tilde{X}_1$ be the $1$st projection. 
For any $1\leq i\leq k$, $l\in\Z_{>0}$, $\vec{a}\in\Z_{\geq 0}^{r_i-1}$, 
$\vec{b}\in\Z_{\geq 0}^{s_i-1}$ and $\lambda\in\R_{\geq 0}$, we have the equality 
\[
\sF_{E_1}^\lambda W_{l,\vec{a},\vec{b}}^i=\left(\sF_{F_1}^\lambda U_{l,\vec{a}}^i\right)
\otimes V_{l,\vec{b}}^i.
\]
This immediately implies that 
\[
S_l\left(W_{\vec{\bullet}}^i; E_1\right)
=\frac{1}{h^0\left(U_{l,\vec{\bullet}}^i\right)h^0\left(V_{l,\vec{\bullet}}^i\right)}
\int_0^\infty\sum_{\vec{a}\in\Z_{\geq 0}^{r_i-1}}\sum_{\vec{b}\in\Z_{\geq 0}^{s_i-1}}
\dim\sF_{F_1}^{l t}U_{l,\vec{a}}^i\dim V_{l,\vec{b}}^i dt
=S_l\left(U_{\vec{\bullet}}^i; F_1\right).
\]
Thus, we get 
\[
\frac{A_{X,B}(E_1)}{\sum_{i=1}^k S\left(W_{\vec{\bullet}}^i; E_1\right)}
=\frac{A_{X_1,B_1}(F_1)}{\sum_{i=1}^k S\left(U_{\vec{\bullet}}^i; F_1\right)}
<\delta_1+\varepsilon, 
\]
which gives the inequality $\delta\leq \delta_1$. Thus we get the desired inequality
$\delta\leq\min\{\delta_1,\delta_2\}$. 

We show the reverse inequality $\delta\geq\min\{\delta_1,\delta_2\}$. 
Let $\pi_j\colon X\to X_j$ be the $j$th projection. 
Take any 
prime divisor $E$ over $X$ and any 
$c\in\Q_{>0}$ with $c<\min\{\delta_1,\delta_2\}$. It is enough to show the inequality 
\[
A_{X, B}(E)> c\sum_{i=1}^k S_l\left(W_{\vec{\bullet}}^i; E\right)
\] 
for any $l\gg 0$. 
For simplicity, let us set 
$P_{l,\vec{a}}^i:=\dim U_{l,\vec{a}}^i$, 
$Q_{l,\vec{b}}^i:=\dim V_{l,\vec{b}}^i$, 
$P_l^i:=h^0\left(U_{l,\vec{\bullet}}^i\right)$, 
$Q_l^i:=h^0\left(V_{l,\vec{\bullet}}^i\right)$, 
and 
\[
\left\{\vec{c}_1^i,\dots,\vec{c}_{Q_l^i}^i\right\}:=
\left\{\left(\vec{b},k\right)\,\,\Big|\,\,
\vec{b}\in\Z_{\geq 0}^{s_i-1}\text{ with }Q_{l,\vec{b}}^i\neq 0, 
\,\,1\leq k\leq Q_{l,\vec{b}}^i\right\}.
\]
Note that $h^0\left(W_{l,\vec{\bullet}}^i\right)=P_l^i Q_l^i$ holds (see Example 
\ref{interior_example}). 

Let us consider the case $\pi_2\left(C_X(E)\right)=X_2$. 
For any $1\leq i\leq k$, $\vec{a}\in\Z_{\geq 0}^{r_i-1}$ and 
$\vec{b}\in\Z_{\geq 0}^{s_i-1}$, let us consider the basis type filtration $\sG'$ 
of $V_{l,\vec{b}}^i$ associated to general points $x_1,\dots,x_{Q_{l,\vec{b}}^i}\in X_2$ 
of type (I) in the sense of Example \ref{filter-system_example} 
\eqref{filter-system_example2}, and let $\sG$ be the filtration of 
$W_{l,\vec{a},\vec{b}}^i$ defined by $\sG'$, i.e., 
$\sG^\lambda W_{l,\vec{a},\vec{b}}^i:=U_{l,\vec{a}}^i\otimes
{\sG'}^\lambda V_{l,\vec{b}}^i$. Take a basis 
\[
\left\{f_{\vec{a},\vec{b},j,k'}^i\right\}_{\substack{1\leq j\leq P_{l,\vec{a}}^i \\
1\leq k'\leq Q_{l,\vec{b}}^i}}
\]
of $W_{l,\vec{a},\vec{b}}^i$ compatible with $\sF_E$ and $\sG$ such that 
the image of $\{f_{\vec{a},\vec{b},j,k'}^i\}_{1\leq j\leq P_{l,\vec{a}}^i}$ on 
$U_{l,\vec{a}}^i\otimes\Bbbk(x_{k'})$ forms a basis for any 
$\vec{a}\in\Z_{\geq 0}^{r_i-1}$, $\vec{b}\in\Z_{\geq 0}^{s_i-1}$ and 
$1\leq k\leq Q_{l,\vec{b}}^i$. 
Take a general point $x\in X_2$ and let us set 
$X_x:=\pi_2^{-1}(x)\simeq X_1$, $B_x:=B|_{X_x}$. Set 
\[
B_{\vec{a},\vec{b},k'}^i:=\sum_{j=1}^{P_{l,\vec{a}}^i}\left(f_{\vec{a},\vec{b},j,k'}^i=0\right).
\]
Then 
\[
D^i:=\frac{1}{l P_l^i Q_l^i}\sum_{\vec{a}\in\Z_{\geq 0}^{r_i-1}}
\sum_{\left(\vec{b},k'\right)\in\{\vec{c}_1^i,\dots,\vec{c}_{Q_l^i}^i\}}
B_{\vec{a},\vec{b},k'}^i
\]
is an $l$-basis type $\Q$-divisor of $W_{\vec{\bullet}}^i$ with 
$\ord_E D^i=S_l\left(W_{\vec{\bullet}}^i; E\right)$. Since $x\in X_2$ is general, 
for any $1\leq h\leq Q_l^i$, 
\[
D_{x,\vec{c}_h^i}^i:=\frac{1}{l P_l^i}\sum_{\vec{a}\in\Z_{\geq 0}^{r_i-1}}
\left(B_{\vec{a},\vec{c}_h^i}^i\right)|_{X_x}
\]
is an $l$-basis type $\Q$-divisor of $U_{\vec{\bullet}}^i$ on $X_x\simeq X_1$. 
Note that 
\[
\sum_{i=1}^k D^i|_{X_x}=
\frac{1}{Q_l^1\dots Q_l^k}\sum_{1\leq h_1\leq Q_l^1}
\cdots\sum_{1\leq h_k\leq Q_l^k
}
\left(D_{x,\vec{c}_{h_1}^1}^1+\cdots+D_{x,\vec{c}_{h_k}^k}^k\right)
\]
and the pair 
\[
\left(X_x, B_x+c\sum_{i=1}^k D_{x,\vec{c}_{h_i}^i}^i\right)
\]
is klt for any $l\gg 0$ and any $h_1,\dots,h_k$, since $c<\delta_1$. 
This implies that the pair 
\[
\left(X_x, B_x+c\sum_{i=1}^kD^i|_{X_x}\right)
\] 
is also klt. By inversion of adjunction, 
the pair $\left(X, B+ c\sum_{i=1}^k D^i\right)$ is klt around a neighborhood of 
$X_x$. Therefore we get the desired inequality
\[
A_{X, B}(E)> c\sum_{i=1}^k \ord_E D^i
=c\sum_{i=1}^k S_l\left(W_{\vec{\bullet}}^i; E\right).
\] 

Let us consider the remaining case $\pi_2\left(C_X(E)\right)\subsetneq X_2$. 
Take a resolution $\sigma_2\colon\tilde{X}_2\to X_2$ of singularities and a 
prime divisor $F_2\subset\tilde{X}_2$ such that the restriction 
$\ord_E|_{\Bbbk(X_2)}$ to 
the function field $\Bbbk(X_2)$ of $X_2$ is proportional to $\ord_{F_2}$. Set 
$\tilde{X}:=X_1\times\tilde{X}_2$, $E_2:=\pi_2^*(F_2)\subset\tilde{X}$ and 
$\sigma\colon\tilde{X}\to X$. 
For any $1\leq i\leq k$, $\vec{a}\in\Z_{\geq 0}^{r_i-1}$ and 
$\vec{b}\in\Z_{\geq 0}^{s_i-1}$, let us consider the basis type filtration $\sG'$ 
of $V_{l,\vec{b}}^i$ associated to general points $x_1,\dots,x_{Q_{l,\vec{b}}^i}\in F_2
\subset\tilde{X}_2$ 
of type (II) in the sense of Example \ref{filter-system_example} 
\eqref{filter-system_example2}, and let $\sG$ be the filtration of 
$W_{l,\vec{a},\vec{b}}^i$ defined by $\sG'$. 
Note that $\sG$ refines $\sF_{E_2}$. 
Take a basis 
\[
\left\{f_{\vec{a},\vec{b},j,k'}^i\right\}_{\substack{1\leq j\leq P_{l,\vec{a}}^i \\
1\leq k'\leq Q_{l,\vec{b}}^i}}
\]
of $W_{l,\vec{a},\vec{b}}^i$ compatible with $\sF_E$ and $\sG$ such that, 
for any 
$\vec{a}\in\Z_{\geq 0}^{r_i-1}$, $\vec{b}\in\Z_{\geq 0}^{s_i-1}$ and 
$1\leq k'\leq Q_{l,\vec{b}}^i$, there exists $m\in\Z_{\geq 0}$ such that 
$\ord_{E_2}\left(f_{\vec{a},\vec{b},j,k'}^i\right)=m$ for any $1\leq j\leq P_{l,\vec{a}}^i$ 
and the image of 
$\{\pi_2^*f^{-m}\sigma^*f_{\vec{a},\vec{b},j,k'}^i\}_{1\leq j\leq P_{l,\vec{a}}^i}$ on 
$U_{l,\vec{a}}^i\otimes\Bbbk(x_{k'})$ forms a basis, where 
$f\in H^0\left(\tilde{X}_2,\sO_{\tilde{X}_2}(F_2)\right)$ is the defining equation of 
$F_2\subset\tilde{X}_2$. 
Take a general point $x\in F_2\subset\tilde{X}_2$ and set 
\[
K_{\tilde{X}}+\tilde{B}+\left(1-A_{X_2,B_2}(F_2)\right)E_2=\sigma^*(K_X+B),
\]
$\tilde{X}_x:=\pi_2^{-1}(x)$, and 
$B_x:=\tilde{B}|_{X_x}$. 
Set 
\[
B_{\vec{a},\vec{b},k'}^i:=\sum_{j=1}^{P_{l,\vec{a}}^i}\left(f_{\vec{a},\vec{b},j,k'}^i=0\right).
\]
Then 
\[
D^i:=\frac{1}{l P_l^i Q_l^i}\sum_{\vec{a}\in\Z_{\geq 0}^{r_i-1}}
\sum_{\left(\vec{b},k'\right)\in\{\vec{c}_1^i,\dots,\vec{c}_{Q_l^i}^i\}}
B_{\vec{a},\vec{b},k'}^i
\]
is an $l$-basis type $\Q$-divisor of $W_{\vec{\bullet}}^i$ with 
$\ord_E D^i=S_l\left(W_{\vec{\bullet}}^i; E\right)$ and  
$\ord_{E_2} D^i=S_l\left(W_{\vec{\bullet}}^i; E_2\right)=
S_l\left(V_{\vec{\bullet}}^i; F_2\right)$. 
Write 
\[
\sigma^*D^i=S_l\left(V_{\vec{\bullet}}^i; F_2\right)E_2+
\frac{1}{l P_l^i Q_l^i}\sum_{\vec{a\in\Z_{\geq 0}^{r_i-1}}}\sum_{h=1}^{Q_l^i}
{B'}_{\vec{a},\vec{c}_h^i}^i, 
\]
where $\sigma^*B^i_{\vec{a},\vec{c}_h^i}$ and ${B'}_{\vec{a},\vec{c}_h^i}^i$ may only 
differ along $E_2$. Since $x\in F_2$ is general, for any $1\leq h\leq Q_l^i$, 
\[
D_{x,\vec{c}_h^i}^i:=\frac{1}{l P_l^i}\sum_{\vec{a}\in\Z_{\geq 0}^{r_i-1}}
\left({B'}_{\vec{a},\vec{c}_h^i}^i\right)|_{\tilde{X}_x}
\]
is an $l$-basis type $\Q$-divisor of $U_{\vec{\bullet}}^i$ on $\tilde{X}_x$. 
Since $c<\delta_1$, for any $l\gg 0$ and for any $h_1,\dots,h_k$, the pair 
\[
\left(\tilde{X}, \tilde{B}_x+c\sum_{i=1}^kD^i_{x,\vec{c}_{h_i}^i}\right)
\]
is klt. Same as the previous argument, the pair 
\[
\left(\tilde{X}_x,\tilde{B}_x+c\sum_{i=1}^k\frac{1}{l P_l^i Q_l^i}
\sum_{\vec{a}\in\Z_{\geq 0}^{r_i-1}}\sum_{h=1}^{Q_l^i}
\left({B'}_{\vec{a},\vec{c}_h^i}^i\right)|_{\tilde{X}_x}\right)
\]
is also klt. By inversion of adjunction, the pair 
\[
\left(\tilde{X},\tilde{B}+E_2+c\sum_{i=1}^k\frac{1}{l P_l^i Q_l^i}
\sum_{\vec{a}\in\Z_{\geq 0}^{r_i-1}}\sum_{h=1}^{Q_l^i}
{B'}_{\vec{a},\vec{c}_h^i}^i\right)
\]
is plt around a neighborhood of $\tilde{X}_x$. For $l\gg 0$, we know that  
\[
1-A_{X_2,B_2}(F_2)+c\sum_{i=1}^k S_l\left(V_{\vec{\bullet}}^i;F_2\right)<1
\]
since $c<\delta_2$. This implies that the pair 
\[
\left(\tilde{X},\tilde{B}+\left(1-A_{X_2,B_2}(F_2)\right)E_2+c\sum_{i=1}^k
\sigma^*D^i\right)
\]
is sub-klt around a neighborhood of $\tilde{X}_x$. 
This gives the desired inequality 
\[
A_{X, B}(E)> c\sum_{i=1}^k \ord_E D^i
=c\sum_{i=1}^k S_l\left(W_{\vec{\bullet}}^i; E\right)
\] 
and then we get the assertion. 
\end{proof}

\section{Toward Abban--Zhuang's methods}\label{AZ_section}

In this section, we assume that the characteristic of $\Bbbk$ is zero. 
Let $X$ be an $n$-dimensional projective variety, 
let $B$ be an effective $\Q$-Weil 
divisor on $X$ and let $\eta\in X$ be a scheme-theoretic point such that 
$(X, B)$ is klt at $\eta$. We set $Z:=\overline{\{\eta\}}\subset X$. 
Take any $c_1,\dots,c_k\in\R_{>0}$. 
For any $1\leq i\leq k$, let $V_{\vec{\bullet}}^i$ be 
the Veronese equivalence class of an $(m\Z_{\geq 0})^{r_i}$-graded linear series 
$V_{m\vec{\bullet}}^i$ on $X$ associated to 
$L_1^i,\dots,L_{r_i}^i\in\CaCl(X)\otimes_\Z\Q$ which has bounded support and 
contains an ample series.

We recall the notion introduced in \cite[Definition 11.10]{r3d28}. 

\begin{definition}\label{delta-ref_definition}
Let $\sigma\colon X'\to X$ be a projective birational morphism with $X'$ normal, 
let $Y\subset X'$ be a prime $\Q$-Cartier divisor on $X'$ and let $e\in\Z_{>0}$ 
with $e Y$ Cartier. 
For any $l\in m\Z_{>0}$ with 
$\prod_{i=1}^k h^0\left(V_{l,m\vec{\bullet}}^{i, (Y,e)}\right)\neq 0$, we set 
\[
\delta_{\eta,l}^{(Y,e)}\left(X,B; \left\{c_i\cdot V_{m\vec{\bullet}}^i\right\}_{i=1}^k\right)
:=\inf_{\substack{{D'}^i\text{ $l$-$(Y,e)$-subbasis type} \\
\text{$\Q$-divisor of }V_{m\vec{\bullet}}^i\\ \text{for all }1\leq i\leq k}}
\lct_\eta\left(X, B; \sum_{i=1}^k c_i{D'}^i\right).
\]
\end{definition}

The proof of the following proposition is essentially same as the proof of 
Proposition \ref{alpha-delta_proposition}. More precisely, 
we apply Lemma \ref{bt-div_lemma} 
\eqref{bt-div_lemma2}. We omit the proof. See \cite[Proposition 11.13 (1)]{r3d28} 
in detail. 

\begin{proposition}[{\cite[Proposition 11.13 (1)]{r3d28}}]\label{delta-ref_proposition}
We have 
\[
\lim_{l\in m\Z_{>0}}
\delta_{\eta,l}^{(Y,e)}\left(X,B; \left\{c_i\cdot V_{m\vec{\bullet}}^i\right\}_{i=1}^k\right)
=\delta_\eta\left(X, B;\left\{c_i\cdot V_{\vec{\bullet}}^i\right\}_{i=1}^k\right).
\]
\end{proposition}

Here is an analogue of \cite[Theorem 3.2]{AZ}. We omit the proof, since 
the proof is essentially same as the proof of \cite[Theorem 11.14]{r3d28}
and applying Propositions \ref{delta-basic_proposition} 
\eqref{delta-basic_proposition5} and \ref{delta-ref_proposition}. 

\begin{thm}[{cf.\ \cite[Theorem 3.2]{AZ} 
and \cite[Theorem 11.14]{r3d28}}]\label{AZ_thm}
Let $Y$ be a primitive prime divisor over $X$ and let $\sigma\colon\tilde{X}\to X$ 
be the associated prime blowup. Assume that there exists an open 
subscheme $\eta\in U\subset X$ such that $Y$ is a plt-type prime divisor 
over $(U, B|_U)$. Let $(Y, B_Y)$ be the associated klt pair over $U$ (see Definition 
\ref{primitive_definition} \eqref{primitive_definition3}). Let 
$Z_0\subset Z\subset X$ be a closed subvariety with $Z_0\subset C_X(Y)$ and 
$Z_0\cap U\neq\emptyset$. Let $\eta_0\in X$ be the generic point of $Z_0$. 
\begin{enumerate}
\renewcommand{\theenumi}{\arabic{enumi}}
\renewcommand{\labelenumi}{(\theenumi)}
\item\label{AZ_thm1}
If $\eta\not\in C_X(Y)$, then we have 
\[
\delta_\eta\left(X,B;\left\{c_i V_{\vec{\bullet}}^i\right\}_{i=1}^k\right)\geq
\inf_{\eta'\in\tilde{X}; \sigma(\eta')=\eta_0}\delta_{\eta'}\left(Y, B_Y; 
\left\{c_i V_{\vec{\bullet}}^{i,(Y)}\right\}_{i=1}^k\right).
\]
\item\label{AZ_thm2}
If $\eta\in C_X(Y)$, then we have 
\[
\delta_\eta\left(X,B;\left\{c_i V_{\vec{\bullet}}^i\right\}_{i=1}^k\right)\geq
\min\left\{\frac{A_{X,B}(Y)}{\sum_{i=1}^kc_i S\left(V_{\vec{\bullet}}^i;Y\right)}, \quad
\inf_{\eta'\in\tilde{X}; \sigma(\eta')=\eta_0}\delta_{\eta'}\left(Y, B_Y; 
\left\{c_i V_{\vec{\bullet}}^{i,(Y)}\right\}_{i=1}^k\right)\right\}.
\]
If moreover the equality holds and there exists a prime divisor $E$ over $X$ with 
$Z\subset C_X(E)$, $C_{\tilde{X}}(E)\subset Y$ and 
\[
\delta_\eta\left(X,B;\left\{c_i V_{\vec{\bullet}}^i\right\}_{i=1}^k\right)
=\frac{A_{X,B}(E)}{\sum_{i=1}^k c_i S\left(V_{\vec{\bullet}}^i; E\right)}, 
\]
then the equality 
\[
\delta_\eta\left(X,B;\left\{c_i V_{\vec{\bullet}}^i\right\}_{i=1}^k\right)
=\frac{A_{X,B}(Y)}{\sum_{i=1}^k c_i S\left(V_{\vec{\bullet}}^i; Y\right)}, 
\]
holds. 
\end{enumerate}
\end{thm}

Assume that there exists a finite set $\Lambda_i$ and a decomposition 
\[
\Delta_{\Supp\left(V_{\vec{\bullet}}^{i,(Y)}\right)}=\overline{\bigcup_{\lambda\in\Lambda_i}
\Delta_{\Supp}^{i,\langle\lambda\rangle}}
\]
is given for any $1\leq i\leq k$. 
We consider $V_{\vec{\bullet}}^{i,(Y),\langle\lambda\rangle}$ in the sense of 
Definition \ref{interior_definition} \eqref{interior_definition4}. 
By Proposition \ref{delta-basic_proposition} \eqref{delta-basic_proposition9} and 
\eqref{delta-basic_proposition10}, we have 
\begin{eqnarray*}
\delta_{\eta'}\left(Y, B_Y; 
\left\{c_i V_{\vec{\bullet}}^{i,(Y)}\right\}_{i=1}^k\right)
&=&\delta_{\eta'}\left(Y, B_Y; 
\left\{c_i \frac{\vol\left(V_{\vec{\bullet}}^{i,(Y),\langle\lambda\rangle}\right)}{\vol
\left(V_{\vec{\bullet}}^{i,(Y)\rangle}\right)}
V_{\vec{\bullet}}^{i,(Y),\langle\lambda\rangle}\right\}_{1\leq i\leq k, \lambda\in\Lambda_i}\right)\\
&\geq&
\left(
\sum_{i=1}^k\sum_{\lambda\in\Lambda_i}
c_i \frac{\vol\left(V_{\vec{\bullet}}^{i,(Y),\langle\lambda\rangle}\right)}{\vol
\left(V_{\vec{\bullet}}^{i,(Y)}\right)}
\delta_{\eta'}\left(Y, B_Y; 
V_{\vec{\bullet}}^{i,(Y),\langle\lambda\rangle}\right)^{-1}
\right)^{-1}.
\end{eqnarray*}
Moreover, by Theorem \ref{adequate_thm} and 
Corollary \ref{adequate-divide_corollary}, we can estimate the values 
$\delta_{\eta'}\left(Y, B_Y; 
V_{\vec{\bullet}}^{i,(Y),\langle\lambda\rangle}\right)$, hence also 
the value $\delta_\eta\left(X, B;\left\{c_i V_{\vec{\bullet}}^i\right\}_{i=1}^k\right)$, 
in many situations. 

We end the article by seeing basic examples. 

\begin{example}[{cf.\ \cite[Corollary 2.17]{AZ}}]\label{curve_example}
Assume that $n=1$ and $\eta$ is a closed point. Set $b:=\ord_\eta B\in\Q\cap
 [0, 1)$. Consider $\R$-Cartier $\R$-divisors $L_1,\dots,L_k$ on $X$ with 
$\deg L_i=d_i\in\R_{>0}$. For any Cartier divisor $L$ on $X$ with $\deg L=1$, we 
know that 
\[
\delta_\eta\left(X,B; L\right)=\frac{1-b}{1/2}=2(1-b).
\]
Thus, by Proposition \ref{delta-basic_proposition}, we have
\[
\delta_\eta\left(X,B; \left\{c_iL_i\right\}_{i=1}^k\right)
=\frac{2(1-b)}{\sum_{i=1}^kc_i d_i}.
\]
\end{example}

\begin{example}[{cf.\ \cite[Corollary A.14]{RTZ}}]\label{hirzebruch_example}
Assume that $X=\pr_{\pr^1}\left(\sO\oplus\sO(m)\right)$ with $m\in\Z_{\geq 0}$ 
and $B=0$. Let $F$, $E\in\CaCl(X)$ be the class of a fiber of $X/\pr^1$, 
$(-m)$-curve, respectively. For any $1\leq i\leq k$, let us consider any  big 
$L_i:=a_i E+b_i F\in\CaCl(X)\otimes_\Z\Q$, i.e., $a_i>0$ and $b_i>0$. 
We compute the value $\delta\left(X;\left\{L_i\right\}_{i=1}^k\right)$. 
If $m=0$, i.e., if $X=\pr^1\times\pr^1$, then we have 
\begin{eqnarray*}
\delta\left(\pr^1\times\pr^1; \left\{a_i E+b_i F\right\}_{i=1}^k\right)
&=&\min\left\{\delta\left(\pr^1;\left\{a_i\cdot\sO(1)\right\}_{i=1}^k\right),
\quad\delta\left(\pr^1;\left\{b_i\cdot\sO(1)\right\}_{i=1}^k\right)\right\}\\
&=&\min\left\{\frac{2}{\sum_{i=1}^k a_i}, \quad\frac{2}{\sum_{i=1}^k b_i}\right\}
\end{eqnarray*}
by Corollary \ref{corollary_ziquan} and Proposition \ref{delta-basic_proposition} 
\eqref{delta-basic_proposition8}. From now on, assume that $m\geq 1$. 
For any $1\leq i\leq k$, let us set 
\begin{eqnarray*}
p_i:=\begin{cases}
a_i-\frac{b_i}{3m} & \text{if }m a_i\geq b_i, \\
\frac{a_i(3b_i-m a_i)}{3(2b_i-m a_i)} & \text{if }m a_i<b_i, 
\end{cases}\quad\quad
q_i:=\begin{cases}
\frac{b_i}{3} & \text{if }m a_i\geq b_i, \\
\frac{3b_i^2-3m a_i b_i+m^2 a_i^2}{3(2b_i-m a_i)} & \text{if }m a_i<b_i. 
\end{cases}
\end{eqnarray*}
Then we have 
\begin{eqnarray*}
p_i&=&S\left(L_i; E\right)=S\left(L_i; F'\triangleright F'\cap E\right), \\
q_i&=&S\left(L_i; F'\right)=S\left(L_i; E_\infty\triangleright F'\cap E_\infty\right)
=S\left(L_i; E\triangleright F'\cap E\right)
\end{eqnarray*}
for any $F'\in |F|$ and for any irreducible $E_\infty\in|E+m F|$ by 
Theorem \ref{toric-okounkov_thm} or Corollary \ref{ad-surface_corollary}. 
Thus we get the equality 
\[
\delta\left(X;\left\{L_i\right\}_{i=1}^k\right)=\min\left\{\frac{1}{\sum_{i=1}^k p_i}, 
\quad\frac{1}{\sum_{i=1}^k q_i}\right\}
\]
by Theorem \ref{AZ_thm}. 
\end{example}

\end{document}